\documentclass[11pt,reqno]{amsart}

\usepackage[a4paper,margin=1.08in]{geometry}
\usepackage{amsmath,amssymb,mathtools,mathrsfs}
\usepackage{amsthm}
\usepackage{bm}
\usepackage{enumitem}
\usepackage{array}
\usepackage{booktabs,tabularx}
\usepackage{microtype}
\usepackage{lmodern}
\usepackage{hyperref}
\usepackage{aliascnt}
\usepackage[nameinlink,capitalise,noabbrev]{cleveref}

\hypersetup{
  hidelinks,
  pdftitle={A Singular Stochastic Wave--Klein--Gordon System with Rough Gaussian Forcing and Distinct Propagation Speeds},
  pdfauthor={Guangqian Zhao},
  pdfsubject={Singular stochastic dispersive equations on the full space},
  pdfkeywords={stochastic wave equation; Klein--Gordon equation; white-in-time Gaussian forcing; Fourier-multiplier covariance; distinct speeds; paracontrolled distributions; Gaussian random operators; compact localization; spectral approximation convergence}
}

\numberwithin{equation}{section}

\newtheorem{theorem}{Theorem}[section]

\newaliascnt{proposition}{theorem}
\newtheorem{proposition}[proposition]{Proposition}
\aliascntresetthe{proposition}

\newaliascnt{lemma}{theorem}
\newtheorem{lemma}[lemma]{Lemma}
\aliascntresetthe{lemma}

\newaliascnt{corollary}{theorem}
\newtheorem{corollary}[corollary]{Corollary}
\aliascntresetthe{corollary}

\theoremstyle{definition}
\newaliascnt{assumption}{theorem}
\newtheorem{assumption}[assumption]{Assumption}
\aliascntresetthe{assumption}

\newaliascnt{definition}{theorem}
\newtheorem{definition}[definition]{Definition}
\aliascntresetthe{definition}

\theoremstyle{remark}
\newaliascnt{remark}{theorem}
\newtheorem{remark}[remark]{Remark}
\aliascntresetthe{remark}

\newaliascnt{example}{theorem}

\aliascntresetthe{example}

\crefname{theorem}{Theorem}{Theorems}
\crefname{proposition}{Proposition}{Propositions}
\crefname{lemma}{Lemma}{Lemmas}
\crefname{corollary}{Corollary}{Corollaries}
\crefname{assumption}{Assumption}{Assumptions}
\crefname{definition}{Definition}{Definitions}
\crefname{remark}{Remark}{Remarks}
\crefname{example}{Example}{Examples}
\Crefname{theorem}{Theorem}{Theorems}
\Crefname{proposition}{Proposition}{Propositions}
\Crefname{lemma}{Lemma}{Lemmas}
\Crefname{corollary}{Corollary}{Corollaries}
\Crefname{assumption}{Assumption}{Assumptions}
\Crefname{definition}{Definition}{Definitions}
\Crefname{remark}{Remark}{Remarks}
\Crefname{example}{Example}{Examples}

\newcommand{\R}{\mathbb R}
\newcommand{\N}{\mathbb N}
\newcommand{\E}{\mathbb E}
\newcommand{\Pp}{\mathbb P}
\newcommand{\cC}{\mathcal C}
\newcommand{\cD}{\mathcal D}
\newcommand{\cB}{\mathcal B}

\newcommand{\cZ}{\mathcal Z}

\newcommand{\Dyd}{2^{\mathbb N_0}}
\newcommand{\la}{\langle}
\newcommand{\ra}{\rangle}
\newcommand{\eps}{\varepsilon}
\newcommand{\dd}{\,\mathrm d}
\newcommand{\ii}{\mathrm i}
\newcommand{\wh}{\widehat}
\newcommand{\wt}{\widetilde}
\newcommand{\supp}{\operatorname{supp}}
\newcommand{\dist}{\operatorname{dist}}

\newcommand{\loc}{\mathrm{loc}}
\newcommand{\W}{\mathrm W}
\newcommand{\K}{\mathrm K}
\newcommand{\one}{\mathbf 1}
\newcommand{\norm}[1]{\left\lVert #1\right\rVert}
\newcommand{\abs}[1]{\left\lvert #1\right\rvert}
\newcommand{\set}[1]{\left\{#1\right\}}
\newcommand{\paren}[1]{\left(#1\right)}

\newcommand{\Btwo}[1]{B^{#1}_{2,\infty}}

\newcommand{\LW}{L_{\W}}
\newcommand{\LK}{L_{\K}}
\newcommand{\wW}{\omega_{\W}}
\newcommand{\wK}{\omega_{\K}}

\newcommand{\speedW}{c_{\W}}
\newcommand{\speedK}{c_{\K}}
\newcommand{\massK}{m}
\newcommand{\speedmax}{c_*}
\newcommand{\speedgap}{\delta_{\mathrm{sp}}}
\newcommand{\parvec}{\mathbf p}
\newcommand{\betaW}{\beta_{\W}}
\newcommand{\betaK}{\beta_{\K}}
\newcommand{\betastar}{\beta_*}
\newcommand{\betasum}{\beta_{\Sigma}}
\newcommand{\betagamma}{\beta_{\Gamma}}
\newcommand{\Sch}{\mathfrak S}
\newcommand{\Sym}{\operatorname{Sym}}

\newcolumntype{Y}{>{\raggedright\arraybackslash}X}
\newcolumntype{P}[1]{>{\raggedright\arraybackslash}p{#1}}

\title[A stochastic wave--Klein--Gordon system with rough forcing]
{A Singular Stochastic Wave--Klein--Gordon System with Rough Gaussian Forcing and Distinct Propagation Speeds}

\author{Guangqian Zhao}
\address{School of Mathematical Sciences, University of Science and Technology of China, Hefei, Anhui 230026, China}
\email{zhaoguangqian@mail.ustc.edu.cn}
\date{}
\subjclass[2020]{60H15, 35L71, 35R60, 60H30, 35Q40}
\keywords{Stochastic wave equation, Klein--Gordon equation, white-in-time Gaussian forcing, Fourier-multiplier covariance, distinct propagation speeds, paracontrolled distributions, centered Gaussian operators, spectral approximation}

\allowdisplaybreaks[2]
\begin{document}
\begin{abstract}
We construct spatially local paracontrolled solutions on $\mathbb{R}^3$ for the
stochastic wave--Klein--Gordon system
\[
(\partial_t^2-c_{\mathrm{W}}^2\Delta)u=uv+\xi_{\mathrm{W}},\qquad
(\partial_t^2-c_{\mathrm{K}}^2\Delta+m^2)v=uv+\xi_{\mathrm{K}},
\]
with $c_{\mathrm{W}}\ne c_{\mathrm{K}}$.  The independent forcings are white in time and
stationary in space, generated by Fourier multipliers of respective orders
$\beta_{\mathrm{W}},\beta_{\mathrm{K}}$, and we assume
$\beta_*:=\max\{\beta_{\mathrm{W}},\beta_{\mathrm{K}}\}<1/8$.

The main analytic step is the construction of the mixed random operators
$I_a(w\prec\Psi_b)\circ\Psi_c$.  At finite cutoff they decompose into a
centered second Gaussian chaos and, when $b=c$, a deterministic covariance
contraction.  In the latter case the two propagation channels are different;
the resulting low--high phase gap gives the diagonal shell bound
$N^{-1+2\beta_b}$.  The centered terms are controlled through four Schatten
flattenings of a continuous-frequency kernel with two physical localizers.
Weighted phase-layer estimates also construct the first Picard and cubic
stochastic terms.

A localized auxiliary system, a finite-cutoff Sobolev bootstrap, and finite
propagation yield a compatible family of local solutions with
compact-dependent measurable lifetimes.  We prove uniqueness in the stated
paracontrolled class, local Lipschitz dependence on the enhanced data and
Cauchy data, independence of the localization pair, and convergence of the
spectral approximations and nonlinear sources.  Fixed cutoff profiles converge
almost surely along the full sequence; arbitrary admissible cofinal cutoffs
converge in probability up to the reference lifetime.
\end{abstract}

\maketitle

\section{Introduction}

Consider the coupled system
\begin{equation}\label{eq:intro-system}
\begin{aligned}
  (\partial_t^2-c_{\W}^2\Delta)u&=uv+\xi_{\W},\\
  (\partial_t^2-c_{\K}^2\Delta+m^2)v&=uv+\xi_{\K},
\end{aligned}
\qquad (t,x)\in[0,T]\times\R^3,
\end{equation}
where $c_{\W},c_{\K},m>0$ and $c_{\W}\ne c_{\K}$.

\begin{definition}[Channel and index convention]\label{def:channel-convention}
Set
\begin{equation}\label{eq:channel-set}
  \mathfrak C:=\{\W,\K\},
  \qquad \W^\perp:=\K,
  \qquad \K^\perp:=\W.
\end{equation}
The letters $a,b,c$ are reserved for elements of $\mathfrak C$; the symbols
$\W$ and $\K$ are the two values of a channel index.  Repeated channel
indices are not summed unless a summation sign is displayed, and
$\delta_{ab}$ denotes the Kronecker symbol on $\mathfrak C$.

The four index triples used below are
\begin{equation}\label{eq:operator-index-set}
\begin{aligned}
 \mathfrak L_{\mathrm{diag}}
 &:=\{(\W;\K,\K),(\K;\W,\W)\},\\
 \mathfrak L_{\mathrm{off}}
 &:=\{(\W;\W,\K),(\K;\K,\W)\},\\
 \mathfrak L
 &:=\mathfrak L_{\mathrm{diag}}\,\dot\cup\,
    \mathfrak L_{\mathrm{off}}.
\end{aligned}
\end{equation}
For a label $(a;b,c)\in\mathfrak L$, the first entry $a$ is the Duhamel
channel, $b$ is the color of the stochastic factor in the inner
low--high paraproduct, and $c$ is the color of the stochastic factor in the
outer resonant product.  Thus $b=c$ on $\mathfrak L_{\mathrm{diag}}$, whereas
$b\ne c$ on $\mathfrak L_{\mathrm{off}}$.

The variables $q,\ell,r,n$ denote spatial frequencies and
$N,Q,R,L,M$ denote dyadic scales.  Within the abstract Gaussian operator argument, $\mu,\nu$ index orthonormal coordinates and
$\varkappa\in\{1,2,3,4\}$ indexes the four flattenings.
\end{definition}

The four labels and the contraction mechanism are summarized by
\begin{equation}\label{eq:four-label-table}
\begin{array}{c|c|c|c}
 (a;b,c)&\text{Duhamel channel}&\text{Gaussian colors}&\text{zeroth-chaos contraction}\\
\hline
 (\W;\K,\K)&\W&(\K,\K)&\mathcal D^{\W;\K}\\
 (\K;\W,\W)&\K&(\W,\W)&\mathcal D^{\K;\W}\\
 (\W;\W,\K)&\W&(\W,\K)&0\\
 (\K;\K,\W)&\K&(\K,\W)&0
\end{array}
\end{equation}
The first two rows have equal Gaussian colors and different propagation
channels; the last two rows have equal inner and Duhamel channels but distinct
Gaussian colors.  Thus no row is simultaneously same-color and same-channel.

The symbols $\mathbf W_{\W},\mathbf W_{\K}$ denote the two independent
real isonormal Gaussian processes over $L^2(\R_+\times\R^3)$.  For
$a\in\mathfrak C$, choose a real even symbol
$\mathfrak h_a\in S^{\beta_a}_{1,0}(\R^3)$ and define
\[
 \xi_a(\varphi):=\mathbf W_a(\mathfrak h_a(D_x)\varphi).
\]
We assume
\[
 \betastar:=\max\{\betaW,\betaK\}<\frac18,
 \qquad
 \betasum:=\betaW+\betaK,
 \qquad
 \betagamma:=\betasum+\betastar.
\]
The covariance is white in time and Fourier diagonal in space:
\[
 \E[\widehat\xi_a(\dd\xi,\dd t)\widehat\xi_b(\dd\eta,\dd s)]
 =\delta_{ab}(2\pi)^3\delta(t-s)\delta(\xi+\eta)
 |\mathfrak h_a(\xi)|^2\dd\xi\dd\eta\dd t\dd s.
\]
This class contains space--time white noise, unequal nonradial colorings, and
lower-order profiles.  The Fourier diagonal and the independence of the
two colors determine the Wick contractions.  The symbols
$\mathfrak h_a$ enter the stochastic estimates through amplitudes; the
dispersive phases depend only on the wave and Klein--Gordon propagators.
Since the color-$a$ stochastic convolution has local regularity
$-1/2-\beta_a-$, the product in \eqref{eq:intro-system} requires a singular
construction.

Write
\[
  \omega_{\W}(\xi)=c_{\W}|\xi|,
  \qquad
  \omega_{\K}(\xi)=\sqrt{m^2+c_{\K}^2|\xi|^2},
\]
and denote by $I_a$ the retarded propagator in channel $a\in\mathfrak C$.
With the convention of Definition~\ref{def:channel-convention}, the four operators are
\begin{equation}\label{eq:intro-four-operators}
\begin{aligned}
 \{T^{a;b,c}:(a;b,c)\in\mathfrak L\}
 &=\{T^{\W;\K,\K},T^{\K;\W,\W},
       T^{\W;\W,\K},T^{\K;\K,\W}\},\\
 T^{a;b,c}(w)&=I_a(w\prec\Psi_b)\circ\Psi_c.
\end{aligned}
\end{equation}
At finite cutoff, for $(a;b,c)\in\mathfrak L$, Wick's formula gives
\begin{equation}\label{eq:intro-wick-split}
 T^{a;b,c}_{\Lambda}=
 \begin{cases}
  \mathcal D^{a;b}_{\Lambda}+\mathcal B^{a;b,b}_{\Lambda},
    &(a;b,c)\in\mathfrak L_{\mathrm{diag}},\\
  \mathcal B^{a;b,c}_{\Lambda},
    &(a;b,c)\in\mathfrak L_{\mathrm{off}}
 \end{cases}
\end{equation}
where every $\mathcal B^{a;b,c}_{\Lambda}$ is centered.  On
$\mathfrak L_{\mathrm{diag}}$ one has $b=c$ and $a=b^\perp$.  On the
low--high support of the paraproduct, and
above a finite frequency threshold,
\begin{equation}\label{eq:intro-phase-gap}
  |\omega_a(\ell+q)-\omega_b(\ell)|\gtrsim |\ell|,
  \qquad a\ne b,
  \qquad |q|\ll|\ell|.
\end{equation}
A Volterra integration by parts supplies one inverse power of the high
frequency.  Since the color-$b$ covariance contributes
$|\mathfrak h_b(\ell)|^2\lesssim N^{2\beta_b}$ on $|\ell|\sim N$, the
resulting diagonal shell has size $N^{-1+2\beta_b}$.  If $b\ne c$, independence
of the colors eliminates the contraction and leaves a centered second chaos.
Multiplication by the inner physical localizer shifts the Duhamel frequency
from $q+\ell$ to $q+\ell+z$; retaining its output shell separates the
principal transfer region from a Schwartz-decaying remainder.

The remaining stochastic objects are the cross quadratic field
\[
  \Theta=\Psi_{\W}\Psi_{\K},
\]
the first Picard terms $V_a=I_a\Theta$, $a\in\mathfrak C$, and the terms
$V_{a^\perp}\circ\Psi_a$.  Continuous weighted phase-layer estimates yield
\[
  V_a\in C_T\mathcal C_{\mathrm{loc}}^{1/2-\betasum-},
  \qquad
  \partial_tV_a\in C_T\mathcal C_{\mathrm{loc}}^{-1/2-\betasum-}.
\]
For $\Gamma_a=V_{a^\perp}\circ\Psi_a$, the centered third-chaos part is
controlled in $\mathcal C^{-\betasum-\beta_a-}$, while the retained
first-chaos contraction is controlled in
$\mathcal C^{1/2-\betasum-\beta_a-}$.  The corresponding regularity exponents are
\begin{equation}\label{eq:intro-regularity-exponents}
\begin{array}{c|cccccc}
 Z&\Psi_a&\Theta&V_a&\partial_tV_a&\Gamma_a^{(3)}&\Gamma_a^{(1)}\\
\hline
s_Z
 &-\frac12-\beta_a-
 &-1-\betasum-
 &\frac12-\betasum-
 &-\frac12-\betasum-
 &-\betasum-\beta_a-
 &\frac12-\betasum-\beta_a-
\end{array}
\end{equation}
Both cubic components are selected by the finite-cutoff Wick algebra.  For a triple $(a;b,c)$, the centered operator carries the order
$\beta_b+\beta_c\le2\betastar$.  The restriction $\betastar<1/8$ leaves a
nonempty range of deterministic Sobolev, Strichartz, and Besov exponents for
all stochastic distributions and operators.

To pass from stationary full-space fields to a local problem, fix a compact set
$K$ and choose $\chi\triangleleft\rho$ equal to one on a neighborhood of the
maximal-speed backward cone.  Set
\[
  B_a^\rho=\rho(\Psi_a+V_a),
  \qquad
  \zeta_a^{\chi,\rho}=L_aB_a^\rho-\chi\Theta.
\]
The auxiliary system
\[
  L_{\W}u=\chi uv+\zeta_{\W}^{\chi,\rho},
  \qquad
  L_{\K}v=\chi uv+\zeta_{\K}^{\chi,\rho}
\]
has a spatially compact nonlinear term and spatially compact forcing.  On the
inner cone, $\chi=\rho=1$ and $\zeta_a^{\chi,\rho}=\xi_a$.  Finite-cutoff
reconstruction and finite propagation identify the solutions obtained from
different localization pairs; their common restriction defines the local
solution on the cone.

The main results are stated in \cref{thm:local-wp,thm:spectral-approximation}.
The first gives local existence, uniqueness among the solutions of
Definition~\ref{def:local-solution} for fixed $(\Theta,\Xi)$, independence of
the localization, and local Lipschitz stability.  For each fixed cutoff
profile, the second gives almost-sure full-sequence dyadic convergence on a
profile-dependent probability-one event.  For an arbitrary admissible cofinal
sequence, the finite-cutoff solutions are defined only on Borel existence
events, and convergence in probability includes convergence of those events
below the reference lifetime.

\subsection*{Scope of the result}
The construction is spatially local: every compact set $K$ has an almost
surely positive measurable lifetime $\tau_K$, and the argument does not require
a common positive lifetime over all of $\R^3$.  Uniqueness is formulated for
fixed enhanced data $(\Theta,\Xi)$ in the class of
Definition~\ref{def:local-solution}.  For a fixed cutoff profile the full
dyadic sequence converges almost surely on a profile-dependent event, whereas
general admissible cofinal cutoffs converge in probability in the sense of
Definition~\ref{def:partial-probability-convergence}.  The threshold
$\betastar<1/8$ is sufficient for the estimates below; optimality and a
stopping-time formulation of the lifetimes are not considered here.

\subsection*{Relation to earlier work}
Hairer's theory of regularity structures \cite{HairerRS} provides a general
local-expansion framework for singular SPDEs.  For the three-dimensional
quadratic stochastic wave equation, the paracontrolled approach was developed
by Gubinelli, Koch, and Oh \cite{GKO}.  The operator-valued Gaussian estimates
used below are related to the random tensor estimates of Deng, Nahmod, and Yue
\cite{DNY}; see also Kaneshiro \cite{Kaneshiro} for a new proof of the abstract
random tensor estimate.

Related periodic and full-space constructions appear in the author's preprints
\cite{ZhaoColorPhase,ZhaoCentered}.  The present paper treats a massless--massive
wave--Klein--Gordon system on $\R^3$ and combines continuous-frequency
stochastic estimates with physical localization, finite propagation, and a
local paracontrolled well-posedness theory.

The deterministic background for quadratic wave--Klein--Gordon systems
includes \cite{Katayama,IonescuPausaderWG,DongWyatt}; different-speed
Klein--Gordon interactions are treated in \cite{Germain,DengMulti}, and the
quadratic normal-form method originates in \cite{Shatah}.  Littlewood--Paley
conventions follow \cite{BCD}, while the paracontrolled notation follows
\cite{GIP}.  The proof has the
following dependency structure:
\[
 (\Psi,\Theta)
 \longrightarrow (V,\partial_tV,\Gamma)
 \longrightarrow (\mathcal D,\mathcal B)
 \longrightarrow \Xi
 \longrightarrow Z
 \longrightarrow (u,v,\mathcal N).
\]
The principal analytic input--output relations are collected below.  Each
regularity exponent is understood with an arbitrarily small strict loss.
\begin{center}
\small
\renewcommand{\arraystretch}{1.16}
\begin{tabularx}{\textwidth}{@{}P{0.18\textwidth}P{0.20\textwidth}P{0.36\textwidth}Y@{}}
\toprule
object or operator & input & output used below & decisive condition\\
\midrule
$\Psi_a$ & $\xi_a$ & $C_T\mathcal C_{\loc}^{-1/2-\beta_a-}$ & single-channel covariance bound\\
$\Theta$ & $(\Psi_{\W},\Psi_{\K})$ & $C_T\mathcal C_{\loc}^{-1-\betasum-}$ & independence of the two colors\\
$V_a=I_a\Theta$ & $\Theta$ & $C_T\mathcal C_{\loc}^{1/2-\betasum-}$ & weighted three-frequency phase estimate\\
$\Gamma_a^{\chi,\rho}$ & $(V_{a^\perp},\Psi_a)$ & $C_T\mathcal C^{-\beta_{\Gamma,a}-\kappa}\cap L_T^\infty B_{2,\infty}^{-\beta_{\Gamma,a}-\kappa}$ & third-chaos convolution and first-chaos phase estimate\\
$\mathcal D^{a;b}$ & $E_T^{2,\sigma}$ & $C_TH^{s-1}\cap L_T^\infty B_{2,\infty}^{\sigma-1}$ & $s<1-2\beta_b$\\
$\mathcal B^{a;b,c}$ & $E_T^{2,\sigma}$ & $C_TH^{s-1}\cap L_T^1B_{2,\infty}^{\sigma-1}$ & $s<1-\beta_{b,c}$ and $\beta_{b,c}<\sigma<1-\beta_{b,c}$\\
fixed-point sources & $Z\in\mathcal Z_T$ & $\mathfrak R_T^{s_1,\sigma}\times\mathfrak S_T^{s_2,\sigma}$ & parameter window \eqref{eq:parameter-window}\\
\bottomrule
\end{tabularx}
\end{center}

Section~\ref{sec:model} introduces the localization, approximation classes,
function spaces, and main theorems.  Section~\ref{sec:paracontrolled} derives
the paracontrolled system and constructs the stochastic distributions and
operators entering it.  The operators $T^{a;b,c}$ are constructed in
Section~\ref{sec:resonant-operators}.  Section~\ref{sec:deterministic-closure}
contains the deterministic fixed point, measurable lifetimes, spectral
approximation, and localization argument.  The appendices prove the first
Picard estimates and the estimates for $\Gamma_a$.

\section{Analytic setting and main theorems}\label{sec:model}
\subsection{Linear operators and paraproducts}

Fix the dispersive parameter vector and the two forcing-order envelopes
\begin{equation}\label{eq:parameter-vector}
\begin{gathered}
  \parvec=(\speedW,\speedK,\massK),
  \qquad
  \speedW,\speedK,\massK>0,
  \qquad
  \speedgap:=|\speedK-\speedW|>0,\\
  \betaW,\betaK\ge0,
  \qquad
  \betastar:=\max\{\betaW,\betaK\}<\frac18,
  \qquad
  \betasum:=\betaW+\betaK,
  \qquad
  \betagamma:=\betasum+\betastar.
\end{gathered}
\end{equation}
For $a\in\mathfrak C$, $\beta_a$ denotes the corresponding channel
order; $\betasum$ counts the two Gaussian factors in $\Theta$ and $V_a$, and
$\beta_{\Gamma,a}=\betasum+\beta_a\le\betagamma$ counts those in
$\Gamma_a$.  Define
\begin{equation}\label{eq:channel-linear-family}
\begin{aligned}
 c_a&:=
 \begin{cases}c_{\W},&a=\W,\\ c_{\K},&a=\K,\end{cases}
 &
 L_a&:=
 \begin{cases}
  \partial_t^2-c_{\W}^2\Delta,&a=\W,\\
  \partial_t^2-c_{\K}^2\Delta+m^2,&a=\K,
 \end{cases}\\
 \omega_a(\xi)&:=
 \begin{cases}
  c_{\W}|\xi|,&a=\W,\\
  \sqrt{m^2+c_{\K}^2|\xi|^2},&a=\K.
 \end{cases}
\end{aligned}
\end{equation}
Then
\begin{equation}\label{eq:main-system}
  L_{\W}u=uv+\xi_{\W},
  \qquad
  L_{\K}v=uv+\xi_{\K},
  \qquad (t,x)\in[0,T]\times\R^3,
\end{equation}
and $\speedmax:=\max_{a\in\mathfrak C}c_a$.  For fixed positive
numbers $\underline c<\overline c$, $\underline m<\overline m$, and
$\delta_*>0$, define the separated parameter class
\begin{equation}\label{eq:separated-parameter-class}
\begin{aligned}
  \mathfrak P(\underline c,\overline c,
  \underline m,\overline m,\delta_*)
  :=\bigl\{(\widetilde c_{\W},\widetilde c_{\K},\widetilde m):\ &
  \underline c\le \widetilde c_{\W},\widetilde c_{\K}\le\overline c,\\
  &\underline m\le\widetilde m\le\overline m,\quad
  |\widetilde c_{\W}-\widetilde c_{\K}|\ge\delta_*\bigr\}.
\end{aligned}
\end{equation}
Scalar linear estimates are uniform when the speed--mass parameters range
over a fixed class in \eqref{eq:separated-parameter-class}.  Stochastic
estimates are uniform when $(\betaW,\betaK)$ ranges over a compact subset of
\[
 \bigl\{(\gamma_{\W},\gamma_{\K})\in[0,\infty)^2:
          \max\{\gamma_{\W},\gamma_{\K}\}<\tfrac18\bigr\}
\]
and $\mathfrak H_{\mathrm{prof}}$ is bounded.  The
phase constants depend only on the dispersive parameters and the fixed
Littlewood--Paley supports.

\begin{proposition}[Uniform polynomial dependence on the speed gap]
\label{prop:speed-gap-dependence}
Fix $0<\underline c<\overline c$, $0<\underline m<\overline m$, fixed
physical localizers, and $\delta_*>0$ such that
$\mathfrak P(\underline c,\overline c,
\underline m,\overline m,\delta_*)$ is nonempty.  Set
$d_*:=\min\{1,\delta_*\}$ and choose the low--high aperture by the fixed rule
\eqref{eq:aperture}, using the class-uniform value of $\delta_0$ from
\cref{prop:uniform-phase-class}.  Write $\mathbf e$ for the finite vector formed by
$\betaW,\betaK$ and the strict regularity exponents occurring in the cited
estimates.  Let $\mathscr C_{\mathrm{ph}}$ be the finite collection of
constants arising from phase division in
\cref{lem:explicit-phase-gap,lem:dyadic-volterra,prop:localized-diagonal,lem:first-picard-low-output-gap,lem:cubic-one-frequency-layer}
and from the associated dyadic phase sums.  If $\mathbf e$ ranges in a compact
subset $\mathfrak E$ of its open admissible region, then there exist
$C<\infty$ and $J<\infty$ such that, uniformly for
$\parvec\in\mathfrak P(\underline c,\overline c,
\underline m,\overline m,\delta_*)$ and $\mathbf e\in\mathfrak E$,
\begin{equation}\label{eq:phase-constant-polynomial-gap}
  \max_{A\in\mathscr C_{\mathrm{ph}}} A
  \le C d_*^{-J}.
\end{equation}
Here $C$ and $J$ depend only on the compact speed--mass box, the fixed
Littlewood--Paley templates, finitely many Schwartz seminorms of the physical
localizers, and $\mathfrak E$; in particular, they are independent of
$\delta_*$, of the particular parameter vector, and of the forcing profiles.
Before exponent-dependent dyadic summation, no compactness assumption on the
forcing orders is needed: the raw phase constants depend only on the
dispersive parameters.  After stochastic amplitudes are inserted, the bounds
also depend on $\mathfrak H_{\mathrm{prof}}$ from
\eqref{eq:forcing-profile-amplitude-norm}.  The constants in
\cref{prop:parameter-uniform-strichartz,lem:covariance-synthesis,thm:centered-gaussian-operator}
may be chosen independently of $\delta_*$.
\end{proposition}

\begin{proof}
All dependence on the separation parameter can be reduced to a finite list of
geometric quantities.  By \cref{prop:uniform-phase-class}, uniformly over the
compact speed--mass box,
\begin{equation}\label{eq:phase-parameter-ledger}
 \delta_0\ge c_1d_*,\qquad
 c_{\mathrm{sp}}\ge c_2d_*,\qquad
 N_{\mathrm{sp}}\le C_1d_*^{-1/2},
\end{equation}
where $c_1,c_2,C_1>0$ are independent of $\delta_*$.  With the fixed choice
\eqref{eq:aperture}, this also gives
$c_{\mathrm{ap}}^{-1}\le C d_*^{-1}$.

Two further frequency partitions depend on the speed gap.  The proofs of
\cref{lem:first-picard-low-output-gap,lem:cubic-one-frequency-layer} use, in
the low-output high--high region, the inequality (up to interchanging the two
channels)
\[
 \bigl|\omega_{\K}(\xi-\eta)-\omega_{\W}(\eta)\bigr|
 \ge \delta_*|\eta|-\overline c|\xi|-C|\eta|^{-1},
\]
with $C$ uniform on the speed--mass box.  If $|\eta|\sim M$, with
$\kappa_{\mathrm{ann}}M\le|\eta|$ for the fixed lower annular support
constant $\kappa_{\mathrm{ann}}>0$, and
$M\ge C_{\mathrm{lo}}\langle\xi\rangle$, then after subtracting the outer
phase the same lower bound is at least
\[
 \kappa_{\mathrm{ann}}\delta_*M
 -C\langle\xi\rangle-CM^{-1}.
\]
Taking $C_{\mathrm{lo}}=C_2d_*^{-1}$ with $C_2$ sufficiently large absorbs
both error terms, since $\langle\xi\rangle\ge1$, and leaves a lower bound
$c_3d_*M$.  Thus the region threshold and normalized phase gap may be chosen
so that
\begin{equation}\label{eq:low-output-parameter-ledger}
 C_{\mathrm{lo}}\le C_2d_*^{-1},\qquad
 c_{\mathrm{lo}}\ge c_3d_*.
\end{equation}
In the localized contraction proof, choose the near-localization radius as
$\eta=\eta_{\mathrm{LP}}\delta_0$, where
$\eta_{\mathrm{LP}}>0$ depends only on the fixed Fourier supports.  Then
$\eta^{-1}\le C d_*^{-1}$, and the far-localization estimate
\eqref{eq:localized-diagonal-far-tail} loses only a fixed power
$\eta^{-L}\lesssim_L d_*^{-L}$.

Consider first a high-frequency phase branch.  Every division appearing in
the cited formulas is by a phase bounded below by either
$c_{\mathrm{sp}}N$ or $c_{\mathrm{lo}}M$.  The number of such divisions is
bounded by a fixed integer determined by the displayed Volterra and endpoint
identities.  Equations \eqref{eq:phase-parameter-ledger} and
\eqref{eq:low-output-parameter-ledger} therefore bound every high-frequency
phase factor by a fixed power of $d_*^{-1}$.  The derivative lower bounds in
the balanced phase-layer regions do not use speed separation: on annular
shells they follow from
\[
 \omega_{\W}'(r)=c_{\W},\qquad
 \omega_{\K}'(r)=
 \frac{c_{\K}^2r}{\sqrt{m^2+c_{\K}^2r^2}},
\]
and are uniform on the compact positive parameter box.  All higher rescaled
phase derivatives are uniformly bounded there as well.

It remains to check the finitely many regions where no phase division is
made.  A direct sum over shells below $N_{\mathrm{sp}}$ has the form
\[
 \sum_{N<N_{\mathrm{sp}}}N^{\rho}
 \le C_{\mathfrak E}
 N_{\mathrm{sp}}^{\max\{\rho,0\}}
 \bigl(1+\log N_{\mathrm{sp}}\bigr),
\]
with $\rho$ ranging in a compact set determined by $\mathfrak E$.  By
\eqref{eq:phase-parameter-ledger} this is bounded by a fixed power of
$d_*^{-1}$; logarithmic factors are absorbed using
$1+|\log d_*|\lesssim_\varepsilon d_*^{-\varepsilon}$.  The finite-overlap
constants caused by the aperture and by the threshold $C_{\mathrm{lo}}$ are
controlled in the same way.  Modulation sums are uniform when their strict
summability margins range in the compact set $\mathfrak E$.

Thus every constant in $\mathscr C_{\mathrm{ph}}$ is obtained from finitely
many sums and products of quantities in
\eqref{eq:phase-parameter-ledger}--\eqref{eq:low-output-parameter-ledger},
fixed symbol seminorms, and compact exponent-margin constants.  This proves
\eqref{eq:phase-constant-polynomial-gap}.  The phase functions contain neither
$\betaW,\betaK$ nor the forcing profiles; these enter only through dyadic
weights and the covariance amplitudes.  Finally,
\cref{prop:parameter-uniform-strichartz,lem:covariance-synthesis} use only one
channel at a time, while \cref{thm:centered-gaussian-operator} is an abstract
Hilbert-space inequality.  Their constants are consequently uniform on the
compact speed--mass box without any lower bound on
$|c_{\W}-c_{\K}|$.
\end{proof}
For $a\in\mathfrak C$ define
\begin{align}
  S_a(t)(f,g)
  &:=\cos(t\omega_a(D))f
   +\frac{\sin(t\omega_a(D))}{\omega_a(D)}g,
  \label{eq:homogeneous-flow}\\
  I_aF(t)
  &:=\int_0^t\frac{\sin((t-s)\omega_a(D))}{\omega_a(D)}F(s)\dd s.
  \label{eq:duhamel-flow}
\end{align}
For the wave channel, the multiplier $\sin(t\speedW|\xi|)/(\speedW|\xi|)$ is defined at $\xi=0$ by its continuous value $t$.  The two principal propagation speeds are the parameters $\speedW$ and $\speedK$.

We use the Fourier convention
\[
  \wh f(\xi)=\int_{\R^3}e^{-\ii x\cdot\xi}f(x)\dd x,
  \qquad
  f(x)=\frac1{(2\pi)^3}\int_{\R^3}e^{\ii x\cdot\xi}
  \wh f(\xi)\dd\xi.
\]
Accordingly, the Fourier product constant used below is fixed once and for all
as
\begin{equation}\label{eq:fourier-product-constant}
 c_{\mathrm F}:=(2\pi)^{-3},
 \qquad
 \widehat{fg}(\xi)=c_{\mathrm F}\int_{\R^3}\widehat f(\eta)
 \widehat g(\xi-\eta)\dd\eta.
\end{equation}
Let $(P_N)_{N\in\Dyd}$ be a smooth radial inhomogeneous Littlewood--Paley decomposition, with $P_1$ containing the unit ball and $P_N$ supported in a fixed annulus for $N\ge2$.  Let $\delta_0=\delta_0(\parvec)>0$ be supplied by
Lemma~\ref{lem:explicit-phase-gap}.  Fix the low--high aperture by the rule
\begin{equation}\label{eq:aperture}
  c_{\mathrm{ap}}:=\frac{\delta_0}{64}.
\end{equation}
When the dispersive parameters range over a separated class, $\delta_0$ and
$c_{\mathrm{ap}}$ always denote the uniform choices supplied by
\cref{prop:uniform-phase-class}.
We write
\begin{equation}\label{eq:bony}
  f\prec g=\sum_N S_{<c_{\mathrm{ap}}N}f\,P_Ng,
  \qquad
  f\succ g=g\prec f,
  \qquad
  f\circ g=\sum_{N\sim_{\mathrm{ap}} M}P_Nf\,P_Mg.
\end{equation}
Here $N\sim_{\mathrm{ap}}M$ denotes the complementary finite-ratio region,
with constants depending on $c_{\mathrm{ap}}$, so that
$fg=f\prec g+f\succ g+f\circ g$.  We use the
standard paraproduct and commutator estimates from \cite{BCD,GIP}.  The aperture in
\eqref{eq:aperture} is chosen so that the actual smooth paraproduct supports lie
inside the proportional phase-gap window $|q|\le\delta_0|\ell|$.  On every
separated parameter class in \eqref{eq:separated-parameter-class}, the aperture
may be chosen uniformly.

\begin{lemma}[Almost-diagonal action of physical cutoffs]
\label{lem:physical-cutoff-almost-diagonal}
Let $\varphi\in C_c^\infty(\R^3)$.  For every $L>0$, every
$1\le q\le\infty$, and every two non-comparable dyadic scales $M,N$,
\begin{equation}\label{eq:physical-cutoff-almost-diagonal}
 \|P_M\varphi P_N\|_{\mathcal L(L^q,L^q)}
 \lesssim_{L,q,\varphi}
 \left(\frac{\min\{M,N\}}{\max\{M,N\}}\right)^L.
\end{equation}
Consequently, for every $r\in\R$,
\begin{align}
 \|\varphi f\|_{H^r}
 &\lesssim_{r,\varphi}\|f\|_{H^r},
 \label{eq:physical-cutoff-Sobolev}\\
 \|\varphi f\|_{B^r_{2,\infty}}
 &\lesssim_{r,\varphi}\|f\|_{B^r_{2,\infty}}.
 \label{eq:physical-cutoff-Besov}
\end{align}
In particular, these estimates hold for negative $r$.  Multiplication by
$\varphi$ preserves the little-Besov space $b^r_{2,\infty}$ and, for path
spaces, the uniform little-Besov tail condition
\eqref{eq:uniform-little-besov-tail} introduced below.  The same conclusions
hold for differences of two fixed physical cutoffs.
\end{lemma}

\begin{proof}
The Fourier kernel is
$\rho_M(\xi)\widehat\varphi(\xi-\eta)\rho_N(\eta)$.  On non-comparable
shells, $|\xi-\eta|\gtrsim\max\{M,N\}$.  Repeated integration by parts in
the Fourier inversion formula, or equivalently the usual
almost-orthogonality argument with the vanishing moments of the larger-scale
kernel, gives for every $L>0$
\[
 \sup_x\int|K_{M,N}(x,y)|\dd y+
 \sup_y\int|K_{M,N}(x,y)|\dd x
 \lesssim_{L,\varphi}
 \left(\frac{\min\{M,N\}}{\max\{M,N\}}\right)^L.
\]
Schur's test proves \eqref{eq:physical-cutoff-almost-diagonal}; weighted
dyadic Schur summation gives \eqref{eq:physical-cutoff-Sobolev} and taking a
dyadic supremum gives \eqref{eq:physical-cutoff-Besov}.  Preservation of the
little-Besov closure follows by density.  For the uniform path tail, write
$f=P_{\le R}f+P_{>R}f$: boundedness controls the second term uniformly, while
the displayed off-diagonal estimate makes the high output of the first term
uniformly small.  Let first the output threshold and then $R$ tend to
infinity.
\end{proof}

\begin{lemma}[Compact-support continuity upgrade]
\label{lem:compact-support-continuity}
Let $K\Subset\R^3$, let $\alpha,\beta,r\in\R$ with $r<\beta$, and suppose
$f:[0,T]\to\mathcal S'(\R^3)$ satisfies
\begin{equation}\label{eq:compact-support-continuity-hypotheses}
 \supp f(t)\subset K,\qquad
 f\in C_T\mathcal C^\alpha\cap L_T^\infty B_{2,\infty}^\beta.
\end{equation}
Then $f\in C_TH^r$, and
\begin{equation}\label{eq:compact-support-continuity-bound}
 \|f\|_{C_TH^r}
 \lesssim_{K,\alpha,\beta,r}
 \|f\|_{C_T\mathcal C^\alpha}
 +\|f\|_{L_T^\infty B_{2,\infty}^\beta}.
\end{equation}
The same assertion holds for differences of two such paths with the same
compact support.
\end{lemma}

\begin{proof}
Choose $\vartheta\in C_c^\infty$ equal to one near $K$.  For each fixed
$N$, the smoothing operator $P_N\vartheta:\mathcal C^\alpha\to L^2$ is
bounded, hence $t\mapsto P_Nf(t)$ is $L^2$-continuous.  On the other hand,
for dyadic $L\ge2$,
\[
 \sup_{t\le T}\|P_{>L}f(t)\|_{H^r}^2
 \lesssim \|f\|_{L_T^\infty B_{2,\infty}^\beta}^2
 \sum_{N>L}N^{2(r-\beta)}
 \lesssim L^{2(r-\beta)}
 \|f\|_{L_T^\infty B_{2,\infty}^\beta}^2.
\]
The same bound applies to time differences.  A finite-shell/high-frequency
split therefore proves continuity in $H^r$ and
\eqref{eq:compact-support-continuity-bound}.  The difference statement is
identical.
\end{proof}

\begin{lemma}[Compactly supported lowering of the Lebesgue exponent]
\label{lem:compact-p-lowering}
Let $\gamma\ge0$, $1<p_0\le p_1<\infty$, and
$\chi_0\in C_c^\infty(\R^3)$.  Then
\begin{equation}\label{eq:compact-p-lowering}
 \|\chi_0 f\|_{W^{\gamma,p_0}}
 \lesssim_{\chi_0,\gamma,p_0,p_1}
 \|f\|_{W^{\gamma,p_1}}.
\end{equation}
The same estimate holds for time-dependent functions after taking any
$L_t^q$ norm on both sides.
\end{lemma}

\begin{proof}
Choose $\chi_1\in C_c^\infty$ equal to one near $\supp\chi_0$.  On
$\supp\chi_1$, finite-measure embedding and boundedness of smooth
multiplication on $W^{\gamma,p_1}$ give
\[
 \|\chi_1\la D\ra^\gamma(\chi_0f)\|_{p_0}
 \lesssim\|f\|_{W^{\gamma,p_1}}.
\]
The operator $(1-\chi_1)\la D\ra^\gamma\chi_0$ has a smooth rapidly
decaying kernel because the two spatial supports are separated, and is
therefore bounded from $L^{p_1}$ to $L^{p_0}$.  Adding the two pieces proves
the claim, pointwise in time.
\end{proof}

Fix once and for all a nonnegative
$\vartheta\in C_c^\infty(\R^3)$ which equals one on $[-1,1]^3$, and an integer
\begin{equation}\label{eq:pseudolocal-J}
 J_{\mathrm{ps}}\ge8.
\end{equation}
For $z\in\mathbb Z^3$ write $\vartheta_z(x)=\vartheta(x-z)$.

\begin{definition}[Weighted local tempered norm]
\label{def:weighted-tempered-norm}
For $f\in\mathcal S'(\R^3)$ set
\begin{equation}\label{eq:weighted-tempered-norm}
 \|f\|_{\mathfrak T_{J_{\mathrm{ps}}}}
 :=\sup_{z\in\mathbb Z^3}\la z\ra^{-J_{\mathrm{ps}}}
 \|\vartheta_z f\|_{H^{-J_{\mathrm{ps}}}}.
\end{equation}
The path space $C_T\mathfrak T_{J_{\mathrm{ps}}}$ is defined using the
corresponding supremum norm and strong continuity in
$\mathfrak T_{J_{\mathrm{ps}}}$.
\end{definition}

\begin{lemma}[Weighted local control and stationary lifting]
\label{lem:weighted-stationary-lifting}
The following assertions hold.
\begin{enumerate}[label=\textup{(\roman*)},leftmargin=2.3em]
\item The space $\mathfrak T_{J_{\mathrm{ps}}}$ is Banach and embeds
continuously into $\mathcal S'(\R^3)$.  Multiplication by a function in
$C_b^\infty$ is bounded on $\mathfrak T_{J_{\mathrm{ps}}}$.  Moreover, if a
convolution kernel $K$ satisfies
\[
 \int_{\R^3}\la x\ra^{2J_{\mathrm{ps}}+6}|K(x)|\dd x<\infty,
\]
then $f\mapsto K*f$ is bounded on $\mathfrak T_{J_{\mathrm{ps}}}$.  In
particular, the admissible multipliers $\pi_\Lambda$ are uniformly bounded on
this space.
\item If $r\ge-J_{\mathrm{ps}}$, then
\begin{equation}\label{eq:H-to-weighted-tempered}
 \|f\|_{\mathfrak T_{J_{\mathrm{ps}}}}
 \lesssim_r\|f\|_{H^r}.
\end{equation}
\item Let $Z$ be a spatially stationary random path.  If, for some
$p\ge2$ with $pJ_{\mathrm{ps}}>3$,
\[
 \sup_{z\in\mathbb Z^3}
 \|\vartheta_zZ\|_{L^p(\Omega;C_TH^{-J_{\mathrm{ps}}})}<\infty,
\]
then $Z\in L^p(\Omega;C_T\mathfrak T_{J_{\mathrm{ps}}})$.
\item If $Z_n-Z$ is spatially stationary and the preceding local norm tends to
zero uniformly in $z$, then
$Z_n\to Z$ in $L^p(\Omega;C_T\mathfrak T_{J_{\mathrm{ps}}})$.
The same convergence holds almost surely whenever
\[
 \sum_n\sup_{z\in\mathbb Z^3}
 \|\vartheta_z(Z_n-Z)\|_{L^p(\Omega;C_TH^{-J_{\mathrm{ps}}})}^{p}<\infty.
\]
\end{enumerate}
\end{lemma}

\begin{proof}
Let $(\psi_z)_{z\in\mathbb Z^3}$ be a smooth lattice partition with
$\supp\psi_z\subset\{\vartheta_z=1\}$.  For $\phi\in\mathcal S$,
\[
 |\langle f,\phi\rangle|
 \le\sum_z\|\vartheta_zf\|_{H^{-J_{\mathrm{ps}}}}
          \|\psi_z\phi\|_{H^{J_{\mathrm{ps}}}}
 \lesssim \|f\|_{\mathfrak T_{J_{\mathrm{ps}}}}p_A(\phi),
\]
which yields the embedding into $\mathcal S'$; local completeness and
compatibility of the cellwise limits give completeness.  The same cell
decomposition, together with
$\la w\ra^{J_{\mathrm{ps}}}\lesssim
 \la z\ra^{J_{\mathrm{ps}}}\la z-w\ra^{J_{\mathrm{ps}}}$, proves the
standard weighted uniformly-local bounds for multiplication and convolution.
The rescaled kernels of admissible multipliers have uniformly bounded
weighted $L^1$ moments, so they are uniformly bounded on
$\mathfrak T_{J_{\mathrm{ps}}}$.  This proves (i), and (ii) follows from the
translation-uniform bound $H^r\to H^{-J_{\mathrm{ps}}}$ after localization.

For (iii), put $A_z=\|\vartheta_zZ\|_{C_TH^{-J_{\mathrm{ps}}}}$.  Since
$pJ_{\mathrm{ps}}>3$,
\[
 \E\|Z\|_{C_T\mathfrak T_{J_{\mathrm{ps}}}}^p
 \le\sum_z\la z\ra^{-pJ_{\mathrm{ps}}}\E A_z^p<\infty.
\]
The weighted $\ell^p$ tail is uniformly small in time, while finitely many
cell coordinates are continuous; hence the path is strongly continuous in
$\mathfrak T_{J_{\mathrm{ps}}}$.  Applying the same estimate to $Z_n-Z$
proves (iv).  Under the stated summability hypothesis, Markov's inequality
and Borel--Cantelli give almost-sure convergence.
\end{proof}

\begin{remark}[Norm localization versus transfer frequency]
\label{rem:norm-versus-transfer}
The preceding lemmas concern boundedness of fixed physical multipliers in
function-space norms.  For the Duhamel term, the propagated frequency in
Lemma~\ref{lem:inner-localization-transfer} is $p=q+\ell+z$, where $z$ is
the Fourier defect of the inner multiplier.  The
principal sector extracts the inverse power of the actual output shell, and
the complementary sector is summable by Schwartz decay in $z$.
\end{remark}

For a Banach space $X$ and $1\le q\le\infty$, write
\[
  C_TX:=C([0,T];X),
  \qquad
  C_T^1X:=C^1([0,T];X),
  \qquad
  L_T^qX:=L^q([0,T];X).
\]
We use
\[
  \cC^s=B^s_{\infty,\infty},
  \qquad
  \Btwo{s}=B^s_{2,\infty}.
\]
All Sobolev and Besov spaces are over $\R^3$ unless a local subscript is
displayed.  An exponent $r-$ means $r-\varepsilon$ for every sufficiently
small $\varepsilon>0$; all theorem statements below use a fixed loss
$\kappa>0$.

\subsection{Localization and spectral approximation}
\label{subsec:localization}

Fix a deterministic reference horizon $0<T_0\le1$.  For $r>0$ and
$K\Subset\R^3$, write
\begin{equation}\label{eq:expanded-compact}
  K^{[r]}:=\set{x\in\R^3:\dist(x,K)\le r}.
\end{equation}
For $s\in\R$, the space $\mathcal C^s_{\mathrm{loc}}(\R^3)$ consists of
distributions $f$ such that $\varphi f\in\mathcal C^s$ for every
$\varphi\in C_c^\infty(\R^3)$.  Local Sobolev and Besov spaces are defined
analogously.

\begin{definition}[Localization pairs]\label{def:adapted-localization}
A pair $(\chi,\rho)\in C_c^\infty(\R^3)^2$ is adapted to $(K,T_0)$ if
\begin{equation}\label{eq:localization-geometry}
  0\le\chi,\rho\le1,
  \qquad
  \chi\equiv1\ \text{on }K^{[\speedmax T_0+3]},
  \qquad
  \rho\equiv1\ \text{on }(\supp\chi)^{[2]}.
\end{equation}
We call $(\chi,\rho)$ an adapted localization pair and write
$\chi\triangleleft\rho$.  Constants may depend on finitely many seminorms of
this fixed pair.
\end{definition}

Let the spatially stationary stochastic objects $\Psi_a,\Theta,V_a$ be defined in
Sections~\ref{sec:baseline-lift} and~\ref{app:first-picard}.  For an adapted localization
pair set
\begin{equation}\label{eq:localized-baseline}
  B_a^\rho:=\rho(\Psi_a+V_a),
  \qquad a\in\mathfrak C,
\end{equation}
and define the auxiliary forcing
\begin{equation}\label{eq:auxiliary-forcing}
  \zeta_a^{\chi,\rho}:=L_aB_a^\rho-\chi\Theta.
\end{equation}
The expression is interpreted distributionally.  Equivalently,
\begin{equation}\label{eq:auxiliary-forcing-expanded}
  \zeta_a^{\chi,\rho}
  =\rho\xi_a+[L_a,\rho]\Psi_a
   +(\rho-\chi)\Theta+[L_a,\rho]V_a,
\end{equation}
where $[L_a,\rho]$ contains only spatial derivatives of the fixed smooth
multiplier.  Hence $\zeta_a^{\chi,\rho}$ is spatially compact.

\begin{lemma}[Local identity of the auxiliary forcing]
\label{lem:auxiliary-forcing-local}
If $U\subset\R^3$ is open and $\chi=\rho=1$ on a neighborhood of $U$, then
\begin{equation}\label{eq:auxiliary-forcing-equals-noise}
  \zeta_a^{\chi,\rho}=\xi_a
  \quad\text{in }\mathcal D'((0,T_0)\times U).
\end{equation}
The same identity holds at every finite spectral cutoff.
\end{lemma}

\begin{proof}
On $U$, the multipliers are constant and their commutators with $L_a$ vanish.
Since $L_a\Psi_a=\xi_a$ and $L_aV_a=\Theta$, one has
$L_a(\Psi_a+V_a)-\Theta=\xi_a$.  Testing against functions supported in
$(0,T_0)\times U$ proves the claim.  The cutoff statement is identical.
\end{proof}

The global auxiliary equation associated with $(\chi,\rho)$ is
\begin{equation}\label{eq:auxiliary-system}
\begin{aligned}
  \LW u^{\chi,\rho}
  &=\chi u^{\chi,\rho}v^{\chi,\rho}
    +\zeta_{\W}^{\chi,\rho},\\
  \LK v^{\chi,\rho}
  &=\chi u^{\chi,\rho}v^{\chi,\rho}
    +\zeta_{\K}^{\chi,\rho}.
\end{aligned}
\end{equation}
It is endowed with the prescribed deterministic Cauchy data.  Since the stochastic convolutions and the first Picard objects have zero Cauchy data, the entire deterministic datum is assigned to the $Y$ coordinates in the decomposition below.
After subtracting the compact baseline \eqref{eq:localized-baseline}, every
nonlinear source in \eqref{eq:auxiliary-system} is multiplied by $\chi$.
Thus the deterministic fixed point is posed in global, unweighted
Sobolev--Besov spaces even though the original spatially stationary convolutions are only
locally controlled.

\begin{lemma}[Finite-cutoff auxiliary solvability]
\label{lem:finite-cutoff-auxiliary-solvability}
Fix $r>3/2$, an admissible cutoff $\pi_\Lambda$, an adapted localization pair
$(\chi,\rho)$, and the smoothed Cauchy data $y_\Lambda$.  For almost every
$\omega\in\Omega$ there exists $T_\Lambda(\omega)>0$ such that
\eqref{eq:cutoff-system} has a unique solution of the form
\[
  u_\Lambda^{\chi,\rho}=B_{\W,\Lambda}^\rho+R_{\W,\Lambda},
  \qquad
  v_\Lambda^{\chi,\rho}=B_{\K,\Lambda}^\rho+R_{\K,\Lambda},
\]
with
\[
  R_{a,\Lambda}\in C([0,T_\Lambda];H^r)
  \cap C^1([0,T_\Lambda];H^{r-1}),
  \qquad a\in\mathfrak C.
\]
If a finite-cutoff fixed point exists on $[0,T]$, then its
reconstruction agrees with this solution on $[0,T\wedge T_\Lambda]$.
\end{lemma}

\begin{proof}
For fixed $\Lambda$, multiplication by $\rho$ makes
$B_{a,\Lambda}^\rho$ almost surely smooth in space with
$C_TH^m\cap C_T^1H^{m-1}$ paths for every $m$.  Substitution
$u_\Lambda=B_{\W,\Lambda}^\rho+R_{\W,\Lambda}$,
$v_\Lambda=B_{\K,\Lambda}^\rho+R_{\K,\Lambda}$ reduces
\eqref{eq:cutoff-system} to a semilinear wave--Klein--Gordon system for
$R_\Lambda$ with smooth compactly supported coefficients.  Since $H^r$ is an
algebra for $r>3/2$, the energy estimate and the contraction principle give a
unique local solution in $C_TH^r\cap C_T^1H^{r-1}$.  A finite-cutoff fixed
point reconstructs a solution of the same mild system by
Proposition~\ref{prop:finite-cutoff-reconstruction}; uniqueness gives the
last assertion.
\end{proof}

\begin{proposition}[Finite-cutoff regularity bootstrap]
\label{prop:finite-cutoff-regularity}
Let $\Lambda<\infty$ and suppose that a finite-cutoff fixed point
$Z_\Lambda$ exists on $[0,T]$.  Put
\[
 R_{a,\Lambda}:=X_{a,\Lambda}+Y_{a,\Lambda},
 \qquad
 u_\Lambda=B_{\W,\Lambda}^\rho+R_{\W,\Lambda},
 \qquad
 v_\Lambda=B_{\K,\Lambda}^\rho+R_{\K,\Lambda}.
\]
Then, almost surely, for every $m\ge0$,
\begin{equation}\label{eq:finite-cutoff-posterior-regularity}
 R_{a,\Lambda}\in C([0,T];H^m)
 \cap C^1([0,T];H^{m-1}),
 \qquad a\in\mathfrak C.
\end{equation}
The reconstructed fields satisfy the same conclusion.  In particular,
\[
 (u_\Lambda,v_\Lambda)\in
 \bigl(C_TH^1\cap C_T^1L^2\cap C_TL^\infty\bigr)^2.
\]
\end{proposition}

\begin{proof}
We first improve the $X$ terms.  Since
$\widehat\Psi_{a,\Lambda}$ is supported in a ball of radius $C\Lambda$, the
paraproduct $f\prec\Psi_{a,\Lambda}$ contains only shells $N\le C\Lambda$.
Let $K_\chi$ be a compact neighborhood of $\supp\chi$.  For every integer
$m\ge0$, Bernstein's inequality and the Leibniz rule give
\begin{align}
 &\sup_{t\le T}
 \|\chi(f\prec\Psi_{a,\Lambda})(t)\|_{H^m}\notag\\
 &\qquad\lesssim_{m,\Lambda,\chi}
 \|f\|_{C_TL^2}
 \sum_{N\le C\Lambda}
 \|P_N\Psi_{a,\Lambda}\|_{C_T C^m(K_\chi)}.
 \label{eq:finite-cutoff-X-smooth-source}
\end{align}
The sum is finite almost surely, and each summand is continuous in time.
Thus the source in the $X$ equation belongs to $C_TH^m$ for every $m$.
The linear energy estimate yields
\begin{equation}\label{eq:finite-cutoff-X-smoothness}
 X_{a,\Lambda}\in C_TH^{m+1}\cap C_T^1H^m
 \qquad(m\ge0).
\end{equation}

Since $y_\Lambda=\pi_\Lambda y$ has compact Fourier support, its two
Cauchy-data components belong to $H^m\times H^{m-1}$ for every $m$.
The fixed-point space and \eqref{eq:finite-cutoff-X-smoothness} therefore give
\begin{equation}\label{eq:finite-cutoff-bootstrap-start}
 R_{a,\Lambda}\in C_TH^{s_2}\cap C_T^1H^{s_2-1},
 \qquad s_2>\frac12.
\end{equation}
By Proposition~\ref{prop:finite-cutoff-reconstruction}, on $[0,T]$ the
remainder satisfies
\begin{equation}\label{eq:finite-cutoff-remainder-equation}
 L_aR_{a,\Lambda}
 =\chi\Bigl[(B_{\W,\Lambda}^\rho+R_{\W,\Lambda})
 (B_{\K,\Lambda}^\rho+R_{\K,\Lambda})-\Theta_\Lambda\Bigr].
\end{equation}
For fixed $\Lambda$, the functions $B_{a,\Lambda}^\rho$ and
$\chi\Theta_\Lambda$ belong to $C_TH^m$ for every $m$.  Suppose that both
remainder components belong to $C_TH^r$ for some $1/2<r<3/2$.  For every
$0<\varepsilon<r-1/2$, the product theorem on $\R^3$ gives
\[
 H^r\cdot H^r\hookrightarrow H^{2r-3/2-\varepsilon}.
\]
Because the finite-cutoff Cauchy data are smooth, the linear energy estimate
applied to \eqref{eq:finite-cutoff-remainder-equation} therefore yields
\begin{equation}\label{eq:finite-cutoff-bootstrap-step}
 R_{a,\Lambda}\in
 C_TH^{2r-1/2-\varepsilon}
 \cap C_T^1H^{2r-3/2-\varepsilon}.
\end{equation}
With
\[
 \varepsilon_n=\frac12\left(r_n-\frac12\right),
 \qquad
 r_{n+1}=2r_n-\frac12-\varepsilon_n,
\]
one has
$r_{n+1}-1/2=\tfrac32(r_n-1/2)$.  Starting from $r_0=s_2$, finitely many
steps give an exponent larger than $3/2$.  For $r>3/2$, $H^r(\R^3)$ is an
algebra; hence \eqref{eq:finite-cutoff-remainder-equation} and the energy
estimate improve $H^r$ to $H^{r+1}$.  Iteration first proves
\eqref{eq:finite-cutoff-posterior-regularity} for every integer $m\ge0$; the
statement for arbitrary real $m\ge0$ follows by Sobolev embedding between
successive integer orders.

Finally, multiplication by the compactly supported function $\rho$ places
$B_{a,\Lambda}^\rho$ in $C_TH^m\cap C_T^1H^{m-1}$ for every $m$.  Adding the
baseline to the remainder proves the assertion for $u_\Lambda$ and
$v_\Lambda$; Sobolev embedding with $m>3/2$ gives the stated $L^\infty$
bound.
\end{proof}

\begin{lemma}[Finite propagation for the cutoff equations]
\label{lem:finite-propagation}
For $\jmath=1,2$, let
\[
 (u_\jmath,v_\jmath)\in
 \bigl(C_TH^1_{\mathrm{loc}}\cap C_T^1L^2_{\mathrm{loc}}\bigr)^2
\]
be energy solutions of
\[
  L_{\W}u_\jmath=\mathfrak b_\jmath u_\jmath v_\jmath+f_{\W,\jmath},
  \qquad
  L_{\K}v_\jmath=\mathfrak b_\jmath u_\jmath v_\jmath+f_{\K,\jmath}.
\]
Assume on the backward cone $\mathcal Q_{K,T}$ that
$\mathfrak b_1=\mathfrak b_2=:\mathfrak b$ and
$f_{a,1}=f_{a,2}$ for every $a\in\mathfrak C$, and
\[
 \mathfrak b v_1,\;\mathfrak b u_2\in L^1([0,T];L^\infty_{\mathrm{loc}}),
\]
and assume that the Cauchy data agree on $K^{[\speedmax T]}$.  Then
\[
  (u_1,v_1)=(u_2,v_2)
  \qquad\text{on }[0,T]\times K.
\]
For reconstructions of finite-cutoff fixed points, the hypotheses
hold on the entire existence interval by
Proposition~\ref{prop:finite-cutoff-regularity}.
\end{lemma}

\begin{proof}
Set $U=u_1-u_2$, $V=v_1-v_2$, and
$G=\mathfrak b(v_1U+u_2V)$ on the backward cone.  For
$\Omega_\tau=K^{[\speedmax(T-\tau)]}$, use a smooth approximation of its
indicator whose space--time gradient satisfies
$\partial_\tau\vartheta+\speedmax|\nabla\vartheta|\le0$.  With
$F_{\W}=U$, $F_{\K}=V$, $\mu_{\W}=0$, and $\mu_{\K}=m$, define
\[
 E_a(\tau)=\frac12\int\vartheta^2
 \bigl(|\partial_tF_a|^2+c_a^2|\nabla F_a|^2+(\mu_a^2+1)|F_a|^2\bigr)\dd x.
\]
The standard localized energy identity for $L_aF_a=G$ contains an interior
term $\int\vartheta^2\partial_tF_a(G+F_a)$ and a boundary flux.  Because
$c_a\le\speedmax$, the transport inequality for $\vartheta$ makes the latter
nonpositive.  Summing the two channel energies and using
$G=\mathfrak b(v_1U+u_2V)$ gives, for almost every $\tau$,
\[
 E'(\tau)\le C\Bigl(1+
 \|\mathfrak b v_1(\tau)\|_{L^\infty(\Omega_\tau)}+
 \|\mathfrak b u_2(\tau)\|_{L^\infty(\Omega_\tau)}\Bigr)E(\tau),
 \qquad E=E_{\W}+E_{\K}.
\]
The initial localized energy vanishes because the Cauchy data agree on
$K^{[\speedmax T]}$.  Gronwall's lemma yields $E\equiv0$.  For energy
solutions the same identity follows by Steklov averaging in time; the assumed
$L_t^1L_x^\infty$ coefficients make the source locally integrable in the
energy dual space.  Letting the smooth cone cutoff tend to the indicator and
then varying the terminal time proves $U=V=0$ on $[0,T]\times K$.  The last
sentence of the statement follows from
Proposition~\ref{prop:finite-cutoff-regularity} and Sobolev embedding.
\end{proof}

Let $(\Omega,\mathcal F,\Pp)$ be a complete probability space.  Put
\[
 \mathfrak H:=L^2(\R_+\times\R^3;\R)
\]
and let $\mathbf W_{\W},\mathbf W_{\K}$ be independent real isonormal Gaussian processes over
$\mathfrak H$; thus $\mathbf W_a(h)$ is centered Gaussian and
$\E[\mathbf W_a(h)\mathbf W_b(g)]=\delta_{ab}\langle h,g\rangle_{\mathfrak H}$.

\begin{definition}[Admissible Fourier-multiplier Gaussian forcing]
\label{def:gaussian-forcing-profile}
Let the channel-dependent upper orders $\betaW,\betaK$ satisfy
\eqref{eq:parameter-vector}.  For $a\in\mathfrak C$, let
\[
 \mathfrak h_a\in S^{\beta_a}_{1,0}(\R^3;\R),
 \qquad
 \mathfrak h_a(-\xi)=\mathfrak h_a(\xi).
\]
For $L\in\N_0$, set
\begin{equation}\label{eq:gaussian-profile-symbol-bound}
 [\mathfrak h_a]_{\beta_a,L}
 :=\max_{|\alpha|\le L}\sup_{\xi\in\R^3}
 \langle\xi\rangle^{-\beta_a+|\alpha|}
 |\partial_\xi^\alpha\mathfrak h_a(\xi)|<\infty.
\end{equation}
Set
\begin{equation}\label{eq:forcing-profile-amplitude-norm}
 \mathfrak H_{\mathrm{prof}}
 :=\max_{a\in\mathfrak C}[\mathfrak h_a]_{\beta_a,0}
 =\max_{a\in\mathfrak C}\sup_{\xi\in\R^3}
   \langle\xi\rangle^{-\beta_a}|\mathfrak h_a(\xi)|.
\end{equation}
For a real test function
$\varphi\in C_c^\infty((0,\infty)\times\R^3)$, the multiplier
$\mathfrak h_a(D_x)$ acts only in the spatial variable.  The map
$\varphi\mapsto\mathfrak h_a(D_x)\varphi$ is continuous from
$C_c^\infty((0,\infty)\times\R^3)$ to the real Gaussian Hilbert space
$\mathfrak H$.  Hence the formula
\begin{equation}\label{eq:isonormal-forcing-definition}
 \xi_a(\varphi):=\mathbf W_a(\mathfrak h_a(D_x)\varphi),
 \qquad a\in\mathfrak C,
\end{equation}
defines a real centered Gaussian random distribution.  Set
\[
 \mathfrak q_a(\xi):=|\mathfrak h_a(\xi)|^2,
 \qquad
 Q_a:=\mathfrak h_a(D_x)^*\mathfrak h_a(D_x)
      =\mathfrak h_a(D_x)^2.
\]
Then $\mathfrak q_a(\xi)\lesssim\langle\xi\rangle^{2\beta_a}$ and
\begin{equation}\label{eq:gaussian-forcing-covariance}
\begin{aligned}
 \E[\xi_a(\varphi)\xi_b(\psi)]
 &=\delta_{ab}
   \langle\mathfrak h_a(D_x)\varphi,
          \mathfrak h_a(D_x)\psi\rangle_{\mathfrak H}\\
 &=\delta_{ab}\int_0^\infty
   \langle Q_a\varphi(t),\psi(t)\rangle_{L_x^2}\dd t.
\end{aligned}
\end{equation}
Equivalently, in the non-conjugated Fourier random-measure convention,
\begin{equation}\label{eq:fourier-gaussian-forcing}
 \E[\widehat\xi_a(\dd\xi,\dd t)
       \widehat\xi_b(\dd\eta,\dd s)]
 =\delta_{ab}(2\pi)^3\delta(t-s)\delta(\xi+\eta)
  \mathfrak q_a(\xi)\dd\xi\dd\eta\dd t\dd s.
\end{equation}
Thus the forcing is white in time, stationary in space, and independent
between the two colors.  Its law depends on the profile only through the
spectral density $\mathfrak q_a$.
\end{definition}

\begin{remark}[Profiles and uniform constants]
\label{rem:constant-dependence}
The orders $\beta_a$ are upper symbol orders; unequal and nonradial examples
such as $\mathfrak h_a(\xi)=\la\xi\ra^{\gamma_a}s_a(\xi)$ with
$\gamma_a\le\beta_a$ and even $s_a\in S^0_{1,0}$ are allowed.  The stochastic
covariance, phase-layer, and Schatten estimates use only
$\mathfrak H_{\mathrm{prof}}$ in
\eqref{eq:forcing-profile-amplitude-norm}; higher symbol seminorms enter only
non-uniform finite-cutoff smoothness bounds.  Uniform estimates are understood
with the speed--mass vector in a fixed separated class, the orders in a
compact subset of $\{\betastar<1/8\}$, bounded
$\mathfrak H_{\mathrm{prof}}$, fixed strict exponent margins, and fixed
localizers.  Only mixed-phase divisions depend on the lower speed gap, as
quantified in Proposition~\ref{prop:speed-gap-dependence}.
\end{remark}

\begin{definition}[Admissible spectral cutoffs]\label{def:admissible-cutoff}
A family $(\pi_\Lambda)_{\Lambda\ge1}$ is called admissible if
$\pi_\Lambda=m_\Lambda(D)$, where $m_\Lambda\in C_c^\infty(\R^3)$ is real and
even, and there exist constants $0<c_0<C_0<\infty$ such that
\begin{equation}\label{eq:cutoff-supports}
  m_\Lambda(\xi)=1\quad\text{for }|\xi|\le c_0\Lambda,
  \qquad
  \operatorname{supp}m_\Lambda\subset\{|\xi|\le C_0\Lambda\}.
\end{equation}
Moreover,
\begin{equation}\label{eq:cutoff-symbol-bounds}
  \sup_{\Lambda\ge1}\Lambda^{|\alpha|}
  \|\partial^\alpha m_\Lambda\|_{L^\infty}<\infty
  \quad\text{for every multi-index }\alpha,
  \qquad
  m_\Lambda(\xi)\longrightarrow1
  \quad(\xi\in\R^3).
\end{equation}
No radiality, monotonicity, or nesting is imposed.  A fixed-profile family is
given by $m_\Lambda(\xi)=m(\xi/\Lambda)$ for a real even
$m\in C_c^\infty(\R^3)$ that equals one near the origin.  A dyadic
fixed-profile family is indexed by $\Lambda=2^n$, $n\in\N$.  An \emph{admissible cofinal cutoff sequence} is a sequence
$(\pi_{\Lambda_n})_{n\ge1}$ drawn from an admissible family, with
$\Lambda_n\to\infty$.  Statements for a continuous cutoff parameter are
understood sequentially along every admissible cofinal sequence.
\end{definition}

\begin{lemma}[Finite-product cutoff reduction]
\label{lem:finite-product-cutoff}
Let $(E,\mu)$ be a measure space, let $J\in\N$, and let
$\zeta_1,\ldots,\zeta_J:E\to\R^3$ be measurable.  For
$\boldsymbol\epsilon=(\epsilon_1,\ldots,\epsilon_J)\in\{0,1\}^J$, set
\[
 \mathcal I_{\boldsymbol\epsilon}:=
 \{\nu: \epsilon_\nu=1\},
 \qquad
 M_{\Lambda}^{\boldsymbol\epsilon}(z)
 :=\prod_{\nu\in\mathcal I_{\boldsymbol\epsilon}}
 m_\Lambda(\zeta_\nu(z)),
\]
with the empty product equal to one.  If
$\mathcal I_{\boldsymbol\epsilon}=\varnothing$, the maxima and indicators below
are understood as zero.  There is a constant $C_J$, depending only on $J$ and
the uniform zeroth-order cutoff bound, such that
\begin{align}
 |M_{\Lambda}^{\boldsymbol\epsilon}(z)-1|^2
 &\le C_J\,
 \one_{\{\max_{\nu\in\mathcal I_{\boldsymbol\epsilon}}
              |\zeta_\nu(z)|\ge c_0\Lambda\}},
 \label{eq:finite-product-cutoff-tail}\\
 |M_{\Lambda}^{\boldsymbol\epsilon}(z)
   -M_{\Lambda'}^{\boldsymbol\epsilon}(z)|^2
 &\le C_J\left(
 \one_{\{\max_{\nu\in\mathcal I_{\boldsymbol\epsilon}}
              |\zeta_\nu(z)|\ge c_0\Lambda\}}
 +
 \one_{\{\max_{\nu\in\mathcal I_{\boldsymbol\epsilon}}
              |\zeta_\nu(z)|\ge c_0\Lambda'\}}
 \right).
 \label{eq:finite-product-two-cutoffs}
\end{align}
Consequently, for every nonnegative measurable $F$,
\begin{equation}\label{eq:finite-product-integral-tail}
 \int_E |M_{\Lambda}^{\boldsymbol\epsilon}-1|^2F\,\dd\mu
 \le C_J
 \int_{\{\max_{\nu\in\mathcal I_{\boldsymbol\epsilon}}
              |\zeta_\nu|\ge c_0\Lambda\}}F\,\dd\mu,
\end{equation}
and the analogous two-tail bound follows from
\eqref{eq:finite-product-two-cutoffs}.  The same statement compares products
formed from two different admissible cutoff families, after replacing $c_0$
by the smaller plateau constant and using a common zeroth-order bound.  The
conclusion applies with $F$ equal to a static squared kernel or to a normalized
time-increment majorant.  Repeated frequency legs are included by listing them
with multiplicity.
\end{lemma}

\begin{proof}
Telescope the active factors:
\[
 \prod_{j=1}^I m_\Lambda(\zeta_{\nu_j})-1
 =\sum_{j=1}^I\left(\prod_{k<j}m_\Lambda(\zeta_{\nu_k})\right)
 \bigl(m_\Lambda(\zeta_{\nu_j})-1\bigr).
\]
Every summand vanishes unless
$|\zeta_{\nu_j}|\ge c_0\Lambda$; the uniform $L^\infty$ cutoff bound gives
\eqref{eq:finite-product-cutoff-tail}.  The two-cutoff estimate follows from
$|M_\Lambda-M_{\Lambda'}|^2\le2|M_\Lambda-1|^2+2|M_{\Lambda'}-1|^2$.
Integration proves \eqref{eq:finite-product-integral-tail}, and the same
argument applies to two cutoff families and to repeated legs.
\end{proof}

For a real test function $\varphi$, define the cutoff forcing directly by
\begin{equation}\label{eq:cutoff-forcing-isonormal}
 \xi_{a,\Lambda}(\varphi)
 :=\mathbf W_a\bigl(\mathfrak h_a(D_x)m_\Lambda(D_x)\varphi\bigr).
\end{equation}
Since both multipliers are real and even, this agrees with the distributional
notation $\xi_{a,\Lambda}=\pi_\Lambda\xi_a$ and has spectral density
$|m_\Lambda\mathfrak h_a|^2$.  Let $y_\Lambda=\pi_\Lambda y$ and define
\[
  \Psi_{a,\Lambda}=I_a\xi_{a,\Lambda},\qquad
  \Theta_\Lambda=\Psi_{\W,\Lambda}\Psi_{\K,\Lambda},\qquad
  V_{a,\Lambda}=I_a\Theta_\Lambda.
\]
For a fixed adapted localization pair,
\begin{equation}\label{eq:finite-auxiliary-forcing}
  \zeta_{a,\Lambda}^{\chi,\rho}
  :=L_a\bigl[\rho(\Psi_{a,\Lambda}+V_{a,\Lambda})\bigr]
    -\chi\Theta_\Lambda.
\end{equation}
The finite-cutoff auxiliary equation is
\begin{equation}\label{eq:cutoff-system}
\begin{aligned}
  \LW u_\Lambda^{\chi,\rho}
  &=\chi u_\Lambda^{\chi,\rho}v_\Lambda^{\chi,\rho}
    +\zeta_{\W,\Lambda}^{\chi,\rho},\\
  \LK v_\Lambda^{\chi,\rho}
  &=\chi u_\Lambda^{\chi,\rho}v_\Lambda^{\chi,\rho}
    +\zeta_{\K,\Lambda}^{\chi,\rho}.
\end{aligned}
\end{equation}
The baseline in \eqref{eq:finite-auxiliary-forcing} is spatially compact,
and every remainder nonlinearity is multiplied by $\chi$.  After subtracting
the baseline, \eqref{eq:cutoff-system} is therefore a global unweighted
Sobolev mild problem with spatially smooth random coefficients.  On the part
of the backward cone where $\chi=\rho=1$, it agrees with the spectrally
regularized form of \eqref{eq:main-system}.

\begin{remark}[Finite-cutoff equation]
The cutoff acts on the Gaussian forcings and on the deterministic Cauchy data.
The nonlinear term in \eqref{eq:cutoff-system} remains the local product
$\chi u_\Lambda v_\Lambda$; hence the cutoff equation retains finite
propagation.  Projected variants introduced later concern individual stochastic
terms only.
\end{remark}

\subsection{Function spaces}

\begin{assumption}[Admissible exponents]\label{ass:exponents}
Let $\betaW,\betaK$ and the derived quantities
$\betastar,\betasum,\betagamma$ be as in \eqref{eq:parameter-vector}.  The
parameters $s_1,s_2,\kappa,\sigma$ satisfy
\begin{equation}\label{eq:parameter-window}
\begin{gathered}
  \frac14+\betastar<s_1<\frac12-\betastar,
  \qquad
  \frac12+\betastar<s_2<
  \min\left\{1-\betagamma,s_1+\frac14\right\},\\
  0<\kappa<\min\left\{
  \frac12-\betastar-s_1,
  \ s_2-\frac12-\betastar,
  \ \frac{1-s_2-\betagamma}{2},
  \ \frac14-\betastar,
  \ \frac12-\betasum-2\betastar
  \right\},\\
  2\betastar<\sigma<\min\left\{
  s_1,s_2,\frac14,\frac12-\betasum-\kappa
  \right\}.
\end{gathered}
\end{equation}
\end{assumption}

\begin{lemma}[Non-emptiness of the exponent range]
\label{lem:exponent-window-nonempty}
If $\betastar<1/8$, then the set of quadruples
$(s_1,s_2,\kappa,\sigma)$ satisfying \eqref{eq:parameter-window} is nonempty.
\end{lemma}

\begin{proof}
Since $\betagamma\le3\betastar$, both
$1/4-2\betastar$ and $1/2-4\betastar$ are positive.  Choose
\[
 0<\varepsilon<\frac14\min\left\{
 \frac14-2\betastar,\frac12-4\betastar\right\},
 \qquad
 s_1=\frac14+\betastar+2\varepsilon,
 \qquad
 s_2=\frac12+\betastar+\varepsilon.
\]
Then the first line of \eqref{eq:parameter-window} holds.  All upper bounds
for $\kappa$ are strictly positive, so $\kappa>0$ may be chosen below their
minimum.  After decreasing $\kappa$ if necessary,
\[
 2\betastar<
 \min\left\{s_1,s_2,\frac14,
 \frac12-\betasum-\kappa\right\},
\]
and one may choose $\sigma$ between these two quantities.
\end{proof}

\begin{remark}[Choice of exponents]
\label{rem:exponent-restrictions}
The threshold $\betastar<1/8$ is imposed only to make the complete stochastic
lift and deterministic fixed point share one strict exponent window; several
individual stochastic objects exist beyond it.  The upper bound on $s_1$
controls $w\prec\Psi_a$, while $s_2>1/2+\betastar$ makes
$Y\circ\Psi_a$ classical.  The inequalities
$s_2<1-\betagamma$ and $s_2<s_1+1/4$ place the cubic symbols and the endpoint
quadratic remainder in the $Y$-source space.  For an operator label
$(a;b,c)$, the stochastic order is
$\beta_{b,c}=\beta_b+\beta_c\le2\betastar$; its independent construction uses
\[
 0<s<1-\beta_{b,c},\qquad
 \beta_{b,c}<\sigma<1-\beta_{b,c}.
\]
The common fixed-point choice strengthens this to the bounds in
\eqref{eq:parameter-window}, in particular
$2\betastar<\sigma<1/2-\betasum-\kappa$.
\end{remark}

\begin{definition}[Choice of auxiliary exponents]\label{def:loss-hierarchy}
Set
\begin{equation}\label{eq:strict-margin}
\begin{aligned}
 \mathfrak m_*:=\min\Bigl\{&
 s_1-\tfrac14-\betastar,
 \tfrac12-\betastar-s_1-\kappa,
 s_2-\tfrac12-\betastar-\kappa,\\
 &1-s_2-\betagamma-2\kappa,
 s_1+\tfrac14-s_2,
 \kappa,
 \tfrac14-\betastar-\kappa,\\
 &\tfrac12-\betasum-2\betastar-\kappa,
 \sigma-2\betastar,
 s_1-\sigma,
 s_2-\sigma,\\
 &\tfrac14-\sigma,
 \tfrac12-\betasum-\kappa-\sigma
 \Bigr\}>0.
\end{aligned}
\end{equation}
Choose a master loss $0<\varepsilon_0<\mathfrak m_*/100$.
Every disposable dyadic loss $\varepsilon$, $\delta$, or $\eta$ is chosen
below $\varepsilon_0/10$.  Let $C_{\mathrm{inc}}\ge1$ be the maximum of $1$ and the finitely many
coefficients multiplying $\theta$ in the frequency exponents of the increment
estimates.  After the dyadic losses are fixed, choose
\[
 0<\theta<\frac{\varepsilon_0}{10C_{\mathrm{inc}}}.
\]
Since $\betasum<1/4$ and $\beta_a<1/8$, this choice also gives
\[
 \theta<\frac12-\betasum,
 \qquad
 2\beta_a+\theta+\varepsilon<1
 \quad(a\in\mathfrak C)
\]
for every disposable loss $\varepsilon<\varepsilon_0/10$.  Finally choose the
Gaussian moment exponent $p_0$ so that $p_0\theta>4$.  These choices are kept
fixed throughout the proof.
\end{definition}

Recall the partition in \eqref{eq:operator-index-set}.  Its algebraic
relations are
\begin{equation}\label{eq:operator-index-relations}
 \begin{cases}
  b=c,\quad a=b^\perp,&(a;b,c)\in\mathfrak L_{\mathrm{diag}},\\
  b\ne c,\quad a=b,&(a;b,c)\in\mathfrak L_{\mathrm{off}}.
 \end{cases}
\end{equation}
For a label $(a;b,c)\in\mathfrak L$, set
\begin{equation}\label{eq:operator-pair-order}
 \beta_{b,c}:=\beta_b+\beta_c,
 \qquad 0\le\beta_{b,c}\le2\betastar.
\end{equation}
For a cubic label $a\in\mathfrak C$, set
\begin{equation}\label{eq:cubic-order}
 \beta_{\Gamma,a}:=\betasum+\beta_a,
 \qquad 0\le\beta_{\Gamma,a}\le\betagamma.
\end{equation}
The common input space for the operators $T^{a;b,c}$ is
\begin{equation}\label{eq:E-space}
\begin{aligned}
  \|w\|_{E_T^{2,\sigma}}
  :={}&\|w\|_{C_TL^2}+\|\partial_tw\|_{C_TH^{-1}}\\
  &+\|w\|_{L_T^\infty B_{2,\infty}^{\sigma}}
   +\|\partial_tw\|_{L_T^\infty B_{2,\infty}^{\sigma-1}}.
\end{aligned}
\end{equation}
\begin{remark}[Energy component of the operator domain]
The exponent $-1$ is matched to the deterministic contraction.  In a
same-color block of color $b$, summation of the high-frequency loop leaves
$Q^{-1+2\beta_b}$ on the differentiated input; hence
\[
 H^{-1}\longrightarrow H^{s-1}
 \qquad\text{for every }0<s<1-2\beta_b.
\]
For all four index triples, one may use
$0<s<1-2\betastar$.
The centered branch uses only the $C_TL^2$ and
$L_T^\infty B_{2,\infty}^{\sigma}$ components.  The localized first Picard
paths belong to $C_TH^{-1}$ by
\cref{lem:compact-support-continuity,lem:complete-and-E}.
\end{remark}
Every spatial occurrence of $B_{2,\infty}^r$ in a stochastic, source, or
operator space denotes the little-Besov closure $b_{2,\infty}^r$, unless the
large space is explicitly named.  Compactly localized $C_T\mathcal C^r$
random variables are likewise taken in the separable closure of smooth
compactly supported paths.  For an $L_T^\infty B_{2,\infty}^r$ component the
closure is characterized by the uniform tail condition
\begin{equation}\label{eq:uniform-little-besov-tail}
 \lim_{L\to\infty}\;
 \sup_{0\le t\le T}\sup_{\substack{N\in\Dyd\\N>L}}
 N^r\|P_Nf(t)\|_{L^2}=0.
\end{equation}
Thus $E_T^{2,\sigma}$, $\widetilde X_T^{s_1,\sigma}$, and
$\widetilde Y_T^{s_2,\sigma}$ are the closures of smooth finite-spatial-block
paths in their displayed intersection norms.  The energy components give
fixed-block time continuity, so spatial truncation followed by time
mollification yields the same closure.  This convention supplies separable
ranges without changing any ambient Besov estimate.

\begin{lemma}[Extension from finite spatial blocks]
\label{lem:finite-block-extension}
Let $X_1,\ldots,X_m$ be Banach spaces obtained as the closures of subspaces
$X_1^0,\ldots,X_m^0$ of smooth finite-spatial-block paths.  If a $m$-linear
map $T$ satisfies
\[
 \|T(f_1,\ldots,f_m)\|_Y
 \le C\prod_{\nu=1}^m\|f_{\nu}\|_{X_{\nu}},
 \qquad f_{\nu}\in X_{\nu}^0,
\]
for a Banach space $Y$, then it extends uniquely to a continuous $m$-linear
map on $X_1\times\cdots\times X_m$.  The same assertion holds with
$Y$ replaced by $L^p(\Omega;Y)$ or by a Banach space of bounded operators.
If the approximating outputs have a uniform vanishing Littlewood--Paley tail,
the extension takes values in the corresponding little-Besov closure.
\end{lemma}

\begin{proof}
Approximate each input in its dense finite-block subspace.  The displayed
multilinear bound makes the images Cauchy, proves independence of the
approximations, and passes unchanged to complete Bochner and operator-norm
targets.  A uniform high-frequency tail places the limit in the corresponding
little-Besov closure.
\end{proof}

\begin{lemma}[Canonical Besov traces in the completed intersections]
\label{lem:completed-besov-traces}
Let $I=[t_0,t_1]$ be compact and let $\alpha,r\in\R$.  Consider the closure of
smooth finite-spatial-block paths in
\[
 C_IH^\alpha\cap L_I^\infty B_{2,\infty}^r
\]
with the intersection norm and the uniform little-Besov convention
\eqref{eq:uniform-little-besov-tail}.  Every element of this closure has a
unique representative in
\[
 C\bigl(I;b_{2,\infty}^r\bigr),
\]
and the endpoint evaluations are bounded.  The canonical
$C_IH^\alpha$ representative satisfies
\begin{equation}\label{eq:completed-besov-trace-bound}
 \sup_{t\in I}\|f(t)\|_{b_{2,\infty}^r}
 \le \|f\|_{L_I^\infty B_{2,\infty}^r},
 \qquad
 f_n\longrightarrow f\ \text{in the intersection norm}
 \Longrightarrow
 f_n\longrightarrow f\ \text{in }C_Ib_{2,\infty}^r.
\end{equation}
In particular, if the $H^\alpha$ trace of $f$ vanishes at an endpoint, then
its $b_{2,\infty}^r$ trace vanishes there as well.
\end{lemma}

\begin{proof}
If $f_n$ is a smooth finite-block approximating sequence, then
$\|f_n-f_m\|_{L_I^\infty B_{2,\infty}^r}$ equals its genuine time supremum,
so $(f_n)$ is Cauchy in $C(I;b_{2,\infty}^r)$.  Its limit there and its limit
in $C_IH^\alpha$ agree in $C(I;\mathcal S')$, which gives the unique
representative and \eqref{eq:completed-besov-trace-bound}.  Closedness of
$b_{2,\infty}^r$ gives the slice statement, and a zero $H^\alpha$ trace is
the zero distribution, hence has zero Besov trace.
\end{proof}

For a compact interval $I=[t_0,t_1]\subset[0,T_0]$, the notation
$E_I^{2,\sigma}$ denotes the time-translated space defined by \eqref{eq:E-space}, with
all time norms taken over $I$.  Put
\begin{equation}\label{eq:zero-history-E}
 E_{I,0}^{2,\sigma}:=
 \{w\in E_I^{2,\sigma}:w(t_0)=0,\ \partial_tw(t_0)=0\}.
\end{equation}
The endpoint equalities are taken in $L^2$ and $H^{-1}$,
respectively.  Lemma~\ref{lem:completed-besov-traces}, applied separately to
$w$ and $\partial_tw$, upgrades the two $L^\infty$ Besov components to
canonical continuous little-Besov paths.  Thus the energy equalities in
\eqref{eq:zero-history-E} force the corresponding Besov traces to vanish as
well.

\begin{lemma}[Time extension, causality, and restricted operator norms]
\label{lem:time-extension-causality}
Let $0<T\le T_0$.  There is a bounded linear extension
\begin{equation}\label{eq:E-time-extension}
 \mathfrak E_{T,T_0}:E_T^{2,\sigma}\longrightarrow E_{T_0}^{2,\sigma},
 \qquad
 \mathfrak R_T\mathfrak E_{T,T_0}=I,
\end{equation}
whose norm is bounded by an absolute constant independently of $T$ and $T_0$.
The analogous assertion holds after translating the interval.

Call an operator $\mathcal A$ on $E_{T_0}^{2,\sigma}$ causal if inputs that
agree on $[0,t]$ have outputs that agree on $[0,t]$, for every $t\le T_0$.
For a causal operator, define
\begin{equation}\label{eq:causal-restriction}
 \mathcal A_{[0,T]}:=\mathfrak R_T\mathcal A\mathfrak E_{T,T_0}.
\end{equation}
This restriction is independent of the chosen bounded extension.  More
generally, if $I=[t_0,t_1]$ and $w\in E_{I,0}^{2,\sigma}$, extend $w$ by zero on
$[0,t_0]$ and by any bounded extension after $t_1$; the restriction of
$\mathcal A$ to $I$ is again independent of the extension after $t_1$.  We
denote this zero-history operator by $\mathcal A_I$.
\end{lemma}

\begin{proof}
Choose a smooth $\psi_T:[0,\infty)\to[0,T/3]$ with
$\psi_T(0)=0$, $\psi_T'(0)=1$, and $|\psi_T'|\le1$, and set
\[
 (\mathfrak E_{T,T_0}w)(T+h)
 :=3w(T-\psi_T(h))-2w(T-2\psi_T(h)).
\]
The value and first derivative match at $h=0$ because $3-2=1$ and
$-3+4=1$.  The energy and Besov suprema increase by at most an absolute
constant, and the uniform little-Besov tail is preserved by this fixed finite
linear combination of time samples.  Define the formula first on smooth
finite-block paths and pass to the completion; translation treats a general
compact interval.

Causality makes the restriction independent of the extension to the right.
For $w\in E_{I,0}^{2,\sigma}$, Lemma~\ref{lem:completed-besov-traces} gives
zero Besov traces at the left endpoint, so extension by zero to the left
belongs to the same completed path space and creates no endpoint Dirac mass
in the time derivative.  Causality again proves the zero-history assertion.
\end{proof}

For $0<T\le1$, define the solution norms
\begin{align}
  \norm{X}_{X_T^{s_1}}
  &:={}
  \norm{X}_{C_TH^{s_1}}
  +\norm{\partial_tX}_{C_TH^{s_1-1}}
  +\norm{X}_{L_T^8W^{s_1-1/4,8/3}},
  \label{eq:X-space}\\
  \norm{Y}_{Y_T^{s_2}}
  &:={}
  \norm{Y}_{C_TH^{s_2}}
  +\norm{\partial_tY}_{C_TH^{s_2-1}}
  +\norm{Y}_{L_T^4W^{s_2-1/2,4}}.
  \label{eq:Y-space}
\end{align}
We also use the intersection norms
\begin{align}
  \norm{X}_{\wt X_T^{s_1,\sigma}}
  &:={}
  \norm{X}_{X_T^{s_1}}
  +\norm{X}_{L_T^\infty B_{2,\infty}^{\sigma}}
  +\norm{\partial_tX}_{L_T^\infty B_{2,\infty}^{\sigma-1}},
  \label{eq:X-intersection-norm}\\
  \norm{Y}_{\wt Y_T^{s_2,\sigma}}
  &:={}
  \norm{Y}_{Y_T^{s_2}}
  +\norm{Y}_{L_T^\infty B_{2,\infty}^{\sigma}}
  +\norm{\partial_tY}_{L_T^\infty B_{2,\infty}^{\sigma-1}}.
  \label{eq:Y-intersection-norm}
\end{align}
Set
\begin{equation}\label{eq:Z-space}
  \cZ_T
  :=(\wt X_T^{s_1,\sigma})^2\times(\wt Y_T^{s_2,\sigma})^2.
\end{equation}

Throughout, the symbols $\varepsilon$, $\delta$, $\eta$, $\theta$, and
$p_0$ are chosen according to Definition~\ref{def:loss-hierarchy}.

\begin{definition}[Localized source topologies]\label{def:source-topology}
For $r\in\{s_1,s_2\}$ set
\[
 \mathfrak R_T^{r,\sigma}
 :=L_T^1H^{r-1}\cap L_T^1B_{2,\infty}^{\sigma-1}.
\]
For the endpoint deterministic quadratic sector set
\[
 \mathfrak D_T^{s_2}
 :=\left\{F:\ \langle D\rangle^{s_2-1/2}F
 \in L_{t,x}^{4/3}\right\},
 \qquad
 \|F\|_{\mathfrak D_T^{s_2}}
 :=\|\langle D\rangle^{s_2-1/2}F\|_{L_{t,x}^{4/3}}.
\]
The $Y$-source sum space is
\begin{equation}\label{eq:Y-source-sum-space}
 \mathfrak S_T^{s_2,\sigma}
 :=\mathfrak R_T^{s_2,\sigma}
 +\left(\mathfrak D_T^{s_2}\cap
 L_T^1B_{2,\infty}^{\sigma-1}\right),
\end{equation}
with the usual infimum norm over decompositions.  The componentwise nonlinear
source topology used below is
\begin{equation}\label{eq:nonlinear-source-topology}
 C_T\mathcal C^{-1-\betasum-\kappa}
 \times\mathfrak R_T^{s_1,\sigma}
 \times\mathfrak S_T^{s_2,\sigma},
\end{equation}
corresponding to $(\chi\Theta,\mathscr F_X^{\chi,\rho},\mathscr F_Y^{\chi,\rho})$.
\end{definition}

\begin{lemma}[Continuity of the nonlinear-source sum]
\label{lem:source-summation-continuity}
The map
\[
 (G_0,G_1,G_2)\longmapsto G_0+G_1+G_2
\]
from the product space in \eqref{eq:nonlinear-source-topology} to
$\mathcal D'((0,T)\times\R^3)$ is continuous.  Consequently, if
\[
 (G_{0,n},G_{1,n},G_{2,n})\longrightarrow(G_0,G_1,G_2)
\]
in that product topology, then the sums converge in space--time
distributions.  In particular, if finite-cutoff fields satisfy
\[
 \chi u_n v_n=G_{0,n}+G_{1,n}+G_{2,n},
\]
then the limit of the right-hand side uniquely determines the distributional
limit of the localized products.
\end{lemma}

\begin{proof}
Pair with $\varphi\in C_c^\infty((0,T)\times\R^3)$.  The first component is
controlled by the continuous embedding
$\mathcal C^{-1-\betasum-\kappa}\hookrightarrow\mathcal D'$.  The
$\mathfrak R_T^{r,\sigma}$ components are bounded by
$\|G\|_{L_T^1H^{r-1}}\sup_t\|\varphi(t)\|_{H^{1-r}}$.  For the endpoint
piece, Fourier duality and H\"older give
\[
 \left|\int\langle G_D,\varphi\rangle\dd t\right|
 \le\|\la D\ra^{s_2-1/2}G_D\|_{L^{4/3}_{t,x}}
     \|\la D\ra^{1/2-s_2}\varphi\|_{L^4_{t,x}}.
\]
Taking the infimum over sum-space decompositions proves continuity and hence
the final assertion.
\end{proof}

The deterministic Cauchy-data space is
\begin{equation}\label{eq:data-space}
  (u_0,u_1),(v_0,v_1)
  \in H^{s_2}\times H^{s_2-1}.
\end{equation}
Since $\sigma<s_2$, this space embeds continuously into the Besov data space
used in the fixed-point space $\mathcal Z_T$.

\begin{lemma}[Strong cutoff convergence of the deterministic data]
\label{lem:data-cutoff-convergence}
For every admissible cutoff family and every data pair in
\eqref{eq:data-space},
\[
  \pi_\Lambda f\to f\quad\text{in }H^{s_2}\cap B_{2,\infty}^{\sigma},
  \qquad
  \pi_\Lambda g\to g\quad\text{in }H^{s_2-1}\cap
  B_{2,\infty}^{\sigma-1}.
\]
\end{lemma}

\begin{proof}
Sobolev convergence follows from dominated convergence.  For the Besov norm,
split at a fixed dyadic $L$: finitely many low blocks converge by pointwise
multiplier convergence, while
\[
 \sup_{N\ge L}N^\sigma\|P_Nf\|_2
 \le L^{\sigma-s_2}\|f\|_{H^{s_2}},
\]
and the analogous estimate holds for $g$.  Let first $L$ and then $\Lambda$
tend to infinity.
\end{proof}

\begin{definition}[Convergence of partially defined random variables]
\label{def:partial-probability-convergence}
Let $(S,d)$ be a separable metric space, let $G\in\mathcal F$, and let
$(A_n,X_n)_{n\ge1}$ consist of events $A_n\in\mathcal F$ and strongly
measurable $S$-valued random variables $X_n$.  For a strongly measurable
$S$-valued random variable $X$, we say that $(A_n,X_n)$ converges to $X$ in
probability on $G$ if, for every $\varepsilon>0$,
\begin{equation}\label{eq:partial-probability-convergence}
 \Pp\bigl(G\cap(A_n^c\cup\{d(X_n,X)>\varepsilon\})\bigr)
 \longrightarrow0.
\end{equation}
The values assigned to $X_n$ on $A_n^c$ do not affect this notion.
\end{definition}

\subsection{Local solution concept and main results}\label{subsec:local-solutions}

For $0<T\le T_0$ define the backward cone
\begin{equation}\label{eq:backward-cone}
  \mathcal Q_{K,T}
  :=\left\{(t,x)\in[0,T]\times\R^3:
  \dist(x,K)\le \speedmax(T-t)\right\}.
\end{equation}

\begin{definition}[Local paracontrolled solution]\label{def:local-solution}
Let $K\Subset\R^3$ and $0<T\le T_0$.  Let $\Xi$ satisfy
Definition~\ref{def:Xi-assumptions}, and let
$\Theta\in C_T\mathcal C_{\mathrm{loc}}^{-1-\betasum-\kappa}$.  For every
$\chi\in C_c^\infty(\R^3)$ one then has
$\chi\Theta\in C_T\mathcal C^{-1-\betasum-\kappa}$; only this localized
quantity enters the fixed-point problem.  A triple $(u,v,\mathcal N)$ is a local paracontrolled solution associated with
$(\Theta,\Xi)$ on $[0,T]\times K$ if there exist an adapted localization pair
$(\chi,\rho)$ and a fixed point
$Z=(X_{\W},X_{\K},Y_{\W},Y_{\K})\in\mathcal Z_T$ of the map defined by
\eqref{eq:mild-X}--\eqref{eq:mild-Y} such that, on $\mathcal Q_{K,T}$,
\begin{equation}\label{eq:local-reconstruction}
\begin{aligned}
  u&=\rho(\Psi_{\W}+V_{\W})+X_{\W}+Y_{\W},\\
  v&=\rho(\Psi_{\K}+V_{\K})+X_{\K}+Y_{\K},
\end{aligned}
\end{equation}
and
\begin{equation}\label{eq:local-source}
  \mathcal N
  =\Theta+\mathscr F_X^{\chi,\rho}(Z)
          +\mathscr F_Y^{\chi,\rho}(Z),
\end{equation}
where the source maps are defined in
\eqref{eq:localized-X-source} and \eqref{eq:localized-Y-source}.  Every localized component of $\Xi$, together with $\chi\Theta$, must be
formed using this same pair $(\chi,\rho)$; components constructed with
different localization pairs are not combined in one fixed-point problem.  The
remainder terms satisfy
\[
 (X_a(0),\partial_tX_a(0))=(0,0),
 \qquad a\in\mathfrak C,
\]
while
\[
 (Y_{\W}(0),\partial_tY_{\W}(0))=(u_0,u_1),
 \qquad
 (Y_{\K}(0),\partial_tY_{\K}(0))=(v_0,v_1).
\]
Moreover,
\[
  L_{\W}u=\mathcal N+\xi_{\W},
  \qquad
  L_{\K}v=\mathcal N+\xi_{\K}
\]
in $\mathcal D'((0,T)\times K)$.  Two such solutions are identified if
their fields and nonlinear sources agree on $[0,T]\times K$.
\end{definition}

\begin{definition}[Solution on a measurable random interval]
\label{def:random-interval-solution}
Let $\tau:\Omega\to[0,1]$ be measurable.  A paracontrolled solution on
$[0,\tau)$ means a compatible family
\[
 \bigl\{Z^{(q)}:q\in\mathbb Q\cap(0,1]\bigr\},
\]
where $Z^{(q)}$ is a strongly measurable $\mathcal Z_q$-valued random
variable, is set equal to $0$ on $\{q\ge\tau\}$, solves the fixed-point
problem on $\{q<\tau\}$, and satisfies
\begin{equation}\label{eq:random-horizon-compatibility}
 \mathfrak R_{q_1}Z^{(q_2)}=Z^{(q_1)}
 \quad\text{on }\{q_2<\tau\},
 \qquad 0<q_1<q_2,
 \quad q_1,q_2\in\mathbb Q.
\end{equation}
The reconstructed fields and source are required to satisfy the same
restriction compatibility.  For a deterministic $T$, the restriction on
$\{T<\tau\}$ is obtained from any rational $q\in(T,\tau)$; the countable
measurable selection and independence of this choice are proved in
Corollary~\ref{cor:horizon-gluing}.  The resulting restriction is independent of the chosen rational $q$.
\end{definition}

\begin{definition}[Compatible local solutions on an exhaustion]
\label{def:compatible-local-solutions}
Let $(K_n)_{n\ge1}$ be an increasing compact exhaustion of $\R^3$ with
$K_n\subset\operatorname{int}K_{n+1}$, and let
$\tau_{n+1}\le\tau_n$ be measurable lifetimes.  A family
\[
 \mathfrak U_n=(u_n,v_n,\mathcal N_n),\qquad n\ge1,
\]
of local paracontrolled solutions on $[0,\tau_n)\times K_n$ is compatible
if, for $n\le m$ and every rational $q>0$,
\begin{equation}\label{eq:local-solution-compatibility}
 \mathfrak U_m|_{[0,q]\times K_n}=\mathfrak U_n|_{[0,q]\times K_n}
 \quad\text{on }\{q<\tau_m\}.
\end{equation}
For a compact $K\Subset\R^3$, let $n(K)$ be the smallest index with
$K\subset\operatorname{int}K_{n(K)}$ and set $\tau_K:=\tau_{n(K)}$.
The restriction of $\mathfrak U_{n(K)}$ defines the local solution on $K$ up
to time $\tau_K$; every larger index gives the same restriction on its
possibly shorter lifetime.
\end{definition}

\begin{theorem}[Local well-posedness]
\label{thm:local-wp}
Let $\parvec$ satisfy \eqref{eq:parameter-vector}, let
Assumption~\ref{ass:exponents} hold, let the Cauchy data satisfy
\eqref{eq:data-space}, and let $\xi_{\W},\xi_{\K}$ be the independent white-in-time, spatially stationary Gaussian forcings generated by admissible Fourier symbols in Definition~\ref{def:gaussian-forcing-profile}.  There is a compatible family of local solutions in the sense of
Definition~\ref{def:compatible-local-solutions}, based on the exhaustion
$K_n=\overline{B(0,n)}$, whose lifetimes are almost surely positive.  For every
compact $K\Subset\R^3$, it induces a measurable lifetime $\tau_K$ and a local
paracontrolled solution $(u,v,\mathcal N)$ on $[0,\tau_K)\times K$ in the
sense of Definitions~\ref{def:local-solution}
and~\ref{def:random-interval-solution}.  The lifetimes may be chosen so that
\[
 K_1\subset K_2\quad\Longrightarrow\quad \tau_{K_1}\ge\tau_{K_2}.
\]
On every common existence interval, the local solution is independent of the
adapted localization pair.  For fixed $(\Theta,\Xi)$, it is unique in the
class of Definition~\ref{def:local-solution}.  The stability assertion is
understood on a common deterministic interval and in a common contraction
ball: for every compact $K$, every $T\in(0,1]$, and every realization in
$\{T<\tau_K\}$, there exist $R_{K,T}<\infty$, $\delta_{K,T}>0$, and
$C_{K,T}<\infty$ such that the reference map is a strict contraction of
the closed radius-$R_{K,T}$ ball of $\mathcal Z_T$.  If
\[
 d_T\bigl((\Xi,y),(\widetilde\Xi,\widetilde y)\bigr)<\delta_{K,T},
\]
then the perturbed map for the same localization pair is also a strict
contraction of that ball, has a unique fixed point on $[0,T]$, and
\[
 \|Z-\widetilde Z\|_{\mathcal Z_T}
 \le C_{K,T}
 d_T\bigl((\Xi,y),(\widetilde\Xi,\widetilde y)\bigr),
\]
where $d_T$ is defined in \eqref{eq:Xi-distance}.
\end{theorem}

\begin{theorem}[Spectral approximation]
\label{thm:spectral-approximation}
Under the assumptions of Theorem~\ref{thm:local-wp}, fix a compact
$K\Subset\R^3$ and the fixed localization pair
$(\chi_K^\circ,\rho_K^\circ)$ from
\eqref{eq:reference-localization-and-lifetime}.  Denote the corresponding solution of the localized auxiliary system and its
fixed-point variables by
$(u^\circ,v^\circ,\mathcal N^\circ,Z^\circ)$, and denote their
finite-cutoff counterparts by
$(u_\Lambda^\circ,v_\Lambda^\circ,
  \mathcal N_\Lambda^\circ,Z_\Lambda^\circ)$.
Their restrictions to the backward cone $\mathcal Q_{K,T}$ represent the
local fields and source.

\begin{enumerate}[label=\textup{(\roman*)},leftmargin=2.45em]
\item \emph{Fixed-profile dyadic cutoffs.}
For every fixed-profile dyadic cutoff there is an event $\Omega_K^\circ$ of
probability one, which may depend on the chosen profile, such that, for every $\omega\in\Omega_K^\circ$ and every
$T<\tau_K(\omega)$, the finite-cutoff fixed point exists for all sufficiently
large dyadic $\Lambda$ and
\begin{align}
  Z_\Lambda^\circ&\longrightarrow Z^\circ
  &&\text{in }\mathcal Z_T,
  \label{eq:main-fixed-point-convergence}\\
  \bigl(\chi_K(u_\Lambda^\circ-u^\circ),
  \chi_K(v_\Lambda^\circ-v^\circ)\bigr)
  &\longrightarrow0
  &&\text{in }C([0,T];\mathcal D'(\R^3))^2
  \label{eq:main-field-convergence}
\end{align}
for every $\chi_K\in C_c^\infty(\R^3)$ equal to one near $K$.  The three source terms converge in the topology
\eqref{eq:nonlinear-source-topology}.  If
$\eta_K\in C_c^\infty(\R^3)$ satisfies
\begin{equation}\label{eq:inner-source-test-cutoff}
 \eta_K\equiv1\ \text{near }K,
 \qquad \eta_K\chi_K^\circ=\eta_K,
\end{equation}
then
\[
 \eta_Ku_\Lambda^\circ v_\Lambda^\circ
 \longrightarrow \eta_K\mathcal N^\circ
 \quad\text{in }\mathcal D'((0,T)\times\R^3).
\]

\item \emph{Arbitrary admissible cofinal cutoffs.}
Let $(\pi_{\Lambda_n})_{n\ge1}$ be an admissible cofinal cutoff sequence and
let $T\in(0,1]$ be deterministic.  There are Borel events
$\mathcal A_{n,T}^K$ and strongly measurable $\mathcal Z_T$-valued random
variables $Z_{n,T}^\circ$ whose restriction to $\mathcal A_{n,T}^K$ is the
finite-cutoff fixed point.  If $Z_T^\circ$ denotes the strongly measurable
restriction of $Z^\circ$ to $[0,T]$ on $\{T<\tau_K\}$, then
$(\mathcal A_{n,T}^K,Z_{n,T}^\circ)$ converges to $Z_T^\circ$ in probability
on $\{T<\tau_K\}$ in the sense of
Definition~\ref{def:partial-probability-convergence}.  The same conclusion
holds for every continuous seminorm on
$C([0,T];\mathcal D'(\R^3))^2$ and for the source norms in
\eqref{eq:nonlinear-source-topology}.
\end{enumerate}

For any other fixed localization pair $(\chi,\rho)$ adapted to $(K,1)$,
the same conclusions hold on its existence interval, and the limiting fields
and source agree with those obtained from the fixed pair
$(\chi_K^\circ,\rho_K^\circ)$ on every common interval.
\end{theorem}

\begin{remark}[Scope of the main theorems]
\label{rem:main-theorem-scope}
The family in Theorem~\ref{thm:local-wp} is local in space, with a measurable
almost surely positive lifetime $\tau_K$ for each compact $K$.  Uniqueness is
understood for fixed enhanced data $(\Theta,\Xi)$ in the class of
Definition~\ref{def:local-solution}.  In
Theorem~\ref{thm:spectral-approximation}, the almost-sure statement is tied to
a fixed cutoff profile, while general cofinal cutoffs are treated through the
partially defined random variables of
Definition~\ref{def:partial-probability-convergence}.  The proof uses only the
measurability of $\tau_K$, and the sufficient threshold $\betastar<1/8$ is not
claimed to be optimal.
\end{remark}

\begin{remark}[Reference realization]
\label{rem:reference-cutoff-realization}
The pathwise construction starts from one smooth cutoff profile.  Its dyadic
approximations give measurable realizations of the stochastic data and the
solution, and the cutoff-tail estimates identify their limits with the
profile-independent objects.  Theorem~\ref{thm:spectral-approximation}(ii)
shows that every admissible cofinal sequence converges to the same solution.
For a fixed profile, full-sequence convergence follows from the finite-state
stabilization of dyadic shells in Lemma~\ref{lem:finite-state-cutoff}.
\end{remark}

\section{Paracontrolled formulation and stochastic terms}\label{sec:paracontrolled}

Fix an adapted localization pair $(\chi,\rho)$.  The identities below are first
established at finite spectral cutoff.  Their right-hand sides then define the
corresponding limiting source terms.  Only the reconstructed baseline is
multiplied by $\rho$.

\subsection{The paracontrolled expansion}

Set
\begin{equation}\label{eq:theta-def}
  \Theta=\Psi_{\W}\Psi_{\K},
  \qquad V_a=I_a(\Theta),\quad a\in\mathfrak C.
\end{equation}
Write
\begin{equation}\label{eq:three-layer}
\begin{aligned}
  u^{\chi,\rho}&=\rho(\Psi_{\W}+V_{\W})+X_{\W}+Y_{\W},\\
  v^{\chi,\rho}&=\rho(\Psi_{\K}+V_{\K})+X_{\K}+Y_{\K},
\end{aligned}
\end{equation}
and define the functions
\begin{equation}\label{eq:regular-input-def}
  \mathscr U_a^\rho:=\rho V_a+X_a+Y_a,
  \qquad a\in\mathfrak C.
\end{equation}
Since $\rho=1$ on a neighborhood of $\supp\chi$, the auxiliary equation
\eqref{eq:auxiliary-system} and the definition of
$\zeta^{\chi,\rho}$ give, for every $a\in\mathfrak C$,
\begin{equation}\label{eq:remainder-equation}
  L_a(X_a+Y_a)
  =\chi\left(
      \sum_{b\in\mathfrak C}\Psi_b\mathscr U_{b^\perp}^\rho
      +\mathscr U_{\W}^\rho\mathscr U_{\K}^\rho
    \right).
\end{equation}
Thus each low--high source uses the full regular input in the opposite
channel, including the localized first Picard term and both deterministic
remainders.

Assign the low--high products to $X_a$:
\begin{equation}\label{eq:X-equation}
  L_aX_a
  =\chi\sum_{b\in\mathfrak C}
      \mathscr U_{b^\perp}^\rho\prec\Psi_b,
  \qquad X_a(0)=\partial_tX_a(0)=0,
  \qquad a\in\mathfrak C.
\end{equation}
The remaining source is assigned to $Y_a$:
\begin{equation}\label{eq:Y-equation}
  L_aY_a
  =\chi\left(
      \sum_{b\in\mathfrak C}
       \mathscr U_{b^\perp}^\rho\succ\Psi_b
      +\mathscr U_{\W}^\rho\mathscr U_{\K}^\rho
      +\sum_{b\in\mathfrak C}
       \mathscr U_{b^\perp}^\rho\circ\Psi_b
    \right),
  \qquad a\in\mathfrak C.
\end{equation}
The deterministic Cauchy data are placed in $Y_{\W}$ and $Y_{\K}$.

For $a\in\mathfrak C$, the cubic term used in the localized equation is
first defined at finite cutoff by
\begin{equation}\label{eq:localized-cubic-cutoff}
  \Gamma_{a,\Lambda}^{\chi,\rho}
  :=\chi\bigl((\rho V_{a^\perp,\Lambda})\circ\Psi_{a,\Lambda}\bigr).
\end{equation}
With the principal unlocalized convention
$\Gamma_{a,\Lambda}=V_{a^\perp,\Lambda}\circ\Psi_{a,\Lambda}$,
bilinearity of the resonant product gives
\begin{equation}\label{eq:localized-cubic-split}
  \Gamma_{a,\Lambda}^{\chi,\rho}
  =\chi\Gamma_{a,\Lambda}
   +\mathcal R_{a,\Lambda}^{\chi,\rho},
  \qquad
  \mathcal R_{a,\Lambda}^{\chi,\rho}
  :=\chi\bigl(((\rho-1)V_{a^\perp,\Lambda})\circ
                    \Psi_{a,\Lambda}\bigr).
\end{equation}
The supports of $\chi$ and $\rho-1$ are separated.  Hence the second term
has a cutoff-independent smoothing estimate, while the first term converges
to the full-space cubic distribution constructed in
\Cref{thm:cubic-fullspace}.  We denote the resulting cutoff limit by
\begin{equation}\label{eq:gamma-def}
  \Gamma_a^{\chi,\rho}
  :=\chi\Gamma_a+\mathcal R_a^{\chi,\rho}.
\end{equation}
The existence of $\mathcal R_a^{\chi,\rho}$ and the convergence of
\eqref{eq:localized-cubic-cutoff} to \eqref{eq:gamma-def} are proved in
\Cref{cor:localized-cubic-terms}.  In particular, the localized coordinate is
not identified with $\chi\Gamma_a$ alone.

For $(a;b,c)\in\mathfrak L$, define the localized operator
\begin{equation}\label{eq:T-def}
  T_{\chi}^{a;b,c}(w)
  :=\chi\Bigl(
      I_a\bigl(\chi(w\prec\Psi_b)\bigr)\circ\Psi_c
    \Bigr).
\end{equation}
After fixing $(\chi,\rho)$ we suppress the subscript $\chi$.  Substituting
the Duhamel formula for $X_{c^\perp}$ gives the uniform identity
\begin{equation}\label{eq:X-resonant-substitution}
  \chi(X_{c^\perp}\circ\Psi_c)
  =\sum_{b\in\mathfrak C}
    T^{c^\perp;b,c}(\mathscr U_{b^\perp}^\rho),
  \qquad c\in\mathfrak C.
\end{equation}
Consequently,
\begin{align}\label{eq:full-resonant-split}
 &\chi\sum_{c\in\mathfrak C}
   \mathscr U_{c^\perp}^\rho\circ\Psi_c\notag\\
 &\quad=
   \sum_{c\in\mathfrak C}\Gamma_c^{\chi,\rho}
   +\chi\sum_{c\in\mathfrak C}Y_{c^\perp}\circ\Psi_c
   +\sum_{(a;b,c)\in\mathfrak L}
     T^{a;b,c}(\mathscr U_{b^\perp}^\rho).
\end{align}
The products involving $Y$ are classical because
$s_2-1/2-\betastar-\kappa>0$ by
Assumption~\ref{ass:exponents}.

For later use, define the localized source functionals
\begin{equation}\label{eq:localized-X-source}
  \mathscr F_X^{\chi,\rho}(Z)
  :=\chi\sum_{b\in\mathfrak C}
      \mathscr U_{b^\perp}^\rho\prec\Psi_b,
\end{equation}
and
\begin{align}\label{eq:localized-Y-source}
  \mathscr F_Y^{\chi,\rho}(Z)
  :={}&\chi\left(
      \sum_{b\in\mathfrak C}
       \mathscr U_{b^\perp}^\rho\succ\Psi_b
      +\mathscr U_{\W}^\rho\mathscr U_{\K}^\rho
      +\sum_{b\in\mathfrak C}Y_{b^\perp}\circ\Psi_b
    \right)\notag\\
  &+\sum_{b\in\mathfrak C}\Gamma_b^{\chi,\rho}
    +\sum_{(a;b,c)\in\mathfrak L}
      T^{a;b,c}(\mathscr U_{b^\perp}^\rho).
\end{align}
For the finite-cutoff equation set
\[
 \mathscr U_{a,\Lambda}^{\rho}
 :=\rho V_{a,\Lambda}+X_{a,\Lambda}+Y_{a,\Lambda}.
\]
The corresponding source terms are
\begin{equation}\label{eq:localized-X-source-cutoff}
  \mathscr F_{X,\Lambda}^{\chi,\rho}(Z_\Lambda)
  :=\chi\sum_{b\in\mathfrak C}
      \mathscr U_{b^\perp,\Lambda}^{\rho}\prec\Psi_{b,\Lambda},
\end{equation}
and
\begin{align}\label{eq:localized-Y-source-cutoff}
  \mathscr F_{Y,\Lambda}^{\chi,\rho}(Z_\Lambda)
  :={}&\chi\left(
      \sum_{b\in\mathfrak C}
       \mathscr U_{b^\perp,\Lambda}^{\rho}\succ\Psi_{b,\Lambda}
      +\mathscr U_{\W,\Lambda}^{\rho}\mathscr U_{\K,\Lambda}^{\rho}
      +\sum_{b\in\mathfrak C}Y_{b^\perp,\Lambda}\circ\Psi_{b,\Lambda}
    \right)\notag\\
  &+\sum_{b\in\mathfrak C}\Gamma_{b,\Lambda}^{\chi,\rho}
    +\sum_{(a;b,c)\in\mathfrak L}
      T_{\Lambda}^{a;b,c}(\mathscr U_{b^\perp,\Lambda}^{\rho}).
\end{align}
The finite-cutoff Bony and Wick identities give
\begin{equation}\label{eq:localized-source-split}
  \chi\bigl(u_\Lambda^{\chi,\rho}v_\Lambda^{\chi,\rho}
            -\Theta_\Lambda\bigr)
  =\mathscr F_{X,\Lambda}^{\chi,\rho}(Z_\Lambda)
   +\mathscr F_{Y,\Lambda}^{\chi,\rho}(Z_\Lambda).
\end{equation}
After the cutoff is removed, \eqref{eq:localized-X-source} and
\eqref{eq:localized-Y-source} define the two limiting source terms.  In
particular, the limiting nonlinear source is defined through these limits and
not through a separate pointwise interpretation of $uv$.
At finite cutoff the notation
$\mathscr F_{X,\Lambda}^{\chi,\rho}$ and
$\mathscr F_{Y,\Lambda}^{\chi,\rho}$ means that every stochastic factor and operator is evaluated with the
cutoff data; in particular the cubic entry is
$\Gamma_{a,\Lambda}^{\chi,\rho}$ from
\eqref{eq:localized-cubic-cutoff}, not $\chi\Gamma_{a,\Lambda}$.

The partition in \eqref{eq:operator-index-set} separates the two cases used below.  For labels in $\mathfrak L_{\mathrm{diag}}$, the two stochastic
colors coincide and the Duhamel channel is the opposite channel
$a=b^\perp$; the finite Wick expansion contains a deterministic Volterra
contraction and a centered same-color term.  For labels in
$\mathfrak L_{\mathrm{off}}$, the Duhamel channel agrees with the inner
color, $a=b$, while $b\ne c$; independence removes the deterministic
contraction.

\subsection{Construction of the stochastic terms}\label{sec:baseline-lift}

The stochastic terms are constructed in the order
\[
 (\Psi_a,\Theta)\longrightarrow(V_a,\partial_tV_a)
 \longrightarrow\Gamma_a\longrightarrow(\mathcal D,\mathcal B).
\]
All estimates retain the channel orders $\betaW,\betaK$ and are uniform when $\mathfrak H_{\mathrm{prof}}$ is bounded.  The phase estimates for the first
Picard and cubic objects are proved in
Appendices~\ref{app:first-picard} and~\ref{app:cubic-terms}.

Define
\begin{equation}\label{eq:stochastic-convolutions}
  \Psi_{a,\Lambda}=I_a(\pi_\Lambda\xi_a),
  \qquad a\in\mathfrak C.
\end{equation}
We describe the full-space field through its covariance spectral density:
\begin{equation}\label{eq:Psi-variance}
\begin{aligned}
  \E\!\left[\wh\Psi_a(t,\xi)
  \overline{\wh\Psi_a(t,\xi')}\right]
  &=(2\pi)^3\delta(\xi-\xi')p_a(t,\xi),\\
  p_a(t,\xi)
  &:=\mathfrak q_a(\xi)\int_0^t\frac{\sin^2((t-s)\omega_a(\xi))}
  {\omega_a(\xi)^2}\dd s
  \lesssim_{T,\parvec}\la\xi\ra^{-2+2\beta_a}.
\end{aligned}
\end{equation}
For every $0<\theta\le1/2$, the same kernel calculation gives the increment
density bound
\begin{equation}\label{eq:Psi-increment-density}
 p_a^\Delta(t,t';\xi)
 \lesssim_{T,\parvec,\theta}
 |t-t'|^{2\theta}\la\xi\ra^{-2+2\beta_a+2\theta}.
\end{equation}
Indeed, on the common time interval one uses
$|K_a(\tau,\xi)-K_a(\tau',\xi)|\lesssim
|\tau-\tau'|^\theta\la\xi\ra^{-1+\theta}$, while the boundary strip is
controlled by its length.
The wave multiplier is interpreted continuously at the zero mode.  After a
compact physical localization, Gaussian block estimates, hypercontractivity,
and Kolmogorov continuity yield
\begin{equation}\label{eq:Psi-baseline}
  \chi\Psi_{a,\Lambda}\longrightarrow\chi\Psi_a
  \quad\text{in }L^p(\Omega;C_T\cC^{-1/2-\beta_a-\kappa})
\end{equation}
for every finite $p\ge2$ and $\kappa>0$, with the corresponding dyadic almost
sure realization.

Set
\begin{equation}\label{eq:theta-cutoff}
  \Theta_\Lambda=\Psi_{\W,\Lambda}\Psi_{\K,\Lambda}.
\end{equation}
Because the colors are independent, this product is already a centered second
homogeneous chaos; no scalar Wick subtraction occurs.  Its covariance spectral density $q_\Theta$ is defined by
\[
  \E\!\left[\wh\Theta(t,\xi)
  \overline{\wh\Theta(t,\xi')}\right]
  =(2\pi)^3\delta(\xi-\xi')q_\Theta(t,\xi),
\]
and independence gives
\begin{equation}\label{eq:theta-convolution}
  q_\Theta(t,\xi)
  \lesssim_{T,\parvec}
  \int_{\R^3}\la\eta\ra^{-2+2\betaW}
  \la\xi-\eta\ra^{-2+2\betaK}\dd\eta
  \lesssim_{T,\parvec,\kappa}
  \la\xi\ra^{-1+2\betasum+\kappa}.
\end{equation}
Combining \eqref{eq:Psi-increment-density} in one factor with the static
bound in the other gives, for every $0<\theta<\frac12-\betasum$ and every
$\eps>0$,
\begin{equation}\label{eq:Theta-increment-density}
 q_\Theta^\Delta(t,t';\xi)
 \lesssim_{T,\parvec,\theta,\eps}
 |t-t'|^{2\theta}\la\xi\ra^{-1+2\betasum+2\theta+\eps}.
\end{equation}
Consequently
\begin{equation}\label{eq:Theta-baseline}
  \chi\Theta_\Lambda\longrightarrow\chi\Theta
  \quad\text{in }L^p(\Omega;C_T\cC^{-1-\betasum-\kappa}).
\end{equation}
Relative to the limiting multiplier one, a cutoff difference contains at
least one factor $m_\Lambda(\zeta)-1$, which vanishes below $c_0\Lambda$;
the same convolution estimate and a fixed-compact/tail split give cutoff
Cauchy convergence for non-nested families.

Define
\begin{equation}\label{eq:first-picard-def}
  V_{a,\Lambda}=I_a(\Theta_\Lambda).
\end{equation}
The elementary Duhamel estimate gives only the baseline bounds
\begin{align}
  \chi V_{a,\Lambda}&\longrightarrow\chi V_a
  &&\text{in }L^p(\Omega;C_T\cC^{-\betasum-\kappa}),
  \label{eq:V-baseline}\\
  \chi\partial_tV_{a,\Lambda}&\longrightarrow\chi\partial_tV_a
  &&\text{in }L^p(\Omega;C_T\cC^{-1-\betasum-\kappa}).
  \label{eq:dV-baseline}
\end{align}
These bounds use only the order $-1$ outer Duhamel multiplier.  The fixed point
requires the sharper conclusions
\[
  V_a\in C_T\cC_{\loc}^{1/2-\betasum-},
  \qquad
  \partial_tV_a\in C_T\cC_{\loc}^{-1/2-\betasum-},
\]
which follow from the three-propagator phase geometry proved in
Appendix~\ref{app:first-picard}.

\begin{lemma}[Cutoff tails for the linear and cross-quadratic fields]
\label{lem:baseline-cutoff-tails}
Let $0\le\theta<1/2-\betasum$.  For the difference from the uncut field, let
$p_{a,>\Lambda}^{(\theta)}$ and
$q_{\Theta,>\Lambda}^{(\theta)}$ be nonnegative majorants of the sum of the
static covariance density and the normalized time-increment density.  For
every $\varepsilon>0$,
\begin{align}
 p_{a,>\Lambda}^{(\theta)}(\xi)
 &\lesssim
 \one_{\{|\xi|\ge c_0\Lambda\}}
 \langle\xi\rangle^{-2+2\beta_a+2\theta},
 \label{eq:Psi-cutoff-tail-majorant}\\
 q_{\Theta,>\Lambda}^{(\theta)}(\xi)
 &\lesssim_{\varepsilon}
 \max\{\langle\xi\rangle,\Lambda\}^{-1+2\betasum+2\theta+\varepsilon}.
 \label{eq:Theta-cutoff-tail-majorant}
\end{align}
The corresponding difference of two non-nested admissible cutoffs is bounded
by the sum of the two right-hand sides.  Moreover, for every
\[
 s<-\frac12-\beta_a
 \quad\text{and}\quad
 r<-1-\betasum,
\]
one can choose $\theta,\varepsilon,\delta>0$ so that
\begin{align}
 \sum_{N\in\Dyd}N^{2s+3+\delta}
 \sup_{|\xi|\sim N}p_{a,>\Lambda}^{(\theta)}(\xi)&\longrightarrow0,
 \label{eq:Psi-cutoff-tail-sum}\\
 \sum_{N\in\Dyd}N^{2r+3+\delta}
 \sup_{|\xi|\sim N}q_{\Theta,>\Lambda}^{(\theta)}(\xi)&\longrightarrow0.
 \label{eq:Theta-cutoff-tail-sum}
\end{align}
Both tails decay polynomially along dyadic $\Lambda$.
\end{lemma}

\begin{proof}
For $\Psi_a$, the cutoff difference has the single factor
$m_\Lambda(\xi)-1$.  It vanishes for $|\xi|<c_0\Lambda$, and
\eqref{eq:Psi-variance}--\eqref{eq:Psi-increment-density} give
\eqref{eq:Psi-cutoff-tail-majorant}.

For $\Theta$, Lemma~\ref{lem:finite-product-cutoff} restricts the covariance
integral to
\[
 \max\{|\eta|,|\xi-\eta|\}\ge c_0\Lambda.
\]
The normalized increment density is bounded by a sum of two integrals in
which one factor carries the additional weight $\langle\cdot\rangle^{2\theta}$.
Put $N=\langle\xi\rangle$.  If $N\gtrsim\Lambda$, the unrestricted convolution
estimate gives
\[
 q_{\Theta,>\Lambda}^{(\theta)}(\xi)
 \lesssim_\varepsilon N^{-1+2\betasum+2\theta+\varepsilon}.
\]
If $N\ll\Lambda$, the tail condition and
$\xi=\eta+(\xi-\eta)$ force the two input scales to be comparable.  On a
common dyadic shell $M\gtrsim\Lambda$, the integral is bounded by
\[
 M^3M^{-2+2\betaW+2\theta}M^{-2+2\betaK}
 +M^3M^{-2+2\betaW}M^{-2+2\betaK+2\theta}
 \lesssim M^{-1+2\betasum+2\theta}.
\]
Because $2\betasum+2\theta<1$, summing $M\gtrsim\Lambda$ proves
\eqref{eq:Theta-cutoff-tail-majorant}; the intermediate region
$N\simeq\Lambda$ is covered by either estimate.  The two-cutoff assertion is
\eqref{eq:finite-product-two-cutoffs}.

For \eqref{eq:Psi-cutoff-tail-sum}, the sum starts at $N\gtrsim\Lambda$ and
is bounded by
\[
 \sum_{N\gtrsim\Lambda}
 N^{2s+1+2\beta_a+2\theta+\delta},
\]
which decays after choosing
$2\theta+\delta<-2s-1-2\beta_a$.
For \eqref{eq:Theta-cutoff-tail-sum}, split at $N=\Lambda$.  The high-output
part is bounded by
\[
 \sum_{N>\Lambda}
 N^{2r+2+2\betasum+2\theta+\varepsilon+\delta}.
\]
For $N\le\Lambda$,
\begin{align*}
 &\Lambda^{-1+2\betasum+2\theta+\varepsilon}
 \sum_{N\le\Lambda}N^{2r+3+\delta}\\
 &\qquad\lesssim
 \Lambda^{-1+2\betasum+2\theta+\varepsilon
 +\max\{2r+3+\delta,0\}+}.
\end{align*}
If the expression in the positive part is positive, the exponent is
$2r+2+2\betasum+2\theta+\varepsilon+\delta+$; otherwise it is
$-1+2\betasum+2\theta+\varepsilon+$.  Both are negative after taking the
auxiliary parameters sufficiently small.  This proves the two weighted tail
limits and their polynomial dyadic versions.
\end{proof}

\begin{proposition}[Baseline local lift]\label{prop:baseline-regularities}
Fix $0<T\le1$, $\chi\in C_c^\infty(\R^3)$, $\kappa>0$, and finite $p\ge2$.
The four convergences
\eqref{eq:Psi-baseline}, \eqref{eq:Theta-baseline},
\eqref{eq:V-baseline}, and \eqref{eq:dV-baseline} hold in $L^p(\Omega)$.
In addition,
\begin{equation}\label{eq:baseline-weighted-tempered-convergence}
 (\Psi_{a,\Lambda},V_{a,\Lambda},\partial_tV_{a,\Lambda})
 \longrightarrow(\Psi_a,V_a,\partial_tV_a)
 \quad\text{in }
 L^p(\Omega;C_T\mathfrak T_{J_{\mathrm{ps}}})^3.
\end{equation}
Along dyadic cutoffs all these convergences hold almost surely on one common
event after taking a countable exhaustion of compact sets, times, and
regularity losses.
\end{proposition}

\begin{proof}
The relevant dyadic calculation is the following.  If a spatially stationary chaos has density
$q(t,\xi)\lesssim\langle\xi\rangle^{-2\rho-3+\eps}$, then on $|\xi|\sim N$,
\[
 \E|P_NZ(t,x)|^2\lesssim N^{-2\rho+\eps}.
\]
After multiplication by a fixed compact cutoff, the same power holds in local
$L_x^2$; a large finite spatial $L^q$ norm and Bernstein give the local
$L_x^\infty$ block bound with an arbitrarily small additional $N^\eps$ loss.
The increment density gives
$|t-t'|^\theta N^{-\rho+\theta+\eps}$ in probability norm.  Gaussian
hypercontractivity and one-dimensional fractional Sobolev embedding in time
therefore yield $C_T\mathcal C^s\cap L_T^\infty B_{2,\infty}^s$ for every
$s<\rho$.  The corresponding path-space statement is proved in
Lemma~\ref{lem:spectral-to-local-paths}.

For $\Psi_a$, \eqref{eq:Psi-variance} has
$\rho=-1/2-\beta_a$; for $\Theta$,
\eqref{eq:theta-convolution} has $\rho=-1-\betasum$.  Applying one outer
Duhamel multiplier to the latter gives $\rho=-\betasum$ for $V_a$, while
differentiating the outer sine kernel returns $\rho=-1-\betasum$.  These four values yield
\eqref{eq:Psi-baseline}, \eqref{eq:Theta-baseline},
\eqref{eq:V-baseline}, and \eqref{eq:dV-baseline}.  Multiplication by $\chi$
only convolves the shell by a Schwartz kernel and hence preserves the powers by
Lemma~\ref{lem:physical-cutoff-almost-diagonal}.

The cutoff differences for $\Psi_a$ and $\Theta$ satisfy the static and
increment tail bounds of Lemma~\ref{lem:baseline-cutoff-tails}.  Applying
Lemma~\ref{lem:spectral-to-local-paths} to
\eqref{eq:Psi-cutoff-tail-sum} and
\eqref{eq:Theta-cutoff-tail-sum} gives the Cauchy property in the two localized
path spaces.  The two-cutoff estimate in
Lemma~\ref{lem:baseline-cutoff-tails} shows that no nesting is used.  Along a
dyadic fixed-profile family, the polynomial tails, hypercontractivity,
Markov's inequality, and Borel--Cantelli give almost-sure convergence.  A
countable exhaustion of compact cutoffs, rational horizons, and rational
strict losses yields one common probability-one event.

Before physical localization, the three fields in
\eqref{eq:baseline-weighted-tempered-convergence} are spatially stationary.
Their local $C_TH^{-J_{\mathrm{ps}}}$ moments and cutoff-difference moments are
bounded by the preceding block estimates, uniformly under translations of a
fixed cutoff.  Lemma~\ref{lem:weighted-stationary-lifting} upgrades the local
convergence to the weighted convergence in that display; the same polynomial
dyadic tails give the almost-sure assertion.
\end{proof}

The gain from $V_a\in\mathcal C^{-\betasum-}$ to
$V_a\in\mathcal C^{1/2-\betasum-}$ is oscillatory.  The phase functions depend
only on $\omega_{\W}$ and $\omega_{\K}$, while the forcing profiles produce the
weights $\langle\eta\rangle^{2\betaW}$ and
$\langle\xi-\eta\rangle^{2\betaK}$ in the covariance integral.
Proposition~\ref{prop:continuous-phase-integral} establishes the corresponding
weighted phase-layer bound.

At finite cutoff put
\begin{equation}\label{eq:cubic-term-baseline}
  \Gamma_{a,\Lambda}
  :=V_{a^\perp,\Lambda}\circ\Psi_{a,\Lambda}.
\end{equation}
The colored product formula produces one centered third-chaos branch and one
same-color first-chaos contraction.  The latter is retained, not subtracted,
because it is part of the finite-cutoff Wick algebra.  Appendix~\ref{app:cubic-terms}
constructs both branches and proves
\begin{equation}\label{eq:cubic-baseline-regularity}
  \Gamma_a\in C_T\cC_{\loc}^{-\betasum-\beta_a-}
  \cap L_T^\infty B_{2,\infty,\loc}^{-\betasum-\beta_a-}.
\end{equation}

\begin{lemma}[Separated-support smoothing]
\label{lem:separated-support-smoothing}
Let $\phi\in C_c^\infty(\R^3)$ and $\psi\in C_b^\infty(\R^3)$ satisfy
$d_0:=\dist(\supp\phi,\supp\psi)>0$.  If
$f,g\in C([0,T];\mathfrak T_{J_{\mathrm{ps}}})$, then the three localized
inhomogeneous Bony series
\[
  \phi((\psi f)\prec g),\qquad \phi((\psi f)\succ g),\qquad \phi((\psi f)\circ g)
\]
converge in $C([0,T];H^M\cap\mathcal C^M)$ for every $M>0$, and
\begin{equation}\label{eq:separated-support-quantitative}
 \sup_{t\le T}\sum_{\star\in\{\prec,\succ,\circ\}}
 \|\phi((\psi f)\star g)(t)\|_{H^M\cap\mathcal C^M}
 \le C_{M,J_{\mathrm{ps}},d_0,\phi,\psi}
 \|f\|_{C_T\mathfrak T_{J_{\mathrm{ps}}}}
 \|g\|_{C_T\mathfrak T_{J_{\mathrm{ps}}}}.
\end{equation}
The same estimate is uniform after admissible spectral truncations are inserted
on either factor.  Convergence in $C_T\mathfrak T_{J_{\mathrm{ps}}}$ passes to
the corresponding localized Bony series in every displayed topology.
\end{lemma}

\begin{proof}
Put $K=\supp\phi$ and write $\Delta_N=P_N$ and
$S_{<N/8}=\sum_{L<N/8}\Delta_L$, with the inhomogeneous unit block understood.
Testing the dyadic kernels against a unit-lattice partition and using
\eqref{eq:weighted-tempered-norm} gives, for every compact $K_1$, every
multi-index $\alpha$, and some finite exponent $C_{J,\alpha}$,
\begin{equation}\label{eq:pseudolocal-local-polynomial-block}
 \sup_{x\in K_1}\bigl(
 |\partial^\alpha\Delta_Nh(x)|
 +|\partial^\alpha S_{<N/8}h(x)|\bigr)
 \lesssim_{K_1,J_{\mathrm{ps}},\alpha}
 N^{C_{J,\alpha}}\|h\|_{\mathfrak T_{J_{\mathrm{ps}}}}.
\end{equation}
For the separated factor one has the stronger estimate: for every $A>0$,
\begin{equation}\label{eq:pseudolocal-separated-block}
 \sup_{x\in K}\bigl(
 |\partial^\alpha\Delta_N(\psi f)(x)|
 +|\partial^\alpha S_{<N/8}(\psi f)(x)|\bigr)
 \lesssim_{A,\alpha,d_0,\phi,\psi,J_{\mathrm{ps}}}
 N^{-A}\|f\|_{\mathfrak T_{J_{\mathrm{ps}}}}.
\end{equation}
Indeed, both $\Delta_N$ and $S_{<N/8}$ have kernels localized at spatial scale
$N^{-1}$.  For $x\in K$ and $y\in\supp\psi$ their kernels, together with every
$x$-derivative, are bounded by
\[
 C_{A,\alpha}N^{3+|\alpha|}(1+N|x-y|)^{-A}.
\]
The distance $|x-y|\ge d_0$, the polynomial lattice growth allowed by
$\mathfrak T_{J_{\mathrm{ps}}}$, and an increase of $A$ prove
\eqref{eq:pseudolocal-separated-block}.  The low-frequency estimate
uses the kernel of the \emph{whole} low-pass $S_{<N/8}$; estimating its dyadic
summands separately would lose the cancellation which makes a function
supported away from $K$ small at resolution $N^{-1}$.

Use the standard grouped Bony formulas
\begin{align*}
 (\psi f)\prec g&=\sum_N S_{<N/8}(\psi f)\,\Delta_Ng,\\
 (\psi f)\succ g&=\sum_N \Delta_N(\psi f)\,S_{<N/8}g,\\
 (\psi f)\circ g&=\sum_N\sum_{N'\sim N}\Delta_N(\psi f)\,\Delta_{N'}g.
\end{align*}
In the first line the separated low-pass is controlled by
\eqref{eq:pseudolocal-separated-block}; in the second and third lines the
separated high block is controlled by the same estimate.  Every remaining
factor has at most the polynomial growth in
\eqref{eq:pseudolocal-local-polynomial-block}.  After differentiating the
products any fixed number of times, choose $A$ larger than all resulting
polynomial powers.  The three series then converge absolutely in $C^k(K)$ for
every integer $k$.  Multiplication by $\phi$ makes their support compact, and
choosing $k$ sufficiently large yields convergence and
\eqref{eq:separated-support-quantitative} in both $H^M$ and $\mathcal C^M$.
The same absolutely summable majorant is uniform in time; continuity of $f$ and
$g$ in $\mathfrak T_{J_{\mathrm{ps}}}$ therefore gives continuity of the
localized outputs.

Admissible spectral multipliers are uniformly bounded on
$\mathfrak T_{J_{\mathrm{ps}}}$ by
Lemma~\ref{lem:weighted-stationary-lifting}.  Hence the same bilinear estimate
holds uniformly for truncated factors.  Finally,
\[
 \mathfrak R(f_1,g_1)-\mathfrak R(f_2,g_2)
 =\mathfrak R(f_1-f_2,g_1)+\mathfrak R(f_2,g_1-g_2)
\]
for each of the three localized Bony maps $\mathfrak R$; the bilinear bound and
dominated dyadic summation prove the convergence assertion.
\end{proof}

\begin{corollary}[Localized cubic coordinate]\label{cor:localized-cubic-terms}
Fix an adapted localization pair $(\chi,\rho)$ and $a\in\mathfrak C$.
For every $M>0$ and every finite $p\ge2$, the sequence
$\mathcal R_{a,\Lambda}^{\chi,\rho}$ in
\eqref{eq:localized-cubic-split} converges in
\begin{equation}\label{eq:localized-cubic-remainder-convergence}
 L^p\!\left(\Omega;
 C_T(H^M\cap\mathcal C^M)\right)
\end{equation}
to a compactly supported path $\mathcal R_a^{\chi,\rho}$.  The limit obeys
\begin{equation}\label{eq:localized-cubic-pseudolocal-bound}
 \|\mathcal R_a^{\chi,\rho}\|_{C_T(H^M\cap\mathcal C^M)}
 \lesssim_{M,\chi,\rho}
 \|V_{a^\perp}\|_{C_T\mathfrak T_{J_{\mathrm{ps}}}}
 \|\Psi_a\|_{C_T\mathfrak T_{J_{\mathrm{ps}}}}.
\end{equation}
More generally, if $f_j,g_j\in C_T\mathfrak T_{J_{\mathrm{ps}}}$, then
with
$\mathcal R^{\chi,\rho}(f,g):=
 \chi(((\rho-1)f)\circ g)$,
\begin{align}
 &\|\mathcal R^{\chi,\rho}(f_1,g_1)
      -\mathcal R^{\chi,\rho}(f_2,g_2)\|_{C_T(H^M\cap\mathcal C^M)}
 \notag\\
 &\quad\lesssim_{M,\chi,\rho}
 \|f_1-f_2\|_{C_T\mathfrak T_{J_{\mathrm{ps}}}}
 \bigl(\|g_1\|_{C_T\mathfrak T_{J_{\mathrm{ps}}}}
      +\|g_2\|_{C_T\mathfrak T_{J_{\mathrm{ps}}}}\bigr)
 \notag\\
 &\qquad\quad+
 \|g_1-g_2\|_{C_T\mathfrak T_{J_{\mathrm{ps}}}}
 \bigl(\|f_1\|_{C_T\mathfrak T_{J_{\mathrm{ps}}}}
      +\|f_2\|_{C_T\mathfrak T_{J_{\mathrm{ps}}}}\bigr).
 \label{eq:localized-cubic-remainder-difference}
\end{align}
The finite-cutoff Wick decomposition from
\eqref{eq:cubic-wick-decomposition} and
\eqref{eq:localized-cubic-split} gives
\begin{equation}\label{eq:localized-cubic-finite-wick-split}
 \Gamma_{a,\Lambda}^{\chi,\rho}
 =\chi\Gamma_{a,\Lambda}^{(3)}
  +\chi\Gamma_{a,\Lambda}^{(1)}
  +\mathcal R_{a,\Lambda}^{\chi,\rho}.
\end{equation}
Set
\begin{equation}\label{eq:localized-cubic-limit}
 \Gamma_a^{\chi,\rho}:=\chi\Gamma_a+\mathcal R_a^{\chi,\rho}.
\end{equation}
Then
\begin{equation}\label{eq:localized-cubic-full-convergence}
 \Gamma_{a,\Lambda}^{\chi,\rho}
 \longrightarrow \Gamma_a^{\chi,\rho}
\end{equation}
in
\[
 L^p\!\left(\Omega;
 C_T\mathcal C^{-\beta_{\Gamma,a}-\kappa}
 \cap L_T^\infty B_{2,\infty}^{-\beta_{\Gamma,a}-\kappa}\right).
\]
Along a fixed-profile dyadic family the convergence holds almost surely along
the full sequence; along an arbitrary admissible cofinal sequence it holds in
probability.  In all cases the decomposition
\begin{equation}\label{eq:localized-cubic-limit-split}
 \Gamma_a^{\chi,\rho}
 =\chi\Gamma_a^{(3)}+\chi\Gamma_a^{(1)}
  +\mathcal R_a^{\chi,\rho}
\end{equation}
is preserved in the limit.
\end{corollary}

\begin{proof}
Apply Lemma~\ref{lem:separated-support-smoothing} to the bilinear map
$(f,g)\mapsto\chi(((\rho-1)f)\circ g)$.  Its quantitative bound gives
\eqref{eq:localized-cubic-pseudolocal-bound}; applying the same estimate to
\[
 \mathcal R^{\chi,\rho}(f_1,g_1)
 -\mathcal R^{\chi,\rho}(f_2,g_2)
 =\mathcal R^{\chi,\rho}(f_1-f_2,g_1)
  +\mathcal R^{\chi,\rho}(f_2,g_1-g_2)
\]
and then symmetrizing the coefficient bounds gives
\eqref{eq:localized-cubic-remainder-difference}.  Taking $L^p(\Omega)$ norms in
\eqref{eq:localized-cubic-remainder-difference} and applying H\"older's
inequality with probability exponents $2p$ and $2p$, the convergence of
$V_{a^\perp,\Lambda}$ and $\Psi_{a,\Lambda}$ in
$L^{2p}(\Omega;C_T\mathfrak T_{J_{\mathrm{ps}}})$ proves
\eqref{eq:localized-cubic-remainder-convergence}, including the corresponding
cutoff-difference estimate for non-nested families.

By \eqref{eq:localized-cubic-split},
\[
 \Gamma_{a,\Lambda}^{\chi,\rho}
 -\Gamma_a^{\chi,\rho}
 =\chi(\Gamma_{a,\Lambda}-\Gamma_a)
  +(\mathcal R_{a,\Lambda}^{\chi,\rho}
    -\mathcal R_a^{\chi,\rho}).
\]
The first difference converges in the asserted critical topology by
\Cref{thm:cubic-fullspace}; the second converges in every positive Sobolev and
Hölder topology by the first part of the proof.  This proves
\eqref{eq:localized-cubic-full-convergence} and
\eqref{eq:localized-cubic-limit-split}.  The almost-sure dyadic and
in-probability cofinal conclusions follow from the corresponding conclusions
for the two summands.
\end{proof}

\begin{center}
\small
\renewcommand{\arraystretch}{1.25}
\begin{tabularx}{\textwidth}{@{}P{0.22\textwidth}P{0.29\textwidth}Y P{0.19\textwidth}@{}}
\toprule
finite-cutoff term & role & limiting term & convergence topology \\
\midrule
$\chi\Gamma_{a,\Lambda}^{(3)}$
& centered third homogeneous chaos
& $\chi\Gamma_a^{(3)}$
& $C_T\mathcal C^{-\beta_{\Gamma,a}-\kappa}
   \cap L_T^\infty B_{2,\infty}^{-\beta_{\Gamma,a}-\kappa}$ \\
$\chi\Gamma_{a,\Lambda}^{(1)}$
& first chaos selected by the same-color contraction
& $\chi\Gamma_a^{(1)}$
& $C_T\mathcal C^{1/2-\beta_{\Gamma,a}-\kappa}
   \cap L_T^\infty B_{2,\infty}^{1/2-\beta_{\Gamma,a}-\kappa}$ \\
$\mathcal R_{a,\Lambda}^{\chi,\rho}$
& separated-support localization remainder
& $\mathcal R_a^{\chi,\rho}$
& $C_T(H^M\cap\mathcal C^M)$ for every $M>0$ \\
$\Gamma_{a,\Lambda}^{\chi,\rho}$
& localized cubic coordinate in the cutoff source
& $\Gamma_a^{\chi,\rho}$
& the critical topology of \eqref{eq:localized-cubic-full-convergence} \\
\bottomrule
\end{tabularx}
\end{center}

\section{Resonant operators}\label{sec:resonant-operators}

The proof is organized in three layers.  First, Wick's formula is applied to
the two Gaussian legs before any deterministic estimate: equal colors produce
a covariance contraction and all remaining terms form a completely centered
second chaos.  Second, a model-independent two-step Gaussian Khintchine
argument bounds a centered operator by four oriented Schatten flattenings.
Third, the localized wave--Klein--Gordon block is written with both physical
multipliers and the Duhamel frequency seen after localization.  The resulting
kernel identifies the four Hilbert-space directions, supplies the $N^{-1}$
propagator gain on the principal transfer sector, and turns the remaining
transfers into Schwartz tails.  The Wick algebra determines which branch
occurs, while the frequency kernel determines its size.

\subsection{Finite-cutoff decomposition}\label{sec:wick-split}

The four labels are obtained from \eqref{eq:X-resonant-substitution}.
If $(a;b,c)\in\mathfrak L_{\mathrm{diag}}$, then $b=c$ and
$a=b^\perp$; the covariance contraction therefore compares the dispersions
$\omega_a$ and $\omega_b$.  If
$(a;b,c)\in\mathfrak L_{\mathrm{off}}$, then $a=b$ and $b\ne c$;
independence of the two colors eliminates the covariance contraction.

Wick's identity is applied to the two stochastic Fourier variables before
any deterministic kernel estimate is made.  The physical multipliers do not
alter the covariance support $r=-\ell$; they only introduce two deterministic
Schwartz convolution variables.  The full two-localizer kernel, including the
frequency actually seen by the Duhamel propagator, is retained below.

Fix an inner stochastic scale $N$, a low input scale
$Q\le c_{\mathrm{ap}}N$, a Duhamel-output scale $R$, an outer stochastic
scale $L$, a final output scale $M$, and times $0\le s\le t\le T$.
The outer resonant product imposes $L\sim_{\mathrm{ap}}R$.  Before the two
physical multipliers in \eqref{eq:T-def} are inserted, every
translation-invariant dyadic core is a finite sum of operators of the form
\begin{align}\label{eq:raw-kernel}
  \wh T^{a;b,c}_{N,R,L,Q,M}(w)(n,t)
  ={}&\int_0^t\int_{q,\ell,r}
  \delta(n-q-\ell-r)
  \chi_M(n)\chi_Q(q)\rho_N(\ell)\varphi_R(q+\ell)\rho_L(r)\notag\\
  &\quad\times K_a(t-s,q+\ell)
  m_{N,R,L,Q,M}(n,q,\ell,r)
  \wh w(q,s)\notag\\
  &\quad\times\wh\Psi_b(\ell,s)\wh\Psi_c(r,t)
  \dd q\dd\ell\dd r\dd s,
\end{align}
We assign the factors $\rho_N(\ell)$ and $\rho_L(r)$ in
\eqref{eq:raw-kernel} to the two covariance synthesis maps.  The deterministic
coefficient kernel is localized by auxiliary plateaux that equal one on the
corresponding supports.  Here
$K_a(\tau,\zeta)=\sin(\tau\omega_a(\zeta))/\omega_a(\zeta)$ and
$\varphi_R$ is an enlarged cutoff to the shell selected on the Duhamel
output.  The low--high support gives $|q+\ell|\sim N$, while
$\varphi_R(q+\ell)$ gives $R\sim N$ in the translation-invariant core.
Together with $L\sim_{\mathrm{ap}}R$, all three high scales are comparable
before localization, up to finitely many inhomogeneous exceptions.  The
remaining multiplier is supported in the displayed dyadic regions and obeys
uniform rescaled symbol bounds; for example,
\[
  |\partial_n^{\alpha_0}\partial_q^{\alpha_1}
    \partial_\ell^{\alpha_2}\partial_r^{\alpha_3}m_{N,R,L,Q,M}|
  \lesssim_{\alpha}
  M^{-|\alpha_0|}Q^{-|\alpha_1|}
  N^{-|\alpha_2|}L^{-|\alpha_3|},
\]
while the normalized derivatives of the explicit shell factor
$\varphi_R(q+\ell)$ are bounded at scale $R$.
A physical multiplier replaces the strict momentum relation by a Schwartz
kernel in the corresponding momentum defect.  The full multiplier factorization for the
two multipliers in \eqref{eq:T-def}, including the resulting shift of the
Duhamel frequency, is retained in Lemma~\ref{lem:inner-localization-transfer} below.
For a single multiplier the resulting factor has the form
\begin{equation}\label{eq:soft-incidence}
  \kappa(n-q-\ell-r),
  \qquad \kappa\in\mathcal S(\R^3),
\end{equation}
with Schwartz seminorms depending on finitely many seminorms of the physical
cutoff.

For $b\in\mathfrak C$, define the covariance density
$\sigma_b(\ell;s,t)$ by
\begin{equation}\label{eq:sigma-def}
  \E\!\left[\wh\Psi_b(s,\ell)\wh\Psi_b(t,r)\right]
  =(2\pi)^3\delta(\ell+r)\sigma_b(\ell;s,t)
  \quad\text{in }\mathcal D'(\R^3_\ell\times\R^3_r).
\end{equation}
At cutoff $\Lambda$,
$\sigma_{b,\Lambda}(\ell;s,t)=m_\Lambda(\ell)^2
\sigma_b(\ell;s,t)$.  Wick's formula, interpreted after pairing with the
frequency kernel in \eqref{eq:raw-kernel}, reads
\begin{align}\label{eq:finite-wick-identity}
  \wh\Psi_{b,\Lambda}(s,\ell)
  \wh\Psi_{c,\Lambda}(t,r)
  ={}&:\wh\Psi_{b,\Lambda}(s,\ell)
    \wh\Psi_{c,\Lambda}(t,r):\notag\\
   &+\delta_{bc}(2\pi)^3\delta(\ell+r)
     \sigma_{b,\Lambda}(\ell;s,t).
\end{align}

\begin{proposition}[Finite-cutoff Wick decomposition]\label{prop:finite-color-split}
For every $(a;b,c)\in\mathfrak L$, every finite cutoff $\Lambda$, and every
localized dyadic tuple $(N,R,L,Q,M)$ with $L\sim_{\mathrm{ap}}R$,
\begin{equation}\label{eq:finite-split}
 T^{a;b,c}_{\Lambda,N,R,L,Q,M}=
 \begin{cases}
  \cD^{a;b}_{\Lambda,N,R,L,Q,M}
   +\cB^{a;b,b}_{\Lambda,N,R,L,Q,M},
   &(a;b,c)=(a;b,b)\in\mathfrak L_{\mathrm{diag}},\\
  \cB^{a;b,c}_{\Lambda,N,R,L,Q,M},
   &(a;b,c)\in\mathfrak L_{\mathrm{off}}.
 \end{cases}
\end{equation}
The centered term is a second homogeneous Gaussian chaos.  Its
two-localizer frequency kernel is displayed in
\eqref{eq:two-localizer-kernel}.  If $b=c$, the deterministic
contraction satisfies $r=-\ell$; hence $L\sim N$, and the resonant relation
$L\sim_{\mathrm{ap}}R$ also forces $R\sim N$.  The strict relation $n=q$ of
the translation-invariant reference core is replaced after localization by a
Schwartz kernel rapidly decaying away from $n=q$ in $n-q$.  If $b\ne c$, the deterministic
branch is absent.
\end{proposition}

\begin{proof}
Pair the Wick identity \eqref{eq:finite-wick-identity} with the full localized
frequency kernel.  The Wick product gives the centered second-chaos block.
The covariance term is present only for $b=c$ and is supported on $r=-\ell$.
The inner and outer physical multipliers are deterministic, so they do not
change this covariance support.  Before localization the external incidence
then reduces to
\[
 n=q+\ell+r\big|_{r=-\ell}=q.
\]
After localization, the two convolution defects give a Schwartz kernel in
$n-q$.  For $b\ne c$, independence gives zero cross covariance.
\end{proof}

\begin{remark}[Same-family centered remainder]\label{rem:same-family-centered}
When $b=c$, subtracting the covariance term in
\eqref{eq:finite-wick-identity} leaves the full second homogeneous chaos on
the common Gaussian Hilbert space.  Its coefficient tensor is first
symmetrized in the two equal-color legs and then estimated through the four
oriented Hilbert-space flattenings.
\end{remark}

\subsection{Deterministic contractions}\label{sec:volterra}

The proportional low--high phase gap used below is proved in
Appendix~\ref{app:phase-gap}.  Its threshold and constants depend on the
parameter vector, while the phase denominator always contributes one
high-frequency power.

For $0\le s\le t\le T$, let $\sigma_b$ be the covariance density in \eqref{eq:sigma-def}.

For $a,b\in\mathfrak C$ with $a\ne b$ and $\varepsilon_0,\varepsilon_1\in\{-1,1\}$, set
\begin{equation}\label{eq:two-frequency-phase}
  \Phi_{a|b}^{\varepsilon_0,\varepsilon_1}(q,\ell)
  :=\varepsilon_0\omega_a(q+\ell)
    +\varepsilon_1\omega_b(\ell).
\end{equation}

\begin{lemma}[Covariance branch decomposition]\label{lem:covariance-decomposition}
Let $a,b\in\mathfrak C$ with $a\ne b$.  For $\ell\ne0$ in the wave case,
\begin{align}\label{eq:covariance-decomposition}
  \sigma_b(\ell;s,t)
  =\mathfrak q_b(\ell)\Bigg[{}&\frac{s}{2\omega_b(\ell)^2}
      \cos((t-s)\omega_b(\ell))\notag\\
  &+\frac{\sin((t-s)\omega_b(\ell))}{4\omega_b(\ell)^3}
   -\frac{\sin((t+s)\omega_b(\ell))}{4\omega_b(\ell)^3}\Bigg].
\end{align}
On $|\ell|\sim N$, the last two branches are bounded by
$C N^{-3+2\beta_b}$.  In the first branch, after multiplication by the outer
Duhamel kernel and expansion into exponentials, each sign component has the
form
\begin{equation}\label{eq:oscillatory-branch}
  e^{\ii(t-s)\Phi_{a|b}^{\varepsilon_0,\varepsilon_1}(q,\ell)}
  \mathfrak a_{a|b}^{\varepsilon_0,\varepsilon_1}(q,\ell,s)\wh w(q,s),
\end{equation}
with
\begin{equation}\label{eq:amplitude-N3}
  \abs{\mathfrak a_{a|b}^{\varepsilon_0,\varepsilon_1}(q,\ell,s)}
  +\abs{\partial_s\mathfrak a_{a|b}^{\varepsilon_0,\varepsilon_1}(q,\ell,s)}
  \lesssim_T N^{-3+2\beta_b}.
\end{equation}
The constants depend on the forcing profiles only through $\mathfrak H_{\mathrm{prof}}$.
\end{lemma}

\begin{proof}
The isonormal covariance and \eqref{eq:gaussian-forcing-covariance} give
\[
  \sigma_b(\ell;s,t)
  =\mathfrak q_b(\ell)\int_0^s
  \frac{\sin((s-r)\omega_b(\ell))}{\omega_b(\ell)}
  \frac{\sin((t-r)\omega_b(\ell))}{\omega_b(\ell)}\dd r.
\]
Use $2\sin A\sin B=\cos(A-B)-\cos(A+B)$ and integrate in $r$ to
obtain \eqref{eq:covariance-decomposition}.  Since
$\mathfrak q_b(\ell)\lesssim\langle\ell\rangle^{2\beta_b}$, the endpoint
branches are $O(N^{-3+2\beta_b})$.  The main branch contains
$\mathfrak q_b(\ell)s\omega_b(\ell)^{-2}$, while the outer Duhamel kernel
contributes $\omega_a(q+\ell)^{-1}$.  On the proportional low--high window
both dispersive frequencies are comparable to $N$, giving
\eqref{eq:amplitude-N3}.  After the oscillation is placed in the exponential,
time differentiation falls only on the scalar factor $s$ and on fixed
frequency cutoffs.
\end{proof}

Fix $a,b\in\mathfrak C$ with $a\ne b$, and let $Q\le c_{\mathrm{ap}}N$.  Each strict diagonal block is a finite sum
of operators of the form
\begin{align}\label{eq:diagonal-block}
  \wh{\cD^{a;b}_{N,Q}w}(q,t)
  ={}&\chi_Q(q)\int_0^t\int_{\R^3}
  \rho_N(\ell)^2
  m^{a;b}_{N,Q}(q,\ell)
  K_a(t-s,q+\ell)\notag\\
  &\qquad\qquad\times\sigma_b(\ell;s,t)
  \wh w(q,s)\dd\ell\dd s,
\end{align}
where $b$ is the contracted color and $a\ne b$, and $m^{a;b}_{N,Q}$ is a uniformly order-zero symbol.  The output frequency equals the input frequency $q$.

The common input space $E_T^{2,\sigma}$ is defined in \eqref{eq:E-space}.

\begin{lemma}[Dyadic Volterra estimate]\label{lem:dyadic-volterra}
Let $a,b\in\mathfrak C$ with $a\ne b$, $Q\le c_{\mathrm{ap}}N$, and $0<s<1-2\beta_b$.
\begin{align}\label{eq:dyadic-volterra-H}
  \norm{\cD^{a;b}_{N,Q}P_Qw}_{C_TH^{s-1}}
  \lesssim{}&N^{-1+2\beta_b}\norm{P_Qw}_{L_T^\infty L^2}\notag\\
  &+N^{-1+2\beta_b}Q^s
    \norm{\partial_tw}_{C_TH^{-1}},
\end{align}
and
\begin{equation}\label{eq:dyadic-volterra-B}
  \norm{\cD^{a;b}_{N,Q}P_Qw}_{L_T^\infty B_{2,\infty}^{\sigma-1}}
  \lesssim_T
  N^{-1+2\beta_b}
  \paren{Q^{-1}\norm{w}_{L_T^\infty B_{2,\infty}^{\sigma}}
  +\norm{\partial_tw}_{L_T^\infty B_{2,\infty}^{\sigma-1}}}.
\end{equation}
The constants are uniform in $N,Q$ and in the spectral cutoff.  If
$[\mathfrak h_b]_{\beta_b,0}\le H$ and $\parvec$ ranges over a separated
parameter class, the implicit constants are bounded by a constant depending
on $H$ and the parameter class.  The phase-gap constants themselves are
independent of the forcing profiles.
\end{lemma}

\begin{proof}
Use \cref{lem:covariance-decomposition}.  The two endpoint covariance branches are estimated absolutely.  Their loop size is
\[
  N^3\cdot N^{-1}\cdot N^{-3+2\beta_b}=N^{-1+2\beta_b}.
\]
For the main covariance branch on $N\ge N_{\mathrm{sp}}(\parvec)$, use
\cref{lem:explicit-phase-gap} and integrate \eqref{eq:oscillatory-branch} by
parts in $s$.  Since
$\partial_s e^{\ii(t-s)\Phi}=-\ii\Phi
 e^{\ii(t-s)\Phi}$, integration by parts gives
\begin{align}\label{eq:ibp-formula}
  \int_0^t e^{\ii(t-s)\Phi}A(s)\wh w(q,s)\dd s
  ={}&\frac{e^{\ii t\Phi}A(0)\wh w(q,0)-A(t)\wh w(q,t)}{\ii\Phi}\notag\\
  &+\int_0^t\frac{e^{\ii(t-s)\Phi}}{\ii\Phi}
  \paren{A'(s)\wh w(q,s)+A(s)\partial_s\wh w(q,s)}\dd s.
\end{align}
The reciprocal phase contributes $C_{\parvec}N^{-1}$.  The amplitude before
this division is $N^{-3+2\beta_b}$, and the $\ell$-integration has volume $N^3$.  Thus
the boundary and amplitude-derivative terms have scalar multiplier size
$C_{\parvec}N^{-1+2\beta_b}$.  The derivative-on-input term has the same high-shell
factor and contains $P_Q\partial_tw$.  For the finitely many shells
$N<N_{\mathrm{sp}}(\parvec)$, no phase division is made: the rough scalar
multiplier is $O_{\parvec}(1)$, which is bounded by
$C_{\parvec}N^{-1+2\beta_b}$ after enlarging the parameter-dependent constant.
This also covers a possible compact wave--Klein--Gordon crossing when
$\speedK<\speedW$.

Since the output remains at frequency $Q$ and $s-1<0$, the $H^{s-1}$ weight on the first line can be discarded.  For the derivative term,
\[
  \norm{P_Q\partial_tw}_{H^{s-1}}
  \lesssim Q^s\norm{\partial_tw}_{H^{-1}}.
\]
This proves \eqref{eq:dyadic-volterra-H}.

For the Besov estimate, fix the output shell $Q$ before summing the high loop.  Boundary and amplitude-derivative terms satisfy
\[
  Q^{\sigma-1}\norm{P_Qw}_{L^2}
  \le Q^{-1}\norm{w}_{B_{2,\infty}^{\sigma}},
\]
while the derivative term satisfies
\[
  Q^{\sigma-1}\norm{P_Q\partial_tw}_{L^2}
  \le\norm{\partial_tw}_{B_{2,\infty}^{\sigma-1}}.
\]
Combining these inequalities with the $N^{-1+2\beta_b}$ shell factor gives \eqref{eq:dyadic-volterra-B}.
\end{proof}

\begin{lemma}[Closed Volterra integration by parts on the completed input space]
\label{lem:closed-volterra-ibp}
Under the hypotheses of Lemma~\ref{lem:dyadic-volterra}, fix a dyadic pair
$(N,Q)$ and one of the finitely many signed main-phase branches obtained from
\eqref{eq:diagonal-block} after the covariance decomposition.  Denote this
branch by $\mathcal V_{N,Q}$.  The four terms on the right-hand side of
\eqref{eq:ibp-formula}, after the $\ell$-integration is restored, extend
separately and continuously from smooth finite-spatial-block paths to
$E_T^{2,\sigma}$.  For every $w\in E_T^{2,\sigma}$ the integration-by-parts
identity holds in
\[
 C_TH^{s-1}\cap L_T^\infty B_{2,\infty}^{\sigma-1},
\]
where the two components are interpreted with the corresponding estimates
below:
\begin{align}
 \|\mathcal V_{N,Q}w\|_{C_TH^{s-1}}
 &\lesssim
 N^{-1+2\beta_b}\Bigl(
 \|P_Qw\|_{C_TL^2}+Q^s\|\partial_tw\|_{C_TH^{-1}}\Bigr),
 \label{eq:closed-volterra-H}\\
 \|\mathcal V_{N,Q}w\|_{L_T^\infty B_{2,\infty}^{\sigma-1}}
 &\lesssim_T
 N^{-1+2\beta_b}\Bigl(
 Q^{-1}\|w\|_{L_T^\infty B_{2,\infty}^{\sigma}}
 +\|\partial_tw\|_{L_T^\infty B_{2,\infty}^{\sigma-1}}\Bigr).
 \label{eq:closed-volterra-B}
\end{align}
If $w_j$ is any smooth finite-block sequence converging to $w$ in
$E_T^{2,\sigma}$, then the pre-integration-by-parts Volterra branch and each
of the four terms in \eqref{eq:ibp-formula} converge to the corresponding
terms for $w$ in the displayed target spaces.  Thus the extension is independent of the approximating sequence, and
completion preserves the integration-by-parts formula.
\end{lemma}

\begin{proof}
For a fixed shell $Q$, the multiplier $P_Q$ maps $H^{-1}$ boundedly to
$L^2$, with norm $O(Q)$.  Hence, if $w_j\to w$ in $E_T^{2,\sigma}$, then
\[
 P_Qw_j\longrightarrow P_Qw\quad\hbox{in }C_TL^2,
 \qquad
 P_Q\partial_tw_j\longrightarrow P_Q\partial_tw
 \quad\hbox{in }C_TL^2.
\]
The boundary evaluations of $P_Qw_j$ converge in $L^2$ by the energy
component of the norm.  The Besov boundary values and the time-continuous
fixed-block representatives are the canonical ones supplied by
Lemma~\ref{lem:completed-besov-traces}.

For the chosen signed branch, the scalar amplitudes $A$ and $A'$ in
\eqref{eq:ibp-formula} are bounded uniformly on $[0,T]$, and the phase gap
gives $|\Phi|^{-1}\lesssim_{\parvec}N^{-1}$.  After the $\ell$-integration,
the multiplier bounds proved in Lemma~\ref{lem:dyadic-volterra} therefore
make the two boundary terms, the $A'w$ term, and the
$A\partial_tw$ term continuous with respect to the convergences above.  They
also give \eqref{eq:closed-volterra-H} and
\eqref{eq:closed-volterra-B} term by term.  For fixed $(N,Q)$, the original
Volterra integral before integration by parts depends continuously on
$P_Qw$ in $C_TL^2$, because its frequency support is compact and its time
kernel is bounded.  Passing to the limit in the classical identity for
$w_j$ proves the identity for $w$.  Since every term has a unique limit in a
complete target space, the result is independent of the approximating
sequence.
\end{proof}

\begin{proposition}[Summed deterministic diagonal]\label{prop:summed-diagonal}
Let $a,b\in\mathfrak C$ with $a\ne b$, $0<s<1-2\beta_b$ and $0<\sigma<1$.  The dyadic series
\begin{equation}\label{eq:diagonal-series}
  \cD^{a;b}w
  :=\sum_N\sum_{Q\le c_{\mathrm{ap}}N}
  \cD^{a;b}_{N,Q}P_Qw
\end{equation}
converges as a bounded deterministic operator
\begin{equation}\label{eq:diagonal-map}
  \cD^{a;b}:E_T^{2,\sigma}
  \longrightarrow
  C_TH^{s-1}\cap L_T^\infty B_{2,\infty}^{\sigma-1}.
\end{equation}
Moreover,
\begin{equation}\label{eq:diagonal-L1-small}
  \norm{\cD^{a;b}}_{E_T^{2,\sigma}\to L_T^1B_{2,\infty}^{\sigma-1}}
  \le T
  \norm{\cD^{a;b}}_{E_T^{2,\sigma}\to L_T^\infty B_{2,\infty}^{\sigma-1}},
\end{equation}
so the Besov source norm tends to zero as $T\downarrow0$ on bounded diagonal classes.
\end{proposition}

\begin{proof}
Assume first that $w$ is smooth with finitely many spatial Littlewood--Paley blocks,
so that the time integration by parts and every shell rearrangement are
classical.  Lemma~\ref{lem:closed-volterra-ibp} extends the decomposition,
termwise estimates, and endpoint traces to $E_T^{2,\sigma}$ without an
approximation-dependent defect.  Since $s-1+2\beta_b<0$, return to the four
terms in the proof of Lemma~\ref{lem:dyadic-volterra} before taking a time
supremum.  At every fixed
input shell $Q$, summation over $N\gtrsim Q$ gives
\begin{equation}\label{eq:diagonal-N-sum}
 \sum_{N\gtrsim Q}N^{-1+2\beta_b}\lesssim Q^{-1+2\beta_b}.
\end{equation}
Thus the boundary terms, the amplitude-derivative term, and the two endpoint
covariance branches are Fourier multipliers on the $Q$-shell with $L^2$
operator norm at most $CQ^{-1+2\beta_b}$; the same is true for the
coefficient of $P_Q\partial_tw$ in the time integral.

At each common time we first sum the output shells.  Littlewood--Paley orthogonality and
\eqref{eq:diagonal-N-sum} give, for the terms containing $w$,
\[
 \left(\sum_Q
 Q^{2(s-1)}Q^{-2+4\beta_b}\|P_Qw(t)\|_2^2\right)^{1/2}
 \lesssim\|w(t)\|_2,
\]
and the same estimate holds for $w(0)$ and after integration in the common
time variable.  For the differentiated input,
\begin{align*}
 &\left(\sum_QQ^{2(s-1)}Q^{-2+4\beta_b}
 \|P_Q\partial_tw(t)\|_2^2\right)^{1/2}\\
 &\qquad\le
 \sup_QQ^{s-1+2\beta_b}
 \left(\sum_QQ^{-2}\|P_Q\partial_tw(t)\|_2^2\right)^{1/2}
 \lesssim\|\partial_tw(t)\|_{H^{-1}}.
\end{align*}
Minkowski's inequality is applied after this spatial square summation and
proves the $C_TH^{s-1}$ bound.

For the Besov component, on every shell,
\begin{align*}
 Q^{\sigma-1}Q^{-1+2\beta_b}\|P_Qw\|_2
 &\le Q^{-2+2\beta_b}\|w\|_{B_{2,\infty}^{\sigma}},\\
 Q^{\sigma-1}Q^{-1+2\beta_b}\|P_Q\partial_tw\|_2
 &\le Q^{-1+2\beta_b}\|\partial_tw\|_{B_{2,\infty}^{\sigma-1}}.
\end{align*}
Taking the dyadic supremum and the appropriate time norms proves the
$L_T^\infty B_{2,\infty}^{\sigma-1}$ estimate.  Fixed-block cutoff convergence
and the common summable majorant \eqref{eq:diagonal-N-sum} give convergence of
the full deterministic operator.  Finally,
\eqref{eq:diagonal-L1-small} is the elementary inequality
$\|f\|_{L_T^1}\le T\|f\|_{L_T^\infty}$.
\end{proof}

The actual operator \eqref{eq:T-def} contains one compact multiplier in the
Duhamel source and one after the resonant product.  The same-color contraction
is not diagonal; the off-diagonal decay is quantified below.

\begin{lemma}[Kernel-to-shell Schur estimate]\label{lem:kernel-to-shell-schur}
Let $T$ be a Fourier integral operator
\[
 \widehat{Tf}(n)=\int_{\R^3}K(n,q)\widehat f(q)\dd q.
\]
Suppose that, for every $B>0$,
\begin{equation}\label{eq:kernel-schwartz-off-diagonal}
 |K(n,q)|\le A_B\langle n-q\rangle^{-B}.
\end{equation}
Then, for every $A>0$,
\begin{equation}\label{eq:kernel-to-shell-bound}
 \|P_MTP_Q\|_{L^2\to L^2}
 \lesssim_A A_{A+4}
 \left(\frac{\min\{M,Q\}}{\max\{M,Q\}}\right)^A,
 \qquad M,Q\in\Dyd.
\end{equation}
The same conclusion holds for a parameter-dependent family of kernels, with
$A_B$ replaced by the corresponding uniform majorant.  In particular, it
applies without change to time increments and to cutoff differences.
\end{lemma}

\begin{proof}
If $M\sim Q$, Schur's test and $B>3$ give
\[
 \sup_n\int|K(n,q)|\dd q+
 \sup_q\int|K(n,q)|\dd n\lesssim A_B.
\]
Assume next that $M\ge C Q$, where $C$ is larger than the annular support
constant.  On the support of $P_M(n)P_Q(q)$ one has $|n-q|\gtrsim M$; hence
both Schur integrals are bounded by
\[
 A_B\int_{|z|\gtrsim M}\langle z\rangle^{-B}\dd z
 \lesssim A_B M^{3-B}.
\]
Taking $B=A+4$ and using $Q\ge1$ yields
$M^{-A}\le(Q/M)^A$.  The case $Q\ge CM$ is symmetric.  If either dyadic
index is the inhomogeneous unit block, the same Schur integral is bounded by
the decay of $\langle n-q\rangle^{-B}$ away from the remaining annulus.
This proves
\eqref{eq:kernel-to-shell-bound}.  Applying the proof to a difference kernel
gives the last assertion.
\end{proof}

\begin{proposition}[Localized deterministic Volterra operator]
\label{prop:localized-diagonal}
Let $a,b\in\mathfrak C$ with $a\ne b$, and let $\chi_1,\chi_2\in C_c^\infty(\R^3)$.  The contraction branch of
\[
  \chi_2\Bigl(I_a\bigl(\chi_1(w\prec\Psi_b)\bigr)\circ\Psi_b\Bigr)
\]
converges to a deterministic operator
\[
  \mathcal D_{\chi_2,\chi_1}^{a;b}:
  E_T^{2,\sigma}\longrightarrow
  C_TH^{s-1}\cap L_T^\infty B_{2,\infty}^{\sigma-1},
  \qquad 0<s<1-2\beta_b,\quad 0<\sigma<1,
\]
with the same locally small $L_T^1B_{2,\infty}^{\sigma-1}$ consequence as
Proposition~\ref{prop:summed-diagonal}.  Its external Fourier kernel decays
faster than any power away from $n=q$, and all cutoff differences satisfy the
same bounds.
\end{proposition}

\begin{proof}
After the Wick contraction $r=-\ell$, let $z_1,z_2$ be the Fourier defects
introduced by the inner and outer physical multipliers.  The propagated and
final frequencies are
\[
 p=q+\ell+z_1,
 \qquad n=q+z_1+z_2,
\]
and the kernel contains
$\widehat\chi_1(z_1)\widehat\chi_2(z_2)$.  After summing the finitely many resonant and enlarged-shell choices, a
representative contracted block has the external kernel
\begin{equation}\label{eq:localized-contraction-kernel}
\begin{aligned}
 \mathcal K_{N,Q}^{a;b}(t,s;n,q)
 ={}&\chi_Q(q)\int_{\ell,z_1}
 \rho_N(\ell)^2\widehat\chi_1(z_1)
 \widehat\chi_2(n-q-z_1)\\
 &\quad\times \mathfrak m_{N,Q,z_1}(n,q,\ell)
 K_a(t-s,q+\ell+z_1)\sigma_b(\ell;s,t)
 \dd\ell\dd z_1,
\end{aligned}
\end{equation}
where $\sup_{N,Q,z_1}\|\mathfrak m_{N,Q,z_1}\|_{L^\infty}\lesssim1$.
The two Littlewood--Paley shell factors on the contracted variables are
displayed as $\rho_N(\ell)^2$; at finite cutoff they are accompanied by the
single covariance factor $m_\Lambda(\ell)^2$.  Bounded admissible spectral
multipliers and the finitely many resonant symbols are included in
$\mathfrak m$.  The same shell factor appears in the Hilbert-space formulation; see
\eqref{eq:same-shell-rho-square}.  Formula~\eqref{eq:localized-contraction-kernel}
follows by
substituting $r=-\ell$, $p=q+\ell+z_1$, and
$z_2=n-p-r=n-q-z_1$ before making any phase estimate.  It also shows directly
that the external momentum defect is $n-q=z_1+z_2$.

Choose
\[
 \eta:=\eta_{\mathrm{LP}}\delta_0,
\]
where $0<\eta_{\mathrm{LP}}\ll1$ depends only on the fixed
Littlewood--Paley support constants, so that
\[
 |q+z_1|\le\delta_0|\ell|
 \quad\text{whenever}\quad
 |q|\lesssim c_{\mathrm{ap}}N,
 \ |z_1|\le\eta N,
 \ |\ell|\sim N.
\]
Split the kernel into $|z_1|\le\eta N$ and $|z_1|>\eta N$.  The same choice
is used in Proposition~\ref{prop:speed-gap-dependence}.

On the near part and for $N\ge N_{\mathrm{sp}}$, one has, for every sign pair,
\begin{equation}\label{eq:localized-all-sign-phase-gap}
 |\varepsilon_0\omega_a(q+\ell+z_1)
   +\varepsilon_1\omega_b(\ell)|
 \gtrsim_{\parvec}N.
\end{equation}
The difference phase follows from Lemma~\ref{lem:explicit-phase-gap} applied
to the displacement $q+z_1$; the sum phases follow from positivity.  Apply
the Volterra identity \eqref{eq:ibp-formula}.  For each of the boundary,
amplitude-derivative, and input-derivative terms, and for every $B>0$, the
resulting external Fourier kernels satisfy
\begin{equation}\label{eq:localized-diagonal-near-kernels}
 \begin{aligned}
 |\mathcal K^{\mathrm{bd}}_{N,Q}(n,q)|
 +|\mathcal K^{\mathrm{amp}}_{N,Q}(n,q)|
 +|\mathcal K^{\mathrm{in}}_{N,Q}(n,q)|
 \lesssim_{B}
 N^{-1+2\beta_b}\mathbf1_{\{|q|\sim Q\}}
 \langle n-q\rangle^{-B}.
 \end{aligned}
\end{equation}
Here $\mathcal K^{\mathrm{in}}_{N,Q}$ is the multiplier acting on
$P_Q\partial_tw$.  The phase division contributes $N^{-1}$; the covariance weight and the loop
integration contribute $N^{2\beta_b}$.  The external decay is also explicit:
for every $B>0$,
\begin{equation}\label{eq:localized-convolution-schwartz}
 \int_{|z_1|\le\eta N}|\widehat\chi_1(z_1)|
 |\widehat\chi_2(n-q-z_1)|\dd z_1
 \lesssim_B\langle n-q\rangle^{-B},
\end{equation}
because the convolution of two Schwartz functions is Schwartz.  Combining
this with the uniform phase denominator and the bounded amplitude in
\eqref{eq:localized-contraction-kernel} gives
\eqref{eq:localized-diagonal-near-kernels} separately for each boundary,
amplitude-derivative, and input-derivative coefficient.

For $N<N_{\mathrm{sp}}$ no phase division is made.  There are only finitely
many such shells, and their rough kernel bounds are absorbed into
\eqref{eq:localized-diagonal-near-kernels} after enlarging the constant.  On
the far part, for arbitrary $B,L>0$,
\begin{equation}\label{eq:localized-diagonal-far-tail}
 \int_{|z_1|>\eta N}|\widehat\chi_1(z_1)|
 |\widehat\chi_2(n-q-z_1)|\dd z_1
 \lesssim_{B,L}\eta^{-L}N^{-L}\langle n-q\rangle^{-B}.
\end{equation}
Indeed, on $|z_1|>\eta N$ one inserts
$1\le(\eta N)^{-L}\langle z_1\rangle^L$ and then uses weighted Schwartz
convolution.  The same bound applies to each of the three Volterra terms after absorbing the
rough compact-frequency factors.  For fixed parameters the factor $\eta^{-L}$ is absorbed into the constant;
uniformly on separated classes it is a polynomial loss in the inverse speed gap by \eqref{eq:phase-parameter-ledger}.  Lemma~\ref{lem:kernel-to-shell-schur}
therefore yields
\begin{equation}\label{eq:localized-diagonal-shell-matrix}
 \|P_M\mathcal D_{\chi_2,\chi_1;N,Q}^{a;b}P_Q\|_{L^2\to L^2}
 \lesssim_{A}
 \bigl(N^{-1+2\beta_b}+N^{-L}\bigr)
 \left(\frac{\min\{M,Q\}}{\max\{M,Q\}}\right)^A,
\end{equation}
separately for the boundary, amplitude-derivative, and input-derivative
pieces.  Thus all shell-volume factors are contained in the Schur tail
integral.

Choose $A$ larger than the Sobolev and Besov weights.  Weighted Schur
summation in $(M,Q)$ reduces \eqref{eq:localized-diagonal-shell-matrix} to
Proposition~\ref{prop:summed-diagonal}, which proves the asserted mapping and
the $L_T^1B_{2,\infty}^{\sigma-1}$ smallness.

At finite cutoff the contraction multiplier contains
$m_\Lambda(\ell)m_\Lambda(-\ell)$.  Its difference from the limiting
multiplier is a finite telescope with one factor
$m_\Lambda(\ell)-1$ or $m_\Lambda(-\ell)-1$.  The bounds
\eqref{eq:localized-diagonal-near-kernels}--
\eqref{eq:localized-diagonal-shell-matrix} remain valid uniformly, while every
fixed shell difference tends to zero and is eventually zero when the shell
lies in the plateau of \eqref{eq:cutoff-supports}.  Dominated dyadic summation
gives convergence to the uncut operator.  Two non-nested admissible cutoffs
are compared through this common limit.
\end{proof}

\subsection{Centered operators and the operator estimate}
\label{sec:centered-operators}

We begin with the Hilbert-space estimate for a centered quadratic Gaussian
operator.  The self-adjoint Gaussian noncommutative Khintchine inequality is
used in its standard form; Hermitian dilation gives the rectangular estimate,
and Gaussian polarization gives the order-two decoupling identity.  The
subsequent lemmas incorporate covariance synthesis, continuous-frequency
localization, time regularity, and dyadic Sobolev--Besov summation.

\subsubsection{Gaussian operator estimate from four Schatten flattenings}

All Hilbert spaces in this subsection are real.  Complex Fourier $L^2$ spaces
are understood through realification, which changes none of the operator or
Schatten norms.  Let $\mathcal A,\mathcal C_1,\mathcal C_0,\mathcal E$ be
separable Hilbert spaces, and let $\mathbf H$ be a finite-rank four-linear form
on
$\mathcal A\times\mathcal C_1\times\mathcal C_0\times\mathcal E$.
Its four oriented flattenings are defined by
\begin{align}
 \langle F_1(x\otimes y\otimes z),w\rangle
 &=\mathbf H(x,y,z,w),
 &F_1&:\mathcal A\otimes_2\mathcal C_1\otimes_2\mathcal C_0\to\mathcal E,
 \label{eq:abstract-flat-1}\\
 \langle F_2(x\otimes y\otimes w),z\rangle
 &=\mathbf H(x,y,z,w),
 &F_2&:\mathcal A\otimes_2\mathcal C_1\otimes_2\mathcal E\to\mathcal C_0,
 \label{eq:abstract-flat-2}\\
 \langle F_3(x\otimes z),y\otimes w\rangle
 &=\mathbf H(x,y,z,w),
 &F_3&:\mathcal A\otimes_2\mathcal C_0\to\mathcal C_1\otimes_2\mathcal E,
 \label{eq:abstract-flat-3}\\
 \langle F_4(y\otimes z),x\otimes w\rangle
 &=\mathbf H(x,y,z,w),
 &F_4&:\mathcal C_1\otimes_2\mathcal C_0\to\mathcal A\otimes_2\mathcal E.
 \label{eq:abstract-flat-4}
\end{align}
Here $\otimes_2$ denotes the Hilbert tensor product.  For $2\le r<\infty$,
$\Sch_r$ denotes the Schatten class.

The four orientations arise from the two Khintchine steps.  The first step in
the $\mathcal C_1$-Gaussian coordinate produces a row operator and a column
operator; the second step in the $\mathcal A$-coordinate gives the following
four square functions:
\begin{center}
\small
\renewcommand{\arraystretch}{1.16}
\begin{tabularx}{0.96\textwidth}{@{}P{0.19\textwidth}P{0.27\textwidth}P{0.21\textwidth}Y@{}}
\toprule
second random sum & coefficient map & row square & column square\\
\midrule
$\sum_\mu g_\mu A_\mu$
& $A_\mu:\mathcal C_1\otimes_2\mathcal C_0\to\mathcal E$
& $\sum_\mu A_\mu A_\mu^*=F_1F_1^*$
& $\sum_\mu A_\mu^*A_\mu=F_4^*F_4$\\
$\sum_\mu g_\mu C_\mu$
& $C_\mu:\mathcal C_0\to\mathcal C_1\otimes_2\mathcal E$
& $\sum_\mu C_\mu C_\mu^*=F_3F_3^*$
& $\sum_\mu C_\mu^*C_\mu=F_2F_2^*$\\
\bottomrule
\end{tabularx}
\end{center}
Thus $F_1,F_4$ are the two orientations attached to the first row operator,
while $F_3,F_2$ are those attached to the first column operator.  These identities
will be matched with the four frequency directions of the localized kernel in
Proposition~\ref{prop:localized-block-hilbert}.

\begin{lemma}[Schatten ideal facts]\label{lem:centered-schatten-ideal}
For $2\le r<\infty$ and $T\in\Sch_2\cap\mathcal L$,
\begin{equation}\label{eq:centered-schatten-interpolation}
 \|T\|_{\Sch_r}
 \le \|T\|_{\mathcal L}^{1-2/r}\|T\|_{\Sch_2}^{2/r}.
\end{equation}
If $A,B$ are bounded and $T\in\Sch_r$, then
\begin{equation}\label{eq:centered-schatten-ideal}
 \|ATB\|_{\Sch_r}\le\|A\|\,\|T\|_{\Sch_r}\,\|B\|.
\end{equation}
Canonical regroupings of Hilbert tensor products are unitary, so a mere
reorientation of tensor legs preserves the Hilbert--Schmidt norm.
\end{lemma}

\begin{proof}
The interpolation estimate follows from
$\sum_j s_j(T)^r\le\|T\|_{\mathcal L}^{r-2}\sum_j s_j(T)^2$.
The two-sided ideal bound is standard for finite-rank operators and passes to
the Schatten completion.  Canonical tensor regroupings are unitary and hence
preserve singular values.
\end{proof}

\begin{lemma}[Uniform compression on compact Schatten sets]
\label{lem:compact-schatten-compression}
Let $2\le r<\infty$, let $P_m$ and $Q_m$ be finite-rank orthogonal
projections on separable Hilbert spaces $\mathcal H_0$ and $\mathcal H_1$,
respectively, and assume that $P_m\to I_{\mathcal H_0}$ and
$Q_m\to I_{\mathcal H_1}$ strongly.  If
$\mathfrak K\subset\Sch_r(\mathcal H_0,\mathcal H_1)$ is compact, then
\begin{equation}\label{eq:compact-schatten-compression}
 \sup_{T\in\mathfrak K}
 \bigl(\|(I-Q_m)T\|_{\Sch_r}+\|T(I-P_m)\|_{\Sch_r}\bigr)
 \longrightarrow0.
\end{equation}
The same conclusion holds for tensor products of strongly convergent
finite-rank projections and simultaneously for any finite collection of
compact Schatten sets.
\end{lemma}

\begin{proof}
For a fixed $T\in\Sch_r$, approximate $T$ in $\Sch_r$ by a finite-rank map
$T^{(0)}$.  Strong convergence is uniform on the finite-dimensional range and
co-range of $T^{(0)}$, so
\[
 \|(I-Q_m)T^{(0)}\|_{\Sch_r}
 +\|T^{(0)}(I-P_m)\|_{\Sch_r}\longrightarrow0.
\]
Contractivity of orthogonal projections in every Schatten norm then gives the
same conclusion for $T$.  To make the convergence uniform on $\mathfrak K$,
choose a finite $\varepsilon$-net $T_1,\dots,T_J$ in $\Sch_r$.  For $T$ within
$\varepsilon$ of $T_j$, the two compression errors for $T-T_j$ are at most
$2\varepsilon$ each, while the errors for the finite list $(T_j)$ tend to
zero uniformly in $j$.  Letting $\varepsilon\downarrow0$ proves
\eqref{eq:compact-schatten-compression}.  Strong convergence of tensor-product
projections follows first on elementary tensors and then by density and
uniform boundedness.  A finite union of compact sets is compact, which proves
the final assertion.
\end{proof}

\begin{lemma}[Simultaneous finite compression of the four flattenings]
\label{lem:four-flat-compression}
Let $2\le r<\infty$ and suppose that the four flattenings of $\mathbf H$
belong to $\Sch_r$.  Let
$P_{\mathcal A}^{(m)},P_{\mathcal C_1}^{(m)},
 P_{\mathcal C_0}^{(m)},P_{\mathcal E}^{(m)}$
be increasing finite-rank orthogonal projections converging strongly to the
identity on the corresponding Hilbert spaces, and define
\[
 \mathbf H_m(x,y,z,w)
 :=\mathbf H(P_{\mathcal A}^{(m)}x,
 P_{\mathcal C_1}^{(m)}y,
 P_{\mathcal C_0}^{(m)}z,
 P_{\mathcal E}^{(m)}w).
\]
Then the four flattenings of $\mathbf H_m$ are
\begin{align*}
 F_{1,m}&=P_{\mathcal E}^{(m)}F_1
 (P_{\mathcal A}^{(m)}\otimes P_{\mathcal C_1}^{(m)}
 \otimes P_{\mathcal C_0}^{(m)}),\\
 F_{2,m}&=P_{\mathcal C_0}^{(m)}F_2
 (P_{\mathcal A}^{(m)}\otimes P_{\mathcal C_1}^{(m)}
 \otimes P_{\mathcal E}^{(m)}),\\
 F_{3,m}&=(P_{\mathcal C_1}^{(m)}\otimes P_{\mathcal E}^{(m)})F_3
 (P_{\mathcal A}^{(m)}\otimes P_{\mathcal C_0}^{(m)}),\\
 F_{4,m}&=(P_{\mathcal A}^{(m)}\otimes P_{\mathcal E}^{(m)})F_4
 (P_{\mathcal C_1}^{(m)}\otimes P_{\mathcal C_0}^{(m)}).
\end{align*}
In particular,
\begin{equation}\label{eq:four-flat-compression-convergence}
 \|F_{\varkappa,m}-F_{\varkappa}\|_{\Sch_r}\longrightarrow0,
 \qquad 1\le\varkappa\le4.
\end{equation}
The same conclusion holds for a difference of two forms.
\end{lemma}

\begin{proof}
Testing against elementary tensors gives the four displayed factorizations.
If $P_m\to I$ strongly and $T\in\Sch_r$, then
$\|(I-P_m)T\|_{\Sch_r}+\|T(I-P_m)\|_{\Sch_r}\to0$: approximate $T$ in
$\Sch_r$ by a finite-rank map and use uniform convergence on its finite
range and co-range.  The tensor products of the projections also converge
strongly, so the same argument applies to each factorization and to
differences of forms.
\end{proof}

\begin{lemma}[Same-field symmetrization preserves the four Schatten norms]
\label{lem:same-field-symmetrization}
Assume $\mathcal A=\mathcal C_1=: \mathcal G$.  Define
\[
 \mathbf H^\tau(x,y,z,w):=\mathbf H(y,x,z,w),
 \qquad
 \mathbf H^{\mathrm{sym}}:=\tfrac12(\mathbf H+\mathbf H^\tau).
\]
Under the canonical Hilbert-tensor identifications,
\begin{align}
 \|F_1(\mathbf H^\tau)\|_{\Sch_r}&=\|F_1(\mathbf H)\|_{\Sch_r},
 &\|F_2(\mathbf H^\tau)\|_{\Sch_r}&=\|F_2(\mathbf H)\|_{\Sch_r},
 \label{eq:sym-flat-12}\\
 F_3(\mathbf H^\tau)&=F_4(\mathbf H),
 &F_4(\mathbf H^\tau)&=F_3(\mathbf H).
 \label{eq:sym-flat-34}
\end{align}
Consequently,
\begin{equation}\label{eq:sym-profile-control}
 \sum_{\varkappa=1}^4\|F_{\varkappa}(\mathbf H^{\mathrm{sym}})\|_{\Sch_r}
 \le \sum_{\varkappa=1}^4\|F_{\varkappa}(\mathbf H)\|_{\Sch_r}.
\end{equation}
Moreover, the antisymmetric part of the coefficient family contributes zero
to the Wick operator:
\begin{equation}\label{eq:antisymmetric-wick-zero}
 \sum_{\mu,\nu}H_{\mu\nu}^{\mathrm{asym}}(G_{\mu}G_{\nu}-\delta_{\mu\nu})=0.
\end{equation}
\end{lemma}

\begin{proof}
On $F_1$ and $F_2$, exchanging the two copies of $\mathcal G$ is
precomposition with the unitary tensor flip.  Testing on elementary tensors
shows that the same exchange interchanges $F_3$ and $F_4$.  The triangle
inequality gives \eqref{eq:sym-profile-control}.  Finally,
$G_\mu G_\nu-\delta_{\mu\nu}$ is symmetric in $(\mu,\nu)$, so its contraction
with the antisymmetric coefficient tensor vanishes, first for finite sums and
then by simultaneous compression.
\end{proof}

The standard self-adjoint Gaussian noncommutative Khintchine inequality
\cite{LPP,DNY} states that if $(h_{\nu})$ are independent standard Gaussians and $(B_{\nu})$ is a
finite family of self-adjoint operators on a finite-dimensional Hilbert space,
then, for $r\ge2$,
\begin{equation}\label{eq:self-adjoint-NCK}
 \left\|\sum_{\nu}h_{\nu}B_{\nu}\right\|_{L^r(\Omega;\Sch_r)}
 \lesssim\sqrt r\,
 \left\|\left(\sum_{\nu}B_{\nu}^2\right)^{1/2}\right\|_{\Sch_r}.
\end{equation}

\begin{lemma}[Hermitian dilation and the rectangular inequality]
\label{lem:rectangular-NCK-dilation}
Let $A_{\nu}:\mathcal H_0\to\mathcal H_1$ be a finite family of operators between
finite-dimensional Hilbert spaces, and let $(h_{\nu})$ be independent standard
Gaussians.  For every $r\ge2$,
\begin{equation}\label{eq:rectangular-NCK}
 \left\|\sum_{\nu}h_{\nu}A_{\nu}\right\|_{L^r(\Omega;\Sch_r)}
 \lesssim\sqrt r\left[
 \left\|\left(\sum_{\nu}A_{\nu}A_{\nu}^*\right)^{1/2}\right\|_{\Sch_r}
 +\left\|\left(\sum_{\nu}A_{\nu}^*A_{\nu}\right)^{1/2}\right\|_{\Sch_r}
 \right].
\end{equation}
\end{lemma}

\begin{proof}
For $A:\mathcal H_0\to\mathcal H_1$, introduce the self-adjoint Hermitian
dilation on $\mathcal H_1\oplus\mathcal H_0$,
\[
 \mathscr D(A):=
 \begin{pmatrix}0&A\\ A^*&0\end{pmatrix}.
\]
Its nonzero singular values are those of $A$, each with multiplicity two, so
\begin{equation}\label{eq:dilation-schatten-norm}
 \|\mathscr D(A)\|_{\Sch_r}=2^{1/r}\|A\|_{\Sch_r}.
\end{equation}
Moreover,
\[
 \sum_{\nu}\mathscr D(A_{\nu})^2
 =\begin{pmatrix}
   \sum_{\nu}A_{\nu}A_{\nu}^*&0\\[1mm]
   0&\sum_{\nu}A_{\nu}^*A_{\nu}
  \end{pmatrix}.
\]
Apply \eqref{eq:self-adjoint-NCK} to $(\mathscr D(A_{\nu}))_{\nu}$, use
\eqref{eq:dilation-schatten-norm}, and then bound the $\ell^r$ sum of the two
diagonal Schatten norms by their ordinary sum.  This gives
\eqref{eq:rectangular-NCK}.
\end{proof}

\begin{lemma}[Elementary Banach-valued order-two Gaussian decoupling]
\label{lem:elementary-order-two-decoupling}
Let $B$ be a real Banach space, let $(x_{\mu\nu})_{1\le \mu,\nu\le d}\subset B$ be a
finite symmetric family, and let $G=(G_{\mu})_{\mu=1}^d$, $g=(g_{\mu})_{\mu=1}^d$, and
$h=(h_{\nu})_{\nu=1}^d$ be standard Gaussian vectors, with $g$ and $h$
independent.  For every $1\le p<\infty$,
\begin{equation}\label{eq:order-two-decoupling}
 \left\|\sum_{\mu,\nu}x_{\mu\nu}(G_{\mu}G_{\nu}-\delta_{\mu\nu})\right\|_{L^p(\Omega;B)}
 \le2
 \left\|\sum_{\mu,\nu}x_{\mu\nu}g_{\mu}h_{\nu}\right\|_{L^p(\Omega\times\Omega';B)}.
\end{equation}
\end{lemma}

\begin{proof}
Let $G'$ be an independent copy of $G$ and put
\[
 Q(G):=\sum_{\mu,\nu}x_{\mu\nu}(G_{\mu}G_{\nu}-\delta_{\mu\nu}).
\]
Since $\E Q(G')=0$, conditional Jensen gives
\begin{equation}\label{eq:decoupling-jensen}
 \|Q(G)\|_{L^p(B)}
 \le \|Q(G)-Q(G')\|_{L^p(B)}.
\end{equation}
Set
\[
 G^+:=\frac{G+G'}{\sqrt2},
 \qquad
 G^-:=\frac{G-G'}{\sqrt2}.
\]
The vectors $G^+$ and $G^-$ are independent standard Gaussians.  The diagonal
centering cancels in the difference, and symmetry of $(x_{\mu\nu})$ gives the
polarization identity
\begin{align*}
 Q(G)-Q(G')
 &=\sum_{\mu,\nu}x_{\mu\nu}(G^+_{\mu}G^-_{\nu}+G^-_{\mu}G^+_{\nu})\\
 &=2\sum_{\mu,\nu}x_{\mu\nu}G^+_{\mu}G^-_{\nu}.
\end{align*}
Combining this identity with \eqref{eq:decoupling-jensen} proves
\eqref{eq:order-two-decoupling}.  The polarization identity also holds after realification of a complex Banach space.  For broader decoupling formulations,
see \cite{deLaPenaGine,Janson}.
\end{proof}

\begin{theorem}[Centered Gaussian operator estimate]
\label{thm:centered-gaussian-operator}
Let $2\le p\le r<\infty$, and suppose the four flattenings
\eqref{eq:abstract-flat-1}--\eqref{eq:abstract-flat-4} belong to $\Sch_r$.
Choose orthonormal bases $(e_{\mu})$ and $(f_{\nu})$ on the two Gaussian legs and
write
\[
 \langle H_{\mu\nu}z,w\rangle=\mathbf H(e_{\mu},f_{\nu},z,w).
\]
For independent Gaussian families $(g_{\mu})$ and $(h_{\nu})$, define
\[
 Z_{\mathbf H}^{\mathrm{dec}}:=\sum_{\mu,\nu}g_{\mu}h_{\nu}H_{\mu\nu}.
\]
Then
\begin{equation}\label{eq:four-flat-bound}
 \left\|\|Z_{\mathbf H}^{\mathrm{dec}}\|_{\mathcal L(\mathcal C_0,\mathcal E)}
 \right\|_{L^p(\Omega)}
 \lesssim r\sum_{\varkappa=1}^4\|F_{\varkappa}\|_{\Sch_r}.
\end{equation}
If $\mathcal A=\mathcal C_1=: \mathcal G$ and both Gaussian legs are the
same isonormal process on $\mathcal G$, the same bound holds for the completely
centered Wick operator
\begin{equation}\label{eq:same-field-wick-operator}
 Z_{\mathbf H}^{\mathrm{Wick}}
 :=\sum_{\mu,\nu}H_{\mu\nu}^{\mathrm{sym}}(G_{\mu}G_{\nu}-\delta_{\mu\nu}),
 \qquad H_{\mu\nu}^{\mathrm{sym}}:=\tfrac12(H_{\mu\nu}+H_{\nu\mu}).
\end{equation}
In both cases the estimate extends uniquely from finite-rank forms to forms
whose four flattenings lie in $\Sch_r$.
\end{theorem}

\begin{proof}
Assume first that all Hilbert spaces are finite-dimensional.  For fixed $g=(g_{\mu})$ put
$S_{\nu}(g)=\sum_{\mu} g_{\mu}H_{\mu\nu}$.  Applying \eqref{eq:rectangular-NCK} in the
$h$ variables gives
\begin{equation}\label{eq:first-Khintchine-step}
 \|Z_{\mathbf H}^{\mathrm{dec}}(g,\cdot)\|_{L_h^r\Sch_r}
 \lesssim\sqrt r\bigl(\|\mathscr S(g)\|_{\Sch_r}
 +\|\mathscr T(g)\|_{\Sch_r}\bigr),
\end{equation}
where
\begin{align*}
 \mathscr S(g)&:\mathcal C_1\otimes_2\mathcal C_0\to\mathcal E,
 &\mathscr S(g)(f_{\nu}\otimes z)&=S_{\nu}(g)z,\\
 \mathscr T(g)&:\mathcal C_0\to\mathcal C_1\otimes_2\mathcal E,
 &\mathscr T(g)z&=\sum_{\nu}f_{\nu}\otimes S_{\nu}(g)z.
\end{align*}
Write $\mathscr S(g)=\sum_{\mu}g_{\mu}A_{\mu}$ and
$\mathscr T(g)=\sum_{\mu}g_{\mu}C_{\mu}$, with
\[
 A_{\mu}(f_{\nu}\otimes z)=H_{\mu\nu}z,
 \qquad
 C_{\mu}z=\sum_{\nu}f_{\nu}\otimes H_{\mu\nu}z.
\]
The four square functions are the flattenings introduced above:
\begin{equation}\label{eq:four-square-identities}
 F_1F_1^*=\sum_{\mu}A_{\mu}A_{\mu}^*,\quad
 F_4^*F_4=\sum_{\mu}A_{\mu}^*A_{\mu},\quad
 F_3F_3^*=\sum_{\mu}C_{\mu}C_{\mu}^*,\quad
 F_2F_2^*=\sum_{\mu}C_{\mu}^*C_{\mu}.
\end{equation}
Let $(z_j)$ be an orthonormal basis of
$\mathcal C_0$, and let $(w_k)$ be an orthonormal basis of $\mathcal E$.
For $w\in\mathcal E$,
\[
 \sum_\mu\|A_\mu^*w\|_{\mathcal C_1\otimes_2\mathcal C_0}^2
 =\sum_{\mu,\nu,j}|\mathbf H(e_\mu,f_\nu,z_j,w)|^2
 =\|F_1^*w\|_{\mathcal A\otimes_2\mathcal C_1\otimes_2\mathcal C_0}^2,
\]
which proves the first identity.  For $y\in\mathcal C_1$ and
$z\in\mathcal C_0$, polarization of
\[
 \sum_\mu\|A_\mu(y\otimes z)\|_{\mathcal E}^2
 =\sum_{\mu,k}|\mathbf H(e_\mu,y,z,w_k)|^2
 =\|F_4(y\otimes z)\|_{\mathcal A\otimes_2\mathcal E}^2
\]
proves the second.  Similarly, for $Y\in\mathcal C_1\otimes_2\mathcal E$,
\[
 \sum_\mu\|C_\mu^*Y\|_{\mathcal C_0}^2
 =\|F_3^*Y\|_{\mathcal A\otimes_2\mathcal C_0}^2,
\]
and, for $z\in\mathcal C_0$,
\[
 \sum_\mu\|C_\mu z\|_{\mathcal C_1\otimes_2\mathcal E}^2
 =\|F_2^*z\|_{\mathcal A\otimes_2\mathcal C_1\otimes_2\mathcal E}^2.
\]
The equalities on elementary tensors extend by density, and polarization gives
the corresponding operator identities.  Consequently
\begin{align}\label{eq:four-square-schatten-identities}
 \left\|\left(\sum_\mu A_\mu A_\mu^*\right)^{1/2}\right\|_{\Sch_r}
 &=\|F_1\|_{\Sch_r},&
 \left\|\left(\sum_\mu A_\mu^*A_\mu\right)^{1/2}\right\|_{\Sch_r}
 &=\|F_4\|_{\Sch_r},\notag\\
 \left\|\left(\sum_\mu C_\mu C_\mu^*\right)^{1/2}\right\|_{\Sch_r}
 &=\|F_3\|_{\Sch_r},&
 \left\|\left(\sum_\mu C_\mu^*C_\mu\right)^{1/2}\right\|_{\Sch_r}
 &=\|F_2\|_{\Sch_r}.
\end{align}
Thus the four Hilbert directions correspond to the row and column squares
produced by the two successive rectangular Khintchine inequalities.
Applying \eqref{eq:rectangular-NCK} once more, now in the $g$ variables, and
using \eqref{eq:four-square-identities}, yields
\[
 \|Z_{\mathbf H}^{\mathrm{dec}}\|_{L^r(\Omega;\Sch_r)}
 \lesssim r\sum_{\varkappa=1}^4\|F_{\varkappa}\|_{\Sch_r}.
\]
Since $p\le r$ and $\|T\|_{\mathcal L}\le\|T\|_{\Sch_r}$, this proves
\eqref{eq:four-flat-bound} in finite dimension.

In the same-field case, \eqref{eq:antisymmetric-wick-zero} permits replacement
of $\mathbf H$ by $\mathbf H^{\mathrm{sym}}$.  Apply the order-two decoupling
inequality \eqref{eq:order-two-decoupling} with
$B=\mathcal K(\mathcal C_0,\mathcal E)$ and then apply the decoupled estimate
to $\mathbf H^{\mathrm{sym}}$.  The profile does not increase by
\eqref{eq:sym-profile-control}.  This proves the same-field estimate for
finite-rank forms.

For general Schatten flattenings, use one simultaneous compression of all
four Hilbert legs.  Lemma~\ref{lem:four-flat-compression} shows that every
flattening of the compressed form, and of the difference of two consecutive
compressions, converges in $\Sch_r$.  The finite-dimensional estimate therefore
makes the associated random operators Cauchy in
$L^p(\Omega;\mathcal K(\mathcal C_0,\mathcal E))$.  Since the compact-operator
space is separable in operator norm, the limit is strongly measurable.  The
same difference estimate proves independence of the chosen projections and
uniqueness of the extension.  On each finite-dimensional compression the
coordinate formula is the contraction of the deterministic tensor with
the corresponding isonormal Gaussian vectors; it is therefore invariant under
orthogonal changes of basis.  Passing to the compression limit proves that the
extended random operator is basis independent as well.  In the same-field case
use the same projection on the two copies of $\mathcal G$;
Lemma~\ref{lem:same-field-symmetrization} commutes with this compression and
gives the identical conclusion.
\end{proof}

\begin{lemma}[Covariance pullback]\label{lem:covariance-pullback}
Let $U_1:\widetilde{\mathcal A}\to\mathcal A$ and
$U_2:\widetilde{\mathcal C}_1\to\mathcal C_1$ be bounded, and put
\[
 \widetilde{\mathbf H}(\widetilde x,\widetilde y,z,w)
 :=\mathbf H(U_1\widetilde x,U_2\widetilde y,z,w).
\]
For $2\le r\le\infty$,
\begin{equation}\label{eq:covariance-pullback}
 \max_{1\le\varkappa\le4}\|\widetilde F_{\varkappa}\|_{\Sch_r}
 \le\|U_1\|\,\|U_2\|\max_{1\le\varkappa\le4}\|F_{\varkappa}\|_{\Sch_r},
\end{equation}
where $\Sch_\infty$ means the operator norm.
\end{lemma}

\begin{proof}
For $\varkappa=1,2$,
$\widetilde F_{\varkappa}=F_{\varkappa}(U_1\otimes U_2\otimes I)$.  In the two cross orientations,
\[
 \widetilde F_3=(U_2^*\otimes I)F_3(U_1\otimes I),
 \qquad
 \widetilde F_4=(U_1^*\otimes I)F_4(U_2\otimes I).
\]
Apply \eqref{eq:centered-schatten-ideal}.
\end{proof}

\begin{lemma}[Two-parameter Banach-valued time lift]
\label{lem:centered-time-lift}
Let $B$ be a separable Banach space and
$\Delta_T=\{(t,s):0\le s\le t\le T\}$, equipped with the Euclidean metric.
Let $X:\Delta_T\to L^{p_0}(\Omega;B)$ satisfy, for some $p_0\ge2$ and
$\theta>0$,
\begin{equation}\label{eq:time-lift-hypotheses}
 \sup_{u\in\Delta_T}\|X(u)\|_{L^{p_0}(\Omega;B)}\le A,
 \qquad
 \|X(u)-X(v)\|_{L^{p_0}(\Omega;B)}\le B_0|u-v|^\theta.
\end{equation}
If $p_0\theta>2$, then $X$ has a $B$-valued continuous modification and,
for every $2\le p\le p_0$,
\begin{equation}\label{eq:centered-time-lift}
 \|X\|_{L^p(\Omega;C(\Delta_T;B))}
 \lesssim_{T,p_0,\theta}A+B_0.
\end{equation}
The same assertion holds for a sequence of kernels, and for their differences,
with constants uniform whenever the two bounds in
\eqref{eq:time-lift-hypotheses} are uniform.
\end{lemma}

\begin{proof}
Let $\mathcal G_m$ be the triangular dyadic grid of mesh $2^{-m}T$ in
$\Delta_T$, joined to the preceding grid and along nearest-neighbor edges.
There are $O(2^{2m})$ such edges.  If $M_m$ is the largest increment on the
level-$m$ edge set, then
\[
 \|M_m\|_{L^{p_0}}
 \lesssim B_0T^\theta 2^{-m(\theta-2/p_0)}.
\]
Since $p_0\theta>2$, $\sum_mM_m<\infty$ almost surely and in $L^{p_0}$.
Chaining along parent edges therefore gives a uniformly continuous limit on
the countable union of the grids and hence a continuous $B$-valued extension
to $\Delta_T$.  The increment hypothesis identifies this extension with the
original process at every fixed point.  The same chain yields
\[
 \left\|\sup_{u\in\Delta_T}\|X(u)\|_B\right\|_{L^{p_0}}
 \lesssim A+B_0,
\]
and monotonicity gives the estimate for $p\le p_0$.  Separability ensures
strong measurability.  Applying the same construction to differences proves
the final assertion.
\end{proof}

\begin{lemma}[Uniform finite-rank completion on the time triangle]
\label{lem:centered-uniform-compression}
Let $2\le p\le p_0\le r<\infty$, let $0<\theta' <\theta$, and assume
$p_0\theta'>2$.  For $u\in\Delta_T$, let $\mathbf H_u$ be a four-linear form
whose four flattenings $F_{\varkappa,u}$ belong to $\Sch_r$.  Let
$\mathbf H_u^{(m)}$ be simultaneous finite-rank compressions as in
\cref{lem:four-flat-compression}, and let $Z_u^{(m)}$ denote either the
decoupled Gaussian operator or, in the same-field case, the completely
centered Wick operator.  Put
\begin{align*}
 A_{m,n}&:=\sum_{\varkappa=1}^4
 \sup_{u\in\Delta_T}
 \|F_{\varkappa,u}^{(m)}-F_{\varkappa,u}^{(n)}\|_{\Sch_r},\\
 B_{m,n}&:=\sum_{\varkappa=1}^4
 \sup_{u\ne v}
 \frac{\|(F_{\varkappa,u}^{(m)}-F_{\varkappa,u}^{(n)})
 -(F_{\varkappa,v}^{(m)}-F_{\varkappa,v}^{(n)})\|_{\Sch_r}}
 {|u-v|^\theta}.
\end{align*}
Then
\begin{equation}\label{eq:centered-uniform-compression-bound}
 \|Z^{(m)}-Z^{(n)}\|_{L^p(\Omega;
 C(\Delta_T;\mathcal K(\mathcal C_0,\mathcal E)))}
 \lesssim_{T,p_0,\theta,\theta'}
 r\left(A_{m,n}
 +A_{m,n}^{1-\theta'/\theta}B_{m,n}^{\theta'/\theta}\right).
\end{equation}
In particular, if $A_{m,n}\to0$ and $(B_{m,n})$ is bounded, the compressed
operators converge in the displayed Bochner space.  The limit is strongly
measurable, basis independent, and independent of the simultaneous
compression scheme.
\end{lemma}

\begin{proof}
Apply \cref{thm:centered-gaussian-operator} at moment $p_0$ to the difference
$\mathbf H_u^{(m)}-\mathbf H_u^{(n)}$.  Uniformly in $u$,
\[
 \|Z_u^{(m)}-Z_u^{(n)}\|_{L^{p_0}(\Omega;\mathcal K)}
 \lesssim r A_{m,n}.
\]
Applying the same theorem to
\[
 (\mathbf H_u^{(m)}-\mathbf H_u^{(n)})
 -(\mathbf H_v^{(m)}-\mathbf H_v^{(n)})
\]
gives the $\theta$-increment bound
\[
 \|(Z_u^{(m)}-Z_u^{(n)})-(Z_v^{(m)}-Z_v^{(n)})\|_{L^{p_0}(\Omega;\mathcal K)}
 \lesssim r B_{m,n}|u-v|^\theta.
\]
The minimum of this estimate and the static bound is bounded by
\[
 C r A_{m,n}^{1-\theta'/\theta}B_{m,n}^{\theta'/\theta}
 |u-v|^{\theta'}.
\]
Indeed, $\min\{A,Bh^\theta\}\le
 A^{1-\theta'/\theta}B^{\theta'/\theta}h^{\theta'}$ for $h\ge0$.
Now apply \cref{lem:centered-time-lift} with exponent $\theta'$ to obtain
\eqref{eq:centered-uniform-compression-bound}.

Every compressed process is a finite Gaussian polynomial with values in a
finite-dimensional subspace of
$C(\Delta_T;\mathcal K(\mathcal C_0,\mathcal E))$, hence is strongly
measurable.  The Bochner limit is therefore strongly measurable.  If two
compression schemes are used, apply the estimate to their common refinement;
this proves independence.  Basis independence holds at finite rank and passes
to the limit.
\end{proof}

\subsubsection{Covariance synthesis and localization}

For every Littlewood--Paley shell multiplier $\rho_N$, fix once and for all an
even auxiliary plateau $\widetilde\rho_N\in C_c^\infty(\R^3)$ such that
\begin{equation}\label{eq:auxiliary-shell-plateau}
 \widetilde\rho_N\equiv1\quad\hbox{on }\supp\rho_N,
 \qquad
 \supp\widetilde\rho_N\subset\{\widetilde\kappa_-N\le|\xi|
 \le\widetilde\kappa_+N\},
 \qquad
 |\partial^\alpha\widetilde\rho_N|\lesssim_\alpha N^{-|\alpha|},
\end{equation}
with the evident inhomogeneous modification for $N=1$.  In particular,
\begin{equation}\label{eq:plateau-times-shell}
 \widetilde\rho_N\rho_N=\rho_N.
\end{equation}
The functions $\widetilde\rho_N$ localize the deterministic Hilbert-space
kernels.  The Littlewood--Paley factors $\rho_N$ and the spectral cutoffs
$m_\Lambda$ are incorporated in the covariance synthesis maps.

Let
\[
 \mathfrak H_a:=L^2([0,T]\times\R_x^3;\R)
\]
be the real Hilbert space underlying the isonormal process $\mathbf W_a$.  Via the
spatial Fourier transform, we identify it unitarily, as a real Hilbert space,
with the Hermitian subspace
\begin{equation}\label{eq:hermitian-gaussian-hilbert-space}
 \widehat{\mathfrak H}_a
 :=\{F\in L^2([0,T]\times\R_\ell^3;\mathbb C):
 F(r,-\ell)=\overline{F(r,\ell)}\},
 \qquad
 \langle F,G\rangle_{\R}:=\Re\int F\overline G.
\end{equation}
The Hermitian symmetry couples the Fourier coordinates at $\ell$ and
$-\ell$ and realizes the distributional covariance in
\eqref{eq:sigma-def} within the real isonormal process.  In the formulas below $F(r,\ell)$ denotes
the Fourier representative in \eqref{eq:hermitian-gaussian-hilbert-space}.
For $N\ge2$ and an even bounded multiplier $m$, define
\begin{equation}\label{eq:covariance-synthesis-map}
 (\mathsf S_{a,N,t}^{(m)}F)(\ell)
 :=\rho_N(\ell)m(\ell)\mathfrak h_a(\ell)
   \int_0^tK_a(t-r,\ell)F(r,\ell)\dd r,
 \qquad F\in\widehat{\mathfrak H}_a.
\end{equation}
We write
\begin{equation}\label{eq:synthesis-cutoff-notation}
 \mathsf S_{a,N,t}:=\mathsf S_{a,N,t}^{(1)},
 \qquad
 \mathsf S_{a,N,t}^{\Lambda}:=
 \mathsf S_{a,N,t}^{(m_\Lambda)},
 \qquad
 m_\infty:=1.
\end{equation}
The evenness of the propagator and shell cutoffs makes these bounded maps of
real Hilbert spaces into the Hermitian Fourier $L^2$ shell.  The inhomogeneous
unit shell is defined by the same formula and is absorbed into the constants;
there and below inverse shell powers are interpreted with Japanese brackets.
For every Hermitian test function $\phi\in L^2(\R^3)$, the definition gives the
first-chaos representation
\begin{equation}\label{eq:first-chaos-synthesis}
 \bigl\langle P_N\Psi_{a,\Lambda}(t),\phi\bigr\rangle_{\R}
 =\mathbf W_a\!\left((\mathsf S_{a,N,t}^{\Lambda})^*\phi\right),
\end{equation}
and the same identity with $\Lambda=\infty$ for $P_N\Psi_a$.  Consequently,
for any two shells $N,L$,
\begin{align}\label{eq:shell-covariance-from-synthesis}
 &\E\!\left[
  \widehat{P_N\Psi_{a,\Lambda}}(s,\ell)
  \widehat{P_L\Psi_{a,\Lambda}}(t,r)\right]\notag\\
 &\qquad=(2\pi)^3\delta(\ell+r)\,
 \rho_N(\ell)\rho_L(r)m_\Lambda(\ell)m_\Lambda(r)
 \sigma_a(\ell;s,t)
\end{align}
in the non-conjugated Fourier convention.  Thus each stochastic leg carries
its Littlewood--Paley multiplier through the corresponding synthesis map.

\begin{lemma}[Synthesis-map bounds]
\label{lem:covariance-synthesis}
Uniformly for $0\le t,t'\le T_0$, on separated parameter classes, and for bounded $\mathfrak H_{\mathrm{prof}}$,
\begin{align}
 \|\mathsf S_{a,N,t}\|_{\mathcal L(\mathfrak H_a,L^2_\ell)}
 &\lesssim N^{-1+\beta_a},
 \label{eq:synthesis-size}\\
 \|\mathsf S_{a,N,t}-\mathsf S_{a,N,t'}\|_{
 \mathcal L(\mathfrak H_a,L^2_\ell)}
 &\lesssim_{\theta}|t-t'|^\theta N^{-1+\beta_a+\theta},
 \qquad 0<\theta\le\frac12.
 \label{eq:synthesis-increment}
\end{align}
If two admissible multipliers are inserted, their difference satisfies the
same size and time-increment majorants, up to a factor two.
For every fixed shell $N$, the difference between the cutoff synthesis map and
the limiting map tends to zero in operator norm along every admissible cofinal
sequence; the proof uses only the plateau and uniform symbol bounds.
\end{lemma}

\begin{proof}
By Cauchy--Schwarz in the noise time,
\[
 \|\mathsf S_{a,N,t}\|^2
 \le \sup_{|\ell|\sim N}\mathfrak q_a(\ell)
       \int_0^t|K_a(t-r,\ell)|^2\dd r
 \lesssim N^{-2+2\beta_a}.
\]
Assume $t'>t$.  On $[0,t]$ use
\[
 |K_a(t-r,\ell)-K_a(t'-r,\ell)|
 \lesssim |t-t'|^\theta\langle\ell\rangle^{-1+\theta}.
\]
On the boundary strip $[t,t']$, its $L^2_r$ norm is bounded by
$|t-t'|^{1/2}N^{-1+\beta_a}$, which is at most
$C_{T_0,\theta}|t-t'|^\theta N^{-1+\beta_a+\theta}$ for
$0<\theta\le1/2$ and $N\ge1$.  This proves
\eqref{eq:synthesis-increment}, including the inhomogeneous wave block because
the wave kernel is interpreted by its continuous value at zero.

Let $m_1,m_2$ be two admissible multipliers and write
$\mathsf S_{a,N,t}^{(m_i)}$ for the corresponding synthesis maps.  Their
difference is
\[
 (\mathsf S_{a,N,t}^{(m_1)}-\mathsf S_{a,N,t}^{(m_2)})F(\ell)
 =\rho_N(\ell)(m_1-m_2)(\ell)\mathfrak h_a(\ell)
   \int_0^tK_a(t-r,\ell)F(r,\ell)\dd r.
\]
Consequently,
\begin{align*}
 \|\mathsf S_{a,N,t}^{(m_1)}-\mathsf S_{a,N,t}^{(m_2)}\|
 &\lesssim
 \|m_1-m_2\|_{L^\infty(\supp\rho_N)}N^{-1+\beta_a},\\
 \|(\mathsf S_{a,N,t}^{(m_1)}-\mathsf S_{a,N,t}^{(m_2)})
 -(\mathsf S_{a,N,t'}^{(m_1)}-\mathsf S_{a,N,t'}^{(m_2)})\|
 &\lesssim_\theta
 \|m_1-m_2\|_{L^\infty(\supp\rho_N)}
 |t-t'|^\theta N^{-1+\beta_a+\theta}.
\end{align*}
Uniform symbol bounds give the asserted factor-two majorants.  For fixed
$N$, the plateau in \eqref{eq:cutoff-supports} contains the enlarged
$N$-shell once $\Lambda$ is sufficiently large.  The cutoff and limiting synthesis maps then agree on that shell.  The same
argument applies to arbitrary, possibly non-nested, cofinal sequences.
\end{proof}

\begin{lemma}[Wick realization after covariance pullback]
\label{lem:wick-pullback-realization}
Let $b,c\in\mathfrak C$.  Let $\mathbf W_b,\mathbf W_c$ be the real isonormal processes on the Gaussian Hilbert spaces
of colors $b,c$.  Let $S_1:\mathfrak H_b\to\mathcal A$ and
$S_2:\mathfrak H_c\to\mathcal C_1$ be bounded, and let $\mathbf H$ be a
finite-rank four-linear form.  Put
\[
 \mathbf H^\sharp(f,g,x,y):=\mathbf H(S_1f,S_2g,x,y).
\]
If $b\ne c$, then the operator obtained by pairing the two pushed-forward
first chaoses with $\mathbf H$ has, for arbitrary orthonormal bases
$(e_\mu)$ and $(f_\nu)$, the coordinate representation
\begin{equation}\label{eq:distinct-color-pullback-realization}
 Z^\sharp
 =\sum_{\mu,\nu}
 \mathbf W_b(e_\mu)\mathbf W_c(f_\nu)H^\sharp_{\mu\nu},
 \qquad
 \langle H^\sharp_{\mu\nu}x,y\rangle
 =\mathbf H^\sharp(e_\mu,f_\nu,x,y).
\end{equation}
If $b=c$ and the two Hilbert spaces are identified with the same
$\mathfrak H_b$, the completely Wick-centered pairing is
\begin{equation}\label{eq:same-color-pullback-realization}
 Z^{\sharp,\mathrm{Wick}}
 =\sum_{\mu,\nu}
 (H^\sharp_{\mu\nu})^{\mathrm{sym}}
 \bigl(\mathbf W_b(e_\mu)\mathbf W_b(e_\nu)-\delta_{\mu\nu}\bigr).
\end{equation}
These identities are basis independent.  If the four flattenings of
$\mathbf H^\sharp$ lie in $\Sch_r$, they remain valid as identities in the
$L^p(\Omega;\mathcal K(\mathcal C_0,\mathcal E))$ limit supplied by
Theorem~\ref{thm:centered-gaussian-operator}.
\end{lemma}

\begin{proof}
For finite-rank projections on the Gaussian Hilbert spaces, expand the two
pushed-forward first chaoses in orthonormal coordinates.  Independence of the
colors gives \eqref{eq:distinct-color-pullback-realization}.  For one color,
the finite-dimensional Wick identity subtracts the scalar covariance
$\langle e_\mu,e_\nu\rangle=\delta_{\mu\nu}$, and symmetry of the
second Wiener integral replaces the coefficient tensor by its symmetrization,
giving \eqref{eq:same-color-pullback-realization}.  Orthogonal changes of basis
leave the isonormal Gaussian vectors and the tensor contraction unchanged.
Finally use simultaneous finite compression, the corresponding difference estimate for the four Schatten norms, and completeness of $L^p(\Omega;\mathcal K)$ to pass to the limit.
The Hermitian realization in
\eqref{eq:hermitian-gaussian-hilbert-space} identifies the two
opposite-frequency coordinates throughout this limiting argument.
\end{proof}

\begin{lemma}[Finite-state stabilization for fixed-profile dyadic cutoffs]
\label{lem:finite-state-cutoff}
Let $m_\Lambda(\xi)=m(\xi/\Lambda)$ be a fixed-profile cutoff from
Definition~\ref{def:admissible-cutoff}, with $\Lambda\in\Dyd$.  Let
$\rho_N$ be supported in an annulus
$\{\kappa_-N\le|\xi|\le\kappa_+N\}$ for $N\ge2$.  There exist integers
$k_-<k_+$, depending only on $(\kappa_-,\kappa_+,c_0,C_0)$, such that
\begin{align*}
 m_\Lambda\rho_N&=0 &&\text{if }\log_2(\Lambda/N)\le k_-,\\
 m_\Lambda\rho_N&=\rho_N &&\text{if }\log_2(\Lambda/N)\ge k_+.
\end{align*}
Consequently the restriction of $m_\Lambda$ to a fixed dyadic shell has only
$k_+-k_-+2$ possible states as $\Lambda$ ranges over the dyadic cutoffs.  The
inhomogeneous unit shell is handled separately and also has only finitely many
states.  More generally, for any fixed number of cutoff factors carried by
arbitrary dyadic shells (not necessarily comparable), their joint restriction
has a scale-independent finite number of states.  This number depends only on
the cutoff profile, the shell support constants, and the number of factors,
not on the ratios of the shell scales.

In particular, if a dyadic random block $X_{\Lambda,N,Q,M}$ depends on
$\Lambda$ only through these cutoff factors and every one of its possible
states satisfies $\|\|X\|\|_{L^p}\le A_{N,Q,M}$, then
\begin{equation}\label{eq:finite-state-maximal-bound}
 \left\|\sup_{\Lambda\in\Dyd}\|X_{\Lambda,N,Q,M}\|\right\|_{L^p}
 \lesssim_m A_{N,Q,M}.
\end{equation}
For every fixed block, $X_{\Lambda,N,Q,M}$ agrees with its uncut value for
all sufficiently large $\Lambda$.
\end{lemma}

\begin{proof}
If $\kappa_+N\le c_0\Lambda$, the plateau condition gives $m_\Lambda=1$ on the
shell.  If $\kappa_-N>C_0\Lambda$, the support condition gives
$m_\Lambda=0$ there.  Since $N$ and $\Lambda$ are dyadic, only finitely many
integer values of $\log_2(\Lambda/N)$ lie between these two regimes.  For a
fixed collection of $J$ shells, order their logarithmic scales.  Each factor
can change only on an interval of uniformly bounded integer length around its
own scale; outside that interval it is identically zero or identically one.
As $\log_2\Lambda$ increases, the joint vector of restrictions can therefore
change only on the union of $J$ such intervals and at their endpoints.  It has
at most $C_mJ+1$ states, independently of all ratios between the shells.  The
maximal estimate follows by bounding the supremum by the sum over this finite
state set and using the triangle inequality in $L^p$.  The plateau regime gives stabilization.
\end{proof}

The inner multiplier shifts the Duhamel frequency.  For a representative
dyadic term, the momentum variables are
\begin{equation}\label{eq:localized-frequency-flow}
 \underbrace{q}_{w_Q}+\underbrace{\ell}_{\Psi_{b,N}}
 \longrightarrow k=q+\ell,
 \qquad
 k+\underbrace{z}_{\widehat\chi_{\mathrm{in}}}
 \longrightarrow p=q+\ell+z,
 \qquad
 p+\underbrace{r}_{\Psi_{c,L}}
 +\underbrace{\zeta}_{\widehat\chi_{\mathrm{out}}}
 \longrightarrow n,
\end{equation}
so that
\begin{equation}\label{eq:outer-localization-defect}
 \zeta=n-q-\ell-r-z.
\end{equation}
The shell $R$ is imposed on $p$, before the outer resonant product.  It is
therefore $p$, not $q+\ell$, that enters the Duhamel multiplier and supplies
the inverse frequency gain.

\begin{lemma}[Inner-localization shell transfer]
\label{lem:inner-localization-transfer}
Fix $\chi_{\mathrm{in}},\chi_{\mathrm{out}}\in C_c^\infty(\R^3)$.
Let $N,Q,R,L,M$ denote, respectively, the low--high stochastic, input,
Duhamel-output, outer stochastic, and final-output shells.  Assume
\[
 Q\le c_{\mathrm{ap}}N,
 \qquad L\sim_{\mathrm{ap}}R.
\]
At finite cutoff, every centered block is obtained by covariance pullback
from a finite sum of deterministic coefficient kernels
\begin{equation}\label{eq:two-localizer-kernel}
\begin{aligned}
 H_{N,R,L,Q,M}^{t,s}(\ell,r,q,n)
 ={}&\widetilde\rho_N(\ell)\widetilde\rho_L(r)
      \chi_Q(q)\chi_M(n)\\
 &\times\int_{\R^3}\widehat\chi_{\mathrm{in}}(z)
 \widehat\chi_{\mathrm{out}}(n-q-\ell-r-z)\\
 &\qquad\times\varphi_R(q+\ell+z)
 K_a(t-s,q+\ell+z)\,
 \mathfrak a_{N,R,L,Q,M,z}(\ell,r,q,n)\dd z.
\end{aligned}
\end{equation}
Here the tilded factors are the auxiliary plateaux from
\eqref{eq:auxiliary-shell-plateau}.  The stochastic legs are composed with
$\mathsf S_{b,N,s}^{\Lambda}$ and $\mathsf S_{c,L,t}^{\Lambda}$, respectively,
so that
\begin{equation}\label{eq:single-shell-factor-after-pullback}
 \widetilde\rho_N(\ell)\widetilde\rho_L(r)
 [\rho_Nm_\Lambda](\ell)[\rho_Lm_\Lambda](r)
 =\rho_N(\ell)\rho_L(r)m_\Lambda(\ell)m_\Lambda(r).
\end{equation}
The factor $m_\Lambda$ is omitted when $\Lambda=\infty$, and the deterministic
coefficient kernel is unchanged.  Here $\varphi_R$ is an enlarged shell cutoff
and
\begin{equation}\label{eq:transfer-amplitude-size}
 \sup_{N,R,L,Q,M,z}
 \|\mathfrak a_{N,R,L,Q,M,z}\|_{L^\infty}\lesssim1.
\end{equation}
The remaining Bony symbols and finite-overlap cutoffs are absorbed into
$\mathfrak a$.  This amplitude is independent of $(t,s)$, and the deterministic
time dependence is displayed in $K_a(t-s,\cdot)$.  The synthesis maps carry
the stochastic shell and spectral-cutoff factors.  The propagated frequency
in \eqref{eq:two-localizer-kernel} is
\begin{equation}\label{eq:actual-duhamel-frequency}
 p=q+\ell+z.
\end{equation}

There is a constant $C_{\mathrm{tr}}>1$, depending only on the
Littlewood--Paley supports and the aperture, such that we write
$R\sim_{\mathrm{tr}}N$ when
$C_{\mathrm{tr}}^{-1}N\le R\le C_{\mathrm{tr}}N$.  Enlarging
$C_{\mathrm{tr}}$ absorbs the finite inhomogeneous exceptions.  This
relation has the following properties.
\begin{enumerate}[label=\textup{(\roman*)}]
\item If $R\sim_{\mathrm{tr}}N$, then, since
$L\sim_{\mathrm{ap}}R$, on the support of
\eqref{eq:two-localizer-kernel},
\begin{equation}\label{eq:shifted-duhamel-gain}
 \la p\ra\sim\la R\ra\sim\la L\ra\sim\la N\ra,
 \qquad
 |K_a(t-s,p)|\lesssim\la N\ra^{-1}.
\end{equation}
Thus the factor
\begin{equation}\label{eq:shifted-normalized-amplitude}
 \mathfrak m_z:=\la N\ra\varphi_R(p)K_a(t-s,p)
 \mathfrak a_{N,R,L,Q,M,z}
\end{equation}
has uniformly bounded modulus, uniformly in $z$.

\item If $R\not\sim_{\mathrm{tr}}N$, then for every $A>0$,
\begin{equation}\label{eq:inner-transfer-tail-L1}
 \sup_{\substack{|q|\lesssim Q,\ |\ell|\sim N}}
 \int_{\R^3}|\widehat\chi_{\mathrm{in}}(z)|
 \one_{\{\varphi_R(q+\ell+z)\ne0\}}\dd z
 \lesssim_A\la N+R\ra^{-A}.
\end{equation}
For $M\lesssim R$, after covariance pullback the four oriented flattenings
of this transfer tail satisfy, for every $A>0$,
\begin{equation}\label{eq:inner-transfer-tail-flat}
 \max_{1\le\varkappa\le4}\|F_{\varkappa,\Lambda,N,R,L,Q,M}^{\sharp}\|_{\mathfrak S_2}
 \lesssim_A
 \la N+R\ra^{-A}
 \la N\ra^{1/2+\beta_b}\la R\ra^{-1/2+\beta_c}\la Q\ra^{3/2}.
\end{equation}
For $M\gg R$ there is, in addition, arbitrary decay in $M$.

\item Let $(t,s),(t',s')\in\Delta_T$ and put
$d_\Delta:=|t-t'|+|s-s'|$.  For every $A>0$ and
$0<\theta\le1/2$, the transfer-tail flattenings satisfy
\begin{equation}\label{eq:inner-transfer-tail-time-increment}
\begin{aligned}
 &\max_{1\le\varkappa\le4}
 \|F_{\varkappa,N,R,L,Q,M}^{\sharp}(t,s)
     -F_{\varkappa,N,R,L,Q,M}^{\sharp}(t',s')\|_{\mathfrak S_2}\\
 &\qquad\lesssim_{A,\theta}
 d_\Delta^\theta\langle N+R\rangle^{-A}
 \langle N\rangle^{1/2+\beta_b}
 \langle R\rangle^{-1/2+\beta_c}
 \langle Q\rangle^{3/2}.
\end{aligned}
\end{equation}
If $F_{\varkappa,\Lambda}^{\sharp}$ denotes the corresponding cutoff flattening,
then
\begin{equation}\label{eq:inner-transfer-tail-cutoff-difference}
 \max_{1\le\varkappa\le4}
 \|F_{\varkappa,\Lambda}^{\sharp}-F_{\varkappa,\infty}^{\sharp}\|_{\mathfrak S_2}
 \lesssim_A
 \langle N+R\rangle^{-A}
 \langle N\rangle^{1/2+\beta_b}\langle R\rangle^{-1/2+\beta_c}
 \langle Q\rangle^{3/2},
\end{equation}
with the analogous bound for the time increment in
\eqref{eq:inner-transfer-tail-time-increment}.  For each fixed shell tuple the
left-hand side of \eqref{eq:inner-transfer-tail-cutoff-difference} tends to
zero along every admissible cofinal family; it is eventually zero once the
two stochastic shells lie in the cutoff plateau.
\end{enumerate}
Consequently the sector $R\not\sim_{\mathrm{tr}}N$ is summable after the
insertion of arbitrary fixed polynomial Sobolev or Besov weights.
\end{lemma}

\begin{proof}
We derive one representative term; the Bony and resonant decompositions
produce only finitely many terms of the same form.

\emph{Step 1: the low--high product.}
Write the low--high symbol as
$\chi_Q(q)\rho_N(\ell)\mathfrak b_{Q,N}^{\circ}(q,\ell)$, where
$\mathfrak b_{Q,N}^{\circ}$ is the remaining smooth factor.  It is uniformly order zero and vanishes unless
$Q\le c_{\mathrm{ap}}N$.  Before physical localization,
\begin{align*}
 \widehat{(P_Qw\prec P_N\Psi_b)}(k,s)
 =c_{\mathrm F}\int_{q,\ell}
 &\delta(k-q-\ell)\chi_Q(q)\rho_N(\ell)
 \mathfrak b_{Q,N}^{\circ}(q,\ell)\\
 &\times\widehat w(q,s)\widehat\Psi_b(\ell,s)\dd q\dd\ell.
\end{align*}
At finite spectral cutoff, $\widehat\Psi_b$ in this formula is replaced by
$m_\Lambda\widehat\Psi_b$.  The factor $\rho_Nm_\Lambda\widehat\Psi_b$ is represented by
$\mathsf S_{b,N,s}^{\Lambda}$ through
\eqref{eq:first-chaos-synthesis}.  Multiplication by
$\chi_{\mathrm{in}}$ convolves the momentum relation with
$\widehat\chi_{\mathrm{in}}$.  Writing its Fourier variable as $z$ and using
\eqref{eq:plateau-times-shell} gives the deterministic coefficient in
\eqref{eq:two-localizer-kernel}, with propagated frequency
$p=q+\ell+z$.

\emph{Step 2: Duhamel propagation and its output shell.}
Set
\[
 \mathcal U_{R,Q,N}(t)
 :=P_RI_a\bigl[\chi_{\mathrm{in}}(P_Qw\prec P_N\Psi_b)\bigr](t).
\]
Since $P_R$ commutes with $I_a$,
\begin{align*}
 \widehat{\mathcal U_{R,Q,N}}(p,t)
 ={}&\varphi_R(p)\int_0^tK_a(t-s,p)\\
 &\quad\times\widehat{\chi_{\mathrm{in}}(P_Qw\prec P_N\Psi_b)}(p,s)
 \dd s.
\end{align*}
Thus the shell factor $\varphi_R$ and the multiplier $K_a$ act on the same
variable $p$.

\emph{Step 3: outer resonance, outer localization, and final output.}
Write the resonant symbol as
$\rho_L(r)\mathfrak r_{R,L}^{\circ}(p,r)$, where
$\mathfrak r_{R,L}^{\circ}$ is uniformly order zero and vanishes unless
$L\sim_{\mathrm{ap}}R$.  Taking the resonant product with
$P_L\Psi_c(t)$, multiplying by $\chi_{\mathrm{out}}$, and finally applying
$P_M$ gives
\begin{align*}
 &\widehat{P_M\{\chi_{\mathrm{out}}
   [\mathcal U_{R,Q,N}\circ P_L\Psi_c]\}}(n,t)\\
 &\quad=c_{\mathrm F}^2\chi_M(n)
 \int_{p,r,\zeta}
 \widehat\chi_{\mathrm{out}}(\zeta)
 \delta(n-p-r-\zeta)\rho_L(r)
 \mathfrak r_{R,L}^{\circ}(p,r)\\
 &\hspace{42mm}\times
 \widehat{\mathcal U_{R,Q,N}}(p,t)
 \widehat\Psi_c(r,t)\dd p\dd r\dd\zeta.
\end{align*}
At cutoff $\Lambda$, the factor
$\rho_Lm_\Lambda\widehat\Psi_c$ is generated by
$\mathsf S_{c,L,t}^{\Lambda}$.  Substituting the preceding formulas,
eliminating $p$ and $\zeta$, and inserting
$\widetilde\rho_N=\widetilde\rho_L=1$ on the two stochastic shell supports
yields the deterministic kernel
\eqref{eq:two-localizer-kernel}, with
$\zeta=n-q-\ell-r-z$.  The Fourier-normalization constant, the symbols
$\mathfrak b_{Q,N}^{\circ}$ and $\mathfrak r_{R,L}^{\circ}$, and the enlarged
deterministic cutoffs are absorbed into
$\mathfrak a_{N,R,L,Q,M,z}$.  Their rescaled symbol bounds give
\eqref{eq:transfer-amplitude-size}, and
\eqref{eq:single-shell-factor-after-pullback} identifies the resulting
finite-cutoff block.

\emph{Step 4: principal transfer and transfer tail.}
The low--high support gives $|q+\ell|\sim N$.  If
$R\sim_{\mathrm{tr}}N$, the factor $\varphi_R(p)$ gives
$\la p\ra\sim\la R\ra\sim\la N\ra$, and
$|K_a(t-s,p)|\lesssim\la p\ra^{-1}$.  This proves
\eqref{eq:shifted-duhamel-gain} and the boundedness of the normalized
amplitude \eqref{eq:shifted-normalized-amplitude}.

If $R\not\sim_{\mathrm{tr}}N$, then $|p|\sim R$ and
$|q+\ell|\sim N$ force
\[
 |z|=|p-(q+\ell)|\gtrsim N+R
\]
after enlarging the finite-ratio constant to absorb the inhomogeneous blocks.
The Schwartz decay of $\widehat\chi_{\mathrm{in}}$ gives
\eqref{eq:inner-transfer-tail-L1}.  Before covariance pullback, Minkowski's
inequality and integration first in the final frequency give
\begin{align}
 \|H_{\Lambda,N,R,L,Q,M}^{\mathrm{tr}}\|_{L^2_{\ell,r,q,n}}
 \lesssim{}&\la R\ra^{-1}
 \|\widehat\chi_{\mathrm{in}}
   \one_{\{|z|\gtrsim N+R\}}\|_{L^1_z}
 \|\widehat\chi_{\mathrm{out}}\|_{L^2}\notag\\
 &\times\la N\ra^{3/2}\la L\ra^{3/2}\la Q\ra^{3/2}.
 \label{eq:transfer-tail-HS-before-pullback}
\end{align}
Every orientation has this Hilbert--Schmidt norm, because it is only a
regrouping of the same four tensor legs.  Composing the two Gaussian legs
with $\mathsf S_{b,N,s}$ and $\mathsf S_{c,L,t}$ contributes
$\la N\ra^{-1+\beta_b}\la L\ra^{-1+\beta_c}$ by
Lemma~\ref{lem:covariance-synthesis}.  Since $L\sim R$, the Schatten ideal
property gives \eqref{eq:inner-transfer-tail-flat}.  If $M\gg R$, then
$|n-p-r|\gtrsim M$ and the outer Schwartz factor supplies arbitrary additional
$M$-decay.

\emph{Step 5: time increments.}
Write the pulled-back form as
$H_{t,s}\circ(\mathsf S_{b,N,s},\mathsf S_{c,L,t})$ and decompose its
increment as
\begin{align*}
&(H_{t,s}-H_{t',s'})\circ
 (\mathsf S_{b,N,s},\mathsf S_{c,L,t})\\
&\quad+H_{t',s'}\circ
 (\mathsf S_{b,N,s}-\mathsf S_{b,N,s'},\mathsf S_{c,L,t})\\
&\quad+H_{t',s'}\circ
 (\mathsf S_{b,N,s'},\mathsf S_{c,L,t}-\mathsf S_{c,L,t'}).
\end{align*}
The first line uses
$|K_a(t-s,p)-K_a(t'-s',p)|
 \lesssim d_\Delta^\theta\la R\ra^{-1+\theta}$; the other two use
\eqref{eq:synthesis-increment}.  Increasing the Schwartz exponent in
\eqref{eq:inner-transfer-tail-L1} absorbs the resulting powers of $N$ and
$R$, proving \eqref{eq:inner-transfer-tail-time-increment}.

\emph{Step 6: cutoff differences.}
The deterministic kernel is unchanged; only the two synthesis maps differ.
Use the two-term identity
\begin{align*}
&H\circ(\mathsf S_{b,N,s}^{\Lambda},\mathsf S_{c,L,t}^{\Lambda})
-H\circ(\mathsf S_{b,N,s},\mathsf S_{c,L,t})\\
&\quad=H\circ(\mathsf S_{b,N,s}^{\Lambda}-\mathsf S_{b,N,s},
              \mathsf S_{c,L,t}^{\Lambda})
 +H\circ(\mathsf S_{b,N,s},
          \mathsf S_{c,L,t}^{\Lambda}-\mathsf S_{c,L,t})
\end{align*}
and Lemma~\ref{lem:covariance-synthesis} give
\eqref{eq:inner-transfer-tail-cutoff-difference} and its increment analogue.
On a fixed pair of stochastic shells the cutoff maps eventually coincide
with the uncut maps, which proves fixed-shell stabilization and comparison of
non-nested profiles through the same limit.
\end{proof}

Fix once and for all $C_{\mathrm{out}}>1$ sufficiently large relative to
the Littlewood--Paley support constants.  On the main-transfer sector
$R\sim_{\mathrm{tr}}N$, there are only finitely many choices of
$R\sim N$ and $L\sim_{\mathrm{ap}}R$.  Retaining the actual outer stochastic
scale $L$ in each summand, its deterministic coefficient kernel has the form
\begin{equation}\label{eq:iterated-soft-incidence}
\begin{aligned}
 H_{N,L}(\ell,r,q,n)
 ={}&A_N\widetilde\rho_N(\ell)\widetilde\rho_L(r)
      \chi_Q(q)\chi_M(n)\\
 &\times\int_{\R^3}\kappa_1(z)
 \kappa_2(n-q-\ell-r-z)\mathfrak m_z(\ell,r,q,n)\dd z,
\end{aligned}
\end{equation}
where $L\sim N$, $|A_N|\lesssim\la N\ra^{-1}$,
$\kappa_1,\kappa_2\in\mathcal S(\R^3)$, and
$\sup_z\|\mathfrak m_z\|_{L^\infty}\lesssim1$.  The factors $\rho_N$ and $\rho_L$ are supplied by the synthesis maps rather
than the deterministic kernel in \eqref{eq:iterated-soft-incidence}.  Since
$L\sim N$, the scale estimates below may replace $L$ by $N$ after pullback.  The four oriented
flattenings are
\begin{equation}\label{eq:flattening-orientations}
 (\ell,r,q)\longrightarrow n,\qquad
 (\ell,r,n)\longrightarrow q,\qquad
 (\ell,q)\longrightarrow(r,n),\qquad
 (r,q)\longrightarrow(\ell,n).
\end{equation}

\paragraph{Hilbert-space realization of a principal block.}
For the kernel in \eqref{eq:iterated-soft-incidence}, let
\begin{equation}\label{eq:localized-block-hilbert-spaces}
 \mathcal A_N:=L^2(\R^3_\ell),\qquad
 \mathcal C_{1,L}:=L^2(\R^3_r),\qquad
 \mathcal C_{0,Q}:=L^2(\R^3_q),\qquad
 \mathcal E_M:=L^2(\R^3_n).
\end{equation}
The auxiliary plateaux $\widetilde\rho_N,\widetilde\rho_L$ are included in
the form, whereas the genuine Littlewood--Paley multipliers are included in
the synthesis maps.  All four spaces are realified when Fourier variables
are complex.  Define
\begin{equation}\label{eq:localized-four-linear-form}
 \mathbf H_{t,s}(x,y,f,g)
 :=\Re\int H_{t,s}(\ell,r,q,n)
 x(\ell)y(r)f(q)\overline{g(n)}
 \dd\ell\dd r\dd q\dd n.
\end{equation}

\begin{proposition}[Hilbert-space realization and covariance subtraction]
\label{prop:localized-block-hilbert}
Let $\mathcal G_a:=\widehat{\mathfrak H}_a$ be the real Gaussian Hilbert
space in \eqref{eq:hermitian-gaussian-hilbert-space}.  Fix either a finite
spectral cutoff $\lambda=\Lambda$ or the uncut state $\lambda=\infty$, set
$m_\infty=1$, and put
\[
 S_b:=\mathsf S_{b,N,s}^{(m_\lambda)}:
 \mathcal G_b\to\mathcal A_N,
 \qquad
 S_c:=\mathsf S_{c,L,t}^{(m_\lambda)}:
 \mathcal G_c\to\mathcal C_{1,L}.
\]
For
\[
 \mathbf H_{t,s}^{\sharp,\lambda}(F,G,f,g)
 :=\mathbf H_{t,s}(S_bF,S_cG,f,g),
\]
the four pulled-back flattenings, with the superscript $\lambda$ suppressed,
have the directions
\begin{align}
 F_1^\sharp&:\mathcal G_b\otimes_2\mathcal G_c\otimes_2\mathcal C_{0,Q}
       \longrightarrow\mathcal E_M,
 &F_1^\sharp&=F_1(S_b\otimes S_c\otimes I_Q),
 \label{eq:pulled-flat-direction-1}\\
 F_2^\sharp&:\mathcal G_b\otimes_2\mathcal G_c\otimes_2\mathcal E_M
       \longrightarrow\mathcal C_{0,Q},
 &F_2^\sharp&=F_2(S_b\otimes S_c\otimes I_M),
 \label{eq:pulled-flat-direction-2}\\
 F_3^\sharp&:\mathcal G_b\otimes_2\mathcal C_{0,Q}
       \longrightarrow\mathcal G_c\otimes_2\mathcal E_M,
 &F_3^\sharp&=(S_c^*\otimes I_M)F_3(S_b\otimes I_Q),
 \label{eq:pulled-flat-direction-3}\\
 F_4^\sharp&:\mathcal G_c\otimes_2\mathcal C_{0,Q}
       \longrightarrow\mathcal G_b\otimes_2\mathcal E_M,
 &F_4^\sharp&=(S_b^*\otimes I_M)F_4(S_c\otimes I_Q).
 \label{eq:pulled-flat-direction-4}
\end{align}
Here $I_Q$ and $I_M$ are the identities on the input and output shell
spaces.  Each factorization contains one copy of $S_b$ and one copy of $S_c$.
When $\lambda=\Lambda$, \eqref{eq:single-shell-factor-after-pullback} gives the
finite-cutoff stochastic multipliers
\begin{equation}\label{eq:two-leg-multiplier}
 \rho_N(\ell)\rho_L(r)m_\Lambda(\ell)m_\Lambda(r).
\end{equation}

Assume now that $b=c$.  Testing the covariance of the two finite-cutoff
synthesis maps against the deterministic kernel gives
\begin{align}
 &\widetilde\rho_N(\ell)\widetilde\rho_L(r)
 (2\pi)^3\delta(\ell+r)
 \rho_N(\ell)\rho_L(r)m_\Lambda(\ell)m_\Lambda(r)
 \sigma_b(\ell;s,t)\notag\\
 &\qquad=(2\pi)^3\delta(\ell+r)
 \rho_N(\ell)\rho_L(r)m_\Lambda(\ell)m_\Lambda(r)
 \sigma_b(\ell;s,t).
 \label{eq:same-color-shell-contraction}
\end{align}
After imposing $r=-\ell$ and using evenness, its scalar shell factor is
\begin{equation}\label{eq:same-color-shell-factor}
 \rho_N(\ell)\rho_L(-\ell)m_\Lambda(\ell)^2
 \sigma_b(\ell;s,t).
\end{equation}
On a same-shell summand $L=N$, this becomes
\begin{equation}\label{eq:same-shell-rho-square}
 \rho_N(\ell)^2m_\Lambda(\ell)^2\sigma_b(\ell;s,t).
\end{equation}
This is the covariance term in Proposition~\ref{prop:finite-color-split} and,
after finite-overlap shell reindexing, the factor in
\eqref{eq:localized-contraction-kernel}.

Let $P$ be a finite-rank orthogonal projection on the common Gaussian space
$\mathcal G_b$.  Choose an orthonormal basis
$(e_\mu)_{\mu=1}^{d_P}$ of $P\mathcal G_b$ and define the finite compressed
covariance contraction by
\begin{equation}\label{eq:finite-compressed-covariance}
 D_P:=\sum_{\mu=1}^{d_P} H_{e_\mu,e_\mu}^{\sharp,\lambda},
 \qquad
 \langle H_{F,G}^{\sharp,\lambda}f,g\rangle
 :=\mathbf H_{t,s}(S_bF,S_cG,f,g).
\end{equation}
The sum is basis independent and equals the covariance of the two
$P$-compressed first chaoses.  Wick centering subtracts $D_P$ at finite rank.
For the dyadic block, \eqref{eq:same-color-shell-contraction} is treated
separately, and Gaussian completion is applied to the centered tensor.
\end{proposition}

\begin{proof}
The four factorizations follow by inserting $S_b$ and $S_c$ into the first
two slots of \eqref{eq:localized-four-linear-form} and testing against
elementary tensors.  A synthesis adjoint occurs when the corresponding Gaussian leg lies on the
output side of the flattening, giving
\eqref{eq:pulled-flat-direction-3}--\eqref{eq:pulled-flat-direction-4}.

At finite cutoff, \eqref{eq:first-chaos-synthesis} identifies the two
pushed-forward Gaussian legs with
$P_N\Psi_{b,\Lambda}(s)$ and $P_L\Psi_{c,\Lambda}(t)$.  By \eqref{eq:plateau-times-shell}, multiplication by the auxiliary plateaux
in \eqref{eq:iterated-soft-incidence} leaves these legs unchanged.  This gives
\eqref{eq:two-leg-multiplier}.

If $b=c$, apply
\eqref{eq:shell-covariance-from-synthesis}.  Testing that covariance
distribution against the two auxiliary plateaux gives the first line of
\eqref{eq:same-color-shell-contraction};
\eqref{eq:plateau-times-shell} gives its second line.  The support relation
$r=-\ell$ and the evenness of $\rho_L,m_\Lambda$ yield
\eqref{eq:same-color-shell-factor} and
\eqref{eq:same-shell-rho-square}.  Hence Wick centering removes the contraction obtained by applying
\eqref{eq:finite-wick-identity} to the raw finite-cutoff kernel
\eqref{eq:raw-kernel}.

For the finite-compression statement, let $\phi\in\mathcal A_N$ and
$\psi\in\mathcal C_{1,L}$.  Parseval's identity on the finite-dimensional
space $P\mathcal G_b$ gives
\begin{align}\label{eq:compressed-covariance-identity}
 \sum_{\mu=1}^{d_P}
 \langle S_be_\mu,\phi\rangle_{\R}
 \langle S_ce_\mu,\psi\rangle_{\R}
 &=\langle P S_b^*\phi,P S_c^*\psi\rangle_{\mathcal G_b}\notag\\
 &=\E\!\left[
   \langle\Psi_{b,P}^{N,\lambda}(s),\phi\rangle_{\R}
   \langle\Psi_{b,P}^{L,\lambda}(t),\psi\rangle_{\R}\right],
\end{align}
where $\Psi_{b,P}^{N,\lambda}$ and $\Psi_{b,P}^{L,\lambda}$ denote the two
first chaoses obtained by restricting the isonormal coordinate to
$P\mathcal G_b$ and applying the corresponding synthesis maps.  Contracting
\eqref{eq:compressed-covariance-identity} with the remaining input and output
legs of $\mathbf H_{t,s}$ gives $D_P$ in
\eqref{eq:finite-compressed-covariance}.  This also proves basis independence.

Writing $\mathbf W_\mu:=\mathbf W_b(e_\mu)$, the finite-dimensional Wick
identity is
\begin{align}\label{eq:finite-dimensional-operator-wick}
 \sum_{\mu,\nu=1}^{d_P}H_{e_\mu,e_\nu}^{\sharp,\lambda}
 \mathbf W_\mu\mathbf W_\nu
 =D_P+
 \sum_{\mu,\nu=1}^{d_P}(H_{e_\mu,e_\nu}^{\sharp,\lambda})^{\mathrm{sym}}
 \bigl(\mathbf W_\mu\mathbf W_\nu-\delta_{\mu\nu}\bigr).
\end{align}
The antisymmetric coefficient tensor drops out because both
$\mathbf W_\mu\mathbf W_\nu$ and $\delta_{\mu\nu}$ are symmetric in
$(\mu,\nu)$.  The second term in
\eqref{eq:finite-dimensional-operator-wick} is therefore centered.

The covariance distribution \eqref{eq:same-color-shell-contraction} is
separated and treated by Proposition~\ref{prop:localized-diagonal}.  The
centered term is then completed through finite-rank Gaussian compressions by
Lemma~\ref{lem:four-flat-compression} and
Theorem~\ref{thm:centered-gaussian-operator}.
\end{proof}

\begin{lemma}[Scalar Wiener realization of a centered block]
\label{lem:scalar-wiener-realization}
Fix a concrete pulled-back dyadic form $\mathbf H^\sharp$ and assume that its
first flattening $F_1^\sharp$ is Hilbert--Schmidt.  For
$f\in\mathcal C_{0,Q}$ and $g\in\mathcal E_M$, define the scalar coefficient
kernel on the two Gaussian Hilbert spaces by
\[
 \mathfrak k_{f,g}(F,G):=\mathbf H^\sharp(F,G,f,g).
\]
Then $\mathfrak k_{f,g}\in\mathcal G_b\otimes_2\mathcal G_c$ and
\begin{equation}\label{eq:scalar-wiener-kernel-bound}
 \|\mathfrak k_{f,g}\|_{\mathcal G_b\otimes_2\mathcal G_c}
 \le \|F_1^\sharp\|_{\Sch_2}\|f\|_{\mathcal C_{0,Q}}
 \|g\|_{\mathcal E_M}.
\end{equation}
For distinct colors, the scalar pairing of the centered operator with $(f,g)$
is the ordered double Wiener integral of $\mathfrak k_{f,g}$.  For equal
colors, it is the second Wiener integral of the symmetrization of
$\mathfrak k_{f,g}$.  In either case, simultaneous finite-rank Gaussian
compressions converge in $L^2(\Omega)$ to this scalar integral.  If, in
addition, all four flattenings of $\mathbf H^\sharp$ lie in $\Sch_r$ for
some $2\le r<\infty$, this scalar limit agrees with the pairing of the
operator-valued completion in
Theorem~\ref{thm:centered-gaussian-operator}.

At finite spectral cutoff, pairing \eqref{eq:finite-wick-identity} with the
localized deterministic kernel gives this scalar Wiener integral.  The
auxiliary plateaux act as the identity on the ranges of the synthesis maps,
so the coefficient has the multiplier in \eqref{eq:two-leg-multiplier}; for
equal colors, the covariance term is \eqref{eq:same-color-shell-contraction}.
\end{lemma}

\begin{proof}
Choose orthonormal bases $(e_\mu)$ and $(f_\nu)$ of the two Gaussian Hilbert
spaces.  By Parseval and the definition of $F_1^\sharp$,
\begin{align*}
 \|\mathfrak k_{f,g}\|_{\mathcal G_b\otimes_2\mathcal G_c}^2
 &=\sum_{\mu,\nu}
   |\langle F_1^\sharp(e_\mu\otimes f_\nu\otimes f),g\rangle|^2\\
 &\le \|F_1^\sharp\|_{\Sch_2}^2\|f\|^2\|g\|^2,
\end{align*}
which is \eqref{eq:scalar-wiener-kernel-bound}.  The isometry for ordered
double Wiener integrals in the distinct-color case, and the second-chaos
isometry after symmetrization in the equal-color case, show that orthogonal
projections converging strongly to the identity approximate the corresponding
scalar integrals in $L^2(\Omega)$.  Pairing the finite-rank operator formula with $(f,g)$ gives these projected
integrals.  Under the additional
four-flattening hypothesis, the operator-valued difference estimate in
Theorem~\ref{thm:centered-gaussian-operator} has a unique limit, so its scalar
pairings coincide with the Wiener-integral limits.  Finally, expansion of the
two pushed-forward first chaoses at finite spectral cutoff gives
\eqref{eq:finite-wick-identity} after deterministic testing.  Identity
\eqref{eq:plateau-times-shell} removes the auxiliary plateaux and leaves one
factor $\rho_Nm_\Lambda$ and one factor $\rho_Lm_\Lambda$.  The zeroth-chaos term is \eqref{eq:same-color-shell-contraction}, the
covariance branch from Proposition~\ref{prop:finite-color-split}, and the
remaining scalar kernel is centered.
\end{proof}

\begin{lemma}[Translated shell-intersection volume]
\label{lem:translated-shell-volume}
Let $\vartheta_N$ and $\chi_M$ be uniformly bounded smooth cutoffs to dyadic
balls or annuli in $\R^3$.  Uniformly in $u\in\R^3$,
\begin{equation}\label{eq:translated-shell-volume}
 \int_{\R^3}\vartheta_N(\xi)^2\chi_M(\xi+u)^2\dd\xi
 \lesssim \min\{N,M\}^3.
\end{equation}
The estimate applies in particular to the genuine shell multipliers
$\rho_N$ and to the auxiliary plateaux $\widetilde\rho_N$, and remains valid
under any fixed finite enlargement of either support.
\end{lemma}

\begin{proof}
The integrand is uniformly bounded and supported in the intersection of a
translate of the $M$-shell with the $N$-shell.  This intersection has volume at
most the smaller of the two ambient shell volumes, which is
$O(\min\{N,M\}^3)$.  Finite support enlargements only change the implicit
constant.
\end{proof}

\begin{lemma}[Flattening bounds for the localized wave--Klein--Gordon kernel]
\label{lem:bounded-amplitude-flattening}
Let $L\sim N$, $Q\le c_{\mathrm{ap}}N$, and $M\lesssim N$.  If
$F_1,\ldots,F_4$ denote
the flattenings in \eqref{eq:flattening-orientations}, then for every
$2\le r_0<\infty$,
\begin{align}
 \|F_1\|_{\mathfrak S_{r_0}}
 &\lesssim |A_N|C_\kappa N^{3/2}Q^{3/2}M^{3/r_0},
 \label{eq:bounded-flat-1}\\
 \|F_2\|_{\mathfrak S_{r_0}}
 &\lesssim |A_N|C_\kappa N^{3/2}M^{3/2}Q^{3/r_0},
 \label{eq:bounded-flat-2}\\
 \|F_3\|_{\mathfrak S_{r_0}}+\|F_4\|_{\mathfrak S_{r_0}}
 &\lesssim |A_N|C_\kappa Q^{3/2}M^{3/2}N^{3/r_0},
 \label{eq:bounded-flat-34}
\end{align}
where
\[
 C_\kappa:=\|\kappa_1\|_{L^1}
 \bigl(\|\kappa_2\|_{L^1}+\|\kappa_2\|_{L^2}\bigr).
\]
The estimates are uniform over all measurable amplitudes $\mathfrak m_z$ with
$|\mathfrak m_z|\le1$; no frequency-derivative bound on $\mathfrak m_z$ is required.
\end{lemma}

\begin{proof}
Fix $z$ and write $k_z(y)=\kappa_2(y-z)$.  Since
$|\mathfrak m_z|\le1$, the modulus of every oriented operator is dominated
pointwise by the positive kernel obtained by replacing $\mathfrak m_z$ with
one and $k_z$ with $|k_z|$.  The soft incidence variable is always
\[
 y=n-q-\ell-r.
\]
All changes of variables below have determinant one.  The four eliminations
used in the operator estimates are:
\begin{center}
\small
\renewcommand{\arraystretch}{1.14}
\begin{tabularx}{0.94\textwidth}{@{}P{0.12\textwidth}P{0.24\textwidth}P{0.24\textwidth}Y@{}}
\toprule
map & input $\to$ output & eliminated variable & surviving fiber volume\\
\midrule
$F_1$ & $(\ell,r,q)\to n$ & $r=n-q-\ell-y$ & $N^3Q^3$\\
$F_2$ & $(\ell,r,n)\to q$ & $r=n-q-\ell-y$ & $N^3M^3$\\
$F_3$ & $(\ell,q)\to(r,n)$ & $\ell=n-q-r-y$ & $Q^3M^3$\\
$F_4$ & $(r,q)\to(\ell,n)$ & $r=n-q-\ell-y$ & $Q^3M^3$\\
\bottomrule
\end{tabularx}
\end{center}

For $F_1$, weighted Cauchy--Schwarz with measure
$|k_z(y)|\dd y\dd\ell\dd q$ gives
\begin{align*}
 |F_1f(n)|^2\lesssim |A_N|^2
 &\left(\iiint |k_z(y)|\widetilde\rho_N(\ell)^2\chi_Q(q)^2
 \widetilde\rho_L(n-q-\ell-y)^2\dd y\dd\ell\dd q\right)\\
 &\times\left(\iiint |k_z(y)|
 |f(\ell,n-q-\ell-y,q)|^2\dd y\dd\ell\dd q\right).
\end{align*}
The first factor is bounded by
$C\|\kappa_2\|_{L^1}N^3Q^3$.  Integrating the second factor in $n$ and using
$r=n-q-\ell-y$ gives $\|\kappa_2\|_{L^1}\|f\|_2^2$.  Hence
\begin{equation}\label{eq:centered-op-F1}
 \|F_1\|_{\mathcal L}\lesssim
 |A_N|\|\kappa_2\|_{L^1}N^{3/2}Q^{3/2}.
\end{equation}

For $F_2$, write its output as $q$ and eliminate the input variable
$r=n-q-\ell-y$.  Weighted Cauchy--Schwarz in $(y,\ell,n)$ gives
\begin{align*}
 |F_2f(q)|^2\lesssim |A_N|^2
 &\left(\iiint |k_z(y)|\widetilde\rho_N(\ell)^2\chi_M(n)^2
 \widetilde\rho_L(n-q-\ell-y)^2\dd y\dd\ell\dd n\right)\\
 &\times\left(\iiint |k_z(y)|
 |f(\ell,n-q-\ell-y,n)|^2\dd y\dd\ell\dd n\right).
\end{align*}
The first factor is at most
$C\|\kappa_2\|_{L^1}N^3M^3$.  After integration in $q$, the second factor is
$\|\kappa_2\|_{L^1}\|f\|_2^2$ by the same unit-Jacobian substitution.  Thus
\begin{equation}\label{eq:centered-op-F2}
 \|F_2\|_{\mathcal L}\lesssim
 |A_N|\|\kappa_2\|_{L^1}N^{3/2}M^{3/2}.
\end{equation}

For $F_3$, eliminate $\ell=n-q-r-y$.  Cauchy--Schwarz in $(y,q)$ gives a
first fiber factor bounded by $C\|\kappa_2\|_{L^1}Q^3$.  After integrating in
the output variables $(r,n)$, change from $n$ to
$\ell=n-q-r-y$.  The remaining translated shell intersection is bounded by
Lemma~\ref{lem:translated-shell-volume}:
\[
 \sup_{\ell,q,y}\int\widetilde\rho_L(r)^2
 \chi_M(\ell+q+r+y)^2\dd r\lesssim M^3,
\]
where $M\lesssim N\sim L$ is used.  Consequently
\[
 \|F_3\|_{\mathcal L}\lesssim
 |A_N|\|\kappa_2\|_{L^1}Q^{3/2}M^{3/2}.
\]

For $F_4$, eliminate $r=n-q-\ell-y$.  Cauchy--Schwarz in $(y,q)$ again gives
the factor $C\|\kappa_2\|_{L^1}Q^3$.  Integrating in the output variables
$(\ell,n)$ and changing from $n$ to $r=n-q-\ell-y$ leaves
\[
 \sup_{r,q,y}\int\widetilde\rho_N(\ell)^2
 \chi_M(\ell+q+r+y)^2\dd\ell\lesssim M^3
\]
by the same translated-intersection estimate.  Hence
\begin{equation}\label{eq:centered-op-F34}
 \|F_3\|_{\mathcal L}+\|F_4\|_{\mathcal L}
 \lesssim |A_N|\|\kappa_2\|_{L^1}Q^{3/2}M^{3/2}.
\end{equation}
This computation shows explicitly why an arbitrary bounded amplitude is
admissible in both cross orientations: no change of variables differentiates
or otherwise regularizes $\mathfrak m_z$.

All four Hilbert--Schmidt norms are equal because the flattenings are unitary
regroupings of one tensor.  For fixed $z$, Tonelli's theorem, the change of
variables $y=n-q-\ell-r$, the translated shell-intersection bound, and
$|\mathfrak m_z|\le1$ give, uniformly for $L\sim N$,
\begin{align}\label{eq:centered-common-HS}
 \|F_\varkappa\|_{\Sch_2}^2
 &=\|H_z\|_{L^2_{\ell,r,q,n}}^2\notag\\
 &\lesssim |A_N|^2\|\kappa_2\|_{L^2}^2
 N^3Q^3M^3,
 \qquad 1\le\varkappa\le4.
\end{align}
Interpolating \eqref{eq:centered-op-F1}, \eqref{eq:centered-op-F2}, and
\eqref{eq:centered-op-F34} with \eqref{eq:centered-common-HS} by
\eqref{eq:centered-schatten-interpolation} yields
\eqref{eq:bounded-flat-1}--\eqref{eq:bounded-flat-34} for fixed $z$.
Minkowski's inequality in $\Sch_{r_0}$ and integration against
$|\kappa_1(z)|\dd z$ produce $C_\kappa$.  No derivative of
$\mathfrak m_z$ is used.
\end{proof}

\begin{proposition}[Schatten bounds for the principal centered kernel]
\label{prop:centered-flattening-verification}
Let the hypotheses of Lemma~\ref{lem:bounded-amplitude-flattening} hold and
put $\beta_{b,c}=\beta_b+\beta_c$.  Before covariance pullback, the operator
and Hilbert--Schmidt bounds, and after pullback the Schatten bounds used in
the Gaussian estimate, are as follows:
\begin{center}
\scriptsize
\setlength{\tabcolsep}{3pt}
\renewcommand{\arraystretch}{1.22}
\begin{tabularx}{0.98\textwidth}{@{}P{0.11\textwidth}P{0.19\textwidth}P{0.18\textwidth}P{0.18\textwidth}Y@{}}
\toprule
Flattening & Orientation & $\|F_\varkappa\|_{\mathcal L}/|A_N|$
& $\|F_\varkappa\|_{\mathfrak S_2}/|A_N|$
& pulled-back $\mathfrak S_{r_0}$ bound, using $|A_N|\lesssim N^{-1}$\\
\midrule
$F_1$ & $(\ell,r,q)\to n$
& $N^{3/2}Q^{3/2}$
& $N^{3/2}Q^{3/2}M^{3/2}$
& $N^{-3/2+\beta_{b,c}}Q^{3/2}M^{3/r_0}$\\
$F_2$ & $(\ell,r,n)\to q$
& $N^{3/2}M^{3/2}$
& $N^{3/2}Q^{3/2}M^{3/2}$
& $N^{-3/2+\beta_{b,c}}M^{3/2}Q^{3/r_0}$\\
$F_3$ & $(\ell,q)\to(r,n)$
& $Q^{3/2}M^{3/2}$
& $N^{3/2}Q^{3/2}M^{3/2}$
& $N^{-3+\beta_{b,c}}Q^{3/2}M^{3/2}N^{3/r_0}$\\
$F_4$ & $(r,q)\to(\ell,n)$
& $Q^{3/2}M^{3/2}$
& $N^{3/2}Q^{3/2}M^{3/2}$
& $N^{-3+\beta_{b,c}}Q^{3/2}M^{3/2}N^{3/r_0}$\\
\bottomrule
\end{tabularx}
\end{center}
The map directions are those in
\eqref{eq:pulled-flat-direction-1}--\eqref{eq:pulled-flat-direction-4}.
The first Khintchine step creates the row/column pair in the second Gaussian
leg, and the second step creates the two complementary squares in the first
Gaussian leg, as recorded in the row/column identities preceding
Lemma~\ref{lem:centered-schatten-ideal}.  The factorizations in
Proposition~\ref{prop:localized-block-hilbert} contribute one synthesis norm
from each stochastic leg, namely
$N^{-1+\beta_b}L^{-1+\beta_c}\lesssim N^{-2+\beta_{b,c}}$ because
$L\sim N$.  The deterministic kernel contains only the auxiliary plateaux.
Consequently all four rows are controlled by
\[
 N^{-3/2+\beta_{b,c}+\varepsilon}
 \bigl(Q^{3/2+\varepsilon}+M^{3/2+\varepsilon}\bigr)
\]
after choosing $r_0$ as in the proof of
Proposition~\ref{prop:principal-centered-block}.
\end{proposition}

\begin{proof}
The operator-norm column is
\eqref{eq:centered-op-F1}--\eqref{eq:centered-op-F34}; the common
Hilbert--Schmidt column is the final integral in the proof of
Lemma~\ref{lem:bounded-amplitude-flattening}.  Interpolation by
\eqref{eq:centered-schatten-interpolation}, covariance pullback by
Lemma~\ref{lem:covariance-pullback}, and
\eqref{eq:synthesis-size} give the final column.  Since $Q,M\lesssim N$,
the two mixed rows satisfy
$Q^{3/2}M^{3/2}\lesssim N^{3/2}(Q^{3/2}+M^{3/2})$.
\end{proof}

\begin{lemma}[Logarithmic Schatten selection]
\label{lem:logarithmic-schatten-selection}
Fix $p\ge2$, $\delta>0$, and $C\ge1$.  For dyadic
$1\le Q,M\le CN$, choose
\[
 r_{N}:=\max\left\{p,\,2,\,\delta^{-1}\log(2+N)\right\}.
\]
Then $p\le r_N<\infty$ and
\begin{equation}\label{eq:logarithmic-schatten-selection}
 r_N\bigl(1+N^{3/r_N}+Q^{3/r_N}+M^{3/r_N}\bigr)
 \lesssim_{p,\delta,C}N^\delta.
\end{equation}
The same assertion holds with $p$ replaced by any prescribed finite moment
$p_0$.  Thus a logarithmically growing Schatten exponent converts all finite
shell-volume remnants and the linear Khintchine factor into an arbitrarily
small power of the principal shell.
\end{lemma}

\begin{proof}
The choice of $r_N$ gives
$N^{3/r_N}\le e^{3\delta}$; since $Q,M\le CN$, the same bound, up to a
constant depending on $C$, holds for $Q^{3/r_N}$ and $M^{3/r_N}$.  Finally,
$\log(2+N)\lesssim_\delta N^\delta$, while the fixed contribution $p+2$ is
also bounded by $C_{p,\delta}N^\delta$ for $N\ge1$.  This proves
\eqref{eq:logarithmic-schatten-selection}.
\end{proof}

\begin{proposition}[Principal centered dyadic estimate]
\label{prop:principal-centered-block}
Let $R\sim_{\mathrm{tr}}N$, $Q\le c_{\mathrm{ap}}N$, and
$M\le C_{\mathrm{out}}N$.  After summing the finitely many
$R\sim N$ and $L\sim_{\mathrm{ap}}R$ shells, let $\mathcal B_{\Lambda,N,Q,M}^{a;b,c}(t,s)$ be the completely
Wick-centered block associated with \eqref{eq:iterated-soft-incidence}.  Put
$\Delta_T=\{(t,s):0\le s\le t\le T\}$.  For every $2\le p<\infty$ and every
$\eps>0$,
\begin{equation}\label{eq:centered-block-profile}
 \sup_\Lambda
 \left\|
 \|\mathcal B_{\Lambda,N,Q,M}^{a;b,c}\|_{
 C(\Delta_T;\mathcal L(L_q^2,L_n^2))}
 \right\|_{L^p_\omega}
 \lesssim_{p,\eps}
 N^{-3/2+\beta_{b,c}+\eps}
 \bigl(Q^{3/2+\eps}+M^{3/2+\eps}\bigr).
\end{equation}
For a fixed-profile dyadic cutoff family, the same estimate holds with
$\sup_\Lambda$ inside the $L^p_\omega$ norm.  For every fixed
$(N,Q,M)$, the blocks converge in
$L^p(\Omega;C(\Delta_T;\mathcal L(L_q^2,L_n^2)))$, and cutoff differences
obey the same dyadic majorant.
\end{proposition}

\begin{proof}
Treat first one of the finitely many outer scales $L\sim N$.  Use the shell
Hilbert spaces
$\mathcal A_N,\mathcal C_{1,L},\mathcal C_{0,Q},\mathcal E_M$ from
\eqref{eq:localized-block-hilbert-spaces}.  The deterministic kernel defines
$\mathbf H_{N,L,Q,M}^{t,s}$ by
\eqref{eq:localized-four-linear-form}; its operator coefficients act from the
input shell $\mathcal C_{0,Q}$ to the output shell $\mathcal E_M$.  At finite
cutoff pull back the two stochastic legs by
\[
 S_b=\mathsf S_{b,N,s}^{\Lambda}:\mathcal G_b\to\mathcal A_N,
 \qquad
 S_c=\mathsf S_{c,L,t}^{\Lambda}:\mathcal G_c\to\mathcal C_{1,L};
\]
for the limiting block use the same formula with $\Lambda=\infty$.  The four resulting maps have the domains, codomains, and factorizations in
\eqref{eq:pulled-flat-direction-1}--\eqref{eq:pulled-flat-direction-4}.
Thus the ideal property places one factor $\|S_b\|$ and one factor
$\|S_c\|$ in every orientation, giving
$N^{-1+\beta_b}L^{-1+\beta_c}\lesssim N^{-2+\beta_{b,c}}$ by
\eqref{eq:synthesis-size}.  The plateau identity \eqref{eq:plateau-times-shell} leaves the ranges of
$S_b,S_c$ unchanged.  Lemma~\ref{lem:wick-pullback-realization}
then realizes the pulled-back form as the centered Gaussian block, while
Lemma~\ref{lem:scalar-wiener-realization} identifies every scalar pairing
with the intrinsic double Wiener integral obtained from the finite-cutoff Wick
identity.  In the equal-color case,
Proposition~\ref{prop:localized-block-hilbert} identifies the diagonal
$\delta_{\mu\nu}$ term with the covariance contraction
\eqref{eq:same-color-shell-contraction}, already separated in
Proposition~\ref{prop:finite-color-split}.  Summing over the finitely many
$L\sim N$ and $R\sim N$ choices changes only the implicit constant.

The actual Duhamel shell contributes $|A_N|\lesssim N^{-1}$ by
\eqref{eq:shifted-duhamel-gain}.  Thus the effective centered envelope is
\begin{equation}\label{eq:centered-envelope}
 \mathcal G_{b,c}(N)\lesssim N^{-1}N^{-1+\beta_b}N^{-1+\beta_c}
 =N^{-3+\beta_{b,c}}.
\end{equation}
For $2\le r_0<\infty$, the four pulled-back Schatten estimates are
listed in Proposition~\ref{prop:centered-flattening-verification}.
Since $Q,M\lesssim N$,
$Q^{3/2}M^{3/2}\lesssim N^{3/2}(Q^{3/2}+M^{3/2})$.
Apply Lemma~\ref{lem:logarithmic-schatten-selection} with an internal loss
$\delta<\eps/4$ and choose the resulting exponent $r_0\ge p$.  The factors
$N^{3/r_0}$, $Q^{3/r_0}$, and $M^{3/r_0}$ are then uniformly bounded, while
the Khintchine factor $r_0$ costs at most $N^\delta$.  Substituting the four
rows of Proposition~\ref{prop:centered-flattening-verification} into
Theorem~\ref{thm:centered-gaussian-operator} proves the fixed-time form of
\eqref{eq:centered-block-profile}, with room to spare in the displayed
$\eps$ loss.
For distinct colors this is the decoupled chaos directly.  For equal colors, apply the same-field part of
Theorem~\ref{thm:centered-gaussian-operator} to the symmetrized tensor; the
covariance factor $\rho_N\rho_Lm_\Lambda^2$ is given in
\eqref{eq:same-color-shell-factor}.

For the time lift, on the principal support $|p|\sim N$,
\begin{equation}\label{eq:normalized-duhamel-increment}
 |N K_a(\tau,p)-N K_a(\tau',p)|
 \lesssim\min\{1,N|\tau-\tau'|\}
 \lesssim N^\theta|\tau-\tau'|^\theta.
\end{equation}
Together with \eqref{eq:synthesis-increment}, the decomposition
\[
\begin{aligned}
 H_{t,s}^{\sharp}-H_{t',s'}^{\sharp}
 ={}&(H_{t,s}-H_{t',s'})\circ
     (\mathsf S_{b,N,s},\mathsf S_{c,L,t})\\
 &+H_{t',s'}\circ
     (\mathsf S_{b,N,s}-\mathsf S_{b,N,s'},\mathsf S_{c,L,t})\\
 &+H_{t',s'}\circ
     (\mathsf S_{b,N,s'},\mathsf S_{c,L,t}-\mathsf S_{c,L,t'}).
\end{aligned}
\]
gives the same four Schatten estimates with the factor
$N^\theta(|t-t'|+|s-s'|)^\theta$.  Each term contains one increment, either in the normalized Duhamel
amplitude or in one covariance synthesis map.  Relative to its static bound, each such
increment costs $N^\theta$, and the resulting bounds are summed.  To obtain the displayed target loss $\eps$, rerun the preceding fixed-time
and increment estimates with internal loss $\eps/4$, choose
$0<\theta<\eps/4$, and take $p_0>\max\{p,2/\theta\}$.  In these estimates apply
Lemma~\ref{lem:logarithmic-schatten-selection} with the prescribed moment
$p_0$ and an internal loss smaller than $\eps/4$.  The selected Schatten
exponent is at least $p_0$, as required by the two-parameter time lift, and the
Khintchine factor costs at most the reserved $\eps/4$ shell loss.  The single increment cost
$N^\theta$ and these two internal losses then sum to strictly less than
$\eps$.  Applying Lemma~\ref{lem:centered-time-lift} with
$B=\mathcal K(L_q^2,L_n^2)$ proves
\eqref{eq:centered-block-profile} in $C(\Delta_T)$.

For a fixed dyadic triple, the two stochastic frequencies are confined to
fixed compact annuli.  Once the cutoff scale is sufficiently large, the
plateau in \eqref{eq:cutoff-supports} contains both annuli, so the two cutoff synthesis maps agree with their limiting maps.  Hence the
finite block, including its time increments and its same-field Wick
subtraction when $b=c$, eventually agrees with the uncut block.  This proves
fixed-block convergence for every admissible cofinal sequence.  For a
fixed-profile dyadic family,
Lemma~\ref{lem:finite-state-cutoff} applies simultaneously to the two
comparable stochastic shells.  It gives a scale-independent finite list of cutoff states, stabilization of
each fixed block, and \eqref{eq:finite-state-maximal-bound}.  This proves the estimate with
$\sup_\Lambda$ inside the $L^p$ norm.
\end{proof}

\begin{lemma}[Schatten continuity of a pulled-back dyadic block]
\label{lem:pulled-back-schatten-continuity}
Fix a principal shell tuple $(N,Q,M)$, one of its finitely many cutoff states,
and $2\le r_0<\infty$.  Let $F_\varkappa^\sharp(t,s)$ be the four
flattenings of the pulled-back form on $\Delta_T$.  For every
$0<\theta\le1/2$,
\begin{equation}\label{eq:principal-schatten-time-continuity}
 \sum_{\varkappa=1}^4
 \|F_\varkappa^\sharp(t,s)-F_\varkappa^\sharp(t',s')\|_{\Sch_{r_0}}
 \lesssim_{r_0,\theta}
 d_\Delta^\theta N^\theta\,\mathfrak P_{N,Q,M}(r_0),
\end{equation}
where
\begin{align}\label{eq:principal-schatten-profile}
 \mathfrak P_{N,Q,M}(r_0):={}&
 N^{-3/2+\beta_{b,c}}
 \bigl(Q^{3/2}M^{3/r_0}+M^{3/2}Q^{3/r_0}\bigr)\notag\\
 &+N^{-3+\beta_{b,c}}Q^{3/2}M^{3/2}N^{3/r_0}.
\end{align}
The same statement holds for differences of cutoff states.  On the transfer
tail, the right-hand side is replaced by the majorant in
\eqref{eq:inner-transfer-tail-time-increment}.  Consequently, for every fixed
shell tuple, $(t,s)\mapsto F_\varkappa^\sharp(t,s)$ is continuous in
$\Sch_{r_0}$, its image is compact, and simultaneous finite-rank compression
converges uniformly on $\Delta_T$ in all four Schatten norms.
\end{lemma}

\begin{proof}
Use the factorizations
\eqref{eq:pulled-flat-direction-1}--\eqref{eq:pulled-flat-direction-4} and
the three-term telescope separating the Duhamel increment from the two
synthesis-map increments.  On the principal shell,
\[
 |NK_a(t-s,p)-NK_a(t'-s',p)|
 \lesssim d_\Delta^\theta N^\theta,
\]
while \eqref{eq:synthesis-increment} costs the same relative factor on either
Gaussian leg.  Apply the Schatten ideal property to each term and then the
static bounds in Proposition~\ref{prop:centered-flattening-verification}; this
proves \eqref{eq:principal-schatten-time-continuity}.  Cutoff differences are
treated by the identical telescope.  The transfer-tail assertion is
Step~5 of Lemma~\ref{lem:inner-localization-transfer}; since $r_0\ge2$,
the displayed $\mathfrak S_2$ majorant also controls the
$\mathfrak S_{r_0}$ norm.

Continuity on the compact triangle makes the image of each flattening compact
in $\Sch_{r_0}$.  Apply
Lemma~\ref{lem:compact-schatten-compression} to the finite union of these four
compact images, including the appropriate tensor-product projections in each
domain and codomain.  The factorizations in Lemma~\ref{lem:four-flat-compression} then give
simultaneous uniform
convergence on $\Delta_T$ in all four Schatten norms.
\end{proof}

\begin{lemma}[Strong measurability of fixed dyadic blocks]
\label{lem:centered-fixed-block-measurability}
Let $H_Q$ and $H_M$ be the separable $L^2$ spaces of fixed input and output
shells.  Every centered dyadic kernel constructed above admits a strongly
measurable representative in
\[
 C\bigl(\Delta_T;\mathcal K(H_Q,H_M)\bigr).
\]
Its Volterraization is strongly measurable with values in
\[
 \mathcal L\bigl(C_TH_Q,\,C_TH_M\cap L_T^1H_M\bigr),
\]
and every finite dyadic sum is strongly measurable and separably valued in
the corresponding finite-block operator topology.
\end{lemma}

\begin{proof}
The compact-operator space $\mathcal K(H_Q,H_M)$, and hence
$C(\Delta_T;\mathcal K(H_Q,H_M))$, is separable.  Compress the two Gaussian
Hilbert legs and the input/output shell spaces simultaneously by finite-rank
orthogonal projections.  Each compressed kernel is a finite sum of scalar
Gaussian polynomials times deterministic continuous finite-rank operator
kernels, so it is strongly measurable in this space.

By Lemma~\ref{lem:pulled-back-schatten-continuity}, the four Schatten images
of a fixed block are compact and the simultaneous compressions converge
uniformly on the time triangle.  Hence $A_{m,n}\to0$ in
Lemma~\ref{lem:centered-uniform-compression}.  Compression is contractive in
every Schatten norm, so the Hölder seminorms of the compressed differences
are bounded by twice the right-hand side of
\eqref{eq:principal-schatten-time-continuity}; thus $(B_{m,n})$ is bounded.
Lemma~\ref{lem:centered-uniform-compression} makes the compressed Gaussian
kernels Cauchy in
\[
 L^p\bigl(\Omega;C(\Delta_T;\mathcal K(H_Q,H_M))\bigr).
\]
Their limit is therefore strongly measurable and independent of the
projections.  For a transfer-tail block, use
\eqref{eq:inner-transfer-tail-time-increment} in the same argument.

The Volterra map
\[
 (\mathfrak VK)(w)(t)=\int_0^tK(t,s)w(s)\dd s
\]
is bounded and linear from the preceding compact-kernel space to the stated
operator space, with norm at most $2T\|K\|_{C(\Delta_T;\mathcal L)}$ for
$T\le1$.  It preserves strong measurability and separable range.  A finite
dyadic sum factors through a finite product of shell spaces and is therefore
again strongly measurable and separably valued.
\end{proof}

\begin{lemma}[Quantitative dyadic Hardy bounds]
\label{lem:quantitative-dyadic-hardy}
Let $(x_Q)_{Q\in\Dyd}$ be a nonnegative square-summable sequence.
For every $\alpha>0$,
\begin{equation}\label{eq:dyadic-hardy-basic}
 \left(\sum_M\left[M^{-\alpha}\sum_{Q\le M}x_Q\right]^2\right)^{1/2}
 \lesssim_\alpha\left(\sum_Qx_Q^2\right)^{1/2}.
\end{equation}
For $0<\gamma<\alpha$ and $L\ge2$,
\begin{equation}\label{eq:dyadic-hardy-tail}
 \left(\sum_{M>L}\left[M^{-\alpha}\sum_{Q\le M}x_Q\right]^2\right)^{1/2}
 \lesssim_{\alpha,\gamma}L^{-\gamma}
 \left(\sum_Qx_Q^2\right)^{1/2}.
\end{equation}
Moreover, for every $\delta>0$,
\begin{equation}\label{eq:dyadic-geometric-tail}
 \sum_{Q>M}Q^{-\delta}\lesssim_\delta M^{-\delta},
 \qquad
 \sum_{Q>L}Q^{-\delta}\lesssim_\delta L^{-\delta}.
\end{equation}
\end{lemma}

\begin{proof}
Write $M=2^m$ and $Q=2^q$, with $m,q\ge0$.  The matrix
$K_{m,q}=2^{-\alpha m}\one_{\{q\le m\}}$ has uniformly bounded row sums,
\[
 \sup_m\sum_qK_{m,q}=\sup_m(m+1)2^{-\alpha m}<\infty,
\]
and uniformly bounded column sums,
\[
 \sup_q\sum_mK_{m,q}=\sup_q\sum_{m\ge q}2^{-\alpha m}<\infty.
\]
Schur's test on $\ell^2(\N_0)$ proves
\eqref{eq:dyadic-hardy-basic}.  For $M>L$ write
$M^{-\alpha}\le L^{-\gamma}M^{-(\alpha-\gamma)}$ and apply the basic
estimate with exponent $\alpha-\gamma$, which gives
\eqref{eq:dyadic-hardy-tail}.  The two estimates in
\eqref{eq:dyadic-geometric-tail} are geometric series.
\end{proof}

\begin{lemma}[Dyadic assembly for the centered profile]
\label{lem:centered-assembly}
Let $0<s<1-\beta_{b,c}$, $\beta_{b,c}<\sigma<1-\beta_{b,c}$, and suppose a family of continuous Volterra
kernels has nonnegative block majorants satisfying
\begin{equation}\label{eq:abstract-centered-block-majorant}
 b_{N,Q,M}\le C_b
 N^{-3/2+\beta_{b,c}+\eta}\bigl(Q^{3/2+\eta}+M^{3/2+\eta}\bigr),
 \qquad Q\le cN,\quad M\le CN,
\end{equation}
where
\begin{equation}\label{eq:assembly-strict-margin}
 \beta_{b,c}+2\eta<\min\{\sigma,1-\sigma,1-s\}.
\end{equation}
Then the Volterra block sum maps
\begin{equation}\label{eq:centered-assembly-map}
 C_TL^2\cap L_T^\infty B_{2,\infty}^{\sigma}
 \longrightarrow C_TH^{s-1}\cap C_TB_{2,\infty}^{\sigma-1}.
\end{equation}
For $0<T\le1$ it satisfies the quantitative estimates
\begin{align}
 \|\mathcal Bw\|_{C_TH^{s-1}}
 +\|\mathcal Bw\|_{C_TB_{2,\infty}^{\sigma-1}}
 &\lesssim T C_b
 \|w\|_{C_TL^2\cap L_T^\infty B_{2,\infty}^{\sigma}},
 \label{eq:centered-assembly-with-time}\\
 \|\mathcal Bw\|_{L_T^1B_{2,\infty}^{\sigma-1}}
 &\lesssim T^2 C_b
 \|w\|_{C_TL^2\cap L_T^\infty B_{2,\infty}^{\sigma}}.
 \label{eq:centered-assembly-L1-time}
\end{align}
The same assertions hold after taking $L^p(\Omega)$ norms of random kernel
majorants.

For dyadic $L\ge2$, set
\begin{equation}\label{eq:anisotropic-finite-set}
 \mathfrak F_L:=\{(N,Q,M):Q\le L,\ M\le L,\ N\le L^2\}.
\end{equation}
For every
\begin{equation}\label{eq:assembly-tail-gamma}
 0<\gamma<\min\{\sigma-\beta_{b,c}-2\eta,\,1-\sigma-\beta_{b,c}-2\eta,\,
 1-s-\beta_{b,c}-2\eta,\,3/2-2\beta_{b,c}-3\eta\},
\end{equation}
the sum over the complement of $\mathfrak F_L$ has operator norm
$O(TC_bL^{-\gamma})$ in the two $C_T$ targets in
\eqref{eq:centered-assembly-map}, and $O(T^2C_bL^{-\gamma})$ in
$L_T^1B_{2,\infty}^{\sigma-1}$.  Thus the finite block sums converge in
operator norm with a deterministic tail modulus.
\end{lemma}

\begin{proof}
It is enough first to work with smooth paths having only finitely many
Littlewood--Paley blocks.  Every spatial sum and Volterra integral is then
classical; the uniform estimates and the deterministic tail modulus below
extend the operator uniquely to the completed source space.

Collapse the high stochastic shell first:
\[
 K_{M,Q}:=\sum_{N\gtrsim\max\{M,Q\}}b_{N,Q,M}.
\]
The dyadic $N$ sum gives
\begin{equation}\label{eq:collapsed-centered-kernel}
 K_{M,Q}\lesssim C_b
 \begin{cases}
  M^{\beta_{b,c}+2\eta},&Q\le M,\\
  Q^{\beta_{b,c}+2\eta},&Q>M.
 \end{cases}
\end{equation}
Fix a source time $r$ and put $x_Q(r)=\|P_Qw(r)\|_{L^2}$.  All spatial
summations are performed at this common time.  Littlewood--Paley orthogonality
then identifies $(x_Q(r))_Q$ with the $L^2$ norm of $w(r)$, after which the
Volterra time integration is taken.

In the region $Q\le M$, put
$\alpha_H:=1-s-\beta_{b,c}-2\eta>0$.  Lemma~\ref{lem:quantitative-dyadic-hardy}
gives the discrete Hardy estimate
\begin{equation}\label{eq:centered-discrete-hardy}
 \left(\sum_M
 \left[M^{-\alpha_H}\sum_{Q\le M}x_Q(r)\right]^2\right)^{1/2}
 \lesssim_{s,\eta}\left(\sum_Qx_Q(r)^2\right)^{1/2}.
\end{equation}
Littlewood--Paley orthogonality identifies the right-hand side with
$\|w(r)\|_{L^2}$; no interchange of conditionally convergent spatial sums is
used.

In the region $Q>M$, use
$x_Q(r)\le Q^{-\sigma}\|w(r)\|_{B_{2,\infty}^{\sigma}}$.  Then
\[
 \left(\sum_M
 \left[M^{s-1}\sum_{Q>M}Q^{-\sigma+\beta_{b,c}+2\eta}\right]^2\right)^{1/2}
 \lesssim \|w(r)\|_{B_{2,\infty}^{\sigma}},
\]
since $\beta_{b,c}+2\eta<\sigma$ and $s-1-\sigma+\beta_{b,c}+2\eta<0$.  Consequently, uniformly in
the kernel time $t\ge r$,
\begin{equation}\label{eq:pointwise-H-assembly}
 \left\|\sum_{N,Q,M}B_{N,Q,M}(t,r)P_Qw(r)\right\|_{H^{s-1}}
 \lesssim C_b\bigl(\|w(r)\|_{L^2}
 +\|w(r)\|_{B_{2,\infty}^{\sigma}}\bigr).
\end{equation}

For the Besov target, use the Besov input in both frequency regions.  From
\eqref{eq:collapsed-centered-kernel},
\begin{align*}
 M^{\sigma-1}\sum_{Q\le M}M^{\beta_{b,c}+2\eta}Q^{-\sigma}
 &\lesssim M^{\sigma-1+\beta_{b,c}+2\eta},\\
 M^{\sigma-1}\sum_{Q>M}Q^{-\sigma+\beta_{b,c}+2\eta}
 &\lesssim M^{-1+\beta_{b,c}+2\eta}.
\end{align*}
Both powers are negative under \eqref{eq:assembly-strict-margin}; in
particular, they tend to zero as $M\to\infty$.  Thus
\begin{equation}\label{eq:pointwise-B-assembly}
 \left\|\sum_{N,Q,M}B_{N,Q,M}(t,r)P_Qw(r)\right\|_
 {B_{2,\infty}^{\sigma-1}}
 \lesssim C_b\|w(r)\|_{B_{2,\infty}^{\sigma}},
\end{equation}
and the output belongs to the little-Besov target.

Minkowski's integral inequality now gives, for every $t\le T$,
\[
 \|\mathcal Bw(t)\|_{H^{s-1}}
 +\|\mathcal Bw(t)\|_{B_{2,\infty}^{\sigma-1}}
 \lesssim C_b\int_0^t
 \bigl(\|w(r)\|_{L^2}+\|w(r)\|_{B_{2,\infty}^{\sigma}}\bigr)\dd r.
\]
This proves \eqref{eq:centered-assembly-with-time}; integration once more in
$t$ proves \eqref{eq:centered-assembly-L1-time}.  Continuity follows first for
finite block sums from continuity of the kernels and then for the full sum
from the uniform tail estimate below.

For the tail estimate, decompose the complement of $\mathfrak F_L$ into the
disjointly ordered regions
\[
 \{M>L\},\qquad \{M\le L<Q\},\qquad
 \{M,Q\le L,\ N>L^2\}.
\]
On $M>L$, apply \eqref{eq:dyadic-hardy-tail} with
$\alpha_H=1-s-\beta_{b,c}-2\eta$ to the $Q\le M$ Sobolev contribution.  The
$Q>M$ Sobolev contribution is bounded using
\eqref{eq:dyadic-geometric-tail}; its exponent is
$1-s+\sigma-\beta_{b,c}-2\eta>\alpha_H$.  Thus every
$0<\gamma<\alpha_H$ gives a factor $L^{-\gamma}$.  In the Besov target, the
two frequency regions have decay exponents
$1-\sigma-\beta_{b,c}-2\eta$ and $1-\beta_{b,c}-2\eta$, respectively, so every
$\gamma<1-\sigma-\beta_{b,c}-2\eta$ is admissible.

On $M\le L<Q$, necessarily $Q>M$.  With
$\delta_Q:=\sigma-\beta_{b,c}-2\eta>0$,
\eqref{eq:dyadic-geometric-tail} gives
\[
 \sum_{Q>L}Q^{-\sigma+\beta_{b,c}+2\eta}
 \lesssim L^{-\delta_Q}.
\]
The remaining $M$-weights are summable in the Sobolev target because $s<1$
and are bounded by one in the Besov target because $\sigma<1$.  Hence this
region has norm $O(C_bL^{-\delta_Q})$.

Finally, if $Q,M\le L$ and $N>L^2$, direct dyadic summation gives
\begin{equation}\label{eq:assembly-high-N-tail}
 \sum_{N>L^2}b_{N,Q,M}
 \lesssim C_bL^{-3+2\beta_{b,c}+2\eta}
 (Q^{3/2+\eta}+M^{3/2+\eta})
 \lesssim C_bL^{-\kappa_N},
\end{equation}
where $\kappa_N:=3/2-2\beta_{b,c}-3\eta>0$.  In the Sobolev target,
Cauchy--Schwarz over the finitely many input shells costs at most
$(1+\log L)^{1/2}$, while
$\sum_{M\le L}M^{2(s-1)}$ is uniformly bounded.  In the Besov target,
$\sum_{Q\le L}Q^{-\sigma}$ is uniformly bounded.  Therefore every
$0<\gamma<\kappa_N$ absorbs the possible logarithm and yields an
$O(C_bL^{-\gamma})$ operator tail.  Taking the minimum of the four positive
margins gives \eqref{eq:assembly-tail-gamma}.  The Volterra integral adds the
factor $T$, and the $L_T^1$ norm adds one further factor $T$, proving the
stated tail moduli.

For random kernels, take the $L^p(\Omega)$ norm of the kernel majorant before
the deterministic sums and use Minkowski's inequality.  By
Lemma~\ref{lem:centered-fixed-block-measurability}, every finite anisotropic
sum is strongly measurable and separably valued.  The quantitative tail makes
these sums Cauchy in the Bochner space of ambient operators; completeness
therefore gives a strongly measurable assembled limit without any
separability assumption on the full ambient operator space.
\end{proof}

\begin{lemma}[Strong measurability of assembled centered operators]
\label{lem:centered-block-measurability}
Under the hypotheses of Lemma~\ref{lem:centered-assembly}, the assembled
centered operator is a strongly measurable random variable in the ambient
Sobolev--Besov operator topology.  The same conclusion holds for operator-norm
limits of cutoff differences and for the locally small
$L_T^1B_{2,\infty}^{\sigma-1}$ component.
\end{lemma}

\begin{proof}
The finite anisotropic sums are strongly measurable by
Lemma~\ref{lem:centered-fixed-block-measurability}.  The deterministic tail
modulus in Lemma~\ref{lem:centered-assembly}, after taking the prescribed
$L^p(\Omega)$ norm, makes them Cauchy in the corresponding Bochner space.
Its limit is strongly measurable and has separable essential range, even
when the ambient operator space itself is nonseparable.  Applying the same
argument to differences and to the $L_T^1$ target proves the remaining
claims.
\end{proof}

\begin{lemma}[Simultaneous dyadic realization and full cutoff convergence]
\label{lem:centered-dyadic-bc}
Fix the physical localizers and the four index triples
$(a;b,c)\in\mathfrak L$.  Assume the fixed-profile estimate
\eqref{eq:centered-block-profile}, with $\sup_{\Lambda\in\Dyd}$ inside every
finite $L^p(\Omega)$ norm.  There is one event
$\Omega_{\mathrm{cen}}$ of probability one such that, simultaneously for all
index triples, all dyadic triples, and every $\eta'>0$,
\begin{equation}\label{eq:pathwise-centered-block-majorant}
 \sup_{\Lambda\in\Dyd}
 \|\mathcal B_{\Lambda,N,Q,M}^{a;b,c}\|_{
 C(\Delta_{T_0};\mathcal L(L_q^2,L_n^2))}
 \le C_{\eta'}(\omega)
 N^{-3/2+\beta_{b,c}+\eta'}(Q^{3/2+\eta'}+M^{3/2+\eta'}).
\end{equation}
On this event, for every $0<s<1-\beta_{b,c}$ and $\beta_{b,c}<\sigma<1-\beta_{b,c}$, the entire fixed-profile dyadic
cutoff sequence converges, without passage to a subsequence, in the assembled
operator topology whenever $\eta'$ is chosen below the strict margins in
\eqref{eq:assembly-strict-margin}.
\end{lemma}

\begin{proof}
Choose a decreasing sequence $\eta_m\downarrow0$ and put
$\eps_m=\delta_m=\eta_m/4$.  For each $m$, take a finite moment $p_m$ so large
that $\delta_mp_m>2$.  Proposition~\ref{prop:principal-centered-block}, in its
fixed-profile maximal form, and Markov's inequality give
\begin{align*}
 &\Pp\left(
 \frac{\sup_\Lambda
 \|\mathcal B_{\Lambda,N,Q,M}^{a;b,c}\|_{C(\Delta_{T_0};\mathcal L)}}
 {N^{-3/2+\beta_{b,c}+\eps_m}
 (Q^{3/2+\eps_m}+M^{3/2+\eps_m})}
 >(NQM)^{\delta_m}\right)\\
 &\hspace{35mm}\lesssim_{m}(NQM)^{-\delta_mp_m}.
\end{align*}
The sum of the right-hand side over dyadic $(N,Q,M)$ is finite.  Since
$Q,M\lesssim N$ on the principal sector, the threshold factor can be absorbed
as
\[
 (NQM)^{\delta_m}N^{-3/2+\beta_{b,c}+\eps_m}
 (Q^{3/2+\eps_m}+M^{3/2+\eps_m})
 \lesssim
 N^{-3/2+\beta_{b,c}+\eta_m}(Q^{3/2+\eta_m}+M^{3/2+\eta_m}).
\]
The first Borel--Cantelli lemma leaves only finitely many exceptional
triples.  Taking the maximum of their normalized sizes together with one gives
a finite random constant $C_{\eta_m}(\omega)$ and hence
\eqref{eq:pathwise-centered-block-majorant} with $\eta' =\eta_m$.  Intersect
the resulting events over the countable set of $m$ and the four index triples.  If $\eta'>0$ is arbitrary, choose $m$ with $\eta_m<\eta'$; because
all dyadic scales are at least one, the $\eta_m$ estimate implies the
$\eta'$ estimate.  Restriction from $\Delta_{T_0}$ gives the same event for every
$0<T\le T_0$.  No independence between different blocks is used.

Fix $s,\sigma$ and choose one $\eta'$ below the assembly margins.  Given
$L$, the complement of the finite anisotropic set $\mathfrak F_L$ in
\eqref{eq:anisotropic-finite-set} is $O_\omega(L^{-\gamma})$ by
Lemma~\ref{lem:centered-assembly}.  On the finite set, Lemma~\ref{lem:finite-state-cutoff} gives stabilization of
every block for all sufficiently large dyadic cutoffs.  A finite-set/tail argument makes
the whole cutoff sequence Cauchy in the assembled operator norm and proves
full-sequence convergence.  The anisotropic finite sums are strongly
measurable by Lemma~\ref{lem:centered-block-measurability}; extending their
pointwise norm limit by zero off $\Omega_{\mathrm{cen}}$ gives a strongly
measurable pathwise realization.  It agrees almost surely with the Bochner
limit constructed in the assembly argument.
\end{proof}

\begin{proposition}[Compatibility of spectral, Gaussian, and dyadic approximations]
\label{prop:centered-approximation-compatibility}
Fix $(a;b,c)\in\mathfrak L$, admissible exponents
\[
 0<s<1-\beta_{b,c},
 \qquad
 \beta_{b,c}<\sigma<1-\beta_{b,c},
\]
and $2\le p<\infty$.  Let $(\Lambda_j)_{j\ge1}$ be any admissible cofinal
spectral cutoff sequence.  On each of the four Hilbert legs in
\eqref{eq:abstract-flat-1}--\eqref{eq:abstract-flat-4}, fix increasing
finite-rank orthogonal projections converging strongly to the identity; in
the equal-color case use the same projection on the two copies of the common
Gaussian space.  For dyadic $L\ge2$, let
$\mathcal B_{j,m,L}^{a;b,c}$ be the finite sum of the principal centered
blocks over $\mathfrak F_L$ in \eqref{eq:anisotropic-finite-set}, with spectral
cutoff $\Lambda_j$ and simultaneous four-leg compression at level $m$.
Write $j=\infty$ for the uncut synthesis maps and $m=\infty$ for the
uncompressed centered Wiener operator.  Set
\begin{equation}\label{eq:centered-compatibility-target}
 \mathscr X_T^{s,\sigma}:=
 \mathcal L\!\left(
 C_TL^2\cap L_T^\infty B_{2,\infty}^{\sigma},
 C_TH^{s-1}\cap L_T^1B_{2,\infty}^{\sigma-1}
 \right).
\end{equation}
Then the following assertions hold.

First, for finite $j,m,L$, the operator
$\mathcal B_{j,m,L}^{a;b,c}$ is a finite Gaussian polynomial with values in a
finite-dimensional subspace of $\mathscr X_T^{s,\sigma}$.  For each $j,m$,
including the uncut or uncompressed values, its limit as $L\to\infty$ exists
in $L^p(\Omega;\mathscr X_T^{s,\sigma})$; denote it by
$\mathcal B_{j,m}^{a;b,c}$.  For every $\gamma$ below the strict margins in
\eqref{eq:assembly-tail-gamma},
\begin{equation}\label{eq:three-layer-uniform-tail}
 \sup_{j,m}
 \|\mathcal B_{j,m}^{a;b,c}-\mathcal B_{j,m,L}^{a;b,c}\|_
 {L^p(\Omega;\mathscr X_T^{s,\sigma})}
 \lesssim_{p,\gamma}L^{-\gamma}.
\end{equation}
Here and below the suprema include $j=\infty$ and $m=\infty$.

Second, the three-parameter family has a joint limit:
\begin{equation}\label{eq:three-layer-joint-limit}
 \mathcal B_{j,m,L}^{a;b,c}
 \longrightarrow
 \mathcal B_{\infty,\infty}^{a;b,c}
 \quad\text{in }L^p(\Omega;\mathscr X_T^{s,\sigma})
 \quad\text{as }j,m,L\to\infty
\end{equation}
with the product ordering on $\N^3$.  Consequently every iterated limit in
the spectral cutoff, Gaussian compression, and dyadic assembly parameters
exists and gives the same centered operator.  In the equal-color case, the common limit is the centered second Wiener
chaos obtained after removing the deterministic covariance distribution.
\end{proposition}

\begin{proof}
At finite $(j,m,L)$, only finitely many spatial shells and finitely many
Gaussian coordinates occur, so the operator is a finite Gaussian polynomial
and is strongly measurable.  Orthogonal compression cannot increase any of
the four Schatten norms in
Lemma~\ref{lem:four-flat-compression}.  The block majorant in
\eqref{eq:abstract-centered-block-majorant} is therefore uniform in $j$ and
$m$.  Applying Lemma~\ref{lem:centered-assembly} to the compressed blocks
proves existence of $\mathcal B_{j,m}^{a;b,c}$ and the uniform tail
\eqref{eq:three-layer-uniform-tail}.

Fix $\varepsilon>0$ and choose $L_*$ so large that the right-hand side of
\eqref{eq:three-layer-uniform-tail}, with $L=L_*$, is smaller than
$\varepsilon$.  The set $\mathfrak F_{L_*}$ is finite.  Cofinality and the
plateau property imply that there is $j_*$ such that, for $j\ge j_*$, every
spectral multiplier occurring in these blocks is identically one on its
stochastic shell.  Thus, for every $m$,
\begin{equation}\label{eq:finite-layer-spectral-stabilization}
 \mathcal B_{j,m,L_*}^{a;b,c}
 =\mathcal B_{\infty,m,L_*}^{a;b,c},
 \qquad j\ge j_*.
\end{equation}

For each block in $\mathfrak F_{L_*}$, choose
$2\le p\le p_0\le r<\infty$ and $0<\theta'<\theta\le1/2$ with
$p_0\theta'>2$.  The time triangle is compact and the four uncut pulled-back
flattenings are continuous in $\Sch_r$ by
Lemma~\ref{lem:pulled-back-schatten-continuity}; their images are therefore
compact Schatten sets, while their $\theta$-Hölder seminorms are uniformly
bounded on the finite collection.  Lemma~\ref{lem:compact-schatten-compression}
applies simultaneously to all four orientations and all blocks.  The uniform
time-triangle completion in Lemma~\ref{lem:centered-uniform-compression} then
gives an $m_*$ such that
\begin{equation}\label{eq:finite-layer-gaussian-compression}
 \|\mathcal B_{\infty,m,L_*}^{a;b,c}
 -\mathcal B_{\infty,\infty,L_*}^{a;b,c}\|_
 {L^p(\Omega;\mathscr X_T^{s,\sigma})}
 <\varepsilon,
 \qquad m\ge m_*.
\end{equation}
The same finite-block estimate bounds the Volterra map from
$C(\Delta_T)$ to the operator target $\mathscr X_T^{s,\sigma}$, so the
convergence passes to that space.

Let now $j\ge j_*$, $m\ge m_*$, and $L\ge L_*$.  Since
$\mathfrak F_{L_*}\subset\mathfrak F_L$, the difference between the two finite
dyadic sums is a sub-sum of the complement of $\mathfrak F_{L_*}$ and hence
is bounded by the same nonnegative majorant as the assembly tail.  Combining
\eqref{eq:three-layer-uniform-tail},
\eqref{eq:finite-layer-spectral-stabilization}, and
\eqref{eq:finite-layer-gaussian-compression} yields
\[
 \|\mathcal B_{j,m,L}^{a;b,c}
 -\mathcal B_{\infty,\infty}^{a;b,c}\|_
 {L^p(\Omega;\mathscr X_T^{s,\sigma})}
 \lesssim\varepsilon.
\]
This proves the joint limit \eqref{eq:three-layer-joint-limit} and therefore
all iterated-limit identities.  For distinct colors, the scalar pairings are
ordered double Wiener integrals.  For equal colors,
Lemma~\ref{lem:scalar-wiener-realization} identifies every scalar pairing with
the symmetrized second Wiener integral, while
Proposition~\ref{prop:localized-diagonal} supplies the separated covariance
distribution.  This identifies the joint limit with the centered second
chaos.
\end{proof}

\begin{theorem}[Main-sector centered operator]
\label{thm:localized-centered}
Fix $0<T\le T_0$, $0<s<1-\beta_{b,c}$, $\beta_{b,c}<\sigma<1-\beta_{b,c}$, compactly supported inner and
outer physical localizers, and $(a;b,c)\in\mathfrak L$.  Let
$\mathcal B_{\Lambda,\mathrm{main}}^{a;b,c}$ be the sum of the principal
blocks $R\sim_{\mathrm{tr}}N$ and $M\le C_{\mathrm{out}}N$ in
Proposition~\ref{prop:principal-centered-block}.  There exists a centered random
operator $\mathcal B_{\mathrm{main}}^{a;b,c}$ such that, for every
$2\le p<\infty$,
\begin{equation}\label{eq:centered-operator-topology}
 \mathcal B_{\Lambda,\mathrm{main}}^{a;b,c}
 \longrightarrow \mathcal B_{\mathrm{main}}^{a;b,c}
 \quad\text{in}\quad
 L^p\!\left(\Omega;
 \mathcal L\!\left(C_TL^2\cap L_T^\infty B_{2,\infty}^{\sigma},
 C_TH^{s-1}\cap L_T^1B_{2,\infty}^{\sigma-1}\right)\right).
\end{equation}
The limit is independent of the admissible cutoff family.  Along a
fixed-profile dyadic family, convergence holds almost surely in the same
operator topology.  The moment bounds are uniform on separated parameter classes and for bounded $\mathfrak H_{\mathrm{prof}}$.
\end{theorem}

\begin{proof}
Choose $\eta>0$ below the strict margins in
Lemma~\ref{lem:centered-assembly}.  Apply
Proposition~\ref{prop:principal-centered-block} and Minkowski's inequality.
This gives a uniform $L^p$ bound for the assembled operators in
\eqref{eq:centered-operator-topology}.  For two arbitrary
admissible cofinal cutoffs, first choose a finite dyadic set whose complement
is uniformly small by the assembly lemma.  On the finite set, fixed-block
convergence from Proposition~\ref{prop:principal-centered-block} gives
convergence in $L^p$ operator norm.  This proves the Cauchy property and also
shows that the limit is independent of the cutoff profile; no nesting of the
multipliers is required.  Proposition~\ref{prop:centered-approximation-compatibility}
shows in addition that spectral removal, Gaussian finite-rank completion,
and dyadic assembly have a joint limit, so the definition of the centered
operator does not depend on an implicit order of approximation.

For a fixed-profile dyadic family, apply
Lemma~\ref{lem:centered-dyadic-bc} and repeat the deterministic
assembly pathwise.  This gives full-sequence almost-sure convergence on one
probability-one event.  Strong measurability was established in the assembly
lemma.  The constants in the centered proof use the individual wave and
Klein--Gordon propagator bounds, the fixed localizers, the dyadic support
constants, and the shellwise bounds for the two forcing profiles.  Hence they
are uniform on separated speed--mass classes, on compact subsets of the
admissible $(\betaW,\betaK)$ region, and on sets with bounded forcing-profile
$\mathfrak H_{\mathrm{prof}}$.  The centered branch uses only single-channel propagator bounds.
\end{proof}

\begin{proposition}[Uniform interval smallness of the centered source]
\label{prop:centered-besov-small}
For almost every realization constructed in
Theorem~\ref{thm:localized-centered}, and every $0<h\le1$,
\begin{equation}\label{eq:centered-uniform-interval-small}
 \sup_{\substack{I\subset[0,T_0]\text{ interval}\\ |I|\le h}}
 \|\mathcal B^{a;b,c}_{\mathrm{main},I}\|_{
 \mathcal L(E_{I,0}^{2,\sigma},L^1(I;B_{2,\infty}^{\sigma-1}))}
 \lesssim C_{\eta'}(\omega)h^2
\end{equation}
for every $\eta'>0$ below the strict assembly margins.  In particular,
\begin{equation}\label{eq:centered-besov-small}
 \|\mathcal B^{a;b,c}_{\mathrm{main},[0,T]}\|_{
 \mathcal L(E_T^{2,\sigma},L_T^1B_{2,\infty}^{\sigma-1})}
 \longrightarrow0
 \qquad\text{as }T\downarrow0.
\end{equation}
Along a fixed-profile dyadic approximation, both assertions hold on the
common probability-one event of Theorem~\ref{thm:localized-centered},
uniformly for the limiting operator, every cutoff level, and all cutoff
differences.
\end{proposition}

\begin{proof}
Work on the event $\Omega_{\mathrm{cen}}$ of
Lemma~\ref{lem:centered-dyadic-bc}.  Choose $\eta'>0$ below the strict
assembly margins.  The pathwise block majorant
\eqref{eq:pathwise-centered-block-majorant} is uniform on the full triangle
$\Delta_{T_0}$ and supplies a finite random constant $C_{\eta'}(\omega)$ in
the hypothesis of Lemma~\ref{lem:centered-assembly}.  If
$I=[t_0,t_1]$ and the input has zero history at $t_0$, the Volterra integral starts
at $t_0$.  Repeating the two time integrations in the assembly lemma over
$t_0\le s\le t\le t_1$ gives the factors $|I|$ and $|I|^2$, with all spatial sums
unchanged.  Hence
\[
 \|\mathcal B^{a;b,c}_{\mathrm{main},I}\|_{
 \mathcal L(E_{I,0}^{2,\sigma},L^1(I;B_{2,\infty}^{\sigma-1}))}
 \lesssim C_{\eta'}(\omega)|I|^2,
\]
which proves \eqref{eq:centered-uniform-interval-small}.  On the initial
interval $I=[0,T]$, the same two integrations apply to an arbitrary input in
$E_T^{2,\sigma}$: no endpoint vanishing is used because the Volterra lower
limit is already zero.  This proves \eqref{eq:centered-besov-small}.  For a
fixed-profile dyadic family, the supremum over all cutoff levels is already
included in \eqref{eq:pathwise-centered-block-majorant}; a cutoff difference
is bounded by the sum of the two corresponding kernel majorants.  Thus the
same interval estimate is uniform over the whole family and its differences.
\end{proof}

\begin{lemma}[Localization remainder]
\label{lem:far-output}
Let $\mathcal R^{a;b,c}_{\Lambda,N,R,Q,M}(t,s)$ denote the sum over the
finitely many outer stochastic shells $L\sim_{\mathrm{ap}}R$ of a centered
block in one of the following two sectors:
\begin{equation}\label{eq:localization-remainder-sectors}
 R\not\sim_{\mathrm{tr}}N,
 \qquad\text{or}\qquad
 R\sim_{\mathrm{tr}}N\ \text{ and }\ M>C_{\mathrm{out}}N.
\end{equation}
For every $2\le p<\infty$, every $A>0$, and every
$0<\theta\le1/2$,
\begin{align}
 \|\mathcal R^{a;b,c}_{\Lambda,N,R,Q,M}(t,s)\|_{
 L^p_\omega\mathcal L(L_q^2,L_n^2)}
 &\lesssim_{p,A}\la N+R+M\ra^{-A},
 \label{eq:localization-tail-block}\\
 \|\mathcal R^{a;b,c}_{\Lambda,N,R,Q,M}(t,s)
 -\mathcal R^{a;b,c}_{\Lambda,N,R,Q,M}(t',s')\|_{
 L^p_\omega\mathcal L(L_q^2,L_n^2)}
 &\lesssim_{p,A,\theta}\la N+R+M\ra^{-A}
 (|t-t'|+|s-s'|)^\theta.
 \label{eq:localization-tail-increment}
\end{align}
The estimates are uniform in the spectral cutoff.  Cutoff differences obey
the same majorants and converge on every fixed block.  For a fixed-profile
dyadic family, the supremum over all cutoff levels may be placed inside each
finite $L^p(\Omega)$ norm in \eqref{eq:localization-tail-block} and
\eqref{eq:localization-tail-increment}.  Consequently, the sum of the
remainder blocks converges after the insertion of arbitrary fixed polynomial
Sobolev or Besov weights, in particular in every operator topology used in
Theorem~\ref{thm:localized-centered}.  Along a fixed-profile dyadic family,
the full remainder sequence converges almost surely, without a subsequence,
on a probability-one event that may be intersected with
$\Omega_{\mathrm{cen}}$.  For every interval $I\subset[0,T_0]$, its
zero-history restriction is causal and maps $E_{I,0}^{2,\sigma}$ to
$L^1(I;B_{2,\infty}^{\sigma-1})$ with norm $O_\omega(|I|^2)$, uniformly over
the limit and all cutoff levels on that event.  In particular, its
initial-interval operator norm tends to zero as $T\downarrow0$.
\end{lemma}

\begin{proof}
In the transfer sector $R\not\sim_{\mathrm{tr}}N$, the estimate follows from
\eqref{eq:inner-transfer-tail-flat}.  When $M\lesssim R$, the arbitrary
Schwartz power in $N+R$ absorbs the displayed polynomial factor and all
Sobolev--Besov weights.  When $M\gg R$, the outer defect
$\widehat\chi_{\mathrm{out}}(n-p-r)$ contributes arbitrary decay in $M$.

It remains to treat the main-transfer far-output sector.  There
$|p|\sim|r|\sim N$ and $|n|\sim M>C_{\mathrm{out}}N$.  Taking
$C_{\mathrm{out}}$ sufficiently large gives
\begin{equation}\label{eq:main-far-output-defect}
 |n-p-r|\gtrsim M.
\end{equation}
The Fourier transform of the outer physical localizer is Schwartz.  Keeping
the phase-space powers explicit, the kernel
\eqref{eq:two-localizer-kernel}, Minkowski in the inner defect $z$, and
\eqref{eq:main-far-output-defect} give, for every $B>0$,
\begin{equation}\label{eq:far-output-HS-explicit}
 \max_{1\le\varkappa\le4}
 \|F_{\varkappa,N,R,Q,M}^{\sharp}(t,s)\|_{\Sch_2}
 \lesssim_B
 \langle N\rangle^{\beta_{b,c}}\langle Q\rangle^{3/2}
 \langle M\rangle^{3/2-B}.
\end{equation}
Indeed, before covariance pullback the Duhamel factor contributes $N^{-1}$,
the two stochastic shells contribute the square-root volume $N^3$, and the
input and output shells contribute $Q^{3/2}M^{3/2}$.  The two synthesis maps
then contribute $N^{-2+\beta_{b,c}}$, which gives the polynomial factor in
\eqref{eq:far-output-HS-explicit}.  Since $M>C_{\mathrm{out}}N$ and
$Q\lesssim N$, choosing $B$ larger than the prescribed target power makes the
right-hand side decay faster than any fixed polynomial in $N+M$.  The
Hilbert--Schmidt estimate controls every Schatten orientation.  Applying
Theorem~\ref{thm:centered-gaussian-operator}, together with complete
same-family Wick centering when necessary, proves
\eqref{eq:localization-tail-block}.

For increments, use
\[
 |K_a(\tau,p)-K_a(\tau',p)|
 \lesssim_\theta|\tau-\tau'|^\theta\la R\ra^{-1+\theta}
 \qquad (\la p\ra\sim\la R\ra)
\]
and the synthesis-map increment estimates at scales $N$ and $R$.  Repeating
the preceding Hilbert--Schmidt calculation yields the right-hand side of
\eqref{eq:far-output-HS-explicit} multiplied by
$d_\Delta^\theta\langle N\rangle^\theta$.  Increasing $B$ absorbs this explicit
increment cost and proves \eqref{eq:localization-tail-increment}.
For each fixed remainder block, all stochastic cutoff factors are supported
on finitely many fixed annuli and therefore stabilize in the cutoff plateau.  The time-lift Lemma~\ref{lem:centered-time-lift}, followed by
absolute dyadic summation, gives the asserted operator convergence.

For a fixed-profile dyadic family, Lemma~\ref{lem:finite-state-cutoff}
applies to every fixed remainder block as well.  It therefore puts the cutoff
supremum inside the $L^p(\Omega)$ norm while preserving an arbitrarily high
Schwartz power.  Choose that power and then a finite moment large enough that
Markov's inequality is summable over the countable tuples $(N,R,Q,M)$ and the
finite set of index triples.  The first Borel--Cantelli lemma yields one event on
which the cutoff-maximal remainder blocks are absolutely summable with every
fixed polynomial Sobolev--Besov weight.  No independence of the blocks is
used.  Since each fixed block stabilizes for all sufficiently large dyadic cutoffs, the finite-set/tail argument gives full-sequence pathwise
convergence of the remainder.

Finally, let $I=[t_0,t_1]\subset[0,T_0]$ and let
$w\in E_{I,0}^{2,\sigma}$.  For a single remainder kernel with
$\sup_{(t,s)\in\Delta_{T_0}}\|K(t,s)\|_{L^2\to L^2}\le A_K$, causality and
zero history give
\begin{align*}
 \left\|t\mapsto\int_{t_0}^tK(t,s)w(s)\dd s\right\|_{L^1(I;L^2)}
 &\le \int_{t_0}^{t_1}\int_{t_0}^tA_K\|w(s)\|_{L^2}\dd s\dd t\\
 &\le \tfrac12|I|^2A_K\|w\|_{C_IL^2}.
\end{align*}
The remainder block majorants are absolutely summable after every required
Sobolev--Besov weight, uniformly over the limit and all fixed-profile cutoff
levels.  Summing the displayed bound therefore gives a uniform
$O_\omega(|I|^2)$ zero-history operator norm.  This proves the initial-interval smallness and the translated interval
modulus.
\end{proof}

\begin{corollary}[Full centered operator]
\label{cor:full-centered}
The full centered operator has the decomposition
\[
 \mathcal B^{a;b,c}_\Lambda
 =\mathcal B^{a;b,c}_{\Lambda,\mathrm{main}}
  +\mathcal R^{a;b,c}_\Lambda,
\]
where the first term is the principal operator of
Theorem~\ref{thm:localized-centered} and the second is the localization
remainder of Lemma~\ref{lem:far-output}.  Consequently, the full operators
converge in the Banach operator space displayed in
\eqref{eq:centered-operator-topology}, are independent of the admissible
cutoff family, and, almost surely, satisfy
\begin{align}
 \sup_{\substack{I\subset[0,T_0]\text{ interval}\\ |I|\le h}}
 \|\mathcal B_I^{a;b,c}\|_{
 \mathcal L(E_{I,0}^{2,\sigma},L^1(I;B_{2,\infty}^{\sigma-1}))}
 &\longrightarrow0 &&(h\downarrow0),
 \label{eq:full-centered-interval-smallness}\\
 \|\mathcal B^{a;b,c}_{[0,T]}\|_{
 \mathcal L(E_T^{2,\sigma},L_T^1B_{2,\infty}^{\sigma-1})}
 &\longrightarrow0 &&(T\downarrow0).
 \label{eq:full-centered-local-smallness}
\end{align}
The fixed-profile almost-sure convergence and the cutoff-uniform interval
smallness statements hold after intersecting the probability-one event of
Lemma~\ref{lem:centered-dyadic-bc} with the remainder event of
Lemma~\ref{lem:far-output}.
\end{corollary}

\begin{proof}
The dyadic partition into the principal sector and the two sectors in
\eqref{eq:localization-remainder-sectors} is exhaustive.  On the principal
sector, Theorem~\ref{thm:localized-centered} gives convergence in the full
operator topology, while Proposition~\ref{prop:centered-besov-small} gives
the initial and translated-interval Besov smallness.  Lemma~\ref{lem:far-output}
gives the same conclusions for the complementary sectors, with absolutely
summable Schwartz tails after every Sobolev--Besov weight used here.  Adding
the two operator limits proves the asserted decomposition and cutoff-profile
independence.  The two smallness limits follow by the triangle inequality;
the common fixed-profile event is the finite intersection of the principal
and remainder events.
\end{proof}

\begin{center}
\small
\renewcommand{\arraystretch}{1.22}
\begin{tabularx}{0.98\textwidth}{@{}P{0.16\textwidth}P{0.13\textwidth}P{0.18\textwidth}P{0.20\textwidth}Y@{}}
\toprule
Label & contraction & centered order & operator input & output used in the fixed point\\
\midrule
$(\W;\K,\K)$ & $\mathcal D^{\W;\K}$ & $2\beta_{\K}$
& $E_T^{2,\sigma}$ & $C_TH^{s_2-1}\cap L_T^1B_{2,\infty}^{\sigma-1}$\\
$(\K;\W,\W)$ & $\mathcal D^{\K;\W}$ & $2\beta_{\W}$
& $E_T^{2,\sigma}$ & $C_TH^{s_2-1}\cap L_T^1B_{2,\infty}^{\sigma-1}$\\
$(\W;\W,\K)$ & none & $\beta_{\W}+\beta_{\K}$
& $E_T^{2,\sigma}$ & $C_TH^{s_2-1}\cap L_T^1B_{2,\infty}^{\sigma-1}$\\
$(\K;\K,\W)$ & none & $\beta_{\W}+\beta_{\K}$
& $E_T^{2,\sigma}$ & $C_TH^{s_2-1}\cap L_T^1B_{2,\infty}^{\sigma-1}$\\
\bottomrule
\end{tabularx}
\end{center}
For a diagonal label, the deterministic component is estimated separately in
$L_T^\infty B_{2,\infty}^{\sigma-1}$ and acquires a factor $T$ when inserted
into the source space.  The centered component is measured directly in
$L_T^1B_{2,\infty}^{\sigma-1}$.

\begin{theorem}[Bounds for the resonant operators]\label{thm:resonant-operators}
Fix $0<T_0\le1$, a label $(a;b,c)\in\mathfrak L$, exponents
\begin{equation}\label{eq:operator-exponent-range}
 0<s<1-\beta_{b,c},
 \qquad
 \beta_{b,c}<\sigma<1-\beta_{b,c},
\end{equation}
and an adapted localization pair.  Construct all kernels on the full time
triangle $\Delta_{T_0}$.  For every admissible spectral cutoff,
\begin{equation}\label{eq:operator-decomposition}
 T^{a;b,c}_{\Lambda}=
 \begin{cases}
  \mathcal D^{a;b}_{\Lambda}+\mathcal B^{a;b,b}_{\Lambda},
    &(a;b,c)=(a;b,b)\in\mathfrak L_{\mathrm{diag}},\\
  \mathcal B^{a;b,c}_{\Lambda},
    &(a;b,c)\in\mathfrak L_{\mathrm{off}}.
 \end{cases}
\end{equation}
For every $0<T\le T_0$, the centered operators
$\mathcal B^{a;b,c}_{\Lambda}$ converge, for every $2\le p<\infty$, in
\begin{equation}\label{eq:centered-full-operator-topology}
  L^p\!\left(\Omega;
  \mathcal L\!\left(
  E_T^{2,\sigma},
  C_TH^{s-1}\cap L_T^1B_{2,\infty}^{\sigma-1}
  \right)\right).
\end{equation}
If $(a;b,c)=(a;b,b)\in\mathfrak L_{\mathrm{diag}}$, then the deterministic
operators $\mathcal D^{a;b}_{\Lambda}$ also converge in
\begin{equation}\label{eq:diagonal-operator-topology}
  \mathcal L\!\left(
  E_T^{2,\sigma},
  C_TH^{s-1}\cap L_T^\infty B_{2,\infty}^{\sigma-1}
  \right).
\end{equation}
The limits are independent of the admissible cutoff profile.  They are causal,
their shorter-time restrictions are intrinsic in the sense of
Lemma~\ref{lem:time-extension-causality}, and
\eqref{eq:operator-decomposition} is preserved by restriction.  On one
event of probability one,
\begin{equation}\label{eq:operator-uniform-interval-bounds}
 \sup_{I\subset[0,T_0]}
 \|T_I^{a;b,c}\|_{E_{I,0}^{2,\sigma}\to C_IH^{s-1}}<\infty.
\end{equation}
For a diagonal label one has, on the same event,
\begin{equation}\label{eq:operator-uniform-diagonal-interval-bound}
 \sup_{I\subset[0,T_0]}
 \|\mathcal D_I^{a;b}\|_{E_{I,0}^{2,\sigma}\to
 L^\infty(I;B_{2,\infty}^{\sigma-1})}<\infty.
\end{equation}
For a fixed-profile dyadic family, the full sequence converges almost surely
in the same operator topologies.  Moreover, almost surely,
\begin{equation}\label{eq:operator-uniform-interval-smallness}
 \sup_{\substack{I\subset[0,T_0]\text{ interval}\\ |I|\le h}}
 \|\mathcal B_I^{a;b,c}\|_{\mathcal L(
 E_{I,0}^{2,\sigma},L^1(I;B_{2,\infty}^{\sigma-1}))}
 \longrightarrow0
 \qquad(h\downarrow0),
\end{equation}
and, in particular,
\begin{equation}\label{eq:operator-local-smallness}
  \|\mathcal B^{a;b,c}_{[0,T]}\|_{\mathcal L(
  E_T^{2,\sigma},L_T^1B_{2,\infty}^{\sigma-1})}
  \longrightarrow0
  \qquad(T\downarrow0).
\end{equation}
Along a fixed-profile dyadic family, the same probability-one event gives
both limits uniformly over all dyadic cutoff levels and their pairwise
differences.  Since the block kernels are realized in
$C(\Delta_{T_0};\mathcal K)$ before Volterraization, this event works
simultaneously for every $0<T\le T_0$.  The estimates are uniform on separated speed--mass classes, on compact subsets of the admissible $(\betaW,\betaK)$ region, and for bounded $\mathfrak H_{\mathrm{prof}}$.
\end{theorem}

\begin{proof}
The finite decomposition \eqref{eq:operator-decomposition} is
Proposition~\ref{prop:finite-color-split}.  For a diagonal label,
$a=b^\perp$ and $c=b$.  The phase gap in
Lemma~\ref{lem:explicit-phase-gap}, followed by
Lemma~\ref{lem:dyadic-volterra} and
Propositions~\ref{prop:summed-diagonal} and
\ref{prop:localized-diagonal}, gives
\eqref{eq:diagonal-operator-topology} and
\eqref{eq:operator-uniform-diagonal-interval-bound}.

The centered component is controlled by
Theorem~\ref{thm:localized-centered} and Corollary~\ref{cor:full-centered}.
Since $E_T^{2,\sigma}$ embeds continuously into
$C_TL^2\cap L_T^\infty B_{2,\infty}^{\sigma}$, these results give
\eqref{eq:centered-full-operator-topology}.  The finite-cutoff Volterra
formulae are causal, and causality passes to the operator-norm limits.
Lemma~\ref{lem:time-extension-causality} therefore makes every restriction
intrinsic.  Repeating the estimates on translated intervals gives
\eqref{eq:operator-uniform-interval-bounds}; in the diagonal case it also gives
\eqref{eq:operator-uniform-diagonal-interval-bound}.  The interval and initial
smallness assertions
\eqref{eq:operator-uniform-interval-smallness}--\eqref{eq:operator-local-smallness}
follow from Corollary~\ref{cor:full-centered}.  The same summable block
estimates applied to cutoff differences give cutoff independence and
fixed-profile almost-sure convergence.  Uniformity follows from
Proposition~\ref{prop:uniform-phase-class}, the channelwise synthesis bounds
in Lemma~\ref{lem:covariance-synthesis}, and the bound for $\mathfrak H_{\mathrm{prof}}$.
\end{proof}

\section{Deterministic closure, approximation, and localization}\label{sec:deterministic-closure}

Throughout this section an adapted localization pair $(\chi,\rho)$ is fixed.  All
stochastic distribution norms are taken after multiplication by the indicated
compact cutoff, and the operators $T^{a;b,c}$ are the localized operators
\eqref{eq:T-def}.  The cubic coordinate $\Gamma_a^{\chi,\rho}$ is the cutoff limit defined in
\eqref{eq:localized-cubic-limit}; we retain the localization superscript
throughout this section.  The
tuple $\Xi$ is treated as deterministic input in this section; probability
enters only when the stochastic terms and operators constructed above are
substituted.  Constants may depend on finitely many seminorms of $(\chi,\rho)$.

\subsection{A priori estimates with prescribed stochastic terms}

The scalar wave and Klein--Gordon flows on $\R^3$ satisfy the standard conic estimates
\begin{align}
  \norm{I_aF}_{X_T^{s_1}}
  &\lesssim\norm{F}_{L_T^1H^{s_1-1}},
  \label{eq:X-linear}\\
  \norm{I_aF}_{Y_T^{s_2}}
  &\lesssim\norm{F}_{L_T^1H^{s_2-1}},
  \label{eq:Y-linear-rough}\\
  \norm{I_aF}_{Y_T^{s_2}}
  &\lesssim
  \norm{\la D\ra^{s_2-1/2}F}_{L_{t,x}^{4/3}}.
  \label{eq:Y-linear-dual}
\end{align}
The same estimates hold for both parameterized phases, with constants locally uniform in the parameter vector.  Elementary dyadic energy propagation gives
\begin{equation}\label{eq:besov-propagation}
  \norm{I_aF}_{L_T^\infty B_{2,\infty}^{\sigma}}
  +\norm{\partial_tI_aF}_{L_T^\infty B_{2,\infty}^{\sigma-1}}
  \lesssim\norm{F}_{L_T^1B_{2,\infty}^{\sigma-1}}.
\end{equation}

\begin{lemma}[Uniform annular dispersive estimate]
\label{lem:uniform-annular-dispersive}
Let
\[
 \mathcal E_a(t):=e^{\ii t\omega_a(D)},
 \qquad a\in\mathfrak C.
\]
For every compact positive parameter class and every dyadic $N\ge2$,
\begin{equation}\label{eq:uniform-annular-dispersive}
 \|\mathcal E_a(t)P_N\|_{L^1_x\to L^\infty_x}
 \lesssim N^3(1+N|t|)^{-1},
 \qquad t\in\R,
\end{equation}
with a constant independent of $N$, $t$, and $a$.  The speed gap is not used.
\end{lemma}

\begin{proof}
Let $\varphi$ be the fixed radial annular cutoff defining $P_N$.  After the
change $\xi=N\eta$, the oscillatory kernel is
\[
 N^3\int_{\R^3}e^{\ii N(x\cdot\eta+t\phi_{a,N}(|\eta|))}
 \varphi(\eta)\dd\eta,
\]
where
\[
 \phi_{\W,N}(r)=c_{\W}r,
 \qquad
 \phi_{\K,N}(r)=\sqrt{(m/N)^2+c_{\K}^2r^2}.
\]
On the support of $\varphi$ there are constants
$0<v_-<v_+<\infty$ such that, uniformly in the parameter class and
$N\ge2$,
\begin{equation}\label{eq:uniform-rescaled-phase-symbol}
 v_-\le\partial_r\phi_{a,N}(r)\le v_+,
 \qquad
 |\partial_r^j\phi_{a,N}(r)|\le C_j,
 \quad j\ge2.
\end{equation}
If $N|t|\le1$, the integral is bounded by $CN^3$.  Suppose $N|t|>1$.
If $|x|\le v_-|t|/2$, then for every angular variable
$\vartheta\in\mathbb S^2$,
\[
 |\partial_r(r x\cdot\vartheta+t\phi_{a,N}(r))|
 \ge v_-|t|/2.
\]
One integration by parts in $r$ therefore gives a factor $(N|t|)^{-1}$.
If $|x|>v_-|t|/2$, radiality of the cutoff and the identity
\[
 \int_{\mathbb S^2}e^{\ii Nr x\cdot\vartheta}\dd\vartheta
 =4\pi\frac{\sin(Nr|x|)}{Nr|x|}
\]
give a factor $(1+N|x|)^{-1}\lesssim(N|t|)^{-1}$ before the bounded radial
integration.  The symbol bounds in
\eqref{eq:uniform-rescaled-phase-symbol} make both estimates uniform.
Multiplication by $N^3$ proves \eqref{eq:uniform-annular-dispersive}.
\end{proof}

\begin{proposition}[Uniform scalar dispersive and Strichartz estimates]
\label{prop:parameter-uniform-strichartz}
The constants in \eqref{eq:X-linear}--\eqref{eq:besov-propagation} are uniform
for $0<T\le1$ and $\parvec$ in every separated parameter class
\eqref{eq:separated-parameter-class}.  The speed-separation lower bound is not
used in these estimates.
\end{proposition}

\begin{proof}
By Lemma~\ref{lem:uniform-annular-dispersive}, unitarity on $L^2$, and the
Keel--Tao theorem~\cite{KeelTao}, for $N\ge2$ one has uniformly in the channel
\begin{align}
 \|\mathcal E_a(t)P_Nf\|_{L_t^8L_x^{8/3}}
 &\lesssim N^{1/4}\|P_Nf\|_2,
 \label{eq:uniform-annular-8-83}\\
 \|\mathcal E_a(t)P_Nf\|_{L_t^4L_x^4}
 &\lesssim N^{1/2}\|P_Nf\|_2.
 \label{eq:uniform-annular-4-4}
\end{align}
The retarded inhomogeneous estimate gives
\begin{equation}\label{eq:uniform-annular-dual-4-4}
 \left\|\int_0^t\mathcal E_a(t-s)P_NF(s)\dd s\right\|_{L_t^4L_x^4}
 \lesssim N\|P_NF\|_{L_{t,x}^{4/3}}.
\end{equation}
The sine propagator is a finite linear combination of half waves multiplied
by $\omega_a(D)^{-1}$; on the $N$-shell this contributes $N^{-1}$.  Hence
Minkowski's inequality and \eqref{eq:uniform-annular-8-83} give
\[
 N^{s_1-1/4}\|P_NI_aF\|_{L_t^8L_x^{8/3}}
 \lesssim N^{s_1-1}\|P_NF\|_{L_t^1L_x^2},
\]
while \eqref{eq:uniform-annular-4-4} gives
\[
 N^{s_2-1/2}\|P_NI_aF\|_{L_t^4L_x^4}
 \lesssim N^{s_2-1}\|P_NF\|_{L_t^1L_x^2}.
\]
In \eqref{eq:uniform-annular-dual-4-4}, the factor $N$ is canceled by
$\omega_a(D)^{-1}$, which yields \eqref{eq:Y-linear-dual} after applying
$\langle D\rangle^{s_2-1/2}$.  The energy components in
\eqref{eq:X-linear} and \eqref{eq:Y-linear-rough} follow from the uniform
single-channel energy estimate.

The inhomogeneous unit block is controlled by energy and Bernstein; for the
wave kernel, $\sin(tc_{\W}|\xi|)/(c_{\W}|\xi|)$ is interpreted by its
continuous value at $\xi=0$.  Finally, the shellwise energy bound for the
Duhamel propagator gives \eqref{eq:besov-propagation}.  Summation of the
annular estimates proves \eqref{eq:X-linear}--\eqref{eq:besov-propagation}.
\end{proof}

\begin{definition}\label{def:Xi-assumptions}
Let Assumption~\ref{ass:exponents} hold.  On $[0,T_0]$, let
\begin{equation}\label{eq:Xi-tuple}
 \Xi:=\left(
  (\Psi_a,V_a,\partial_tV_a,\Gamma_a^{\chi,\rho})_{a\in\mathfrak C},
  (\mathcal D^{a;b})_{(a;b,b)\in\mathfrak L_{\mathrm{diag}}},
  (\mathcal B^{a;b,c})_{(a;b,c)\in\mathfrak L}
 \right)
\end{equation}
and assume the following conditions.
\begin{enumerate}[label=\textup{(\roman*)},leftmargin=2.3em]
\item For every $a\in\mathfrak C$, the global representatives satisfy
\[
 \Psi_a,V_a,\partial_tV_a\in
 C_{T_0}\mathfrak T_{J_{\mathrm{ps}}},
 \qquad
 \rho\Psi_a\in C_{T_0}\mathcal C^{-1/2-\beta_a-\kappa}.
\]
The weighted norm is used only for pseudolocal smoothing remainders; all
singular estimates use compactly localized norms.
\item For every $a\in\mathfrak C$,
\[
 \rho V_a\in C_{T_0}\mathcal C^{1/2-\betasum-\kappa}
 \cap L_{T_0}^\infty B_{2,\infty}^{1/2-\betasum-\kappa},
\]
\[
 \rho\partial_tV_a\in C_{T_0}\mathcal C^{-1/2-\betasum-\kappa}
 \cap L_{T_0}^\infty B_{2,\infty}^{-1/2-\betasum-\kappa}.
\]
\item For every $a\in\mathfrak C$,
\[
 \Gamma_a^{\chi,\rho}\in C_{T_0}\mathcal C^{-\beta_{\Gamma,a}-\kappa}
 \cap L_{T_0}^\infty B_{2,\infty}^{-\beta_{\Gamma,a}-\kappa}.
\]
\item For $(a;b,b)\in\mathfrak L_{\mathrm{diag}}$,
\[
 \mathcal D^{a;b}\in\mathcal L\bigl(
 E_{T_0}^{2,\sigma},L_{T_0}^\infty B_{2,\infty}^{\sigma-1}\bigr),
\]
and for every $(a;b,c)\in\mathfrak L$,
\[
 \mathcal B^{a;b,c}\in\mathcal L\bigl(
 E_{T_0}^{2,\sigma},L_{T_0}^1B_{2,\infty}^{\sigma-1}\bigr).
\]
Define
\begin{equation}\label{eq:stochastic-operator}
 T^{a;b,c}:=
 \begin{cases}
  \mathcal D^{a;b}+\mathcal B^{a;b,b},
    &(a;b,c)=(a;b,b)\in\mathfrak L_{\mathrm{diag}},\\
  \mathcal B^{a;b,c},
    &(a;b,c)\in\mathfrak L_{\mathrm{off}}.
 \end{cases}
\end{equation}
Then
\[
 T^{a;b,c}\in\mathcal L(E_{T_0}^{2,\sigma},C_{T_0}H^{s_2-1})
 \qquad ((a;b,c)\in\mathfrak L).
\]
All these operators are causal, their restrictions are understood through
Lemma~\ref{lem:time-extension-causality}, and
\begin{align}\label{eq:uniform-interval-operator-bounds}
 &\sup_{I\subset[0,T_0]}
 \Bigg[
 \sum_{(a;b,c)\in\mathfrak L}
 \|T_I^{a;b,c}\|_{E_{I,0}^{2,\sigma}\to C_IH^{s_2-1}}\notag\\
 &\hspace{35mm}+
 \sum_{(a;b,b)\in\mathfrak L_{\mathrm{diag}}}
 \|\mathcal D_I^{a;b}\|_{E_{I,0}^{2,\sigma}\to
 L^\infty(I;B_{2,\infty}^{\sigma-1})}
 \Bigg]<\infty.
\end{align}
\item Define
\begin{equation}\label{eq:centered-interval-modulus}
 \omega_{\mathrm{cen},\Xi}(h):=
 \sum_{(a;b,c)\in\mathfrak L}
 \sup_{\substack{I\subset[0,T_0]\text{ interval}\\ |I|\le h}}
 \|\mathcal B_I^{a;b,c}\|_{\mathcal L(
 E_{I,0}^{2,\sigma},L^1(I;B_{2,\infty}^{\sigma-1}))}.
\end{equation}
Then
\begin{equation}\label{eq:centered-operator-smallness}
 \omega_{\mathrm{cen},\Xi}(h)\longrightarrow0
 \qquad(h\downarrow0),
\end{equation}
and, separately,
\begin{equation}\label{eq:centered-operator-initial-smallness}
 \sum_{(a;b,c)\in\mathfrak L}
 \|\mathcal B^{a;b,c}_{[0,T]}\|_{\mathcal L(
 E_T^{2,\sigma},L_T^1B_{2,\infty}^{\sigma-1})}
 \longrightarrow0
 \qquad(T\downarrow0).
\end{equation}
The first limit is used on translated zero-history intervals; the second is
used for the fixed-point construction from prescribed Cauchy data.
\end{enumerate}
\end{definition}

For a tuple $\Xi$ satisfying Definition~\ref{def:Xi-assumptions}, set
\begin{align}
 M_T^{\mathrm{ps}}(\Xi)
 &:={\sum}_{a\in\mathfrak C}
 \bigl(\|\Psi_a\|_{C_T\mathfrak T_{J_{\mathrm{ps}}}}
 +\|V_a\|_{C_T\mathfrak T_{J_{\mathrm{ps}}}}
 +\|\partial_tV_a\|_{C_T\mathfrak T_{J_{\mathrm{ps}}}}\bigr),
 \label{eq:M-pseudolocal}\\
 M_T^{\mathrm{op},H}(\Xi)
 &:={\sum}_{(a;b,c)\in\mathfrak L}
 \|T^{a;b,c}\|_{E_T^{2,\sigma}\to C_TH^{s_2-1}},
 \label{eq:M-op-H}\\
 M_{T,\mathrm{diag}}^{\mathrm{op},B}(\Xi)
 &:={\sum}_{(a;b,b)\in\mathfrak L_{\mathrm{diag}}}
 \|\mathcal D^{a;b}\|_{E_T^{2,\sigma}\to
 L_T^\infty B_{2,\infty}^{\sigma-1}},
 \label{eq:M-op-diag-B}\\
 M_{T,\mathrm{cen}}^{\mathrm{op},B}(\Xi)
 &:={\sum}_{(a;b,c)\in\mathfrak L}
 \|\mathcal B^{a;b,c}\|_{E_T^{2,\sigma}\to
 L_T^1B_{2,\infty}^{\sigma-1}}.
 \label{eq:M-op-cen-B}
\end{align}
We further define
\begin{equation}\label{eq:Xi-size}
\begin{aligned}
 M_T(\Xi):={}&
 \sum_{a\in\mathfrak C}\Bigl(
 \|\rho\Psi_a\|_{C_T\mathcal C^{-1/2-\beta_a-\kappa}}
 +\|\rho V_a\|_{C_T\mathcal C^{1/2-\betasum-\kappa}}\\
 &\qquad
 +\|\rho V_a\|_{L_T^\infty B_{2,\infty}^{1/2-\betasum-\kappa}}
 +\|\rho\partial_tV_a\|_{C_T\mathcal C^{-1/2-\betasum-\kappa}}\\
 &\qquad
 +\|\rho\partial_tV_a\|_{L_T^\infty B_{2,\infty}^{-1/2-\betasum-\kappa}}
 +\|\Gamma_a^{\chi,\rho}\|_{C_T\mathcal C^{-\beta_{\Gamma,a}-\kappa}}\\
 &\qquad
 +\|\Gamma_a^{\chi,\rho}\|_{L_T^\infty B_{2,\infty}^{-\beta_{\Gamma,a}-\kappa}}
 \Bigr)\\
 &+M_T^{\mathrm{ps}}(\Xi)+M_T^{\mathrm{op},H}(\Xi)\\
 &+M_{T,\mathrm{diag}}^{\mathrm{op},B}(\Xi)
 +M_{T,\mathrm{cen}}^{\mathrm{op},B}(\Xi).
\end{aligned}
\end{equation}
For deterministic data $y$, let
\begin{equation}\label{eq:bounded-size}
 M_T^{\mathrm{bd}}(\Xi,y)
 :=1+M_T(\Xi)+\|y\|_{\mathcal H^{s_2,\sigma}},
\end{equation}
where $\|y\|_{\mathcal H^{s_2,\sigma}}$ is defined in
\eqref{eq:data-norm}.  The term $M_T^{\mathrm{ps}}$ records the pseudolocal smoothing
remainders and their cutoff differences in the deterministic size and
stability norm.

\begin{definition}[Quantities used in the short-time estimate]\label{def:local-smallness-modulus}
Let $0<T\le T_0$.  Define
\begin{align}
 \mu_T^{\mathrm{op}}(\Xi)
 &:={}T M_T^{\mathrm{op},H}(\Xi)
 +T M_{T,\mathrm{diag}}^{\mathrm{op},B}(\Xi)
 +M_{T,\mathrm{cen}}^{\mathrm{op},B}(\Xi),
 \label{eq:mu-op}\\
 \mu_T(\Xi,y)
 &:={}\mu_T^{\mathrm{op}}(\Xi)
 +T\sum_{a\in\mathfrak C}\left(
 \|\Gamma_a^{\chi,\rho}\|_{C_T\mathcal C^{-\beta_{\Gamma,a}-\kappa}}
 +\|\Gamma_a^{\chi,\rho}\|_{L_T^\infty B_{2,\infty}^{-\beta_{\Gamma,a}-\kappa}}
 \right)\notag\\
 &\quad+T^\delta\left(1+M_{T_0}^{\mathrm{bd}}(\Xi,y)
 +(M_{T_0}^{\mathrm{bd}}(\Xi,y))^2\right),
 \label{eq:mu-full}
\end{align}
where $\delta>0$ is chosen smaller than the time exponents in
Lemmas~\ref{lem:rough-products} and~\ref{lem:deterministic-quadratic}.
\end{definition}

\begin{definition}[Uniform short-time bounds]\label{def:localized-smallness-class}
A family $\mathfrak K$ of pairs $(\Xi,y)$ on $[0,T_0]$ satisfies the
uniform bound $A$ if
\[
 \sup_{(\Xi,y)\in\mathfrak K}M_{T_0}^{\mathrm{bd}}(\Xi,y)\le A
\]
and
\begin{equation}\label{eq:family-smallness}
 \nu_{\mathfrak K}(T):=
 \sup_{(\Xi,y)\in\mathfrak K}\mu_T(\Xi|_{[0,T]},y)
 \longrightarrow0
 \qquad\text{as }T\downarrow0.
\end{equation}
\end{definition}

\begin{lemma}[Completeness, separability, and the common operator input]
\label{lem:complete-and-E}
The spaces $\widetilde X_T^{s_1,\sigma}$,
$\widetilde Y_T^{s_2,\sigma}$, and $\mathcal Z_T$ are separable Banach spaces.
Moreover,
if $V_a$ is the first Picard term contained in $\Xi$ and
$X_a\in\widetilde X_T^{s_1,\sigma}$,
$Y_a\in\widetilde Y_T^{s_2,\sigma}$, then
\begin{equation}\label{eq:regular-input-in-E}
 \mathscr U_a^\rho:=\rho V_a+X_a+Y_a\in E_T^{2,\sigma},
 \qquad
 \|\mathscr U_a^\rho\|_{E_T^{2,\sigma}}
 \lesssim M_T(\Xi)+\|X_a\|_{\widetilde X_T}
 +\|Y_a\|_{\widetilde Y_T}.
\end{equation}
\end{lemma}

\begin{proof}
Completeness follows by intersecting finitely many Banach spaces with the sum
norm, using the little-Besov convention following \eqref{eq:E-space}.
Separability follows from the same completion convention.  At a fixed finite
set of spatial shells, all energy and Strichartz components take values in
separable Sobolev or Lebesgue spaces, and the corresponding continuous- or
finite-$p$ time path spaces are separable.  On those finite shells the two
$L_T^\infty B_{2,\infty}$ seminorms are equivalent to finite maxima of the
$C_TL^2$ block norms already supplied by the energy components, so no
nonseparable ambient $L^\infty$ space is introduced.  Taking rational linear
combinations from countable dense subsets and then the countable union over
finite shell sets gives a countable dense family in each of these completed spaces.  Finite products preserve separability.

The support of $\rho V_a(t)$ and $\rho\partial_tV_a(t)$ is contained in the
fixed compact set $\supp\rho$.  Apply
Lemma~\ref{lem:compact-support-continuity} first with
\[
 (\alpha,\beta,r)=
 (1/2-\betasum-\kappa,1/2-\betasum-\kappa,0)
\]
and then with
\[
 (\alpha,\beta,r)=
 (-1/2-\betasum-\kappa,-1/2-\betasum-\kappa,-1).
\]
The inequalities $0<1/2-\betasum-\kappa$ and
$-1<-1/2-\betasum-\kappa$ give
\[
 \rho V_a\in C_TL^2,
 \qquad
 \rho\partial_tV_a\in C_TH^{-1}.
\]
Moreover, $\sigma<1/2-\betasum-\kappa$ implies the continuous little-Besov
embeddings
\[
 B_{2,\infty}^{1/2-\betasum-\kappa}\hookrightarrow B_{2,\infty}^{\sigma},
 \qquad
 B_{2,\infty}^{-1/2-\betasum-\kappa}\hookrightarrow
 B_{2,\infty}^{\sigma-1}.
\]
By the completion convention following \eqref{eq:E-space}, the paths
$\rho V_a$ and $\rho\partial_tV_a$ admit spatially smooth finite-block
approximants in their displayed intersection norms.  Convergence in the
Besov component passes the uniform tail condition
\eqref{eq:uniform-little-besov-tail} to the limit, and lowering the Besov
exponent preserves that condition.  For the stochastic objects constructed below, the localized finite-cutoff paths are smooth and compactly supported but
need not have finite Fourier support.  An additional output
Littlewood--Paley truncation produces the required finite-block approximants,
and the uniform high-frequency tail removes this auxiliary truncation.  The
paths $X_a$ and $Y_a$ satisfy the same condition by the completion convention;
their $C_TL^2$ and $C_TH^{-1}$ components follow from the Sobolev energy
norms.  Hence the sum belongs to the completed space $E_T^{2,\sigma}$.  Summing the bounds proves \eqref{eq:regular-input-in-E}.
\end{proof}

\begin{lemma}[Measurable contraction selection]
\label{lem:measurable-contraction-selection}
Let $(\Omega,\mathcal F)$ be a measurable space, let $B$ be a separable Banach
space, let $G\in\mathcal F$, and let $R>0$.  Suppose
\[
 \Phi:\Omega\times\overline B_R\longrightarrow B
\]
is jointly Borel measurable and that there is a deterministic $0\le\lambda<1$
such that, for every $\omega\in G$, the map $z\mapsto\Phi(\omega,z)$ sends
$\overline B_R$ into itself and is $\lambda$-Lipschitz.  Then the unique fixed
point $z_*(\omega)$ on $G$, extended by $0$ on $G^c$, is strongly measurable
as a $B$-valued map.  On $G$ it is the pointwise and norm limit of the measurable Picard iterates
\[
 z_0=0,
 \qquad
 z_{m+1}(\omega)=
 \mathbf1_G(\omega)\Phi(\omega,z_m(\omega)).
\]
\end{lemma}

\begin{proof}
Extend the iteration by setting $z_m=0$ on $G^c$ at every stage.  Joint Borel
measurability of $\Phi$ and measurability of $G$ show inductively that every
$z_m$ is Borel measurable.  On $G$, the contraction estimate gives
\[
 \|z_{m+1}-z_m\|_B
 \le \lambda^m\|z_1-z_0\|_B
 \le 2R\lambda^m,
\]
so the iterates converge in $B$ to the unique fixed point.  A pointwise norm
limit of Borel maps into a separable Banach space is strongly measurable.
\end{proof}

\begin{corollary}[Measurability of the paracontrolled fixed point]
\label{cor:measurable-paracontrolled-fixed-point}
Fix $T$ and a deterministic radius $R$.  Let $G$ be a Borel event on which the
random map $\mathscr P_{\Xi(\omega),y}$ is a self-map and a contraction of
the closed radius-$R$ ball in $\mathcal Z_T$, with a deterministic contraction
constant strictly smaller than one.  Choose the strongly measurable realizations of the stochastic paths and
operators constructed above.  Then the
fixed point on $G$, extended by $0$ on $G^c$, is a strongly measurable
$\mathcal Z_T$-valued random variable.
\end{corollary}

\begin{proof}
By Lemma~\ref{lem:complete-and-E}, $\mathcal Z_T$ is separable.  For each fixed
$Z$, every term in $\mathscr P_{\Xi(\omega),y}Z$ is obtained by evaluating a
strongly measurable stochastic path or operator component and then applying
a deterministic continuous multilinear map; it is therefore strongly
measurable in $\omega$.  For each fixed $\omega$, the source formulas show that
the map is continuous in $Z$.  The Carath\'eodory measurability criterion on a
separable metric domain therefore makes
$(\omega,Z)\mapsto\mathscr P_{\Xi(\omega),y}Z$ jointly Borel measurable.
Lemma~\ref{lem:measurable-contraction-selection} applies.
\end{proof}

\begin{lemma}[Products with a negative H\"older coefficient]
\label{lem:negative-holder-products}
Let $0<\alpha<1$, $f\in\mathcal C^{-\alpha}(\R^3)$, and
$\varepsilon>0$.  For smooth finite-block functions the following estimates
hold, and they extend to the corresponding completed spaces.
\begin{enumerate}[label=\textup{(\roman*)},leftmargin=2.3em]
\item If $g\in L^2$, then
\begin{equation}\label{eq:negative-holder-low-high}
 \|g\prec f\|_{B_{2,\infty}^{-\alpha}}
 +\|g\prec f\|_{H^{-\alpha-\varepsilon}}
 \lesssim_\varepsilon
 \|g\|_{L^2}\|f\|_{\mathcal C^{-\alpha}}.
\end{equation}
\item If $g\in H^s\cap B_{2,\infty}^{\tau}$, then
\begin{align}
 \|f\prec g\|_{H^{s-\alpha}}
 &\lesssim \|f\|_{\mathcal C^{-\alpha}}\|g\|_{H^s},
 \label{eq:negative-holder-high-low-sobolev}\\
 \|f\prec g\|_{B_{2,\infty}^{\tau-\alpha}}
 &\lesssim \|f\|_{\mathcal C^{-\alpha}}
              \|g\|_{B_{2,\infty}^{\tau}}.
 \label{eq:negative-holder-high-low-besov}
\end{align}
If only the Besov norm of $g$ is used, then
\begin{equation}\label{eq:negative-holder-high-low-besov-to-sobolev}
 \|f\prec g\|_{H^{\tau-\alpha-\varepsilon}}
 \lesssim_\varepsilon
 \|f\|_{\mathcal C^{-\alpha}}
 \|g\|_{B_{2,\infty}^{\tau}}.
\end{equation}
\item If $s>\alpha$ and $g\in H^s$, then the resonant product is defined and
\begin{equation}\label{eq:negative-holder-resonant}
 \|g\circ f\|_{B_{2,\infty}^{s-\alpha}}
 +\|g\circ f\|_{H^{s-\alpha-\varepsilon}}
 \lesssim_\varepsilon
 \|g\|_{H^s}\|f\|_{\mathcal C^{-\alpha}}.
\end{equation}
\end{enumerate}
Multiplication of the left-hand sides by a fixed function in
$C_c^\infty(\R^3)$ preserves all estimates.
\end{lemma}

\begin{proof}
The fixed paraproduct aperture changes only finite-overlap constants.  Since
$f\in\mathcal C^{-\alpha}$,
\[
 \|P_Nf\|_{L^\infty}\lesssim
 N^\alpha\|f\|_{\mathcal C^{-\alpha}},
 \qquad
 \|S_{<\mathrm{cap}N}f\|_{L^\infty}
 \lesssim N^\alpha\|f\|_{\mathcal C^{-\alpha}}.
\]
For the low--high term, output frequency $M$ is comparable to the high
frequency $N$, and therefore
\[
 \|P_M(g\prec f)\|_{L^2}
 \lesssim M^\alpha\|g\|_{L^2}
          \|f\|_{\mathcal C^{-\alpha}}.
\]
Taking the $B_{2,\infty}^{-\alpha}$ supremum and then summing
$M^{-2\varepsilon}$ proves \eqref{eq:negative-holder-low-high}.

For the high--low term,
\[
 \|P_M(f\prec g)\|_{L^2}
 \lesssim M^\alpha\|\widetilde P_Mg\|_{L^2}
          \|f\|_{\mathcal C^{-\alpha}},
\]
where $\widetilde P_M$ is a fixed enlargement of $P_M$.  Square summation
with the weight $M^{2(s-\alpha)}$ proves
\eqref{eq:negative-holder-high-low-sobolev}; taking a dyadic supremum proves
\eqref{eq:negative-holder-high-low-besov}.  The embedding
$B_{2,\infty}^{r}\hookrightarrow H^{r-\varepsilon}$ gives
\eqref{eq:negative-holder-high-low-besov-to-sobolev}.

For the resonant term, the output shell $M$ receives contributions from
comparable input shells $N$ with $N\gtrsim M$.  Hence
\[
 \|P_M(g\circ f)\|_{L^2}
 \lesssim \|f\|_{\mathcal C^{-\alpha}}
 \sum_{N\gtrsim M}N^\alpha\|P_Ng\|_{L^2}.
\]
With $a_N=N^s\|P_Ng\|_{L^2}$ and $s>\alpha$,
\[
 M^{s-\alpha}\|P_M(g\circ f)\|_{L^2}
 \lesssim \|f\|_{\mathcal C^{-\alpha}}
 \sum_{N\gtrsim M}\left(\frac{M}{N}\right)^{s-\alpha}a_N
 \lesssim \|f\|_{\mathcal C^{-\alpha}}\|g\|_{H^s},
\]
by Cauchy--Schwarz and a geometric series.  This proves the Besov part of
\eqref{eq:negative-holder-resonant}; the Sobolev part follows from the same
Besov-to-Sobolev embedding.  Smooth multiplication is bounded on all spaces
in the statement.
\end{proof}

We collect the deterministic estimates used in the fixed point.  Every
source below is multiplied by $\chi$.  Singular products use the compactly
supported coordinates $\rho\Psi$ and $\rho V$; the separated-support
remainders are controlled by Lemma~\ref{lem:separated-support-smoothing} and
recorded in $M_T^{\mathrm{ps}}$.

\begin{lemma}[Localized rough stochastic products]
\label{lem:rough-products}
Assume \eqref{eq:parameter-window}.  Let
$U=\rho V+X+Y$ with $X\in\wt X_T^{s_1,\sigma}$,
$Y\in\wt Y_T^{s_2,\sigma}$, and $V$ as in
\cref{def:Xi-assumptions}.  For $a\in\mathfrak C$, put
\[
 \mathfrak Y_a(Y):=Y\succ\Psi_a+Y\circ\Psi_a.
\]
Then
\begin{align}
  \norm{\chi(U\prec\Psi_a)}_{L_T^1H^{s_1-1}}
  +\norm{\chi(U\prec\Psi_a)}_{L_T^1B_{2,\infty}^{\sigma-1}}
  &\lesssim T^\delta M_T(\Xi)\notag\\
  &\quad\times
  \paren{1+\norm{X}_{\widetilde X_T^{s_1,\sigma}}
  +\norm{Y}_{\widetilde Y_T^{s_2,\sigma}}},
  \label{eq:lowhigh-estimate}\\
  \norm{\chi((\rho V)\succ\Psi_a)}_{L_T^1H^{s_2-1}}
  +\norm{\chi((\rho V)\succ\Psi_a)}_{L_T^1B_{2,\infty}^{\sigma-1}}
  &\lesssim T M_T(\Xi)^2,
  \label{eq:V-highlow}\\
  \norm{\chi(X\succ\Psi_a)}_{L_T^1H^{s_2-1}}
  +\norm{\chi(X\succ\Psi_a)}_{L_T^1B_{2,\infty}^{\sigma-1}}
  &\lesssim T^\delta M_T(\Xi)\norm{X}_{\widetilde X_T^{s_1,\sigma}},
  \label{eq:X-highlow}\\
  \norm{\chi\mathfrak Y_a(Y)}_{L_T^1H^{s_2-1}}
  +\norm{\chi\mathfrak Y_a(Y)}_{L_T^1B_{2,\infty}^{\sigma-1}}
  &\lesssim T M_T(\Xi)
  \norm{Y}_{\widetilde Y_T^{s_2,\sigma}}.
  \label{eq:Y-random-products}
\end{align}
The estimates have the corresponding Lipschitz form, including differences of
the weighted pseudolocal terms.
\end{lemma}

\begin{proof}
Set
\[
 \alpha_a:=\frac12+\beta_a+\kappa,
 \qquad
 \alpha_V:=\frac12-\betasum-\kappa.
\]
The parameter window gives $0<\alpha_a<1$, $\alpha_V>0$, and
$s_2>\alpha_a$.  Split
\[
 \Psi_a=\rho\Psi_a+(1-\rho)\Psi_a.
\]
Whenever $(1-\rho)\Psi_a$ occurs, the supports of this factor and $\chi$ are
separated.  Lemma~\ref{lem:separated-support-smoothing} then gives an
arbitrarily smooth output after multiplication by $\chi$.  By
\eqref{eq:H-to-weighted-tempered} and
\cref{lem:complete-and-E},
\[
 \|U\|_{C_T\mathfrak T_{J_{\mathrm{ps}}}}
 +\|U\|_{C_TL^2}
 \lesssim M_T(\Xi)+\|X\|_{\wt X_T^{s_1,\sigma}}
 +\|Y\|_{\wt Y_T^{s_2,\sigma}}.
\]
The separated-support terms are therefore bounded in every source space by
this quantity times $M_T^{\mathrm{ps}}(\Xi)$, with an explicit factor $T$
after time integration.  It remains to consider $\rho\Psi_a$.

For the low--high product, choose $\varepsilon_{\mathrm{lh}}>0$ so small that
\[
 s_1-1< -\alpha_a-\varepsilon_{\mathrm{lh}},
 \qquad
 \sigma-1< -\alpha_a.
\]
These inequalities follow respectively from
$\kappa<1/2-\betastar-s_1$ and
$\sigma<1/2-\betasum-\kappa$.  Applying
\eqref{eq:negative-holder-low-high} pointwise in time gives
\begin{align*}
 &\|\chi(U\prec\rho\Psi_a)\|_{H^{s_1-1}}
 +\|\chi(U\prec\rho\Psi_a)\|_{B_{2,\infty}^{\sigma-1}}\\
 &\qquad\lesssim
 \|\rho\Psi_a\|_{\mathcal C^{-\alpha_a}}\|U\|_{L^2}.
\end{align*}
Integration over $[0,T]$ produces a factor $T$.  Since
$0<T\le1$ and the auxiliary exponent satisfies $0<\delta<1$, one has
$T\le T^\delta$, proving \eqref{eq:lowhigh-estimate}.

For the first-Picard high--low term, use
\eqref{eq:negative-holder-high-low-besov}--
\eqref{eq:negative-holder-high-low-besov-to-sobolev} with
$f=\rho\Psi_a$, $g=\rho V$, and $\tau=\alpha_V$.  Thus, for every
$\varepsilon_V>0$,
\begin{align*}
 &\|\chi((\rho V)\succ\rho\Psi_a)\|_{H^{\alpha_V-\alpha_a-\varepsilon_V}}
 +\|\chi((\rho V)\succ\rho\Psi_a)\|_{B_{2,\infty}^{\alpha_V-\alpha_a}}\\
 &\qquad\lesssim_{\varepsilon_V}
 \|\rho\Psi_a\|_{\mathcal C^{-\alpha_a}}
 \|\rho V\|_{B_{2,\infty}^{\alpha_V}}.
\end{align*}
Here
\[
 \alpha_V-\alpha_a=-\beta_{\Gamma,a}-2\kappa.
\]
The strict inequality
$\beta_{\Gamma,a}+2\kappa<1-s_2$ permits a choice of
$\varepsilon_V>0$ such that
$s_2-1<\alpha_V-\alpha_a-\varepsilon_V$.  The inequality
$\sigma-1<\alpha_V-\alpha_a$ follows from
\eqref{eq:strict-margin}.  Time integration gives
\eqref{eq:V-highlow}.

For $X\succ\rho\Psi_a=(\rho\Psi_a)\prec X$,
\eqref{eq:negative-holder-high-low-sobolev}--
\eqref{eq:negative-holder-high-low-besov} yield
\begin{align*}
 \|\chi(X\succ\rho\Psi_a)\|_{H^{s_1-\alpha_a}}
 &\lesssim
 \|\rho\Psi_a\|_{\mathcal C^{-\alpha_a}}\|X\|_{H^{s_1}},\\
 \|\chi(X\succ\rho\Psi_a)\|_{B_{2,\infty}^{\sigma-\alpha_a}}
 &\lesssim
 \|\rho\Psi_a\|_{\mathcal C^{-\alpha_a}}
 \|X\|_{B_{2,\infty}^{\sigma}}.
\end{align*}
Since $s_2<s_1+1/4$ and $\beta_a+\kappa<1/4$,
\[
 s_2-1<s_1-\alpha_a,
\]
and since $\alpha_a<1$,
$\sigma-1<\sigma-\alpha_a$.  After integration in time this gives a factor
$T$, which is bounded by $T^\delta$, and proves
\eqref{eq:X-highlow}.

The same high--low estimate with $X$ replaced by $Y$ gives
\begin{align*}
 \|\chi(Y\succ\rho\Psi_a)\|_{H^{s_2-\alpha_a}}
 +\|\chi(Y\succ\rho\Psi_a)\|_{B_{2,\infty}^{\sigma-\alpha_a}}
 \lesssim M_T(\Xi)\|Y\|_{\wt Y_T^{s_2,\sigma}}.
\end{align*}
For the resonant term, $s_2>\alpha_a$ and
\eqref{eq:negative-holder-resonant} give, for
$0<\varepsilon_Y<1-\alpha_a$,
\begin{align*}
 \|\chi(Y\circ\rho\Psi_a)\|_{H^{s_2-\alpha_a-\varepsilon_Y}}
 +\|\chi(Y\circ\rho\Psi_a)\|_{B_{2,\infty}^{s_2-\alpha_a}}
 \lesssim_{\varepsilon_Y}
 M_T(\Xi)\|Y\|_{C_TH^{s_2}}.
\end{align*}
The first exponent is larger than $s_2-1$, and the second is positive while
$\sigma-1<0$.  Both terms therefore lie in the two $Y$-source spaces.
Time integration proves \eqref{eq:Y-random-products}.

All estimates are bilinear in the displayed factors.  Subtracting two data
sets one factor at a time, together with the difference estimate in
Lemma~\ref{lem:separated-support-smoothing}, gives the stated Lipschitz
bounds.
\end{proof}

\begin{lemma}[First Picard multiplier products]\label{lem:V-products}
Put $\alpha_V=1/2-\betasum-\kappa$ and let $V\in C_T\cC^{\alpha_V}$ be spatially localized.  If $G\in C_TH^\lambda\cap L_T^\infty B_{2,\infty}^{\sigma}$ with $\lambda\ge0$, $0<\sigma<\alpha_V$, and
\[
  r<\min\{\alpha_V,\lambda\},
\]
then
\begin{align}\label{eq:V-holder-multiplier}
  \norm{VG}_{C_TH^r}
  +\norm{VG}_{L_T^\infty B_{2,\infty}^{\sigma}}
  \lesssim
  \norm{V}_{C_T\cC^{\alpha_V}}
  \paren{\norm{G}_{C_TH^\lambda}+\norm{G}_{L_T^\infty B_{2,\infty}^{\sigma}}}.
\end{align}
The same conclusion holds when $G$ is another localized $\cC^{\alpha_V}$ function, with any $r<\alpha_V$.  Consequently every term in $\mathscr U_{\W}^\rho \mathscr U_{\K}^\rho$ containing at least one first Picard factor belongs to
\[
  L_T^1H^{s_2-1}\cap L_T^1B_{2,\infty}^{\sigma-1},
\]
with an explicit factor $T$ and the corresponding bilinear Lipschitz bounds.
\end{lemma}

\begin{proof}
Decompose $VG=V\prec G+V\succ G+V\circ G$.  Since $\alpha_V>0$, the
low--high term is bounded in $H^\lambda$ and in $B_{2,\infty}^{\sigma}$.  In the
high--low term, the high block of $V$ contributes $N^{-\alpha_V}$ in $L^\infty$ and
the low block of $G$ is controlled in $L^2$; the $H^r$ sum is finite for
$r<\alpha_V$.  The resonant term is controlled for $r<\alpha_V+\lambda$.  Since $\alpha_V,\lambda\ge0$, the
assumption $r<\min\{\alpha_V,\lambda\}$ implies all three requirements and gives the
Sobolev estimate.  The Besov estimate follows by taking a dyadic supremum and
using $\sigma<\alpha_V$.  In particular, no lower bound of the form $r>-\alpha_V$ is
needed when the second factor already has nonnegative Sobolev regularity.

Compact localization also gives
$\cC^{\alpha_V}\hookrightarrow H^r\cap B_{2,\infty}^{\sigma}$ for $r<\alpha_V$ and
$\sigma<\alpha_V$, which handles $V_{\W}V_{\K}$.

In the application, take $r=s_2-1<0$ and $\lambda=s_1$ or $s_2$.  Since
$s_2<1$, $\lambda\ge0$, and $0<\sigma<\alpha_V$, all hypotheses hold in
\eqref{eq:parameter-window}; the estimate in $B_{2,\infty}^{\sigma}$ is
stronger than the required source estimate in
$B_{2,\infty}^{\sigma-1}$.  Integration over $[0,T]$ supplies the factor
$T$.  Difference estimates follow from the bilinear identity
$VG-\widetilde V\widetilde G=(V-\widetilde V)G+\widetilde V(G-\widetilde G)$.
\end{proof}

\begin{lemma}[Endpoint localized deterministic quadratic products]
\label{lem:deterministic-quadratic}
Let
\[
  \gamma=s_2-\frac12>0.
\]
For $A,B\in\set{X,Y}$,
\begin{equation}\label{eq:quadratic-dual}
  \norm{\la D\ra^\gamma\bigl(\chi AB\bigr)}_{L_{t,x}^{4/3}}
  \lesssim T^{1/4}
  \norm{A}_{\wt A_T}\norm{B}_{\wt B_T},
\end{equation}
where $\wt A_T$ denotes the corresponding $X$ or $Y$ intersection norm.  Moreover,
\begin{equation}\label{eq:quadratic-Besov}
  \norm{\chi AB}_{L_T^1B_{2,\infty}^{\sigma-1}}
  \lesssim T^{1/4}
  \norm{A}_{\wt A_T}\norm{B}_{\wt B_T}.
\end{equation}
Both estimates have the corresponding bilinear difference form.
\end{lemma}

\begin{proof}
Put $\alpha_X=s_1-1/4$.  The parameter condition gives
\[
  0<\gamma=s_2-\frac12<\alpha_X.
\]
The fractional Leibniz rule and the Strichartz components of the solution
norms give the following three space--time estimates before lowering the
spatial Lebesgue exponent:
\begin{align*}
 XX:&\quad
 \|XX\|_{L_T^4W^{\gamma,4/3}}
 \lesssim\|X\|_{L_T^8W^{\alpha_X,8/3}}^2,\\
 XY:&\quad
 \|XY\|_{L_T^{8/3}W^{\gamma,8/5}}
 \lesssim
 \|X\|_{L_T^8W^{\alpha_X,8/3}}
 \|Y\|_{L_T^4W^{\gamma,4}},\\
 YY:&\quad
 \|YY\|_{L_T^2W^{\gamma,2}}
 \lesssim\|Y\|_{L_T^4W^{\gamma,4}}^2.
\end{align*}
Multiplication by $\chi$ is bounded on each displayed Sobolev space.
Lemma~\ref{lem:compact-p-lowering} then lowers the spatial exponents
$8/5$ and $2$ to $4/3$ after multiplication by $\chi$.  Finally,
\[
 L_T^4\hookrightarrow L_T^{4/3},\qquad
 L_T^{8/3}\hookrightarrow L_T^{4/3},\qquad
 L_T^2\hookrightarrow L_T^{4/3}
\]
with respective factors $T^{1/2}$, $T^{3/8}$, and $T^{1/4}$.  Thus the time powers in the dual conic norm are
\begin{equation}\label{eq:quadratic-dual-powers}
 XX:T^{1/2},\qquad XY:T^{3/8},\qquad YY:T^{1/4}.
\end{equation}
The smallest one proves \eqref{eq:quadratic-dual} uniformly for all three
products.

For the Besov estimate, Bernstein gives the sharp local embedding
\[
 L_x^{4/3}\hookrightarrow B_{2,\infty}^{-3/4}.
\]
The strict inequality $\sigma<1/4$ implies
$B_{2,\infty}^{-3/4}\hookrightarrow b_{2,\infty}^{\sigma-1}$: indeed the
additional negative dyadic power forces the high-frequency tail to vanish.
H\"older in time from $L_T^{4/3}$ to $L_T^1$ supplies another factor
$T^{1/4}$.  Hence the Besov-source powers are
\begin{equation}\label{eq:quadratic-besov-powers}
 XX:T^{3/4},\qquad XY:T^{5/8},\qquad YY:T^{1/2}.
\end{equation}
For $T\le1$ all six powers are bounded by the common factor
$T^{1/4}$ used in \eqref{eq:quadratic-dual}--\eqref{eq:quadratic-Besov}.
Bilinearity proves the difference bounds.
\end{proof}

For a triple $(a;b,c)\in\mathfrak L$, the corresponding regular
input is $\mathscr U_{b^\perp}^\rho$.  We write
\begin{equation}\label{eq:regular-input-E-norm}
 \|\mathscr U^\rho\|_{E_T^{2,\sigma}}
 :=\sum_{a\in\mathfrak C}
   \|\mathscr U_a^\rho\|_{E_T^{2,\sigma}}.
\end{equation}

\begin{lemma}[Operator and cubic terms]
\label{lem:operators-and-cubic}
For $\mathscr U_a^\rho=\rho V_a+X_a+Y_a$ one has
\begin{align}
 \sum_{(a;b,c)\in\mathfrak L}
 \|T^{a;b,c}(\mathscr U_{b^\perp}^\rho)\|_{L_T^1H^{s_2-1}}
 &\le T M_T^{\mathrm{op},H}(\Xi)
 \|\mathscr U^\rho\|_{E_T^{2,\sigma}},
 \label{eq:operator-H-bound}\\
 \sum_{(a;b,c)\in\mathfrak L}
 \|T^{a;b,c}(\mathscr U_{b^\perp}^\rho)\|_{L_T^1B_{2,\infty}^{\sigma-1}}
 &\le\left(T M_{T,\mathrm{diag}}^{\mathrm{op},B}(\Xi)
 +M_{T,\mathrm{cen}}^{\mathrm{op},B}(\Xi)\right)
 \|\mathscr U^\rho\|_{E_T^{2,\sigma}},
 \label{eq:operator-B-bound}\\
 \sum_{a\in\mathfrak C}\|\Gamma_a^{\chi,\rho}\|_{\mathfrak R_T^{s_2,\sigma}}
 &\lesssim T\sum_{a\in\mathfrak C}\left(
 \|\Gamma_a^{\chi,\rho}\|_{C_T\mathcal C^{-\beta_{\Gamma,a}-\kappa}}
 +\|\Gamma_a^{\chi,\rho}\|_{L_T^\infty B_{2,\infty}^{-\beta_{\Gamma,a}-\kappa}}
 \right).
 \label{eq:cubic-source-bound}
\end{align}
All three bounds have Lipschitz analogues for differences of the components of
$\Xi$ and of the inputs.
\end{lemma}

\begin{proof}
The Sobolev operator output is initially in $C_TH^{s_2-1}$ and is converted to
an $L_T^1$ source by the explicit factor $T$.  The deterministic diagonal has
the stronger $L_T^\infty B_{2,\infty}^{\sigma-1}$ output and is treated in the
same way.  The centered Besov component is already a direct $L_T^1$ source and
is locally small by Corollary~\ref{cor:full-centered}.  Finally, choose
$\delta>0$ so small that $-\beta_{\Gamma,a}-\kappa-\delta>s_2-1$.  The $B_{2,\infty}^{-\beta_{\Gamma,a}-\kappa}$ component of the stochastic
data gives
\[
 B_{2,\infty}^{-\beta_{\Gamma,a}-\kappa}\hookrightarrow H^{-\beta_{\Gamma,a}-\kappa-\delta}
 \hookrightarrow H^{s_2-1},
 \qquad
 B_{2,\infty}^{-\beta_{\Gamma,a}-\kappa}\hookrightarrow
 B_{2,\infty}^{\sigma-1},
\]
where the last embedding uses $-\beta_{\Gamma,a}-\kappa>\sigma-1$.  Thus the terms $\Gamma_a^{\chi,\rho}$
belong to both displayed source spaces, and time integration gives the
factor $T$.  The $C_T\mathcal C^{-\beta_{\Gamma,a}-\kappa}$ component supplies the
stated path topology.
\end{proof}

\subsection{The fixed point and stability}

Let $y_{\W}=(u_0,u_1)$ and $y_{\K}=(v_0,v_1)$ satisfy
\eqref{eq:data-space}.  We use
\begin{equation}\label{eq:data-norm}
 \|y\|_{\mathcal H^{s_2,\sigma}}
 :=\sum_{a\in\mathfrak C}\left(
 \|y_{a,0}\|_{H^{s_2}}+\|y_{a,1}\|_{H^{s_2-1}}
 +\|y_{a,0}\|_{B_{2,\infty}^{\sigma}}
 +\|y_{a,1}\|_{B_{2,\infty}^{\sigma-1}}\right).
\end{equation}
This norm is finite under \eqref{eq:data-space} because $\sigma<s_2$.

For two tuples $\Xi,\widetilde\Xi$ and two sets of Cauchy data, define
\begin{equation}\label{eq:Xi-distance}
 d_T\bigl((\Xi,y),(\widetilde\Xi,\widetilde y)\bigr)
 :=M_T(\Xi-\widetilde\Xi)
  +\|y-\widetilde y\|_{\mathcal H^{s_2,\sigma}}.
\end{equation}
Here the difference of the operator components is measured in the operator
norms appearing in \eqref{eq:M-op-H}--\eqref{eq:M-op-cen-B}.

The extension construction in Lemma~\ref{lem:time-extension-causality} gives,
for any two causal operator components $\mathcal A,\widetilde{\mathcal A}$,
\begin{equation}\label{eq:restriction-of-operator-difference}
 \|\mathcal A_{[0,T]}-\widetilde{\mathcal A}_{[0,T]}\|
 \lesssim
 \|\mathcal A-\widetilde{\mathcal A}\|_{[0,T_0]},
\end{equation}
with the corresponding domain and target norms.  For a zero-history interval
$I$, zero extension to the left gives the same inequality with constant one.
The distributional components and weighted pseudolocal terms restrict by ordinary norm
monotonicity.  Consequently, the short-interval moduli pass uniformly to cutoff limits.

For $Z=(X_{\W},X_{\K},Y_{\W},Y_{\K})\in\mathcal Z_T$, put
$\mathscr U_a^\rho=\rho V_a+X_a+Y_a$.  Define $\mathscr P_{\Xi,y}Z=\widetilde Z$ by
\begin{align}
 \widetilde X_a
 &:=I_a\bigl(\mathscr F_X^{\chi,\rho}(Z)\bigr),
 \label{eq:mild-X}\\
 \widetilde Y_a
 &:=S_a(t)y_a+I_a\bigl(\mathscr F_Y^{\chi,\rho}(Z)\bigr),
 \qquad a\in\mathfrak C.
 \label{eq:mild-Y}
\end{align}
Here the resonant terms are understood through
\eqref{eq:localized-Y-source}; their finite-cutoff form is given by
\eqref{eq:full-resonant-split}.

The estimates of
\cref{lem:rough-products,lem:V-products,lem:deterministic-quadratic,lem:operators-and-cubic}
control all terms in
$\mathscr F_X^{\chi,\rho}$ and $\mathscr F_Y^{\chi,\rho}$ after the
expansion $\mathscr U_a^\rho=\rho V_a+X_a+Y_a$.  The operator and cubic contributions
are measured by the first two lines of \eqref{eq:mu-full}; every remaining
term carries an explicit positive power of $T$ and is bounded by the final
line of \eqref{eq:mu-full}.  The separated-support remainders are included in
$M_T^{\mathrm{ps}}$.

\begin{proposition}[Uniform fixed-point estimate]
\label{prop:fixed-point-estimate}
Let $\mathfrak K$ satisfy Definition~\ref{def:localized-smallness-class}
with bound $A$.
For every $R\ge1$ there is a nondecreasing polynomial $P_R$ such that, for
$(\Xi,y),(\widetilde\Xi,\widetilde y)\in\mathfrak K$ and
$\|Z\|_{\mathcal Z_T},\|\widetilde Z\|_{\mathcal Z_T}\le R$,
\begin{align}
 \|\mathscr P_{\Xi,y}Z\|_{\mathcal Z_T}
 &\le C(A+A^2)+\mu_T(\Xi,y)P_R,
 \label{eq:self-map-estimate}\\
 \|\mathscr P_{\Xi,y}Z-
 \mathscr P_{\Xi,y}\widetilde Z\|_{\mathcal Z_T}
 &\le \mu_T(\Xi,y)P_R
 \|Z-\widetilde Z\|_{\mathcal Z_T}.
 \label{eq:contraction-estimate}
\end{align}
For simultaneous perturbations of $\Xi$ and the Cauchy data,
\begin{align}
 &\|\mathscr P_{\Xi,y}Z-
 \mathscr P_{\widetilde\Xi,\widetilde y}\widetilde Z\|_{\mathcal Z_T}
 \notag\\
 &\quad\le \nu_{\mathfrak K}(T)P_R
 \|Z-\widetilde Z\|_{\mathcal Z_T}
 +C_{A,R}\,d_T\big((\Xi,y),(\widetilde\Xi,\widetilde y)\big),
 \label{eq:perturbed-fixed-point}
\end{align}
where $d_T$ is defined in \eqref{eq:Xi-distance}.
\end{proposition}

\begin{proof}
The two low--high terms
in $\mathscr F_X$ are bounded by \eqref{eq:lowhigh-estimate}; applying
\eqref{eq:X-linear} and \eqref{eq:besov-propagation} gives the $X$ part of
\eqref{eq:self-map-estimate}, with coefficient
$T^\delta M_{T_0}^{\mathrm{bd}}$.

In $\mathscr F_Y$, split the rough stochastic products according to the
identity $\mathscr U=\rho V+X+Y$.  The $V\succ\Psi$, $\chi(X\succ\Psi)$, and
$\chi(Y\succ\Psi+Y\circ\Psi)$ terms are respectively
\eqref{eq:V-highlow}, \eqref{eq:X-highlow}, and
\eqref{eq:Y-random-products}.  Products in $\mathscr U_{\W}^\rho\mathscr U_{\K}^\rho$ that contain a
first Picard factor are covered by Lemma~\ref{lem:V-products}; the remaining
$XX$, $XY$, and $YY$ terms are controlled by
Lemma~\ref{lem:deterministic-quadratic} followed by
\eqref{eq:Y-linear-dual} and \eqref{eq:besov-propagation}.  These estimates produce
only the explicit $T$ or $T^\delta$ terms in \eqref{eq:mu-full}.

For the four resonant substitutions, Lemma~\ref{lem:complete-and-E} puts every
input $\mathscr U_a^\rho$ in $E_T^{2,\sigma}$.  The Sobolev output of $T^{a;b,c}$
is converted from $C_TH^{s_2-1}$ to $L_T^1H^{s_2-1}$ by $T$; the deterministic
Besov diagonal is converted from $L_T^\infty$ to $L_T^1$ by $T$; and the
centered Besov output is already measured in the direct $L_T^1$ operator norm.
These are the three terms in $\mu_T^{\mathrm{op}}$.  The terms $\Gamma_a^{\chi,\rho}$
contribute the second line of \eqref{eq:mu-full} by
\eqref{eq:cubic-source-bound}.  The homogeneous $Y$ flow is bounded by the
Cauchy-data component of $A$.  Collecting the terms independent of $Z$ gives
$C(A+A^2)$, and collecting all terms containing $Z$ gives
$\mu_T(\Xi,y)P_R$.  This proves \eqref{eq:self-map-estimate}.

For \eqref{eq:contraction-estimate}, every bilinear product is subtracted by
replacing one factor at a time.  For $T^{a;b,c}$ use the identity
\[
 T(f)-\widetilde T(\widetilde f)
 =(T-\widetilde T)f+\widetilde T(f-\widetilde f),
\]
and use the analogous linear identity for the cubic term.  This yields
the same source decomposition with one solution factor replaced by $Z-\widetilde Z$, hence
the coefficient $\mu_TP_R$.  Allowing $\Xi$ and the Cauchy data to vary leaves, in addition, the componentwise distance
\eqref{eq:Xi-distance}; the uniform bound in Definition~\ref{def:localized-smallness-class} gives
the constant $C_{A,R}$.  This proves \eqref{eq:perturbed-fixed-point}.
\end{proof}

\begin{proposition}[Continuity of the localized source maps]
\label{prop:source-map-continuity}
Fix $0<T\le T_0$, an adapted localization pair $(\chi,\rho)$, and constants
$A,R\ge1$.  Let $\Xi,\widetilde\Xi$ satisfy Definition~\ref{def:Xi-assumptions} on $[0,T]$ and assume
\[
 M_T(\Xi)+M_T(\widetilde\Xi)\le A,
\]
and let $Z,\widetilde Z\in\mathcal Z_T$ satisfy
$\|Z\|_{\mathcal Z_T}+\|\widetilde Z\|_{\mathcal Z_T}\le R$.  To display the dependence on $\Xi$, write
$\mathscr F_{X;\Xi}^{\chi,\rho}$ and
$\mathscr F_{Y;\Xi}^{\chi,\rho}$ for the functionals in
\eqref{eq:localized-X-source}--\eqref{eq:localized-Y-source}.  Then
\begin{align}
 &\bigl\|\mathscr F_{X;\Xi}^{\chi,\rho}(Z)
       -\mathscr F_{X;\widetilde\Xi}^{\chi,\rho}(\widetilde Z)
  \bigr\|_{\mathfrak R_T^{s_1,\sigma}}
 \notag\\
 &\quad+
 \bigl\|\mathscr F_{Y;\Xi}^{\chi,\rho}(Z)
       -\mathscr F_{Y;\widetilde\Xi}^{\chi,\rho}(\widetilde Z)
  \bigr\|_{\mathfrak S_T^{s_2,\sigma}}
 \notag\\
 &\qquad\le C_{A,R,\chi,\rho}
 \left(M_T(\Xi-\widetilde\Xi)
       +\|Z-\widetilde Z\|_{\mathcal Z_T}\right).
 \label{eq:source-map-continuity}
\end{align}
Consequently, if in addition
$\Theta,\widetilde\Theta\in
 C_T\mathcal C_{\mathrm{loc}}^{-1-\betasum-\kappa}$, then the source triples satisfy
\begin{align}
 &\|\chi(\Theta-\widetilde\Theta)\|_{
       C_T\mathcal C^{-1-\betasum-\kappa}}
 +\bigl\|\mathscr F_{X;\Xi}^{\chi,\rho}(Z)
       -\mathscr F_{X;\widetilde\Xi}^{\chi,\rho}(\widetilde Z)
  \bigr\|_{\mathfrak R_T^{s_1,\sigma}}
 \notag\\
 &\quad+
 \bigl\|\mathscr F_{Y;\Xi}^{\chi,\rho}(Z)
       -\mathscr F_{Y;\widetilde\Xi}^{\chi,\rho}(\widetilde Z)
  \bigr\|_{\mathfrak S_T^{s_2,\sigma}}
 \notag\\
 &\qquad\le C_{A,R,\chi,\rho}
 \left(
 \|\chi(\Theta-\widetilde\Theta)\|_{
       C_T\mathcal C^{-1-\betasum-\kappa}}
 +M_T(\Xi-\widetilde\Xi)
 +\|Z-\widetilde Z\|_{\mathcal Z_T}
 \right).
 \label{eq:source-triple-continuity}
\end{align}
\end{proposition}

\begin{proof}
Put
\[
 \mathscr U_a^\rho=\rho V_a+X_a+Y_a,
 \qquad
 \widetilde{\mathscr U}_a^\rho
 =\rho\widetilde V_a+\widetilde X_a+\widetilde Y_a.
\]
Lemma~\ref{lem:complete-and-E}, applied also to differences, gives
\begin{equation}\label{eq:regular-input-difference}
 \sum_{a\in\mathfrak C}
 \|\mathscr U_a^\rho-\widetilde{\mathscr U}_a^\rho
 \|_{E_T^{2,\sigma}}
 \lesssim M_T(\Xi-\widetilde\Xi)
 +\|Z-\widetilde Z\|_{\mathcal Z_T}.
\end{equation}
The corresponding undifferenced inputs are bounded by $C_{A,R}$ in
$E_T^{2,\sigma}$.

For the $X$-source, subtract each low--high product as
\[
 \mathscr U_{b^\perp}^\rho\prec\Psi_b
 -\widetilde{\mathscr U}_{b^\perp}^\rho\prec\widetilde\Psi_b
 = (\mathscr U_{b^\perp}^\rho
    -\widetilde{\mathscr U}_{b^\perp}^\rho)\prec\Psi_b
 +\widetilde{\mathscr U}_{b^\perp}^\rho
    \prec(\Psi_b-\widetilde\Psi_b).
\]
The difference form of Lemma~\ref{lem:rough-products}, together with
\eqref{eq:regular-input-difference}, bounds the sum of these terms in
$\mathfrak R_T^{s_1,\sigma}$ by the right-hand side of
\eqref{eq:source-map-continuity}.

For the non-operator part of the $Y$-source, use the same two-factor
subtraction for the high--low and resonant stochastic products.  For the
regular quadratic term use
\[
 \mathscr U_{\W}^\rho\mathscr U_{\K}^\rho
 -\widetilde{\mathscr U}_{\W}^\rho
  \widetilde{\mathscr U}_{\K}^\rho
 = (\mathscr U_{\W}^\rho-
    \widetilde{\mathscr U}_{\W}^\rho)\mathscr U_{\K}^\rho
 +\widetilde{\mathscr U}_{\W}^\rho
  (\mathscr U_{\K}^\rho-
    \widetilde{\mathscr U}_{\K}^\rho).
\]
The difference estimates in
Lemmas~\ref{lem:rough-products}, \ref{lem:V-products}, and
\ref{lem:deterministic-quadratic} place these contributions in
$\mathfrak S_T^{s_2,\sigma}$ with the required bound.  The separated-support
pieces are estimated by the difference form of
Lemma~\ref{lem:separated-support-smoothing}; their norms are included in
$M_T(\Xi-\widetilde\Xi)$.

For a triple $(a;b,c)$, the identity
\begin{align*}
 &T_{\Xi}^{a;b,c}(\mathscr U_{b^\perp}^\rho)
 -T_{\widetilde\Xi}^{a;b,c}
   (\widetilde{\mathscr U}_{b^\perp}^\rho)\\
 &\qquad=(T_{\Xi}^{a;b,c}-T_{\widetilde\Xi}^{a;b,c})
       (\mathscr U_{b^\perp}^\rho)
 +T_{\widetilde\Xi}^{a;b,c}
       (\mathscr U_{b^\perp}^\rho
        -\widetilde{\mathscr U}_{b^\perp}^\rho)
\end{align*}
combines the operator-difference component of $M_T(\Xi-\widetilde\Xi)$
with \eqref{eq:regular-input-difference}.  The cubic contribution is linear
in the coordinate $\Gamma_a^{\chi,\rho}$, so its difference is controlled directly by
the corresponding component of $M_T(\Xi-\widetilde\Xi)$.  Summing the finite
set of source terms proves \eqref{eq:source-map-continuity}.
Adding the independent $\chi\Theta$ coordinate gives
\eqref{eq:source-triple-continuity}.
\end{proof}

\begin{theorem}[Deterministic fixed point]
\label{thm:deterministic-closure}
Let $\mathfrak K$ satisfy Definition~\ref{def:localized-smallness-class}
with bound $A$.
There is $T_*=T_*(A,\nu_{\mathfrak K})>0$ such that, for every
$(\Xi,y)\in\mathfrak K$ and $0<T\le T_*$, the map
$\mathscr P_{\Xi,y}$ has a unique fixed point in $\mathcal Z_T$.  The fixed points depend locally Lipschitz continuously on
$(\Xi,y)$ on the common interval $[0,T_*]$.
\end{theorem}

\begin{proof}
Choose $R\ge2C(A+A^2)+1$.  By \eqref{eq:family-smallness}, choose $T_*$ so
that $\nu_{\mathfrak K}(T_*)P_R\le1/4$.  Then
\eqref{eq:self-map-estimate} maps the radius-$R$ ball to itself and
\eqref{eq:contraction-estimate} is a strict contraction.  Completeness is
Lemma~\ref{lem:complete-and-E}.  This first gives uniqueness in the
fixed-point ball.  Proposition~\ref{prop:interval-local-uniqueness} below
rules out any additional fixed point in $\mathcal Z_T$.
The perturbative estimate \eqref{eq:perturbed-fixed-point} and absorption of
the first term give local Lipschitz dependence.
\end{proof}

\begin{proposition}[Uniqueness on successive time intervals]
\label{prop:interval-local-uniqueness}
Let $\Xi$ satisfy Definition~\ref{def:Xi-assumptions}, and let $Z,\widetilde Z\in\mathcal Z_T$ be two
fixed points with the same Cauchy data.  Then $Z=\widetilde Z$ on $[0,T]$.
Let $I=[t_0,t_1]\subset[0,T]$ and suppose the two solutions agree
on $[0,t_0]$.  With $R$ larger than their two $\mathcal Z_T$ norms, the
translated difference satisfies
\begin{equation}\label{eq:interval-local-difference}
 \|Z-\widetilde Z\|_{\mathcal Z_I}
 \le \varpi_{\Xi,R}(|I|)
 \|Z-\widetilde Z\|_{\mathcal Z_I},
\end{equation}
where, with $\delta_{\mathrm{loc}}:=\min\{\delta,1/4\}>0$,
\begin{equation}\label{eq:interval-contraction-modulus}
 \varpi_{\Xi,R}(h)
 \le C_{\Xi,R}h^{\delta_{\mathrm{loc}}}
 +C_R\omega_{\mathrm{cen},\Xi}(h)
 \longrightarrow0
 \qquad(h\downarrow0).
\end{equation}
Here $\mathcal Z_I$ denotes the translated solution norm on
$I$, with zero Cauchy data at the left endpoint for the difference.
\end{proposition}

\begin{proof}
Because the two solutions agree before $t_0$, every component of their
difference has zero value and zero first time derivative at $t_0$.  The
corresponding operator inputs therefore lie in
$E_{I,0}^{2,\sigma}$.  Repeat the source estimates on the translated interval.
The rough stochastic products gain at least the positive power $|I|^\delta$
fixed in Definition~\ref{def:local-smallness-modulus}; products containing a
first Picard factor and the pseudolocal smoothing tails gain the explicit
factor $|I|$.  The endpoint quadratic estimates give the powers listed in
\eqref{eq:quadratic-dual-powers} and
\eqref{eq:quadratic-besov-powers}, whose minimum is $|I|^{1/4}$.
The deterministic Sobolev output and the diagonal Besov output acquire
an additional factor $|I|$ when converted to $L^1(I)$, uniformly by
\eqref{eq:uniform-interval-operator-bounds}.  The centered contribution is bounded by
$\omega_{\mathrm{cen},\Xi}(|I|)$.  The scalar linear
estimates are invariant under time translation.  Subtracting the two mild
systems therefore gives \eqref{eq:interval-local-difference} and
\eqref{eq:interval-contraction-modulus}.

Choose $h_0>0$ so that $\varpi_{\Xi,R}(h_0)<1$.  Equation
\eqref{eq:interval-local-difference} implies equality on every interval of
length at most $h_0$ once equality is known at its left endpoint.  A finite partition of $[0,T]$ into such intervals propagates equality from
time zero to time $T$.  Thus the argument applies to arbitrary fixed points in
the solution class, independently of the contraction ball.
\end{proof}

\begin{corollary}[Causal consistency and gluing of horizons]
\label{cor:horizon-gluing}
Fix a tuple $\Xi$ satisfying Definition~\ref{def:Xi-assumptions} and fix the Cauchy data.  If $0<T_1<T_2$ and fixed points
$Z^{(T_j)}\in\mathcal Z_{T_j}$ exist, then
\begin{equation}\label{eq:deterministic-horizon-consistency}
 \mathfrak R_{T_1}Z^{(T_2)}=Z^{(T_1)}.
\end{equation}
Consequently, let $\tau$ be measurable and suppose that for every rational
$q\in(0,1]$ a strongly measurable fixed point $Z^{(q)}$ is given on
$\{q<\tau\}$, extended by $0$ outside this event.  Then these fixed points
glue uniquely to a solution on $[0,\tau)$ in the sense of
Definition~\ref{def:random-interval-solution}.  For every deterministic
$T\in(0,1]$, its restriction to $[0,T]$ on $\{T<\tau\}$ admits a strongly
measurable $\mathcal Z_T$-valued representative.
\end{corollary}

\begin{proof}
Causality of the Duhamel maps and of every restriction of $T^{a;b,c}$ shows
that $\mathfrak R_{T_1}Z^{(T_2)}$ satisfies the same fixed-point equation as
$Z^{(T_1)}$ on $[0,T_1]$.  Full-class uniqueness from
Proposition~\ref{prop:interval-local-uniqueness} gives
\eqref{eq:deterministic-horizon-consistency}.  The reconstruction and source
functionals are causal, so their restrictions satisfy the corresponding
compatibility identities.

For measurability, enumerate the deterministic-horizon selectors as follows.
Enumerate $\mathbb Q\cap(T,1]$ as $(q_m)_{m\ge1}$ and define the disjoint
measurable sets
\[
 A_m:=\{T<q_m<\tau\}\setminus
 \bigcup_{\ell<m}\{T<q_\ell<\tau\}.
\]
Their union is $\{T<\tau\}$.  Extend every $Z^{(q_m)}$ by $0$ off
$\{q_m<\tau\}$ and set
\begin{equation}\label{eq:measurable-rational-horizon-selector}
 Z^{[T]}:=\sum_{m\ge1}\mathbf 1_{A_m}\,
 \mathfrak R_T Z^{(q_m)}.
\end{equation}
At most one summand is nonzero, so $Z^{[T]}$ is strongly measurable.  Equation
\eqref{eq:deterministic-horizon-consistency} makes it independent of the
enumeration and of the selected rational horizon.  This proves the asserted
gluing and measurability.
\end{proof}

\subsection{Spectral approximation and localization}
\label{sec:noise-to-solution}

Fix a compact set $K$ and an adapted localization pair $(\chi,\rho)$.  For
the stochastic terms relevant to the nonlinear source, define
\begin{align}\label{eq:stochastic-source-distance}
 &d_T^{\Theta}
 \bigl((\Theta,\Xi),(\widetilde\Theta,\widetilde\Xi)\bigr)
 \notag\\
 &\qquad:=M_T(\Xi-\widetilde\Xi)
 +\|\chi(\Theta-\widetilde\Theta)\|_{
 C_T\mathcal C^{-1-\betasum-\kappa}}.
\end{align}
The second term is used only when identifying the limit of the nonlinear
source.

For one admissible spectral cutoff, set
\begin{equation}\label{eq:Xi-cutoff}
 \Xi_\Lambda:=
 \left(
 (\Psi_{a,\Lambda},V_{a,\Lambda},\partial_tV_{a,\Lambda},
  \Gamma_{a,\Lambda}^{\chi,\rho})_{a\in\mathfrak C},
 (\mathcal D_\Lambda^{a;b})_{(a;b,b)\in\mathfrak L_{\mathrm{diag}}},
 (\mathcal B_\Lambda^{a;b,c})_{(a;b,c)\in\mathfrak L}
 \right).
\end{equation}
We consider this tuple together with $\Theta_\Lambda$ when proving
convergence of the nonlinear source.

\begin{proposition}[Cutoff factors in the finite-cutoff construction]
\label{prop:cutoff-component-table}
For every finite-cutoff component used in the solution map, the stochastic
cutoff factors occur only at the frequency legs listed below.  The difference
from the uncut component is a finite sum in which at least one listed factor
is replaced by $m_\Lambda-1$.  Hence every such difference is supported where
one listed frequency is at least $c_0\Lambda$.

\begin{center}
\normalfont\small
\renewcommand{\arraystretch}{1.2}
\begin{tabularx}{\textwidth}{@{}P{0.18\textwidth}P{0.29\textwidth}Y@{}}
\toprule
component & active stochastic cutoff legs & estimate used for the tail\\
\midrule
$\Psi_{a,\Lambda}$
& output/noise frequency $\xi$
& \Cref{lem:baseline-cutoff-tails}\\
$\Theta_\Lambda$
& $\eta$ and $\xi-\eta$
& \Cref{lem:baseline-cutoff-tails}\\
$V_{a,\Lambda}$, $\partial_tV_{a,\Lambda}$
& the two input legs $\eta$, $\xi-\eta$; an auxiliary comparison may also
  carry the output leg $\xi$
& \Cref{lem:first-picard-tail-integral,thm:first-picard-fullspace}\\
$\Gamma_{a,\Lambda}^{\chi,\rho}$
& third chaos: $\eta,\zeta,r$ and optional $\eta+\zeta,\xi$;
  first chaos: $r,-r,\xi$ and optional $\xi-r,\xi$
& \Cref{prop:cubic-third-spectral,lem:cubic-first-cutoff-tail,cor:localized-cubic-terms}\\
$\mathcal D_\Lambda^{a;b}$
& the contracted pair $\ell,-\ell$
& \Cref{prop:localized-diagonal}\\
$\mathcal B_\Lambda^{a;b,c}$
& the two Gaussian legs $\ell,r$ in the covariance synthesis maps
& \Cref{prop:principal-centered-block,lem:far-output,thm:resonant-operators}\\
$\zeta_{a,\Lambda}^{\chi,\rho}$
& only the legs inherited from $\xi_{a,\Lambda}$,
  $\Psi_{a,\Lambda}$, $\Theta_\Lambda$, and $V_{a,\Lambda}$
& the identity \eqref{eq:auxiliary-forcing-expanded} and the preceding rows\\
$\mathscr F_{X,\Lambda}^{\chi,\rho}$,
$\mathscr F_{Y,\Lambda}^{\chi,\rho}$
& only the legs of their displayed stochastic coordinates and operators;
  there is no projection of the full nonlinear source
& \Cref{prop:source-map-continuity,prop:finite-cutoff-reconstruction}\\
\bottomrule
\end{tabularx}
\end{center}

For a fixed-profile dyadic family, the distributional components converge
almost surely along the full sequence and the operator components converge in
the pathwise operator topologies of
\Cref{thm:resonant-operators}.  For an arbitrary admissible cofinal sequence,
the distributional components converge in every finite probability moment,
the centered operator components converge in the stated $L^p(\Omega)$
operator norms, and the resulting solution data converge in probability.  Two
non-nested cutoff families have the same limits.
\end{proposition}

\begin{proof}
The displayed finite-cutoff formulas give the listed legs:
\eqref{eq:stochastic-convolutions}, \eqref{eq:theta-cutoff},
\eqref{eq:first-picard-def},
\eqref{eq:cubic-cutoff-multiplier}--
\eqref{eq:cubic-contracted-cutoff-multiplier}, and the kernels in
Sections~\ref{sec:wick-split}--\ref{sec:resonant-operators}.  Applying the
finite-product identity in Lemma~\ref{lem:finite-product-cutoff} to each
formula gives the asserted support property, including repeated legs with
multiplicity.  The cited tail estimates are uniform over admissible cutoff
profiles and are summable in their respective path or operator norms.

The forcing identity \eqref{eq:auxiliary-forcing-expanded} preserves the
stochastic frequency support of its displayed factors.  The finite-cutoff
source formulas \eqref{eq:localized-X-source-cutoff}--
\eqref{eq:localized-Y-source-cutoff} are assembled componentwise, so their
convergence follows from the corresponding tails and the continuity of the
source maps.  Fixed-profile dyadic
almost-sure convergence follows from the polynomial tail bounds and the
finite-state stabilization of fixed blocks.  For two arbitrary cutoff
families, \eqref{eq:finite-product-two-cutoffs} bounds the difference by the
sum of the two tails relative to the uncut object, proving the final
assertion.
\end{proof}

\begin{corollary}[Convergence of the localized cutoff data]
\label{cor:Xi-convergence}
For every deterministic $T_0\le1$, there exist a tuple $\Xi$ satisfying
Definition~\ref{def:Xi-assumptions} and a distribution $\Theta$ such that
\begin{equation}\label{eq:Xi-convergence}
 d_{T_0}^{\Theta}
 \bigl((\Theta_\Lambda,\Xi_\Lambda),(\Theta,\Xi)\bigr)
 \longrightarrow0
 \qquad\text{in probability}.
\end{equation}
The convergence holds in every finite probability moment for the distributional
components and in the stated operator modes for $T^{a;b,c}$.  A
fixed-profile dyadic family converges simultaneously almost surely on a common
event; a general admissible family converges in probability.  The localized
limit is independent of the admissible spectral profile.
\end{corollary}

\begin{proof}
For each $a\in\mathfrak C$, the convergence of
$\Psi_{a,\Lambda}$ and $\Theta_\Lambda$ follows from
Proposition~\ref{prop:baseline-regularities}.  Theorem~\ref{thm:first-picard-fullspace}
gives convergence of $V_{a,\Lambda}$ and $\partial_tV_{a,\Lambda}$ in every
norm entering $M_{T_0}$.  Corollary~\ref{cor:localized-cubic-terms} gives convergence of
$\Gamma_{a,\Lambda}^{\chi,\rho}$ to the localized limit
$\Gamma_a^{\chi,\rho}$, including the smoothing remainder in
\eqref{eq:localized-cubic-limit-split}.  Finally,
Theorem~\ref{thm:resonant-operators} gives convergence of the two deterministic
diagonal components and the four centered operator components in the
operator norms used in \eqref{eq:M-op-H}--\eqref{eq:M-op-cen-B}.
There are only finitely many components in \eqref{eq:Xi-cutoff}.
Adding their convergence to the convergence of $\chi\Theta_\Lambda$ gives
\eqref{eq:Xi-convergence}.  The same finite-sum argument gives
convergence in every finite moment for the distributional components.

The joint cutoff modes and the comparison of non-nested families are given by
Proposition~\ref{prop:cutoff-component-table}.  For the fixed-profile dyadic
statement, fix the exhaustion $K_n=\overline{B(0,n)}$, one adapted localization
pair for each $K_n$, rational horizons, rational strict losses, and integer
probability moments.  This is a countable collection, so the corresponding
Borel--Cantelli events have a common probability-one intersection.  On that
event every component of $\Xi_\Lambda$ converges along the full dyadic
sequence, hence so does the finite product metric
$d_{T_0}^{\Theta}$.  For any localization pair outside the reference exhaustion, the same statement
holds on its corresponding probability-one event.
\end{proof}

\begin{lemma}[Stability under deterministic approximation]
\label{lem:cutoff-localized-small}
Fix $T_0$, a localized limit $(\Xi,y)$, and a deterministic sequence of cutoff
data $(\Xi_n,y_n)$ such that
\[
 d_{T_0}((\Xi_n,y_n),(\Xi,y))\longrightarrow0.
\]
After discarding finitely many indices, the family
\[
 \{(\Xi,y)\}\cup\{(\Xi_n,y_n):n\ge n_0\}
\]
satisfies Definition~\ref{def:localized-smallness-class} on a sufficiently
short common subinterval of $[0,T_0]$.
\end{lemma}

\begin{proof}
Convergence in $d_{T_0}$ gives a common bound for
$M_{T_0}^{\mathrm{bd}}$.  Every deterministic source contribution in
\eqref{eq:mu-full} is therefore bounded by a common polynomial times its
explicit positive power of $T$.  The diagonal Besov contribution is uniformly
bounded in the stronger $L_T^\infty B_{2,\infty}^{\sigma-1}$ operator norm and
is multiplied by $T$.

For the centered component, use the restriction inequality
\eqref{eq:restriction-of-operator-difference}.  First choose $h$ so that the
limit interval modulus $\omega_{\mathrm{cen},\Xi}(h)$ is small.  Then make the
global operator differences small; their zero-history restrictions on every
interval of length at most $h$ are no larger.  For the initial interval,
\eqref{eq:restriction-of-operator-difference} and
\eqref{eq:centered-operator-initial-smallness} give, uniformly in the tail,
small operator norm on the full domain $E_T^{2,\sigma}$ once $T$ is small.  Thus the deterministic tail has the uniform translated-interval modulus
required for the interval-by-interval uniqueness argument.  It also has the
initial-interval smallness required for the fixed-point argument.  The weighted pseudolocal terms
restrict by norm monotonicity, and the cubic path norm is converted to a source
norm by the explicit factor $T$.  These bounds are the terms entering
$\mu_T$; hence the supremum over the deterministic tail tends to zero.
\end{proof}

\begin{proposition}[Measurable lifetimes and compatibility across compact sets]
\label{prop:measurable-lifetimes}
Let $K_n=\overline{B(0,n)}$ and fix deterministic adapted localization pairs
$(\chi_n,\rho_n)$ for $(K_n,1)$.  On the common probability-one event on which
all dyadic approximating tuples are defined, there are measurable random variables
$\tau_n>0$ such that $\tau_{n+1}\le\tau_n$ and the localized fixed points for
$(\chi_n,\rho_n)$ form a strongly measurable solution on $[0,\tau_n)$ in the
sense of Definition~\ref{def:random-interval-solution}.  The solutions and
nonlinear sources for different $n$ agree on overlaps and therefore form a
compatible family of local solutions in the sense of
Definition~\ref{def:compatible-local-solutions}.  For an arbitrary compact $K$, set
\begin{equation}\label{eq:reference-localization-and-lifetime}
 n(K):=\min\{n:K\subset B(0,n)\},\qquad
 (\chi_K^\circ,\rho_K^\circ):=(\chi_{n(K)},\rho_{n(K)}),\qquad
 \tau_K:=\tau_{n(K)}.
\end{equation}
Then $K_1\subset K_2$ implies $\tau_{K_1}\ge\tau_{K_2}$.  Every other pair
$(\chi,\rho)$ adapted to $(K,1)$ has a measurable existence time
$\tau_K^{\chi,\rho}>0$, obtained by the same formula below, and its solution
agrees with the solution obtained from the fixed localization pair on every interval
$[0,T]$ with $T<\tau_K\wedge\tau_K^{\chi,\rho}$.
\end{proposition}

\begin{proof}
Write $\Xi^{(n)}$ for the tuple constructed with
$(\chi_n,\rho_n)$ and put
\begin{equation}\label{eq:lifetime-integer-radius}
 A_n(\omega):=1+\left\lceil M_1^{\mathrm{bd}}(\Xi^{(n)},y)\right\rceil,
 \qquad
 R_{n,k}:=2C_n\bigl((k+1)+(k+1)^2\bigr)+1,
 \qquad
 R_n:=R_{n,A_n}.
\end{equation}
Here $C_n$ and $P_R^{(n)}$ are the deterministic constant and polynomial in
\eqref{eq:self-map-estimate} for the fixed localizer $(\chi_n,\rho_n)$; they
need not be uniform in $n$.  For rational $q\in(0,1]$, define the monotone
measurable envelope
\[
 \overline\mu_n(q):=
 \sup_{r\in\mathbb Q\cap(0,q]}\mu_r(\Xi^{(n)}|_{[0,r]},y).
\]
All quantities are measurable.  Norms of the stochastic distributions are limits of
countably many dyadic or test-function seminorms.  By
Lemma~\ref{lem:centered-block-measurability}, each centered operator is an
operator-norm limit of strongly measurable finite-block sums.  Set
\begin{equation}\label{eq:lifetime-theta-tau}
 \theta_n:=\sup\left\{q\in\mathbb Q\cap(0,1]:
 \overline\mu_n(q)P_{R_n}^{(n)}\le\frac1{16}\right\},
 \qquad
 \tau_n:=\min_{1\le j\le n}\theta_j,
\end{equation}
with the convention $\sup\varnothing=0$.  A supremum over a countable set is
measurable.  Moreover $M_1^{\mathrm{bd}}<\infty$ and $\mu_r\to0$ almost
surely, so $\theta_n>0$ and hence $\tau_n>0$.

We construct the measurable fixed points on the countable family of
contraction events.  For $k\in\mathbb N$ and rational $q\in(0,1]$, let
\begin{equation}\label{eq:lifetime-rational-contraction-event}
 \mathcal C_{n,k,q}:=
 \left\{A_n=k,\quad
 \overline\mu_n(q)P_{R_{n,k}}^{(n)}\le\frac1{16}\right\}.
\end{equation}
On this Borel event, the source estimate makes
$\mathscr P_{\Xi^{(n)},y}$ a self-map and a contraction of the deterministic
radius-$R_{n,k}$ ball in $\mathcal Z_q$.  Corollary~\ref{cor:measurable-paracontrolled-fixed-point}
therefore gives a strongly measurable fixed point
$Z_{n,k}^{(q)}$, extended by $0$ off $\mathcal C_{n,k,q}$.  The events in
\eqref{eq:lifetime-rational-contraction-event} are disjoint in $k$, so
\begin{equation}\label{eq:lifetime-rational-fixed-point}
 Z_n^{(q)}:=\mathbf1_{\{q<\tau_n\}}
 \sum_{k\ge1}\mathbf1_{\mathcal C_{n,k,q}}Z_{n,k}^{(q)}.
\end{equation}
This is strongly measurable and, as required in
Definition~\ref{def:random-interval-solution}, is identically zero off
$\{q<\tau_n\}$.  If $q<\tau_n$, then $q<\theta_n$.  Choose $q'>q$ in the defining set of $\theta_n$; monotonicity of
$\overline\mu_n$ shows that $q$ itself satisfies the contraction inequality.
Thus the sum in \eqref{eq:lifetime-rational-fixed-point} selects the fixed
point on $\{q<\tau_n\}$.  Full-class uniqueness makes the family
compatible under restriction.  Corollary~\ref{cor:horizon-gluing} therefore glues it to a
strongly measurable random-interval solution on $[0,\tau_n)$.

The definition makes $\tau_n$ decreasing.  Finite-cutoff localization
independence and full dyadic convergence imply that the limiting fields and
sources for $K_n$ and $K_m$ agree on every common cylinder with
$T<\min\{\tau_n,\tau_m\}$.  Since only countably many $n,m$ and rational
horizons are involved, these identities hold simultaneously on the common
dyadic event.  This constructs the required compatible family on the exhaustion.  For a further adapted localization pair, repeat the construction with its
own deterministic constants $C_{\chi,\rho}$ and $P_R^{\chi,\rho}$; this gives
$\tau_K^{\chi,\rho}$ and measurable rational-horizon fixed points.  The same
finite-propagation argument identifies the two limits on the common interval.
This proves all assertions, including the reference notation in
\eqref{eq:reference-localization-and-lifetime}.  On the null complement of the common
dyadic event, set all $\tau_n$ equal to zero and all rational-horizon fixed
points equal to zero.  The resulting variables are globally measurable and
the lifetimes are positive almost surely.
\end{proof}

\begin{remark}[Measurable lifetimes]
\label{rem:measurable-not-stopping}
The definition of $\tau_K$ uses the terminal control norm
$M_1^{\mathrm{bd}}$.  Thus $\tau_K$ is measurable and almost surely positive;
no stopping-time property is asserted or used.
\end{remark}

\begin{corollary}[Pathwise convergence of fixed points]
\label{cor:localized-cutoff-passage}
Whenever $d_{T_0}((\Xi_n,y_n),(\Xi,y))\to0$ pathwise, there is a
path-dependent common interval on which the fixed points $Z_n$ converge to $Z$
in $\mathcal Z_T$.  In particular, for a fixed-profile dyadic family this holds
almost surely along the full dyadic sequence.
\end{corollary}

\begin{proof}
Fix a realization on which
$d_{T_0}((\Xi_n,y_n),(\Xi,y))\to0$.  By
Lemma~\ref{lem:cutoff-localized-small}, after discarding finitely many indices
there are $T>0$, $A<\infty$, and an admissible family
\[
 \mathfrak K=\{(\Xi,y)\}\cup
 \{(\Xi_n,y_n):n\ge n_0\}
\]
on $[0,T]$.  Shrink $T$ so that the common fixed-point construction of
Theorem~\ref{thm:deterministic-closure} takes place in a ball of radius $R$
and
\[
 \nu_{\mathfrak K}(T)P_R\le\frac12.
\]
Let $Z$ and $Z_n$ be the corresponding fixed points.  Applying
\eqref{eq:perturbed-fixed-point} to $(Z_n,Z)$ gives
\[
 \|Z_n-Z\|_{\mathcal Z_T}
 \le \nu_{\mathfrak K}(T)P_R\|Z_n-Z\|_{\mathcal Z_T}
 +C_{A,R}d_T((\Xi_n,y_n),(\Xi,y)).
\]
After absorption,
\[
 \|Z_n-Z\|_{\mathcal Z_T}
 \le 2C_{A,R}d_T((\Xi_n,y_n),(\Xi,y))\longrightarrow0.
\]
For a fixed-profile dyadic family, $\Xi_\Lambda$ converges on the
probability-one event of Corollary~\ref{cor:Xi-convergence}; the
preceding deterministic argument applies to every realization in that event.
\end{proof}

For general admissible cutoff families, convergence is first formulated on
events on which the limiting fixed point has a deterministic existence time.
Let $C$ be the constant in \eqref{eq:self-map-estimate}.  The radius is chosen
with a one-unit perturbative margin:
\begin{equation}\label{eq:contraction-radius}
 R_A:=2C\bigl((A+1)+(A+1)^2\bigr)+1,
 \qquad A\in\mathbb N.
\end{equation}
For deterministic $0<T\le T_0$, define
\begin{equation}\label{eq:deterministic-contraction-event}
 \mathcal G_{A,T}:=
 \left\{
 M_{T_0}^{\mathrm{bd}}(\Xi,y)\le A,
 \quad
 \mu_T(\Xi,y)P_{R_A}\le\frac1{16}
 \right\}.
\end{equation}

\begin{definition}[Measurable fixed-point domain]
\label{def:measurable-fixed-point-domain}
For a data pair $(\Upsilon,z)$ on $[0,T_0]$, set
\begin{equation}\label{eq:measurable-fixed-point-events}
 \mathcal C_{k,T}(\Upsilon,z):=
 \left\{
 M_{T_0}^{\mathrm{bd}}(\Upsilon,z)\le k+1,
 \quad
 \mu_T(\Upsilon,z)P_{R_k}\le\frac18
 \right\},
 \qquad k\in\mathbb N,
\end{equation}
and
\begin{equation}\label{eq:measurable-fixed-point-domain}
 \mathcal A_T(\Upsilon,z):=
 \bigcup_{k\ge1}\mathcal C_{k,T}(\Upsilon,z).
\end{equation}
On $\mathcal C_{k,T}(\Upsilon,z)$, let $Z_{k,T}(\Upsilon,z)$ denote the
radius-$R_k$ fixed point, extended by zero off this event, and define
\[
 D_{k,T}(\Upsilon,z)
 :=\mathcal C_{k,T}(\Upsilon,z)\setminus
   \bigcup_{\ell<k}\mathcal C_{\ell,T}(\Upsilon,z).
\]
The associated partial solution map, extended by zero outside
$\mathcal A_T(\Upsilon,z)$, is
\begin{equation}\label{eq:measurable-fixed-point-map}
 \mathscr S_T(\Upsilon,z)
 :=\sum_{k\ge1}
 \mathbf1_{D_{k,T}(\Upsilon,z)}Z_{k,T}(\Upsilon,z).
\end{equation}
\end{definition}

\begin{lemma}[Measurability of the partial solution map]
\label{lem:measurable-partial-solution-map}
The event $\mathcal A_T(\Upsilon,z)$ is Borel and
$\mathscr S_T(\Upsilon,z)$ is strongly measurable.  On
$\mathcal A_T(\Upsilon,z)$ it is the unique fixed point in $\mathcal Z_T$.
At finite cutoff, its reconstruction is a high-regularity solution of the
localized auxiliary equation by
\cref{prop:finite-cutoff-regularity,prop:finite-cutoff-reconstruction}; on
the interval supplied by \cref{lem:finite-cutoff-auxiliary-solvability}, it
agrees with the solution constructed there.
\end{lemma}

\begin{proof}
On $\mathcal C_{k,T}(\Upsilon,z)$, the map
$\mathscr P_{\Upsilon,z}$ is a self-map and a contraction of the
radius-$R_k$ ball.  Corollary~\ref{cor:measurable-paracontrolled-fixed-point}
gives a strongly measurable fixed point there.  The sets
$D_{k,T}(\Upsilon,z)$ are disjoint and Borel, so
\eqref{eq:measurable-fixed-point-map} is strongly measurable.  Full-class
uniqueness makes its value independent of the first successful radius.
Proposition~\ref{prop:finite-cutoff-reconstruction} identifies the finite-
cutoff fixed point with the solution of the localized auxiliary system.
\end{proof}

\begin{proposition}[Cutoff convergence in probability]
\label{prop:in-probability-cutoff-passage}
Suppose $(\Xi_\Lambda,y_\Lambda)\to(\Xi,y)$ in probability in
$d_{T_0}$.  For every fixed $A\in\mathbb N$, $T$, and every $\varepsilon>0$,
\begin{align}
 \Pp\Bigl(&\mathcal G_{A,T}\cap
 \Bigl(\mathcal A_T(\Xi_\Lambda,y_\Lambda)^c
 \cup\Bigl(\mathcal A_T(\Xi_\Lambda,y_\Lambda)\cap
 \bigl\{\|\mathscr S_T(\Xi_\Lambda,y_\Lambda)-Z\|_{\mathcal Z_T}>\varepsilon\bigr\}\Bigr)\Bigr)\Bigr)
 \longrightarrow0.
 \label{eq:contraction-event-fixed-point-convergence}
\end{align}
Moreover,
\begin{equation}\label{eq:contraction-events-exhaust}
 \lim_{T\downarrow0}\Pp(\mathcal G_{A,T})
 =\Pp\{M_{T_0}^{\mathrm{bd}}(\Xi,y)\le A\},
 \qquad
 \lim_{A\to\infty}\lim_{T\downarrow0}\Pp(\mathcal G_{A,T})=1.
\end{equation}
This gives local convergence in probability for arbitrary admissible cofinal
families, without choosing a cutoff-dependent lifespan.
\end{proposition}

\begin{proof}
On $\mathcal G_{A,T}$, the limiting map is a contraction on the ball of radius
$R_A$.  The deterministic estimates used to prove
Proposition~\ref{prop:fixed-point-estimate} are polynomially locally Lipschitz
in all components of $\Xi$.  Consequently there is a deterministic
$\eta=\eta(A,T)>0$ such that, on $\mathcal G_{A,T}$,
\[
 d_{T_0}((\Xi_\Lambda,y_\Lambda),(\Xi,y))<\eta
\]
implies
\[
 M_{T_0}^{\mathrm{bd}}(\Xi_\Lambda,y_\Lambda)\le A+1,
 \qquad
 \mu_T(\Xi_\Lambda,y_\Lambda)P_{R_A}\le\frac18.
\]
The choice \eqref{eq:contraction-radius} then gives
\[
 C\bigl((A+1)+(A+1)^2\bigr)+\frac18<R_A,
\]
so the cutoff map is a self-map and a contraction on the same radius-$R_A$
ball.  The perturbative estimate therefore yields
\[
 \|\mathscr S_T(\Xi_\Lambda,y_\Lambda)-Z\|_{\mathcal Z_T}
 \le C_{A,T}
 d_{T_0}((\Xi_\Lambda,y_\Lambda),(\Xi,y)).
\]
Convergence in probability proves
\eqref{eq:contraction-event-fixed-point-convergence}.

Almost surely,
\[
 M_{T_0}^{\mathrm{bd}}(\Xi,y)<\infty,
 \qquad
 \lim_{T\downarrow0}\mu_T(\Xi,y)=0
\]
by Definition~\ref{def:Xi-assumptions} and
Corollary~\ref{cor:full-centered}.  Hence, for fixed $A$, the indicators of
$\mathcal G_{A,T}$ converge pointwise as $T\downarrow0$ to the indicator of
$\{M_{T_0}^{\mathrm{bd}}\le A\}$.  Dominated convergence gives the first
identity in \eqref{eq:contraction-events-exhaust}, and then $A\to\infty$ gives the
second.
\end{proof}

\begin{proposition}[Cutoff passage up to the reference lifetime]
\label{prop:random-lifespan-cutoff-passage}
Fix $n$ and use the reference pair $(\chi_n,\rho_n)$ from
Proposition~\ref{prop:measurable-lifetimes}.  Let an arbitrary admissible
cofinal cutoff family generate $(\Xi_\Lambda,y_\Lambda)$.  Its measurable
finite-cutoff fixed point is $\mathscr S_T(\Xi_\Lambda,y_\Lambda)$ on the
existence event $\mathcal A_T(\Xi_\Lambda,y_\Lambda)$.  Whenever the limiting
coordinate is defined only on $\{T<\tau_n\}$, extend it by zero outside that
event.  Then, for every deterministic $T\in(0,1]$ and every
$\varepsilon>0$,
\begin{equation}\label{eq:random-lifespan-fixed-point-passage}
 \Pp\left(
  \{T<\tau_n\}\cap
  \left(\mathcal A_T(\Xi_\Lambda,y_\Lambda)^c
  \cup\left(\mathcal A_T(\Xi_\Lambda,y_\Lambda)\cap
  \left\{\|\mathscr S_T(\Xi_\Lambda,y_\Lambda)-Z\|_{\mathcal Z_T}>\varepsilon\right\}\right)\right)
 \right)\longrightarrow0.
\end{equation}
The same conclusion holds for the reconstructed fields in every continuous
seminorm of $C([0,T];\mathcal D'(\R^3))^2$ and for the three source terms
in the topology of Definition~\ref{def:source-topology}.  For any other
adapted localization pair, the corresponding assertion holds with its own measurable time
$\tau_K^{\chi,\rho}$.
\end{proposition}

\begin{proof}
For integers $k\ge1$ and rationals $q\in(T,1]$, put
\[
 R_{n,k}:=2C_n\bigl((k+1)+(k+1)^2\bigr)+1
\]
and
\[
 \mathcal H_{k,q}^{(n)}
 :=\left\{A_n=k,\quad
 \overline\mu_n(q)P_{R_{n,k}}^{(n)}\le\frac1{16}\right\}.
\]
Because $R_n=R_{n,A_n}$, the definition of $\theta_n$ gives the countable
cover
\begin{equation}\label{eq:theta-countable-cover}
 \{T<\theta_n\}
 =\bigcup_{\substack{q\in\mathbb Q\cap(T,1]\\ k\in\mathbb N}}
 \mathcal H_{k,q}^{(n)}.
\end{equation}
On $\mathcal H_{k,q}^{(n)}$ one has
$M_1^{\mathrm{bd}}(\Xi^{(n)},y)\le k$ and
$\mu_qP_{R_{n,k}}^{(n)}\le1/16$.  Thus this event is contained in the
event $\mathcal G_{k,q}^{(n)}$, and
Proposition~\ref{prop:in-probability-cutoff-passage} applies on $[0,q]$, hence
also on $[0,T]$ by restriction.

Enumerate the countable family in \eqref{eq:theta-countable-cover} and let
$B_m$ be the union of its first $m$ members.  Since
$\{T<\tau_n\}\subset\{T<\theta_n\}$,
\[
 \Pp(\{T<\tau_n\}\setminus B_m)\longrightarrow0
 \qquad(m\to\infty).
\]
For fixed $m$, the probability of the exceptional event on $B_m$ tends to zero by a finite union
bound and Proposition~\ref{prop:in-probability-cutoff-passage}.  Taking first $\Lambda\to\infty$ and
then $m\to\infty$ proves
\eqref{eq:random-lifespan-fixed-point-passage}.  Field convergence follows
from the locally Lipschitz reconstruction.  Source convergence uses in addition
the joint convergence of $\Theta_\Lambda$ and $\Xi_\Lambda$ from
Corollary~\ref{cor:Xi-convergence} and the locally Lipschitz source
maps.  Repeating the same countable-cover argument for another fixed adapted localization pair gives the last assertion.
\end{proof}

\begin{proposition}[Finite-cutoff reconstruction]
\label{prop:finite-cutoff-reconstruction}
Let $Z_\Lambda=(X_{\W,\Lambda},X_{\K,\Lambda},
Y_{\W,\Lambda},Y_{\K,\Lambda})$ solve the finite-cutoff fixed-point equation for an
adapted localization pair, and set
\begin{equation}\label{eq:finite-reconstruction}
\begin{aligned}
 u_\Lambda^{\chi,\rho}
 &=\rho(\Psi_{\W,\Lambda}+V_{\W,\Lambda})
   +X_{\W,\Lambda}+Y_{\W,\Lambda},\\
 v_\Lambda^{\chi,\rho}
 &=\rho(\Psi_{\K,\Lambda}+V_{\K,\Lambda})
   +X_{\K,\Lambda}+Y_{\K,\Lambda}.
\end{aligned}
\end{equation}
Then, in the sense of finite-cutoff distributions,
\begin{equation}\label{eq:finite-source-identity}
 \chi u_\Lambda^{\chi,\rho}v_\Lambda^{\chi,\rho}
 =\chi\Theta_\Lambda
  +\mathscr F_{X,\Lambda}^{\chi,\rho}(Z_\Lambda)
  +\mathscr F_{Y,\Lambda}^{\chi,\rho}(Z_\Lambda).
\end{equation}
Consequently the reconstructed fields solve the finite-cutoff auxiliary system
\eqref{eq:cutoff-system}.  Conversely, suppose that a solution of
\eqref{eq:cutoff-system} can be written as
\[
 u_\Lambda=B_{\W,\Lambda}^\rho+R_{\W,\Lambda},
 \qquad
 v_\Lambda=B_{\K,\Lambda}^\rho+R_{\K,\Lambda},
\]
with
\[
 R_{a,\Lambda}\in C_TH^r\cap C_T^1H^{r-1}
 \quad(a\in\mathfrak C)
\]
for some $r>3/2$.  Then it determines a unique solution of the finite-cutoff fixed-point equation
through the decomposition below.
\end{proposition}

\begin{proof}
On $\supp\chi$ one has $\rho=1$.  Hence the product of the reconstructed
fields there is
\[
 (\Psi_{\W,\Lambda}+\mathscr U_{\W,\Lambda}^\rho)
 (\Psi_{\K,\Lambda}+\mathscr U_{\K,\Lambda}^\rho).
\]
Expand it and apply the Bony identity to the two cross products.  The
low--high terms give \eqref{eq:localized-X-source-cutoff}.  Substitute the two
$X$ Duhamel formulae in the resonant terms.  The resulting four blocks are those in \eqref{eq:T-def}; the
$\rho V\circ\Psi$ terms are the finite-cutoff localized coordinates $\Gamma_{a,\Lambda}^{\chi,\rho}$ from
\eqref{eq:localized-cubic-cutoff}.  Their decomposition includes the retained
first-chaos contraction and the smoothing term
$\mathcal R_{a,\Lambda}^{\chi,\rho}$ in
\eqref{eq:localized-cubic-split}.  The remaining high--low, regular quadratic, and classical
$Y\circ\Psi$ terms give \eqref{eq:localized-Y-source}.  This proves
\eqref{eq:finite-source-identity}.  Adding the definition of the finite-cutoff auxiliary forcing
\eqref{eq:finite-auxiliary-forcing} gives \eqref{eq:cutoff-system}.

Conversely, let $(u_\Lambda,v_\Lambda)$ satisfy the hypotheses in the
second part of the proposition and define
\[
 \mathscr U_{\W,\Lambda}^\rho:=u_\Lambda-\rho\Psi_{\W,\Lambda},
 \qquad
 \mathscr U_{\K,\Lambda}^\rho:=v_\Lambda-\rho\Psi_{\K,\Lambda}.
\]
Define $X_{a,\Lambda}$ as the retarded solution with zero Cauchy data of the
low--high source in \eqref{eq:localized-X-source-cutoff}, evaluated at these two
known inputs, and set
\[
 Y_{a,\Lambda}
 :=\mathscr U_{a,\Lambda}^\rho-\rho V_{a,\Lambda}-X_{a,\Lambda}.
\]
The Bony identity for the high-regularity remainders and the smooth
finite-cutoff stochastic factors (which need not remain band-limited after
physical localization), together with the finite Wick decomposition, splits
the remaining source into
\eqref{eq:localized-Y-source-cutoff}; the auxiliary equation supplies the prescribed
Cauchy data for $Y_{a,\Lambda}$.  Hence $(X_\Lambda,Y_\Lambda)$ satisfies the finite-cutoff fixed-point equation and reconstructs the original fields.  The split
is unique because the $X$ term is the prescribed zero-data retarded
low--high solution and $Y$ is the remaining term.
\end{proof}

\begin{proposition}[Convergence of the localized nonlinear source]
\label{prop:source-convergence}
On the common interval of
Corollary~\ref{cor:localized-cutoff-passage},
\[
 (\chi\Theta_\Lambda,
   \mathscr F_{X,\Lambda}^{\chi,\rho},
   \mathscr F_{Y,\Lambda}^{\chi,\rho})
 \longrightarrow
 (\chi\Theta,
   \mathscr F_X^{\chi,\rho},
   \mathscr F_Y^{\chi,\rho})
\]
in the componentwise topology of Definition~\ref{def:source-topology}.  Hence
\begin{equation}\label{eq:localized-source-definition}
  \mathcal N^{\chi,\rho}
  :=\chi\Theta
   +\mathscr F_X^{\chi,\rho}
   +\mathscr F_Y^{\chi,\rho}
\end{equation}
defines the compactly supported nonlinear source of the auxiliary equation and
is the limit of
$\chi u_\Lambda^{\chi,\rho}v_\Lambda^{\chi,\rho}$.  General admissible
cofinal families have the same convergence in probability on every event $\mathcal G_{A,T}$ defined in
\eqref{eq:deterministic-contraction-event}.
\end{proposition}

\begin{proof}
Let $Z_\Lambda\to Z$ be the fixed-point convergence on the common interval.
Corollary~\ref{cor:Xi-convergence} gives
\[
 \|\chi(\Theta_\Lambda-\Theta)\|_{
 C_T\mathcal C^{-1-\betasum-\kappa}}
 +M_T(\Xi_\Lambda-\Xi)\longrightarrow0.
\]
The fixed points and the tuples $\Xi_\Lambda$ remain in bounded subsets of their
respective spaces.  Applying \eqref{eq:source-triple-continuity} with
$(\Theta,\Xi,Z)$ and $(\Theta_\Lambda,\Xi_\Lambda,Z_\Lambda)$ yields
\begin{align*}
 &\|\chi(\Theta_\Lambda-\Theta)\|_{
 C_T\mathcal C^{-1-\betasum-\kappa}}\\
 &\quad+
 \|\mathscr F_{X,\Lambda}^{\chi,\rho}(Z_\Lambda)
   -\mathscr F_X^{\chi,\rho}(Z)\|_{\mathfrak R_T^{s_1,\sigma}}\\
 &\quad+
 \|\mathscr F_{Y,\Lambda}^{\chi,\rho}(Z_\Lambda)
   -\mathscr F_Y^{\chi,\rho}(Z)\|_{\mathfrak S_T^{s_2,\sigma}}
 \longrightarrow0.
\end{align*}
At finite cutoff, Proposition~\ref{prop:finite-cutoff-reconstruction} gives
\[
 \chi u_\Lambda^{\chi,\rho}v_\Lambda^{\chi,\rho}
 =\chi\Theta_\Lambda
  +\mathscr F_{X,\Lambda}^{\chi,\rho}(Z_\Lambda)
  +\mathscr F_{Y,\Lambda}^{\chi,\rho}(Z_\Lambda).
\]
Lemma~\ref{lem:source-summation-continuity} turns the componentwise convergence into convergence of the sums in $\mathcal D'((0,T)\times\R^3)$.  Passing to the limit in the finite-cutoff identity therefore proves the asserted identification with $\mathcal N^{\chi,\rho}$.

On a deterministic contraction event for a general cofinal family,
Proposition~\ref{prop:in-probability-cutoff-passage} gives convergence of the
fixed points in probability, while Corollary~\ref{cor:Xi-convergence}
gives the joint convergence of $\Theta_\Lambda$ and $\Xi_\Lambda$.  The deterministic estimate
\eqref{eq:source-triple-continuity} then gives convergence of all three components in Definition~\ref{def:source-topology}
in probability.
\end{proof}

\begin{proposition}[Finite-cutoff independence of the localization pair]
\label{prop:finite-localization-independence}
Let $(\chi_1,\rho_1)$ and $(\chi_2,\rho_2)$ be adapted to $(K,T_0)$.  For a
common spectral cutoff and common Cauchy data, suppose the two reconstructed
solutions exist on $[0,T]$, $T\le T_0$.  Then
\[
 (u_\Lambda^{\chi_1,\rho_1},v_\Lambda^{\chi_1,\rho_1})
 =
 (u_\Lambda^{\chi_2,\rho_2},v_\Lambda^{\chi_2,\rho_2})
 \quad\text{on }[0,T]\times K.
\]
Their ordinary product sources agree there as well.
\end{proposition}

\begin{proof}
Write $(u_r,v_r)$ for the reconstruction associated with
$(\chi_r,\rho_r)$, $r=1,2$.  By
Proposition~\ref{prop:finite-cutoff-regularity}, both reconstructions belong to
\[
 (C_TH^1\cap C_T^1L^2\cap C_TL^\infty)^2.
\]
In particular, all regularity and coefficient hypotheses of
Lemma~\ref{lem:finite-propagation} hold on compact cones.

Let $\mathcal Q_{K,T}$ be the maximal-speed backward cone from
$[0,T]\times K$.  Cone adaptation gives
$\chi_1=\chi_2=\rho_1=\rho_2=1$ on a neighborhood of
$\mathcal Q_{K,T}$.  Hence, on this cone, the two cutoff systems have the
same nonlinear coefficient $\mathfrak b=1$.  By
Lemma~\ref{lem:auxiliary-forcing-local}, their forcing terms are the same
regularized Gaussian forcings $\xi_{a,\Lambda}$.  The Cauchy data agree on
$K^{[\speedmax T]}$.  Moreover,
\[
 v_1,\ u_2\in L^1([0,T];L^\infty(\mathcal Q_{K,T}))
\]
by the preceding Sobolev regularity.  Lemma~\ref{lem:finite-propagation}
therefore yields $(u_1,v_1)=(u_2,v_2)$ on $[0,T]\times K$.
Since the finite-cutoff fields are functions, their ordinary products agree
there as well.
\end{proof}

\begin{proposition}[Passage of finite propagation to the cutoff limits]
\label{prop:limit-localization-comparison}
Let $(\chi_1,\rho_1)$ and $(\chi_2,\rho_2)$ be adapted to $(K,T_0)$, and let
$(u_\jmath,v_\jmath,\mathcal N_\jmath)$, $\jmath=1,2$, be the corresponding
limiting localized solutions with existence times $\tau_1,\tau_2$.  Fix
$T<\tau_1\wedge\tau_2$.  Choose
$\eta\in C_c^\infty(\R^3)$ such that $\eta=1$ on a neighborhood of $K$ and
$\operatorname{supp}\eta\subset K^{[1]}$.  For every sufficiently large
integer $J$, on the finite intersection of the two fixed-profile dyadic
probability-one events,
\begin{equation}\label{eq:limit-localization-comparison}
 \eta(u_1-u_2)=\eta(v_1-v_2)=0
 \quad\text{in }C_TH^{-J},
 \qquad
 \eta(\mathcal N_1-\mathcal N_2)=0
 \quad\text{in }L_T^1H^{-J}.
\end{equation}
For an arbitrary admissible cofinal cutoff sequence and deterministic
$T$, the same conclusion holds on the event
$\{T<\tau_1\wedge\tau_2\}$ in the partial convergence sense of
\cref{def:partial-probability-convergence}: the two localized approximations
converge in probability in the same Banach spaces to the common triple in
\eqref{eq:limit-localization-comparison}.
\end{proposition}

\begin{proof}
Because $\operatorname{supp}\eta\subset K^{[1]}$ and $T\le T_0$, the
maximal-speed backward cone from $[0,T]\times\operatorname{supp}\eta$ is
contained in $[0,T]\times K^{[\speedmax T_0+1]}$.  Both localization pairs
are identically one on a neighborhood of this larger cylinder.  Hence, for
every cutoff level at which both fixed points exist, the proof of
\cref{prop:finite-localization-independence}, with
$\operatorname{supp}\eta$ in place of $K$, gives
\[
 \eta u_{1,\Lambda}=\eta u_{2,\Lambda},
 \qquad
 \eta v_{1,\Lambda}=\eta v_{2,\Lambda}.
\]
On $\operatorname{supp}\eta$ one has $\chi_1=\chi_2=1$, so the
finite-cutoff source identity also gives
\[
 \eta\mathcal N_{1,\Lambda}
 =\eta u_{1,\Lambda}v_{1,\Lambda}
 =\eta u_{2,\Lambda}v_{2,\Lambda}
 =\eta\mathcal N_{2,\Lambda}.
\]

Choose $J$ so large that multiplication by $\eta$ maps every local
stochastic distribution appearing in the reconstruction continuously into
$H^{-J}$, and every source space in \cref{def:source-topology} continuously
into $L_T^1H^{-J}$.  Along a fixed-profile dyadic sequence, the reconstructed
fields converge in $C_TH^{-J}$ and the sources converge in $L_T^1H^{-J}$ by
\cref{cor:localized-cutoff-passage,prop:source-convergence}.  Passing to the limit in the preceding identities proves
\eqref{eq:limit-localization-comparison}.

For a general cofinal sequence, intersect with
$\{T<\tau_1\wedge\tau_2\}$.  On the common deterministic contraction events
the two finite-cutoff triples are again equal whenever both are defined, while
the probability of the complementary failure event tends to zero.  Each
triple converges in probability in
\[
 (C_TH^{-J})^2\times L_T^1H^{-J}
\]
by \cref{prop:in-probability-cutoff-passage,prop:source-convergence}.  If
$X_\Lambda=Y_\Lambda$ outside events of probability tending to zero and
$X_\Lambda\to X$, $Y_\Lambda\to Y$ in probability in a Banach space, then
for every $\varepsilon>0$,
\[
 \Pp(\|X-Y\|>\varepsilon)
 \le \Pp(\|X-X_\Lambda\|>\varepsilon/3)
   +\Pp(X_\Lambda\ne Y_\Lambda)
   +\Pp(\|Y_\Lambda-Y\|>\varepsilon/3),
\]
and the right-hand side tends to zero.  Hence $X=Y$ almost surely.  Applying
this observation to the localized field-source triples proves the last
assertion.
\end{proof}

\begin{theorem}[Localized solutions and independence of localization]
\label{thm:noise-to-solution}
For each adapted localization pair and each fixed-profile dyadic cutoff
family, there are an event of probability one and a strictly positive random
time such that the finite-cutoff fixed points, reconstructed fields, and
nonlinear sources converge along the full dyadic sequence.  For an arbitrary
admissible cofinal sequence, the corresponding convergence holds locally in
probability on the deterministic contraction events
\eqref{eq:deterministic-contraction-event}; these events exhaust probability one as
in \eqref{eq:contraction-events-exhaust}.

If two localization pairs are adapted to the same compact cone, their limiting
fields and sources agree on the overlap up to the minimum existence time.
For every compact $K$ and every $T$ below the two existence times, the
common restriction to $[0,T]\times K$ satisfies \eqref{eq:main-system} in the
sense of Definition~\ref{def:local-solution}.  It is independent of the
localization pair and of the admissible spectral cutoff profile.
\end{theorem}

\begin{proof}
Fix an adapted localization pair $(\chi,\rho)$.  On the common dyadic event of
Corollary~\ref{cor:Xi-convergence}, choose a realization and a time
$T$ below its existence lifetime.  Corollary~\ref{cor:localized-cutoff-passage}
gives $Z_\Lambda\to Z$ in $\mathcal Z_T$ along the full dyadic sequence.
The decomposition
\[
 u_\Lambda^{\chi,\rho}-u^{\chi,\rho}
 =\rho(\Psi_{\W,\Lambda}-\Psi_{\W})
  +\rho(V_{\W,\Lambda}-V_{\W})
  +(X_{\W,\Lambda}-X_{\W})
  +(Y_{\W,\Lambda}-Y_{\W})
\]
(and the analogous identity in the $\K$ channel) shows convergence in
$C_T\mathcal D'(\R^3)$.  Proposition~\ref{prop:source-convergence} gives
convergence of the localized nonlinear terms.

For two localization pairs adapted to the same compact cone,
\cref{prop:limit-localization-comparison} passes the finite-cutoff
finite-propagation identity to the field and source limits in fixed Banach
spaces.  It also shows that every admissible cofinal cutoff sequence converges
in probability to this same common restriction.  Thus both the localization
pair and the spectral profile give the same fields and nonlinear source on
$[0,T]\times K$.

It remains to identify the equation.  For every test function
$\phi\in C_c^\infty((0,T)\times\R^3)$,
\[
 \langle\zeta_{a,\Lambda}^{\chi,\rho},\phi\rangle
 =\langle B_{a,\Lambda}^\rho,L_a\phi\rangle
  -\langle\chi\Theta_\Lambda,\phi\rangle.
\]
The baseline and $\Theta$ convergence permit passage to the limit, so
$\zeta_a^{\chi,\rho}=L_aB_a^\rho-\chi\Theta$ in distributions.  Passing to
the limit in the cutoff equations then gives
\[
 L_a u_a^{\chi,\rho}
 =\mathcal N^{\chi,\rho}+\zeta_a^{\chi,\rho}
 \quad\text{in }\mathcal D'((0,T)\times\R^3),
\]
with $u_{\W}^{\chi,\rho}=u^{\chi,\rho}$ and
$u_{\K}^{\chi,\rho}=v^{\chi,\rho}$.  On every inner cylinder on which
$\chi=\rho=1$, Lemma~\ref{lem:auxiliary-forcing-local} gives
$\zeta_a^{\chi,\rho}=\xi_a$, while
$\mathcal N^{\chi,\rho}$ is the selected local limit of the unweighted
products.  Thus the restriction satisfies \eqref{eq:main-system} in
the sense of Definition~\ref{def:local-solution}.  Applying the preceding
argument on the fixed compact exhaustion produces the compatible family of local solutions.
\end{proof}

\begin{proof}[Proof of Theorem~\ref{thm:local-wp}]
The required stochastic distributions and operators are constructed in
\cref{thm:resonant-operators,thm:first-picard-fullspace,thm:cubic-fullspace}.
The deterministic solution map and uniqueness in $\mathcal Z_T$ follow from
\cref{thm:deterministic-closure,prop:interval-local-uniqueness}.
Strong measurability and compatibility across deterministic horizons follow
from
\cref{cor:measurable-paracontrolled-fixed-point,cor:horizon-gluing}.
\Cref{prop:measurable-lifetimes} constructs compatible measurable lifetimes,
and \cref{thm:noise-to-solution} proves independence of the localization
pair.  The perturbative estimate \eqref{eq:perturbed-fixed-point} gives local
Lipschitz dependence.
\end{proof}

\begin{proof}[Proof of Theorem~\ref{thm:spectral-approximation}]
For fixed-profile dyadic cutoffs, Corollary~\ref{cor:Xi-convergence},
the deterministic stability
estimate, and Corollary~\ref{cor:localized-cutoff-passage} give
\eqref{eq:main-fixed-point-convergence}.  For the fields, write the difference
as the sum of the localized $\Psi$-difference, the localized $V$-difference,
and the four fixed-point differences, as in the proof of
Theorem~\ref{thm:noise-to-solution}.  Every summand converges in
$C_T\mathcal D'(\R^3)$, which proves \eqref{eq:main-field-convergence}.
Proposition~\ref{prop:source-convergence} gives convergence of the three
source terms.  Multiplying the finite-cutoff identity
\eqref{eq:finite-source-identity} by $\eta_K$ and using
$\eta_K\chi_K^\circ=\eta_K$ yields
\[
 \eta_Ku_\Lambda^\circ v_\Lambda^\circ
 =\eta_K\bigl(\chi_K^\circ\Theta_\Lambda
 +\mathscr F_{X,\Lambda}^{\chi_K^\circ,\rho_K^\circ}(Z_\Lambda^\circ)
 +\mathscr F_{Y,\Lambda}^{\chi_K^\circ,\rho_K^\circ}(Z_\Lambda^\circ)\bigr).
\]
The source convergence permits passage to the limit and proves the localized
product statement.

For an arbitrary admissible cofinal sequence
$(\pi_{\Lambda_n})_{n\ge1}$, set
\[
 \mathcal A_{n,T}^K
 :=\mathcal A_T(\Xi_{\Lambda_n}^\circ,y_{\Lambda_n}),
 \qquad
 Z_{n,T}^\circ
 :=\mathscr S_T(\Xi_{\Lambda_n}^\circ,y_{\Lambda_n}).
\]
Proposition~\ref{prop:random-lifespan-cutoff-passage} gives convergence in
the sense of Definition~\ref{def:partial-probability-convergence}; the field
and source statements follow from the locally Lipschitz reconstruction and
source maps.  Theorem~\ref{thm:noise-to-solution} identifies the limits formed
with different localization pairs on every common backward cone.
\end{proof}

\begin{remark}[The cubic contraction]
\label{rem:cubic-renormalization}
The independent cross quadratic source has no scalar Wick constant.  At
cubic order, the same-color contraction produces the first-chaos components
$\Gamma_{\W}^{(1)}$ and $\Gamma_{\K}^{(1)}$.  They are included in
$\Gamma_{\W}$ and $\Gamma_{\K}$ because they occur in the finite-cutoff
product formula.  Thus $\Gamma_a^{\chi,\rho}$ is the sum of the centered
third chaos and the integrated first-chaos contraction.
\end{remark}

\appendix

\section{Phase geometry and first Picard smoothing}\label{app:first-picard}
\subsection{The distinct-speed phase gap}\label{app:phase-gap}

We prove the low--high phase estimate uniformly on separated parameter
classes.  The argument is symmetric in the two propagation speeds.

For $a,b\in\mathfrak C$ with $a\ne b$, set
\begin{equation}\label{eq:two-frequency-phases}
  \Phi_{a|b}^{\eps_1,\eps_2}(q,\ell)
  =\eps_1\omega_a(q+\ell)+\eps_2\omega_b(\ell),
  \qquad \eps_1,\eps_2\in\{-1,1\}.
\end{equation}

\begin{lemma}[Proportional distinct-speed phase gap]\label{lem:explicit-phase-gap}
Let $\parvec$ satisfy \eqref{eq:parameter-vector}.  There exist
$\delta_0=\delta_0(\parvec)>0$, $N_{\mathrm{sp}}=N_{\mathrm{sp}}(\parvec)\ge2$,
and $c_{\mathrm{sp}}=c_{\mathrm{sp}}(\parvec)>0$ such that, whenever
\[
  |\ell|\sim N,\qquad N\ge N_{\mathrm{sp}},
  \qquad |q|\le\delta_0|\ell|,
\]
one has
\begin{equation}\label{eq:parametric-difference-gap}
  |\omega_a(\ell+q)-\omega_b(\ell)|
  \ge c_{\mathrm{sp}}N,
  \qquad a\ne b.
\end{equation}
Every sum phase in \eqref{eq:two-frequency-phases} has the same type of lower
bound.  Consequently
\begin{equation}\label{eq:all-phase-gap}
  |\Phi_{a|b}^{\eps_1,\eps_2}(q,\ell)|
  \ge c_{\mathrm{sp}}N
\end{equation}
for every sign branch generated by a same-color block.  The constants
$\delta_0$, $N_{\mathrm{sp}}$, and $c_{\mathrm{sp}}$ depend only on the
dispersive parameters and the Littlewood--Paley support constants; they are
independent of $\mathfrak h_{\W}$ and $\mathfrak h_{\K}$.
\end{lemma}

\begin{proof}
Write
\[
  e_m(r):=\sqrt{\massK^2+\speedK^2r^2}-\speedK r.
\]
For $r>0$,
\begin{equation}\label{eq:mass-error}
  0\le e_m(r)
  =\frac{\massK^2}
  {\sqrt{\massK^2+\speedK^2r^2}+\speedK r}
  \le \frac{\massK^2}{2\speedK r}.
\end{equation}
Let $r=|\ell|$.  In either orientation of the wave--Klein--Gordon difference,
the reverse triangle inequality and
$||\ell+q|-|\ell||\le|q|$ give
\begin{equation}\label{eq:raw-parametric-gap}
  |\omega_a(\ell+q)-\omega_b(\ell)|
  \ge \speedgap r-\speedmax|q|-C_{\parvec}r^{-1},
  \qquad a\ne b,
\end{equation}
provided $|q|\le r/2$; in the Klein--Gordon-at-$\ell+q$ orientation, the last
term follows from \eqref{eq:mass-error} and $|\ell+q|\ge r/2$.
Choose
\[
  0<\delta_0<\frac{\speedgap}{4\speedmax}
\]
and then choose $N_{\mathrm{sp}}$ so large that
$C_{\parvec}r^{-1}\le(\speedgap/4)r$ for $r\ge N_{\mathrm{sp}}/2$.
Equation \eqref{eq:raw-parametric-gap} then gives
\[
  |\omega_a(\ell+q)-\omega_b(\ell)|
  \ge \frac{\speedgap}{2}|\ell|
  \gtrsim_{\parvec} N.
\]
For a sum phase, put $c_{\min}:=\min\{\speedW,\speedK\}$.  Since
$\omega_{\W}(\zeta)\ge c_{\min}|\zeta|$ and
$\omega_{\K}(\zeta)\ge c_{\min}|\zeta|$, one has
\[
  \omega_a(\ell+q)+\omega_b(\ell)
  \ge c_{\min}(|\ell+q|+|\ell|)\gtrsim_{\parvec}N.
\]
This proves all sign branches.
\end{proof}

\begin{proposition}[Uniform phase constants on separated parameter classes]
\label{prop:uniform-phase-class}
Fix a class $\mathfrak P(\underline c,\overline c,
\underline m,\overline m,\delta_*)$.  The constants in
Lemma~\ref{lem:explicit-phase-gap} may be chosen uniformly on that class.  One
may take, after changing the fixed Littlewood--Paley constants,
\[
  \delta_0=\min\left\{\frac14,
  \frac{\delta_*}{8\overline c}\right\},
  \qquad
  c_{\mathrm{sp}}
  \ge \kappa_{\mathrm{ann}}
      \min\{\delta_*,\underline c\}>0,
\]
and
\[
  N_{\mathrm{sp}}
  \le C_{\mathrm{ann}}\left(1+
  \frac{\overline m}{\sqrt{\underline c\,\delta_*}}\right),
\]
where $\kappa_{\mathrm{ann}},C_{\mathrm{ann}}>0$ depend only on the fixed
annular support constants.
The radial derivative lower bounds used in the continuous phase-layer
arguments and the scalar wave/Klein--Gordon Strichartz constants are likewise
uniform on the class.  These constants are independent of the forcing orders
and profiles.
\end{proposition}

\begin{proof}
In \eqref{eq:mass-error}, the coefficient of $r^{-1}$ is bounded by
$\overline m^2/(2\underline c)$.  With the displayed choice of $\delta_0$, the
low-frequency displacement costs at most a fixed fraction of $\delta_*r$.
Choosing $N_{\mathrm{sp}}^2\gtrsim
\overline m^2/(\underline c\,\delta_*)$ absorbs the mass error uniformly.  The
KG radial derivative is
\[
  \omega_{\K}'(r)=
  \frac{\speedK^2r}{\sqrt{\massK^2+\speedK^2r^2}},
\]
which has a positive uniform lower bound once $r$ is above a class-dependent
unit threshold.  The remaining linear estimates depend continuously on the
parameters on compact positive sets.
\end{proof}

\begin{corollary}[Normalized coercivity above the finite threshold]
\label{cor:normalized-coercivity}
With the aperture in \eqref{eq:aperture},
\[
  \inf_{N\ge N_{\mathrm{sp}}}
  \inf_{|\ell|\sim N,\ |q|\le\delta_0|\ell|}
  N^{-1}|\Phi_{a|b}^{\eps_1,\eps_2}(q,\ell)|
  \ge c_{\mathrm{sp}}>0.
\]
The shells $N<N_{\mathrm{sp}}$ are finite in number and are estimated without
phase division.
\end{corollary}

\begin{remark}[Possible compact crossing]
If $\speedK<\speedW$, the equation
$\speedW r=\sqrt{\massK^2+\speedK^2r^2}$ can have one positive solution.  This
is a compact-frequency phenomenon.  The proof above starts only at
$N_{\mathrm{sp}}$, and the Volterra estimate treats every lower shell directly;
therefore the main theorem requires only $\speedW\ne\speedK$, not an ordering.
\end{remark}

\begin{remark}[Equal-speed phase]
If $\speedW=\speedK=:c_0$, then
\[
 \sqrt{\massK^2+c_0^2|\ell|^2}-c_0|\ell|
 \simeq |\ell|^{-1}
 \qquad(|\ell|\gg1).
\]
Thus the difference phase used in the same-color Volterra contraction is no
longer of order $|\ell|$.  The analysis below is consequently formulated on
separated-speed parameter classes.
\end{remark}

The phase denominator is used only above the high-frequency threshold.  At
$\xi=0$,
\[
  \frac{\sin(t\speedW|\xi|)}{\speedW|\xi|}\longrightarrow t,
\]
so the massless Duhamel multiplier extends continuously.  The inhomogeneous
unit block is treated directly.

This appendix proves the half-derivative gain for the first Picard terms.
Fourier coefficients are replaced by covariance spectral densities, and the
inhomogeneous unit block containing the wave zero mode is estimated directly
before the propagators are decomposed into sign branches.  The high-frequency
interactions are governed by the phase
\begin{equation}\label{eq:picard-phase}
  \Phi_a^{\bm\eps}(\xi,\eta)
  :=\eps_0\omega_a(\xi)
    +\eps_1\wW(\eta)
    +\eps_2\wK(\xi-\eta),
  \qquad \bm\eps\in\{\pm1\}^3.
\end{equation}
The decisive geometric simplification is that one input is always Klein--Gordon.  In two-center coordinates, differentiation in the Klein--Gordon radius has size comparable to one on every high shell.  The only region not covered by such a phase-layer estimate is the low-output high--high region, where the speed gap gives a pointwise phase lower bound.

\subsection{Kernel representation and oscillatory estimates}

Let $(m_\Lambda)_\Lambda$ be an admissible spectral family and put
$\pi_\Lambda=m_\Lambda(D)$.  The cutoff argument uses the uniform symbol
bounds, the plateau and support conditions, and pointwise convergence.  A
cutoff difference is expanded against the limiting multiplier; each term then
contains a factor $m_\Lambda(\zeta)-1$ supported on
$|\zeta|\ge c_0\Lambda$.  The almost-sure statement is formulated for a
fixed-profile dyadic family.
The stochastic convolutions and first Picard objects are those in
\eqref{eq:stochastic-convolutions}--\eqref{eq:first-picard-def}.  The variant
with the additional output projection
\[
  V^{\mathrm{out}}_{a,\Lambda}
  :=I_a\pi_\Lambda
      (\Psi_{\W,\Lambda}\Psi_{\K,\Lambda})
\]
is the output-projected first-Picard coordinate used in the cutoff comparison.

\begin{theorem}[Localized first Picard smoothing]\label{thm:first-picard-fullspace}
Fix an admissible spectral family and interpret $\Lambda\to\infty$ along
an arbitrary admissible cofinal sequence in the sense of
Definition~\ref{def:admissible-cutoff}.  Let $T<\infty$,
$a\in\mathfrak C$, $\chi\in C_c^\infty(\R^3)$,
$\kappa>0$, and $2\le p<\infty$.  Then there are cutoff-independent random
distributions $V_a$ and $\partial_tV_a$ such that
\begin{align}
  \chi V_{a,\Lambda}&\longrightarrow \chi V_a
  &&\text{in }L^p\!\left(\Omega;
  C_T\cC^{\frac12-\betasum-\kappa}
  \cap L_T^\infty B_{2,\infty}^{\frac12-\betasum-\kappa}\right),
  \label{eq:first-picard-convergence}\\
  \chi\partial_tV_{a,\Lambda}&\longrightarrow \chi\partial_tV_a
  &&\text{in }L^p\!\left(\Omega;
  C_T\cC^{-\frac12-\betasum-\kappa}
  \cap L_T^\infty B_{2,\infty}^{-\frac12-\betasum-\kappa}\right).
  \label{eq:first-picard-derivative-convergence}
\end{align}
The same limits are obtained from $V^{\mathrm{out}}_{a,\Lambda}$.  If
the family is fixed-profile and $\Lambda=2^j$, the convergence holds almost
surely along the entire dyadic sequence on one event, simultaneously for any
prescribed countable family of compactly supported cutoffs, rational time
horizons, rational losses $\kappa>0$, and integer probability moments.  In
particular, one may use the predetermined localizers attached to the compact
exhaustion in Proposition~\ref{prop:measurable-lifetimes}.  Consequently,
\begin{equation}\label{eq:first-picard-regularity}
  V_a\in C_T\cC_{\loc}^{\frac12-\betasum-},
  \qquad
  \partial_tV_a\in C_T\cC_{\loc}^{-\frac12-\betasum-},
\end{equation}
and both objects possess the corresponding local $B_{2,\infty}$ shell envelopes.
\end{theorem}

The proof gives the following more precise spectral statement.  For $\nu\in\{0,1\}$, write
\[
  \partial_t^0V_{a,\Lambda}=V_{a,\Lambda},
  \qquad
  \partial_t^1V_{a,\Lambda}=\partial_tV_{a,\Lambda}.
\]
At finite cutoff, stationarity yields a nonnegative covariance density $q_{a,\nu,\Lambda}$ satisfying
\begin{equation}\label{eq:spectral-density-definition-intro}
\begin{aligned}
 &\E\left[
   \widehat{\partial_t^\nu V_{a,\Lambda}}(t,\xi)
   \overline{
   \widehat{\partial_t^\nu V_{a,\Lambda}}(t,\xi')}
   \right]\\
 &\hspace{4cm}
   =(2\pi)^3\delta(\xi-\xi')
   q_{a,\nu,\Lambda}(t,\xi),
\end{aligned}
\end{equation}
where the identity is understood after testing in $(\xi,\xi')$.  There is an analogous increment density $q^{\Delta}_{a,\nu,\Lambda}(t,t';\xi)$.

\begin{proposition}[Spectral first Picard bounds]\label{prop:first-picard-spectral}
For every $\eps>0$,
\begin{equation}\label{eq:first-picard-spectral-static}
  \sup_{\Lambda\ge1}\sup_{t\le T}
  q_{a,\nu,\Lambda}(t,\xi)
  \lesssim_{T,\eps}
  \la\xi\ra^{-4+2\nu+2\betasum+\eps},
  \qquad \nu=0,1.
\end{equation}
For every $0<\theta<\frac12-\betasum$,
\begin{equation}\label{eq:first-picard-spectral-increment}
  \sup_{\Lambda\ge1}
  q^{\Delta}_{a,\nu,\Lambda}(t,t';\xi)
  \lesssim_{T,\eps,\theta}
  |t-t'|^{2\theta}
  \la\xi\ra^{-4+2\nu+2\betasum+2\theta+\eps}.
\end{equation}
For every finite choice of input and output cutoff flags, the difference from
the limiting kernel satisfies the same static and increment estimates with
the weighted phase integral restricted to the union of the corresponding
high-frequency tail sets.  This reduction, and comparison of two non-nested
cutoffs through the same limiting kernel, follow from
\cref{lem:finite-product-cutoff}.
\end{proposition}

The exponent $-4+2\betasum$ in
\eqref{eq:first-picard-spectral-static} is the full-space covariance-density
form of the dyadic estimate.  On $|\xi|\sim N$ it gives
\[
  \E|P_NV_a(t,x)|^2
  \lesssim \int_{|\xi|\sim N}N^{-4+2\betasum+}\dd\xi
  \lesssim N^{-1+2\betasum+},
\]
so a dyadic block has size $N^{-1/2+\betasum+}$, corresponding to
$\cC^{1/2-\betasum-}$.

We use the conjugated covariance convention in this section.  Let $\mathcal W_{\W}$ and $\mathcal W_{\K}$ be independent complex Gaussian random measures on $[0,T]\times\R^3$, with the reality relation
\[
  \mathcal W_a(\dd r,-\dd\eta)
  =\overline{\mathcal W_a(\dd r,\dd\eta)},
\]
and covariance
\begin{equation}\label{eq:conjugated-isonormal-covariance}
  \E\bigl[
  \mathcal W_a(\dd r,\dd\eta)
  \overline{\mathcal W_b(\dd s,\dd\zeta)}
  \bigr]
  =\delta_{ab}\delta(r-s)\delta(\eta-\zeta)
  \dd r\dd s\dd\eta\dd\zeta.
\end{equation}
This normalized random-measure realization is equivalent to the isonormal
formulation in \eqref{eq:fourier-gaussian-forcing}; the fixed Fourier
normalization and the forcing profiles are absorbed into the deterministic
Wiener kernels.  Independence means that
the product of one first-chaos integral of color $\W$ and one of color
$\K$ is an ordered colored double Wiener integral with no contraction.

Set
\begin{equation}\label{eq:picard-kernels-linear}
  K_a^{[0]}(t,\xi)
  :=\frac{\sin(t\omega_a(\xi))}{\omega_a(\xi)},
  \qquad
  K_a^{[1]}(t,\xi)
  :=\cos(t\omega_a(\xi)),
\end{equation}
with the usual continuous interpretation in the wave channel at $\xi=0$.  For $0\le r_1,r_2\le t$, define the deterministic three-propagator kernel
\begin{equation}\label{eq:three-propagator-kernel}
\begin{aligned}
  \mathcal K_{a,\nu}
  (\xi,\eta;t,r_1,r_2)
  :=\int_{r_1\vee r_2}^{t}
  &K_a^{[\nu]}(t-s,\xi)
   K_{\W}^{[0]}(s-r_1,\eta)\\
  &\times K_{\K}^{[0]}(s-r_2,\xi-\eta)\dd s.
\end{aligned}
\end{equation}
The kernel is set equal to zero when $r_1\vee r_2>t$.

\begin{lemma}[Stochastic Fubini and covariance density]\label{lem:picard-stochastic-fubini}
At finite cutoff, $\partial_t^\nu V_{a,\Lambda}$ is a homogeneous colored second chaos.  Its covariance spectral density is
\begin{equation}\label{eq:picard-density-kernel}
\begin{aligned}
  q_{a,\nu,\Lambda}(t,\xi)
  =c_{\mathrm F}
  \int_{\R^3}\int_0^t\int_0^t
  &|m_\Lambda(\eta)\mathfrak h_{\W}(\eta)|^2
   |m_\Lambda(\xi-\eta)\mathfrak h_{\K}(\xi-\eta)|^2\\
  &\times
  |\mathcal K_{a,\nu}
  (\xi,\eta;t,r_1,r_2)|^2
  \dd r_1\dd r_2\dd\eta,
\end{aligned}
\end{equation}
where $c_{\mathrm F}>0$ depends only on the Fourier normalization.  The formula remains valid with an extra factor $|m_\Lambda(\xi)|^2$ for the auxiliary output projection.
\end{lemma}

\begin{proof}
Let $\varphi\in\mathcal S(\R^3)$.  At finite cutoff, substitute the two stochastic convolution formulas into
\[
  \langle\partial_t^\nu V_{a,\Lambda}(t),\varphi\rangle
  =\int_0^t
  \left\langle
  K_a^{[\nu]}(t-s,D)
  (\Psi_{\W,\Lambda}(s)\Psi_{\K,\Lambda}(s)),
  \varphi\right\rangle\dd s.
\]
The resulting deterministic kernel belongs to
\[
  L^2([0,T]\times\R^3)_{\W}
  \widehat\otimes
  L^2([0,T]\times\R^3)_{\K}.
\]
The ordered colored double Wiener integral is a continuous isometry on this Hilbert tensor product.  Hence the deterministic $s$-integration can be commuted with the Wiener integral.  The $s$-integration is \eqref{eq:three-propagator-kernel}.  The two deterministic forcing profiles contribute the factors $\mathfrak h_{\W}(\eta)$ and $\mathfrak h_{\K}(\zeta)$.  Applying the colored double Wiener isometry and changing variables from the two input frequencies $(\eta,\zeta)$ to $(\xi,\eta)$ with $\xi=\eta+\zeta$ gives
\[
  \E\bigl|
  \langle\partial_t^\nu V_{a,\Lambda}(t),\varphi\rangle
  \bigr|^2
  =c_{\mathrm F}
  \int_{\R^3}|\widehat\varphi(\xi)|^2
  q_{a,\nu,\Lambda}(t,\xi)\dd\xi.
\]
Polarization yields the tested form of \eqref{eq:spectral-density-definition-intro} and proves \eqref{eq:picard-density-kernel}.  The output projection is a deterministic Fourier multiplier and contributes the stated factor.
\end{proof}

Put
\[
  N_0:=\la\xi\ra,
  \qquad
  N_1:=\la\eta\ra,
  \qquad
  N_2:=\la\xi-\eta\ra,
  \qquad
  N_{\max}:=\max(N_0,N_1,N_2),
\]
and
\begin{equation}\label{eq:modulation-weight}
  \mathfrak m(\lambda):=\min\{1,|\lambda|^{-1}\},
  \qquad \mathfrak m(0):=1.
\end{equation}
The low-frequency unit blocks are estimated without splitting the wave sine kernel into exponentials.

\begin{lemma}[Rough and oscillatory kernel bounds]\label{lem:picard-kernel-bounds}
Uniformly for $0\le r_1,r_2\le t\le T$ and $\nu=0,1$,
\begin{equation}\label{eq:picard-kernel-rough}
  |\mathcal K_{a,\nu}
  (\xi,\eta;t,r_1,r_2)|
  \lesssim_T
  \frac{N_0^\nu}{N_0N_1N_2}.
\end{equation}
If
\[
  |\xi|\ge1,
  \qquad |\eta|\ge1,
  \qquad |\xi-\eta|\ge1,
\]
then
\begin{equation}\label{eq:picard-kernel-oscillatory}
  |\mathcal K_{a,\nu}|
  \lesssim_T
  \frac{N_0^\nu}{N_0N_1N_2}
  \sum_{\bm\eps\in\{\pm1\}^3}
  \mathfrak m(\Phi_a^{\bm\eps}(\xi,\eta)).
\end{equation}
Moreover, if $0<\theta\le1/2$ and
$0\le r_1,r_2\le t\wedge t'$, then the corresponding time difference satisfies
\begin{equation}\label{eq:picard-kernel-increment}
\begin{aligned}
  &|\mathcal K_{a,\nu}
  (\xi,\eta;t,r_1,r_2)
  -\mathcal K_{a,\nu}
  (\xi,\eta;t',r_1,r_2)|\\
  &\quad\lesssim_{T,\theta}
  |t-t'|^\theta
  \frac{N_0^\nu N_{\max}^\theta}{N_0N_1N_2}
  \sum_{\bm\eps\in\{\pm1\}^3}
  \mathfrak m(\Phi_a^{\bm\eps}(\xi,\eta))
\end{aligned}
\end{equation}
whenever the three displayed frequencies are at least one.  Without that high-frequency condition, the same estimate holds with the sign sum on the right replaced by $1$.
\end{lemma}

\begin{proof}
The elementary estimates
\[
  \left|\frac{\sin(t\speedW|z|)}{\speedW|z|}\right|
  \lesssim_{T,\parvec}\la z\ra^{-1},
  \qquad
  \left|\frac{\sin(t\wK(z))}{\wK(z)}\right|
  \lesssim_{\parvec}\la z\ra^{-1},
  \qquad
  |\cos(t\omega_a(z))|\le1
\]
give \eqref{eq:picard-kernel-rough} after integration in $s$.

On the high-frequency region, expand the three propagators into exponentials.  Every sign branch equals a unimodular endpoint factor times
\[
  \kappa_{\nu,\bm\eps}
  \frac{N_0^\nu}{N_0N_1N_2}
  \int_{r_1\vee r_2}^{t}
  e^{\ii s\Phi_a^{\bm\eps}(\xi,\eta)}\dd s,
\]
up to constants uniform on the dyadic supports.  Since
\[
  \left|\int_A^B e^{\ii s\lambda}\dd s\right|
  \lesssim_T\mathfrak m(\lambda),
\]
we obtain \eqref{eq:picard-kernel-oscillatory}.

For the increment, the change of the endpoint factor is bounded by
\[
  |t-t'|^\theta N_0^\theta.
\]
The change in the upper integration limit obeys
\[
  \left|\int_t^{t'}e^{\ii s\lambda}\dd s\right|
  \lesssim_{T,\theta}
  |t-t'|^\theta
  N_{\max}^\theta\mathfrak m(\lambda),
\]
because $|\lambda|\lesssim N_{\max}$ and
$\mathfrak m(\lambda)^{1-\theta}
\lesssim N_{\max}^\theta\mathfrak m(\lambda)$.  This proves \eqref{eq:picard-kernel-increment}.  In a unit block, use the rough bound and the inequalities
$|e^{\ii h\lambda}-1|\lesssim|h|^\theta\la\lambda\ra^\theta$ and
$|h|\lesssim_T|h|^\theta$.
\end{proof}

\subsection{Phase layers and covariance estimates}

The next estimates are formulated directly in continuous frequency.  For fixed $\xi\ne0$, set
\[
  d:=|\xi|,
  \qquad r:=|\eta|,
  \qquad \lambda:=|\xi-\eta|.
\]
On the triangle region $|r-\lambda|\le d\le r+\lambda$, two-center coordinates give
\begin{equation}\label{eq:two-center-volume-wavekg}
  \dd\eta=2\pi\frac{r\lambda}{d}\dd r\dd\lambda.
\end{equation}
The phase in \eqref{eq:picard-phase}, with $\xi$ fixed, is
\begin{equation}\label{eq:phase-two-center}
  \Phi_a^{\bm\eps}
  =\eps_0\omega_a(\xi)
   +\eps_1\speedW r
   +\eps_2\sqrt{\massK^2+\speedK^2\lambda^2}.
\end{equation}

\begin{lemma}[Balanced and high--low phase layers]\label{lem:continuous-phase-layers}
There is a constant $C\ge1$ such that the following statements hold for every sign branch and every $L\ge1$.

\begin{enumerate}[label=\textup{(\roman*)},leftmargin=2.4em]
\item If
\[
  d\sim r\sim\lambda\sim M,
  \qquad M\ge2,
\]
then
\begin{equation}\label{eq:balanced-layer-volume}
  \left|
  \left\{\eta:
  d\sim r\sim\lambda\sim M,
  \ |\Phi_a^{\bm\eps}(\xi,\eta)|\le L
  \right\}
  \right|
  \lesssim M^2\min\{L,M\}.
\end{equation}

\item If $d\sim\lambda\sim M$, $r\sim R$, and
$2\le R\le M/C$, then
\begin{equation}\label{eq:wave-low-layer-volume}
  \left|
  \left\{\eta:
  r\sim R,\ \lambda\sim M,
  \ |\Phi_a^{\bm\eps}(\xi,\eta)|\le L
  \right\}
  \right|
  \lesssim R^2\min\{L,R\}.
\end{equation}

\item If $d\sim r\sim M$, $\lambda\sim R$, and
$2\le R\le M/C$, then
\begin{equation}\label{eq:kg-low-layer-volume}
  \left|
  \left\{\eta:
  r\sim M,\ \lambda\sim R,
  \ |\Phi_a^{\bm\eps}(\xi,\eta)|\le L
  \right\}
  \right|
  \lesssim R^2\min\{L,R\}.
\end{equation}
\end{enumerate}
\end{lemma}

\begin{proof}
In the balanced region the density in \eqref{eq:two-center-volume-wavekg} is comparable to $M$.  For fixed $r$,
\[
  \left|\partial_\lambda
  \sqrt{\massK^2+\speedK^2\lambda^2}\right|
  =\frac{\speedK^2\lambda}
  {\sqrt{\massK^2+\speedK^2\lambda^2}}
  \gtrsim_{\parvec}1.
\]
Thus the admissible $\lambda$-set has length $O(\min\{L,M\})$.  The $r$-interval has length $O(M)$, proving \eqref{eq:balanced-layer-volume}.

In the wave-low region, the triangle inequalities force the effective $\lambda$-interval to have length $O(R)$ around $d$.  The two-center density is comparable to $R$, the $r$-interval has length $O(R)$, and the same derivative in $\lambda$ gives an admissible interval of length $O(\min\{L,R\})$.  This proves \eqref{eq:wave-low-layer-volume}.

In the Klein--Gordon-low region, the effective $r$-interval has length $O(R)$ around $d$, and the density is again comparable to $R$.  For fixed $\lambda$, the derivative of \eqref{eq:phase-two-center} in $r$ is $\eps_1\speedW$.  Hence the admissible $r$-interval has length $O(\min\{L,R\})$, proving \eqref{eq:kg-low-layer-volume}.
\end{proof}

\begin{lemma}[Low-output high--high gap]\label{lem:first-picard-low-output-gap}
There exist $C_0=C_0(\parvec)\gg1$ and $c_0=c_0(\parvec)>0$ such that, whenever
\[
  |\eta|\sim|\xi-\eta|\sim M,
  \qquad
  M\ge C_0\la\xi\ra,
\]
one has, for every $a\in\mathfrak C$ and every sign branch,
\begin{equation}\label{eq:first-picard-low-output-gap}
  |\Phi_a^{\bm\eps}(\xi,\eta)|\ge c_0M.
\end{equation}
\end{lemma}

\begin{proof}
The outer term is $O_{\parvec}(\la\xi\ra)$.  If the two input signs agree, their sum has size comparable to $M$.  If they disagree, use
\[
  \left|\wK(\xi-\eta)-\wW(\eta)\right|
  \ge \speedgap|\eta|-\speedmax|\xi|
       -C_{\parvec}|\eta|^{-1}.
\]
Since
\[
  \bigl||\xi-\eta|-|\eta|\bigr|\le|\xi|,
\]
the input difference is bounded below by $c_{\parvec}M-C_{\parvec}(|\xi|+M^{-1})$.  Taking $C_0=C_0(\parvec)$ sufficiently large and absorbing finitely many unit shells makes both this difference and the same-sign sum dominate the outer term.  This proves \eqref{eq:first-picard-low-output-gap}.
\end{proof}

\begin{lemma}[Exhaustion of the three-frequency configurations]
\label{lem:first-picard-region-exhaustion}
Assume
$N_0=\langle\xi\rangle$, $N_1=\langle\eta\rangle$, and
$N_2=\langle\xi-\eta\rangle$ are all at least two.  After enlarging the
fixed comparability constants, the dyadic frequency space is the union of the
following four regions, with only finite overlap:
\begin{align*}
 &N_0\sim N_1\sim N_2,\\
 &N_0\sim N_2\gg N_1,\\
 &N_0\sim N_1\gg N_2,\\
 &N_1\sim N_2\gg N_0.
\end{align*}
In the first three regions the sublevel estimates of
\cref{lem:continuous-phase-layers} apply to every sign vector
$\bm\eps\in\{\pm1\}^3$.  In the last region the pointwise estimate of
\cref{lem:first-picard-low-output-gap} applies to every sign vector.
\end{lemma}

\begin{proof}
The triangle inequalities imply that the two largest numbers among
$N_0,N_1,N_2$ are comparable.  If the smallest one is comparable to the
other two, the first region occurs.  If it is strictly smaller, it is one of
$N_0,N_1,N_2$, giving respectively the fourth, second, or third region.
The derivative estimates in \cref{lem:continuous-phase-layers} are estimates
for absolute values and are therefore independent of the signs.  In the
low-output region, equal signs on the two input phases give their sum, while
opposite signs give their difference; these are the two cases in
the proof of \cref{lem:first-picard-low-output-gap}.
\end{proof}

The estimates used in the weighted phase integral are summarized below.  The
last column is the contribution after the modulation sum on one fixed dyadic
configuration; the remaining small-scale summation is performed in the proof
of \cref{prop:continuous-phase-integral}.
\begin{center}
\small
\begin{tabularx}{\textwidth}{@{}p{0.19\textwidth}p{0.20\textwidth}p{0.20\textwidth}Y@{}}
\toprule
region & phase control & sublevel volume & weighted contribution \\
\midrule
$N_0\sim N_1\sim N_2\sim M$
& differentiate in $\lambda=|\xi-\eta|$
& $M^2\min\{L,M\}$
& $M^{-4+2\nu+2\beta_\Sigma+2\theta}$ \\
$N_0\sim N_2\sim M$, $N_1\sim R\ll M$
& differentiate in $\lambda$
& $R^2\min\{L,R\}$
& $M^{-4+2\nu+2\beta_{\K}+2\theta}R^{2\beta_{\W}}$ \\
$N_0\sim N_1\sim M$, $N_2\sim R\ll M$
& differentiate in $r=|\eta|$
& $R^2\min\{L,R\}$
& $M^{-4+2\nu+2\beta_{\W}+2\theta}R^{2\beta_{\K}}$ \\
$N_1\sim N_2\sim M\gg N_0$
& $|\Phi_a^{\bm\eps}|\gtrsim M$
& $O(M^3)$
& $N_0^{-2+2\nu}M^{-3+2\beta_\Sigma+2\theta}$ \\
\bottomrule
\end{tabularx}
\end{center}

Let
\begin{equation}\label{eq:picard-amplitude-profile}
\begin{aligned}
  \mathfrak A_a(\xi,\eta)
  :={}&\one_{\{\min(|\xi|,|\eta|,|\xi-\eta|)<1\}}\\
  &+\one_{\{\min(|\xi|,|\eta|,|\xi-\eta|)\ge1\}}
  \sum_{\bm\eps\in\{\pm1\}^3}
  \mathfrak m(\Phi_a^{\bm\eps}(\xi,\eta)).
\end{aligned}
\end{equation}
For $\nu\in\{0,1\}$ and $0\le\theta<\frac12-\betasum$, define the weighted phase
integral
\begin{equation}\label{eq:continuous-phase-integral}
  \mathfrak S_a^{(\nu,\theta)}(\xi)
  :=\int_{\R^3}
  \frac{N_0^{2\nu}N_{\max}^{2\theta}
        N_1^{2\betaW}N_2^{2\betaK}}
       {N_0^2N_1^2N_2^2}
  \mathfrak A_a(\xi,\eta)^2\dd\eta.
\end{equation}
The phase and its sublevel sets depend only on the dispersions.  The forcing
profiles enter \eqref{eq:continuous-phase-integral} through the anisotropic
weights $N_1^{2\betaW}N_2^{2\betaK}$.

\begin{proposition}[Weighted three-frequency phase integral]
\label{prop:continuous-phase-integral}
For every $\eps>0$, $\nu\in\{0,1\}$, and
$0\le\theta<\frac12-\betasum$,
\begin{equation}\label{eq:continuous-phase-integral-bound}
  \mathfrak S_a^{(\nu,\theta)}(\xi)
  \lesssim_{\eps,\theta,\parvec,\betaW,\betaK}
  \la\xi\ra^{-4+2\nu+2\betasum+2\theta+\eps}.
\end{equation}
\end{proposition}

\begin{proof}
The inhomogeneous unit blocks are bounded directly.  If $N_0\sim1$, then
$N_1\sim N_2$ outside a bounded set, and the rough part of
\eqref{eq:continuous-phase-integral} is bounded by
\[
  \int_{\R^3}\la\eta\ra^{-4+2\betasum+2\theta}\dd\eta
  \lesssim_{\betaW,\betaK,\theta}1,
\]
since $2\betasum+2\theta<1$.  If $N_0\gg1$ and $N_1\sim1$, then
$N_2\sim N_0$ and the contribution is
$O(N_0^{-4+2\nu+2\betaK+2\theta})$; if $N_2\sim1$, it is
$O(N_0^{-4+2\nu+2\betaW+2\theta})$.  Both are stronger than
\eqref{eq:continuous-phase-integral-bound}.  It remains to assume that all
three Euclidean frequencies are at least one.

For a fixed sign branch, use
\begin{equation}\label{eq:continuous-modulation-decomposition}
  \mathfrak m(\Phi)^2
  \lesssim
  \sum_{\substack{L\in\Dyd\\1\le L\lesssim N_{\max}}}
  L^{-2}\one_{\{|\Phi|\le L\}}.
\end{equation}
Since the sign set is finite, it suffices to estimate one branch.  Decompose
$N_1$ and $N_2$ dyadically.  The triangle relation implies that the two
largest among $N_0,N_1,N_2$ are comparable.

\emph{Balanced region.}
Suppose $N_0\sim N_1\sim N_2\sim M$.  The nonoscillatory factor is
$M^{-6+2\nu+2\betasum+2\theta}$.  Since the modulation parameter is dyadic,
\[
 \sum_{\substack{L\in\Dyd\\1\le L\lesssim M}}
 L^{-2}M^2\min\{L,M\}
 =M^2\sum_{1\le L\lesssim M}L^{-1}\lesssim M^2.
\]
Consequently
\[
 \mathrm{Bal}_M
 \lesssim M^{-4+2\nu+2\betasum+2\theta},
\]
which is stronger than the asserted bound.

\emph{Wave-low region.}
Suppose $N_0\sim N_2\sim M$ and $N_1\sim R\ll M$.  The wave-colored factor
is at scale $R$ and the Klein--Gordon-colored factor is at scale $M$, so the
prefactor is
\[
 P_{M,R}:=M^{-4+2\nu+2\betaK+2\theta}R^{-2+2\betaW}.
\]
For modulations $L\le R$, \eqref{eq:wave-low-layer-volume} gives
\[
 P_{M,R}\sum_{1\le L\le R}L^{-2}R^2L
 \lesssim M^{-4+2\nu+2\betaK+2\theta}R^{2\betaW}.
\]
For $L>R$, the cap has volume $O(R^3)$, and hence
\[
 P_{M,R}\sum_{R<L\lesssim M}L^{-2}R^3
 \lesssim M^{-4+2\nu+2\betaK+2\theta}R^{-1+2\betaW}.
\]
The latter term is no larger than the former for $R\ge1$.  Thus the fixed
$R$-shell is bounded by
\[
 M^{-4+2\nu+2\betaK+2\theta}R^{2\betaW}.
\]
The dyadic sum over $R\le M/C$ is bounded by $CM^{2\betaW}$ if
$\betaW>0$ and by $C\log(2+M)$ if $\betaW=0$.  In both cases it is
$O_\eps(M^{2\betaW+\eps})$, which gives the required power.

\emph{Klein--Gordon-low region.}
Suppose $N_0\sim N_1\sim M$ and $N_2\sim R\ll M$.  The prefactor is
\[
 \widetilde P_{M,R}:=
 M^{-4+2\nu+2\betaW+2\theta}R^{-2+2\betaK}.
\]
Using \eqref{eq:kg-low-layer-volume} for $L\le R$ and the cap volume
$O(R^3)$ for $L>R$ gives, as above,
\[
 \widetilde P_{M,R}\left(
 \sum_{L\le R}L^{-2}R^2L+
 \sum_{L>R}L^{-2}R^3\right)
 \lesssim
 M^{-4+2\nu+2\betaW+2\theta}R^{2\betaK}.
\]
Summing the dyadic $R$-shells gives
$O_\eps(M^{-4+2\nu+2\betasum+2\theta+\eps})$.

\emph{Low-output high--high region.}
Suppose $N_1\sim N_2\sim M$ and $M\ge C_0N_0$.  By
Lemma~\ref{lem:first-picard-low-output-gap},
$\mathfrak m(\Phi_a^{\bm\eps})\lesssim M^{-1}$.  The $M$-shell has volume
$O(M^3)$, and therefore
\[
  \mathrm{LOHH}_M
  \lesssim
  N_0^{-2+2\nu}M^{-4+2\betasum+2\theta}M^{-2}M^3
  =N_0^{-2+2\nu}M^{-3+2\betasum+2\theta}.
\]
The dyadic sum over $M\ge C_0N_0$ is bounded by
$N_0^{-5+2\nu+2\betasum+2\theta}$, which is stronger than
\eqref{eq:continuous-phase-integral-bound}.  If $M\lesssim N_0$, the
triangle relation forces $M\sim N_0$, and the interaction was already
included in the balanced region.  This completes the proof.
\end{proof}

\begin{proof}[Proof of \cref{prop:first-picard-spectral}]
By \cref{lem:picard-stochastic-fubini,lem:picard-kernel-bounds},
\[
  q_{a,\nu,\Lambda}(t,\xi)
  \lesssim_T\mathfrak S_a^{(\nu,0)}(\xi),
\]
uniformly in $\Lambda$.  Proposition~\ref{prop:continuous-phase-integral} with $\theta=0$ proves \eqref{eq:first-picard-spectral-static}.

Assume $t'<t$.  In the double Wiener representation, split $[0,t]^2$ into the common square $[0,t']^2$ and the boundary strip.  On the common square, \eqref{eq:picard-kernel-increment} and the Wiener isometry give
\[
  \text{common-square contribution}
  \lesssim
  |t-t'|^{2\theta}
  \mathfrak S_a^{(\nu,\theta)}(\xi).
\]
The boundary strip has measure $O_T(|t-t'|)$.  The static kernel bound,
$2\theta\le1$, and $N_{\max}^{2\theta}\ge1$ give the same upper bound there.
Applying \cref{prop:continuous-phase-integral} proves
\eqref{eq:first-picard-spectral-increment}.  For any finite collection of
input and output cutoff flags, list the corresponding frequency legs with
multiplicity and apply \cref{lem:finite-product-cutoff} to the squared static
kernel and to the normalized increment kernel.  Terms in which an input leg
is active are bounded by the weighted phase integral restricted to
$\max\{|\eta|,|\xi-\eta|\}\gtrsim\Lambda$ and hence by
\cref{lem:first-picard-tail-integral}.  If an output flag is active, then
$|\xi|\gtrsim\Lambda$; the unrestricted estimates
\eqref{eq:first-picard-spectral-static}--\eqref{eq:first-picard-spectral-increment}
therefore give a dyadically summable tail starting at the output scale
$N\gtrsim\Lambda$.  Summing the finitely many telescope terms proves
convergence for every admissible cofinal sequence.  By
\eqref{eq:finite-product-two-cutoffs}, two non-nested cutoffs are controlled
by the sum of their respective tails.
\end{proof}

The following lemma converts stationary spectral-density bounds into localized path-space estimates.

\begin{lemma}[Spectral density to local H\"older--Besov paths]
\label{lem:spectral-to-local-paths}
Let $(Z_\Lambda)_\Lambda$ be spatially stationary random distributions on
$\R^3$, taking values in a fixed finite sum of Wiener chaoses of order at most
$m$.  Assume that every finite-cutoff field admits a realization in
$C([0,T];C^\infty(\R^3))$; this is the case for the smooth, frequency-compact
spectral multipliers used below.  Assume further that there exists $\eps_*>0$ such that their covariance and
increment spectral densities satisfy, for every $0<\eps\le\eps_*$ and
$0<\theta\le\theta_*$,
\begin{align}
  q_\Lambda(t,\xi)&\lesssim\la\xi\ra^{-2\rho-3+\eps},
  \label{eq:path-criterion-static}\\
  q_\Lambda^\Delta(t,t';\xi)
  &\lesssim |t-t'|^{2\theta}
  \la\xi\ra^{-2\rho-3+2\theta+\eps}.
  \label{eq:path-criterion-increment}
\end{align}
Then, for every $s<\rho$, every compactly supported $\chi$, and every finite $p\ge2$,
\begin{equation}\label{eq:path-criterion-conclusion}
  \sup_\Lambda
  \norm{\chi Z_\Lambda}_{L^p(\Omega;
  C_T\cC^s\cap L_T^\infty B_{2,\infty}^s)}<\infty.
\end{equation}
If the corresponding covariance densities for cutoff differences tend to zero pointwise, let $q_{>\Lambda}$ and
$q_{>\Lambda}^{\Delta}$ be nonnegative majorants for their static and
increment densities and set
\begin{equation}\label{eq:path-tail-majorant-definition}
 q_{>\Lambda}^{(\theta)}(\xi)
 :=\sup_{0\le t\le T}q_{>\Lambda}(t,\xi)
 +\sup_{\substack{0\le t,t'\le T\\ t\ne t'}}
 \frac{q_{>\Lambda}^{\Delta}(t,t';\xi)}{|t-t'|^{2\theta}}.
\end{equation}
If this tail majorant satisfies
\begin{equation}\label{eq:path-criterion-tail}
  \lim_{\Lambda\to\infty}
  \sum_{N\in\Dyd}
  N^{2s+3+\delta}
  \sup_{|\xi|\sim N}q_{>\Lambda}^{(\theta)}(\xi)
  =0
\end{equation}
for some $\delta>0$, then the cutoff family is Cauchy in the space in \eqref{eq:path-criterion-conclusion}.  A polynomial tail bound in \eqref{eq:path-criterion-tail} implies almost sure convergence along dyadic cutoffs.
\end{lemma}

\begin{proof}
For a nonnegative spectral density $q$, write
\[
 q_M:=\sup_{|\xi|\sim M}q(\xi),
 \qquad
 \mathfrak w_L(N,M):=
 \left(\frac{\min\{N,M\}}{\max\{N,M\}}\right)^L,
\]
with the inhomogeneous unit blocks interpreted in the usual way, and set
\begin{equation}\label{eq:spectral-local-shell-envelope}
 \mathfrak a_N(q):=
 \sum_{M\in\Dyd}\mathfrak w_L(N,M)(M^3q_M)^{1/2}.
\end{equation}
For the localized block estimate, let
$2\le q_x\le p_0<\infty$ and every $L>0$,
\begin{equation}\label{eq:localized-spectral-block}
 \|P_N(\chi Z(t))\|_{L^{p_0}(\Omega;L_x^{q_x})}
 \lesssim_{p_0,q_x,L,\chi,m}\mathfrak a_N(q(t)).
\end{equation}
Indeed, decompose $Z=\sum_MP_MZ$.  Stationarity and the covariance identity
give
\[
 \|P_MZ(t,x)\|_{L^2(\Omega)}
 \lesssim(M^3q_M(t))^{1/2},
\]
uniformly in $x$; finite-chaos hypercontractivity gives the same bound in
$L^{p_0}(\Omega)$.  We claim that
\begin{equation}\label{eq:localized-spectral-two-shell}
 \|P_N(\chi P_MZ(t))\|_{L^{p_0}_\omega L_x^{q_x}}
 \lesssim_{L,p_0,q_x,\chi,m}
 \mathfrak w_L(N,M)(M^3q_M(t))^{1/2}.
\end{equation}
For comparable $M,N$, boundedness of $P_N$ on $L^{q_x}$, Minkowski's
inequality (recall $p_0\ge q_x$), compactness of $\supp\chi$, and the preceding
pointwise chaos bound prove \eqref{eq:localized-spectral-two-shell}.  For noncomparable scales, write the stochastic Fourier kernel as
\[
 G_{N,M}(x,\eta)
 :=\varphi_M(\eta)\int_{\R^3}e^{ix\cdot\xi}
 \varphi_N(\xi)\widehat\chi(\xi-\eta)\dd\xi.
\]
The covariance isometry and the definition of $q_M$ give
\begin{equation}\label{eq:localized-spectral-kernel-isometry}
 \|P_N(\chi P_MZ(t))(x)\|_{L^2(\Omega)}^2
 \le q_M(t)\int_{\R^3}|G_{N,M}(x,\eta)|^2\dd\eta.
\end{equation}
On the support of the integrand,
$|\xi-\eta|\gtrsim\max\{M,N\}$.  Writing $\widehat\chi$ as the Fourier transform of the compactly supported
function $\chi$, differentiating the shell symbols, and integrating by parts in
$\xi$ against $e^{\ii(x-y)\cdot\xi}$ yields, for every $L,A>0$,
\begin{equation}\label{eq:localized-spectral-deterministic-kernel}
 \left(\int_{\R^3}|G_{N,M}(x,\eta)|^2\dd\eta\right)^{1/2}
 \lesssim_{L,A,\chi}
 \mathfrak w_L(N,M)M^{3/2}
 \la\dist(x,\supp\chi)\ra^{-A}.
\end{equation}
The powers created by differentiating the shell symbols and the two shell
volumes are absorbed by taking the Schwartz exponent larger; the inhomogeneous
unit shell is included by changing the constant.  Combining
\eqref{eq:localized-spectral-kernel-isometry} and
\eqref{eq:localized-spectral-deterministic-kernel}, applying finite-chaos
hypercontractivity, and choosing $Aq_x>3$ proves
\eqref{eq:localized-spectral-two-shell}.  Summing in $M$ proves
\eqref{eq:localized-spectral-block}.  The identical calculation
for $Z(t)-Z(t')$ yields
\begin{equation}\label{eq:localized-spectral-increment-block}
 \|P_N\chi(Z(t)-Z(t'))\|_{L^{p_0}_\omega L_x^{q_x}}
 \lesssim |t-t'|^\theta
 \mathfrak a_N(q^{(\theta)}),
\end{equation}
where $q^{(\theta)}$ is the normalized increment density.

Set
\[
 q^{(0)}(\xi):=\sup_{t\le T}q(t,\xi),
 \qquad
 q^{(\theta)}(\xi):=
 \sup_{t\ne t'}\frac{q^\Delta(t,t';\xi)}{|t-t'|^{2\theta}}.
\]
Under \eqref{eq:path-criterion-static}--
\eqref{eq:path-criterion-increment}, weighted dyadic Schur summation in
\eqref{eq:spectral-local-shell-envelope} gives, for every $0<\eps_1<\eps_*/4$,
\[
 \mathfrak a_N(q^{(0)})\lesssim N^{-\rho+\eps_1},
 \qquad
 \mathfrak a_N(q^{(\theta)})
 \lesssim N^{-\rho+\theta+\eps_1},
\]
provided $L$ is chosen sufficiently large.  Choose
\[
 0<\alpha<\theta,\qquad \alpha p_0>1.
\]
The vector-valued fractional Sobolev embedding
$W^{\alpha,p_0}([0,T];L_x^{q_x})\hookrightarrow C_TL_x^{q_x}$, together with
\eqref{eq:localized-spectral-block}--
\eqref{eq:localized-spectral-increment-block}, gives
\begin{equation}\label{eq:spectral-block-CT-Lq}
 \|P_N(\chi Z_\Lambda)\|_{L^{p_0}(\Omega;C_TL_x^{q_x})}
 \lesssim N^{-\rho+\theta+\eps_1}.
\end{equation}
We use two spatial branches.  For the H\"older estimate we take a large
finite exponent $q_x$ and then apply Bernstein.  For the Hilbert--Besov
estimate we set $q_x=2$ in
\eqref{eq:localized-spectral-two-shell} from the outset.

\emph{H\"older branch.}
Take $q_x<\infty$ large.  Since the output is localized at frequency $N$,
Bernstein and \eqref{eq:spectral-block-CT-Lq} give
\[
 \|P_N(\chi Z_\Lambda)\|_{L^{p_0}(\Omega;C_TL_x^\infty)}
 \lesssim N^{-\rho+\theta+3/q_x+\eps_1}.
\]
Choose $\theta$, $q_x^{-1}$, and $\eps_1$ so small that
$\theta+3/q_x+\eps_1<\rho-s$.  Since $p_0\ge2$, Minkowski's inequality in
$L^{p_0}(\Omega;\ell_N^2)$ and the bound of a dyadic supremum by its
$\ell^2$ norm yield
\begin{equation}\label{eq:spectral-holder-branch}
 \left\|
  \sup_N N^s\|P_N(\chi Z_\Lambda)\|_{C_TL_x^\infty}
 \right\|_{L^{p_0}(\Omega)}
 \lesssim
 \left(\sum_N
  N^{-2(\rho-s-\theta-3/q_x-\eps_1)}\right)^{1/2}<\infty.
\end{equation}

\emph{Hilbert--Besov branch.}
Now take $q_x=2$ directly in
\eqref{eq:localized-spectral-block}--
\eqref{eq:localized-spectral-increment-block}, and do not apply a spatial
Bernstein inequality.  Choose $\theta+\eps_1<\rho-s$.  The same fractional
Sobolev and Minkowski argument gives
\begin{equation}\label{eq:spectral-besov-branch}
 \left\|
  \sup_N N^s\|P_N(\chi Z_\Lambda)\|_{C_TL_x^2}
 \right\|_{L^{p_0}(\Omega)}
 \le
 \left\|
  \left(\sum_NN^{2s}
   \|P_N(\chi Z_\Lambda)\|_{C_TL_x^2}^2\right)^{1/2}
 \right\|_{L^{p_0}(\Omega)}
 \lesssim1.
\end{equation}
Thus the Besov conclusion is obtained in its native Hilbert spatial norm and
carries no $3/q_x$ loss.  The summable high-frequency majorants in
\eqref{eq:spectral-holder-branch} and
\eqref{eq:spectral-besov-branch}, together with continuity of every fixed
Littlewood--Paley block, also place the paths in the completed spatial
topologies.  Hypercontractivity allows $p_0$ to be chosen larger than
$\max\{p,q_x\}$, and H\"older in probability gives
\eqref{eq:path-criterion-conclusion}.

For cutoff differences, apply the preceding time argument with the
combined majorant $q_{>\Lambda}^{(\theta)}$ from
\eqref{eq:path-tail-majorant-definition}.  We use the following explicit
weighted Schur estimate: for every $\gamma\in\R$, once $L>|\gamma|+1$,
\begin{equation}\label{eq:spectral-envelope-weighted-schur}
 \sum_NN^{2\gamma}\mathfrak a_N(q)^2
 \lesssim_{L,\gamma}
 \sum_MM^{2\gamma+3}q_M.
\end{equation}
Indeed, with $d_M=(M^3q_M)^{1/2}$, conjugation by the weight $N^\gamma$
reduces the map in \eqref{eq:spectral-local-shell-envelope} to the dyadic
matrix
\[
 K_{N,M}:=(N/M)^\gamma\mathfrak w_L(N,M).
\]
Both its row and column sums are bounded by convergent geometric series, so the
ordinary $\ell^2$ Schur test proves
\eqref{eq:spectral-envelope-weighted-schur}.  Specializing first to
$\gamma=s+3/q_x$ and then to $\gamma=s$ gives two distinct estimates.  For the
H\"older branch,
\begin{align}
 &\sum_NN^{2s+6/q_x}
 \mathfrak a_N(q_{>\Lambda}^{(\theta)})^2\notag\\
 &\qquad\lesssim_{L,s,q_x}
 \sum_MM^{2s+3+6/q_x}
 \sup_{|\xi|\sim M}q_{>\Lambda}^{(\theta)}(\xi).
 \label{eq:spectral-tail-schur-holder}
\end{align}
For the Hilbert--Besov branch, taking $q_x=2$ before any Bernstein step gives
instead
\begin{align}
 &\sum_NN^{2s}
 \mathfrak a_N(q_{>\Lambda}^{(\theta)})^2\notag\\
 &\qquad\lesssim_{L,s}
 \sum_MM^{2s+3}
 \sup_{|\xi|\sim M}q_{>\Lambda}^{(\theta)}(\xi).
 \label{eq:spectral-tail-schur-besov}
\end{align}
Choose $q_x$ so large that $6/q_x<\delta$.  The right-hand side of
\eqref{eq:spectral-tail-schur-holder} is then bounded by the tail in
\eqref{eq:path-criterion-tail}; the right-hand side of
\eqref{eq:spectral-tail-schur-besov} is bounded by the same tail without using
any part of the $\delta$ margin.  Combining the first estimate with Bernstein
proves Cauchy convergence in $L^p(\Omega;C_T\mathcal C^s)$, while the second,
with $q_x=2$ throughout, proves Cauchy convergence in
$L^p(\Omega;L_T^\infty B_{2,\infty}^s)$.  Together these two estimates prove convergence in the intersection in
\eqref{eq:path-criterion-conclusion}; the Hilbert--Besov component is obtained
entirely in $L_x^2$.

Each localized finite-cutoff path is spatially smooth and compactly supported.
To identify the completed target, fix one cutoff $\Lambda_0$.  Its uniform
little-Besov tail vanishes.  The Cauchy estimate just proved makes the limit
arbitrarily close to this fixed path in the full
$L_T^\infty B_{2,\infty}^s$ norm; taking the high-frequency limsup and then
letting the approximation error tend to zero gives
\eqref{eq:uniform-little-besov-tail}.  Thus the limit belongs to the completed little-Besov path space.
The same fixed-approximation argument in the full
$C_T\mathcal C^s$ norm places the limit in the little-H\"older path closure
specified after \eqref{eq:E-space}: the finite-cutoff paths are smooth and compactly supported, and convergence
holds in the defining norm.  If the right-hand sides of
\eqref{eq:spectral-tail-schur-holder} and
\eqref{eq:spectral-tail-schur-besov} decay polynomially at dyadic cutoff
$\Lambda=2^j$, choose a sufficiently high probability moment.  Markov's
inequality makes the exceptional probabilities summable, and Borel--Cantelli gives
almost-sure convergence of the complete dyadic sequence.
\end{proof}

\subsection{Cutoff convergence and completion of the proof}

Let $\mathfrak S_{a,>\Lambda}^{(\nu,\theta)}(\xi)$ denote the
integral in \eqref{eq:continuous-phase-integral} restricted to
\[
  \max\{|\eta|,|\xi-\eta|\}>c_0\Lambda,
\]
where $c_0$ is the uniform inner-plateau constant in
\eqref{eq:cutoff-supports}.  After expanding the difference from the formal
uncut multiplier $1$, every term contains at least one factor
$m_\Lambda(\zeta)-1$ and is therefore supported in this region.  Thus
$\mathfrak S_{a,>\Lambda}^{(\nu,\theta)}$ controls the corresponding
covariance density up to uniform constants.  Two arbitrary cofinal cutoff elements are compared through this common
limiting kernel.

\begin{lemma}[Tail phase integrals]\label{lem:first-picard-tail-integral}
Let
\[
  s<\frac12-\betasum-\nu,
  \qquad \nu\in\{0,1\}.
\]
Let $\theta,\eps,\delta>0$ satisfy
\[
 0<\theta<\frac12-\betasum,
 \qquad
 2\theta+\eps+\delta
 <1-2s-2\nu-2\betasum.
\]
Then
\begin{equation}\label{eq:first-picard-tail-sum}
  \lim_{\Lambda\to\infty}
  \sum_{N\in\Dyd}
  N^{2s+3+\delta}
  \sup_{|\xi|\sim N}
  \mathfrak S_{a,>\Lambda}^{(\nu,\theta)}(\xi)
  =0.
\end{equation}
Along dyadic $\Lambda$, the left-hand side is bounded by
$C\Lambda^{-\gamma}$ for some $\gamma>0$.
\end{lemma}

\begin{proof}
For each fixed output shell, the restricted positive integral tends to zero
by dominated convergence.  Proposition~\ref{prop:continuous-phase-integral}
gives the summable majorant
\[
  N^{2s+3+\delta}N^{-4+2\nu+2\betasum+2\theta+\eps}
  =N^{2s-1+2\nu+2\betasum+2\theta+\eps+\delta}.
\]
The imposed inequality on $\theta,\eps,\delta$ makes this dyadic power
strictly negative.

For the polynomial tail, in the balanced and high--low regions the tail
condition forces the dominant scale to be $\gtrsim\Lambda$, so the same
summable majorant starts at the cutoff scale.  In the low-output high--high
region, the proof of Proposition~\ref{prop:continuous-phase-integral} gives
\[
  \mathfrak S_{a,>\Lambda}^{(\nu,\theta)}(\xi)
  \lesssim
  N^{-2+2\nu}\max\{N,\Lambda\}^{-3+2\betasum+2\theta}.
\]
For $N\gtrsim\Lambda$ this is controlled by the full tail.  For
$N\lesssim\Lambda$, dyadic summation gives
\[
\begin{aligned}
 &\Lambda^{-3+2\betasum+2\theta}
  \sum_{N\lesssim\Lambda}N^{2s+1+2\nu+\delta}\\
 &\qquad\lesssim
 \Lambda^{-3+2\betasum+2\theta
 +\max\{2s+1+2\nu+\delta,0\}+}.
\end{aligned}
\]
If the expression inside the positive part is positive, the exponent reduces
to $2s-2+2\nu+2\betasum+2\theta+\delta+$ and is negative by the strict
hypothesis on $s$.  If it is nonpositive, the exponent is
$-3+2\betasum+2\theta+$, which is also negative.  This proves the polynomial
claim for the full range stated in the lemma.
\end{proof}

\begin{proof}[Proof of \cref{thm:first-picard-fullspace}]
For $\nu=0$, \cref{prop:first-picard-spectral} has the form
\[
  q_{a,0,\Lambda}(t,\xi)
  \lesssim\la\xi\ra^{-2(1/2-\betasum)-3+\eps},
\]
so \cref{lem:spectral-to-local-paths} applies with
$\rho=1/2-\betasum$.  For $\nu=1$, it applies with
$\rho=-1/2-\betasum$.  The cutoff Cauchy property follows from \cref{lem:first-picard-tail-integral}.  This proves \eqref{eq:first-picard-convergence}--\eqref{eq:first-picard-derivative-convergence} in every finite probability moment and almost surely along dyadic cutoffs.

The extra output multiplier in $V^{\mathrm{out}}_{a,\Lambda}$ differs
from one only on output frequencies $|\xi|\gtrsim\Lambda$.  Its contribution
is controlled by the same weighted tail lemma: the additional output-shell
restriction is inserted into the positive spectral integral, and the balanced,
high--low, and low-output high--high sectors used in the proof of
Lemma~\ref{lem:first-picard-tail-integral} already exhaust that integral.  Thus
both cutoff conventions have the same limit.

At finite cutoff,
\[
  V_{a,\Lambda}(t)-V_{a,\Lambda}(0)
  =\int_0^t\partial_sV_{a,\Lambda}(s)\dd s.
\]
Passing to the limit in a slightly lower local regularity identifies the
second limiting process as the time derivative of $V_a$.  For the almost-sure
statement, prescribe the reference exhaustion localizers, rational horizons
and losses, and integer moments.  Intersecting the resulting countable family
of Borel--Cantelli events gives the asserted common event.  For any additional
fixed spatial cutoff, the polynomial tail estimate and Borel--Cantelli yield a
separate probability-one event.
\end{proof}

\section{Cubic resonant products}\label{app:cubic-terms}
We construct the cubic stochastic objects required in the fixed-point equation.
For $a\in\mathfrak C$, recall the opposite channel $a^\perp$ from
\eqref{eq:channel-set}.  The two singular cubic products are
\begin{equation}\label{eq:cubic-terms}
  \Gamma_a=V_{a^\perp}\circ\Psi_a.
\end{equation}
The colored Wick expansion has a centered third-chaos component and one
first-chaos contraction.  The third chaos is controlled by the first Picard
covariance density, while the contraction is represented by an integrated
Volterra kernel.  Both branches use phase functions determined solely by the
wave and Klein--Gordon dispersions.  Their amplitude order is
\[
 \beta_{\Gamma,a}=\betasum+\beta_a.
\]

\subsection{Wick decomposition and the third chaos}

Let $\chi^{\mathrm{res}}(m,r)$ denote the smooth order-zero multiplier associated
with the fixed Bony resonant product, so that
\begin{equation}\label{eq:resonant-fourier-symbol}
  \widehat{f\circ g}(\xi)
  =c_{\mathrm F}\int_{\R^3}
  \chi^{\mathrm{res}}(\xi-r,r)
  \widehat f(\xi-r)\widehat g(r)\dd r.
\end{equation}
Its support is contained in
\begin{equation}\label{eq:resonant-support}
  \la \xi-r\ra\sim\la r\ra,
  \qquad
  \la r\ra\gtrsim\la\xi\ra,
\end{equation}
with constants depending only on the Littlewood--Paley convention.

At finite cutoff define the principal convention
\begin{equation}\label{eq:cubic-cutoff-definition}
  \Gamma_{a,\Lambda}^{0,0}
  :=V_{a^\perp,\Lambda}\circ\Psi_{a,\Lambda}.
\end{equation}
For the cutoff comparison, introduce the projected variants indexed by
$\epsilon_{\mathrm P},\epsilon_{\mathrm O}\in\{0,1\}$:
\begin{equation}\label{eq:cubic-cutoff-variants}
 \Gamma_{a,\Lambda}^{\epsilon_{\mathrm P},\epsilon_{\mathrm O}}
 :=\pi_\Lambda^{\epsilon_{\mathrm O}}
 \Bigl(V_{a^\perp,\Lambda}^{\epsilon_{\mathrm P}}\circ
       \Psi_{a,\Lambda}\Bigr),
 \qquad
 V_{b,\Lambda}^{\epsilon_{\mathrm P}}
 :=I_b\pi_\Lambda^{\epsilon_{\mathrm P}}
   (\Psi_{\W,\Lambda}\Psi_{\K,\Lambda}).
\end{equation}
Thus $(\epsilon_{\mathrm P},\epsilon_{\mathrm O})=(0,0)$ is the convention
used in the fixed-point formulation, $\epsilon_{\mathrm P}=1$ is the auxiliary
output projection in the first Picard factor, and $\epsilon_{\mathrm O}=1$
adds a final source projection.  All four finite-cutoff conventions have the
same local limit.  We suppress the superscripts whenever the argument applies
uniformly to them.

\begin{theorem}[Cubic resonant products]\label{thm:cubic-fullspace}
Fix $T<\infty$, $a\in\mathfrak C$, $\chi\in C_c^\infty(\R^3)$,
$\kappa>0$, and $2\le p<\infty$.  Let an admissible spectral family tend to
the identity along an arbitrary admissible cofinal sequence.  For every
$(\epsilon_{\mathrm P},\epsilon_{\mathrm O})\in\{0,1\}^2$ and every finite
cutoff has the colored Wick decomposition
\begin{equation}\label{eq:cubic-wick-decomposition}
  \Gamma_{a,\Lambda}^{\epsilon_{\mathrm P},\epsilon_{\mathrm O}}
  =\Gamma_{a,\Lambda}^{(3),\epsilon_{\mathrm P},\epsilon_{\mathrm O}}
   +\Gamma_{a,\Lambda}^{(1),\epsilon_{\mathrm P},\epsilon_{\mathrm O}},
\end{equation}
where the first term is a centered third homogeneous chaos and the second is
a centered first chaos of color $a^\perp$.  The limits below are independent
of the cutoff family, the cofinal sequence, and the two projection flags.  On
suppressing the flags, one has
\begin{align}
  \chi\Gamma_{a,\Lambda}^{(3)}
  &\longrightarrow \chi\Gamma_a^{(3)}
  &&\text{in }L^p\!\left(\Omega;
  C_T\cC^{-\beta_{\Gamma,a}-\kappa}
  \cap L_T^\infty B_{2,\infty}^{-\beta_{\Gamma,a}-\kappa}\right),
  \label{eq:cubic-third-convergence}\\
  \chi\Gamma_{a,\Lambda}^{(1)}
  &\longrightarrow \chi\Gamma_a^{(1)}
  &&\text{in }L^p\!\left(\Omega;
  C_T\cC^{\frac12-\beta_{\Gamma,a}-\kappa}
  \cap L_T^\infty B_{2,\infty}^{\frac12-\beta_{\Gamma,a}-\kappa}\right).
  \label{eq:cubic-first-convergence}
\end{align}
Consequently
\begin{equation}\label{eq:cubic-total-regularity}
  \Gamma_a=\Gamma_a^{(3)}+\Gamma_a^{(1)}
  \in C_T\cC_{\loc}^{-\beta_{\Gamma,a}-}
  \cap L_T^\infty B_{2,\infty,\loc}^{-\beta_{\Gamma,a}-}.
\end{equation}
Along a fixed-profile dyadic family, all four projection conventions converge
almost surely, along the complete dyadic sequence, on a common
probability-one event obtained by intersecting the event from
\cref{thm:first-picard-fullspace} with the probability-one events supplied by
the third-chaos and first-chaos cubic tail estimates below.

For every exponent quadruple in \eqref{eq:parameter-window}, the total symbol
satisfies the source bounds
\begin{equation}\label{eq:cubic-source-regularity}
  \Gamma_a\in
  L_T^1H_{\loc}^{s_2-1}
  \cap L_T^1B_{2,\infty,\loc}^{\sigma-1}.
\end{equation}
\end{theorem}

The spectral bounds behind the theorem are
\begin{align}
  q_{a,\Lambda}^{(3)}(t,\xi)
  &\lesssim_{T,\eps}\la\xi\ra^{-3+2\beta_{\Gamma,a}+\eps},
  \label{eq:cubic-third-spectral-summary}\\
  q_{a,\Lambda}^{(1)}(t,\xi)
  &\lesssim_{T,\eps}\la\xi\ra^{-4+2\beta_{\Gamma,a}+\eps}.
  \label{eq:cubic-first-spectral-summary}
\end{align}
Thus the cubic order is channel dependent.  For the symbol indexed by $a$,
the centered third chaos is the regularity-limiting component: it belongs
locally to $\mathcal C^{-\beta_{\Gamma,a}-}$, whereas the contraction belongs
locally to $\mathcal C^{1/2-\beta_{\Gamma,a}-}$.  Uniformly in $a$, one may
replace $\beta_{\Gamma,a}$ by $\betagamma$.

The cutoff factors and the estimates used below are summarized in the
following table.  The optional factors are present only when the corresponding
projection flag equals one.
\begin{center}
\small
\begin{tabularx}{\textwidth}{@{}p{0.18\textwidth}p{0.28\textwidth}YY@{}}
\toprule
component & cutoff frequencies & principal bound & limiting regularity \\
\midrule
$\Gamma_{a,\Lambda}^{(3)}$
& $\eta,\zeta,r$; optionally $\eta+\zeta$ and $\xi$
& $q_a^{(3)}(\xi)\lesssim
  \la\xi\ra^{-3+2\beta_{\Gamma,a}+}$
& $\mathcal C^{-\beta_{\Gamma,a}-}_{\mathrm{loc}}$ \\
$\Gamma_{a,\Lambda}^{(1)}$
& $\xi$, the repeated pair $r,-r$; optionally $\xi-r$ and a second $\xi$
& $|\mathcal K_a(\xi;t,u)|\lesssim
  \la\xi\ra^{-2+\beta_{\Gamma,a}+}$
& $\mathcal C^{1/2-\beta_{\Gamma,a}-}_{\mathrm{loc}}$ \\
$\Gamma_{a,\Lambda}^{\chi,\rho}$
& the preceding stochastic legs; $\chi,\rho$ are physical multipliers
& $\Gamma_{a,\Lambda}^{\chi,\rho}
  =\chi\Gamma_{a,\Lambda}+\mathcal R_{a,\Lambda}^{\chi,\rho}$
& $\mathcal C^{-\beta_{\Gamma,a}-}$ after localization \\
\bottomrule
\end{tabularx}
\end{center}
For every cutoff difference, the finite-product identity of
Lemma~\ref{lem:finite-product-cutoff} places at least one of the listed
frequencies above the cutoff scale.  The third-chaos and first-chaos tails are
then estimated separately below.

The algebraic decomposition is made before passing to the limit.  Reorder the two
internal colors in $V_{a^\perp}$ so that the color-$a$ factor is written first.
With the non-conjugated Fourier covariance convention, the expanded symbol is
schematically
\begin{equation}\label{eq:cubic-expanded-generic}
\begin{aligned}
  \widehat\Gamma_{a,\Lambda}(\xi,t)
  =c_{\mathrm F}^2
  \int_{\eta+\zeta+r=\xi}
  &\mathfrak m_\Lambda^{[3]}(\eta,\zeta,r;\xi)
   \chi^{\mathrm{res}}(\eta+\zeta,r)\\
  &\times\int_0^t
  K_{a^\perp}(t-s,\eta+\zeta)
  \widehat\Psi_a(\eta,s)
  \widehat\Psi_{a^\perp}(\zeta,s)
  \widehat\Psi_a(r,t)\dd s
  \dd\eta\dd\zeta\dd r.
\end{aligned}
\end{equation}
Here and below,
\[
  \int_{\eta+\zeta+r=\xi}F(\eta,\zeta,r)
  \dd\eta\dd\zeta\dd r
  :=\int_{\R^3}\int_{\R^3}
  F(\eta,\zeta,\xi-\eta-\zeta)\dd\eta\dd\zeta.
\]
For the convention in \eqref{eq:cubic-cutoff-variants}, the cutoff multiplier is
\begin{equation}\label{eq:cubic-cutoff-multiplier}
 \mathfrak m_{\Lambda,\epsilon_{\mathrm P},\epsilon_{\mathrm O}}^{[3]}
 (\eta,\zeta,r;\xi)
 =m_\Lambda(\eta)m_\Lambda(\zeta)m_\Lambda(r)
  m_\Lambda(\eta+\zeta)^{\epsilon_{\mathrm P}}
  m_\Lambda(\xi)^{\epsilon_{\mathrm O}},
 \qquad \xi=\eta+\zeta+r.
\end{equation}
This formula will be used both before and after the Wick contraction.  Since
$m_\Lambda$ is real and even, contraction of the two color-$a$ legs at
$\eta+r=0$ gives
\begin{equation}\label{eq:cubic-contracted-cutoff-multiplier}
 \mathfrak m_{\Lambda,\epsilon_{\mathrm P},\epsilon_{\mathrm O}}^{[1]}(\xi,r)
 =m_\Lambda(\xi)m_\Lambda(r)^2
  m_\Lambda(\xi-r)^{\epsilon_{\mathrm P}}
  m_\Lambda(\xi)^{\epsilon_{\mathrm O}}.
\end{equation}
All factors are uniformly bounded and converge pointwise to one.  In the formulas below,
$\mathfrak m_\Lambda^{[3]}$ and $\mathfrak m_\Lambda^{[1]}$ always denote the multipliers with
the fixed pair of flags.  Formula
\eqref{eq:cubic-contracted-cutoff-multiplier} records the repeated internal
frequency $r$ explicitly.  The relevant frequencies and cutoff factors are summarized in the following
table.

\begin{center}
\small
\renewcommand{\arraystretch}{1.25}
\begin{tabularx}{\textwidth}{@{}P{0.16\textwidth}P{0.26\textwidth}Y P{0.19\textwidth}@{}}
\toprule
branch & Gaussian frequencies before contraction & active cutoff legs & random frequency after contraction \\
\midrule
third chaos
& color $a$ at $\eta$, color $a^\perp$ at $\zeta$, color $a$ at $r$
& $\eta,\zeta,r$; optionally $\eta+\zeta$ and $\xi$
& three legs with total momentum $\xi=\eta+\zeta+r$ \\
first chaos
& the two color-$a$ legs at $\eta$ and $r$ are paired
& $r$ appears twice; the remaining color-$a^\perp$ leg is at $\xi$; optionally $\xi-r$ and a final $\xi$
& color $a^\perp$ at $\zeta=\xi$ because $\eta+r=0$ \\
\bottomrule
\end{tabularx}
\end{center}

Every cutoff difference is reduced by
\Cref{lem:finite-product-cutoff} to a tail on at least one of the listed
legs.  These are all the projection factors in the finite-cutoff algebra.

The algebraic formula below is conveniently written in the non-conjugated
Fourier covariance convention \eqref{eq:fourier-gaussian-forcing}, whereas covariance
densities are computed with the conjugated random-measure convention
\eqref{eq:conjugated-isonormal-covariance}.  These are two presentations of the same real isonormal processes; the
Hermitian reality relation identifies the opposite-frequency coordinates.  The passage to momentum fibers below is
made only after complexifying the real Gaussian Hilbert spaces; the
complexification is the full complex Fourier $L^2$ space and therefore admits
the usual direct-integral decomposition.

Let $\mathfrak H_a$ and $\mathfrak H_{a^\perp}$ be the two orthogonal real
Gaussian Hilbert spaces generated by the independent noises.  If
$F\in\mathfrak H_a\otimes\mathfrak H_{a^\perp}$ and
$g\in\mathfrak H_a$, then the colored product formula is
\begin{equation}\label{eq:colored-cubic-product-formula}
  I_{(1,1)}(F)I_{(1,0)}(g)
  =I_{(2,1)}\!\left(\operatorname{Sym}_a(F\otimes g)\right)
   +I_{(0,1)}(F\otimes_1 g).
\end{equation}
The first term is the third homogeneous chaos.  In the second term the two
color-$a$ legs are contracted and the color-$a^\perp$ leg remains.  There is no
other contraction because the two color spaces are orthogonal.

For the stochastic convolutions, write
\begin{equation}\label{eq:cubic-same-color-covariance}
  \E\left[
  \widehat\Psi_a(\eta,s)\widehat\Psi_b(r,t)
  \right]
  =\delta_{ab}(2\pi)^3\delta(\eta+r)\sigma_a(\eta;s,t),
\end{equation}
where
\begin{equation}\label{eq:cubic-sigma-definition}
  \sigma_a(\eta;s,t)
  :=\mathfrak q_a(\eta)
  \int_0^sK_a(s-\tau,\eta)K_a(t-\tau,-\eta)\dd\tau,
  \qquad 0\le s\le t.
\end{equation}
Thus the contraction retains the covariance profile
$\mathfrak q_a=|\mathfrak h_a|^2$ of the repeated color.  This factor is kept
inside $\sigma_a$ throughout the Volterra analysis.
The contraction in \eqref{eq:cubic-expanded-generic} forces $\eta+r=0$ and
therefore $\zeta=\xi$.  Hence \eqref{eq:cubic-wick-decomposition} takes the form
\begin{equation}\label{eq:cubic-third-explicit}
\begin{aligned}
  \widehat\Gamma_{a,\Lambda}^{(3)}(\xi,t)
  =c_{\mathrm F}^2
  \int_{\eta+\zeta+r=\xi}
  &\mathfrak m_\Lambda^{[3]}(\eta,\zeta,r;\xi)
   \chi^{\mathrm{res}}(\eta+\zeta,r)\\
  &\times\int_0^t
  K_{a^\perp}(t-s,\eta+\zeta)
  :\widehat\Psi_a(\eta,s)
   \widehat\Psi_{a^\perp}(\zeta,s)
   \widehat\Psi_a(r,t):\dd s
  \dd\eta\dd\zeta\dd r,
\end{aligned}
\end{equation}
while the lower-chaos part is
\begin{equation}\label{eq:cubic-contraction-before-fubini}
  \widehat\Gamma_{a,\Lambda}^{(1)}(\xi,t)
  =\int_0^t M_{a,\Lambda}(\xi;t,s)
  \widehat\Psi_{a^\perp}(\xi,s)\dd s,
\end{equation}
where
\begin{equation}\label{eq:cubic-formal-multiplier}
  M_{a,\Lambda}(\xi;t,s)
  =c_{\mathrm F}\int_{\R^3}
  \mathfrak m_\Lambda^{[1]}(\xi,r)
  \chi^{\mathrm{res}}(\xi-r,r)
  K_{a^\perp}(t-s,\xi-r)
  \sigma_a(r;s,t)\dd r.
\end{equation}
The endpoint estimate is performed after inserting the remaining
color-$a^\perp$ stochastic convolution.  This yields the integrated kernel
used below.

\begin{lemma}[Finite-cutoff cubic Wick formula]
\label{lem:cubic-finite-wick-formula}
For every cutoff and every pair of projection flags,
\eqref{eq:cubic-wick-decomposition},
\eqref{eq:cubic-third-explicit}, and
\eqref{eq:cubic-contraction-before-fubini} are identities in
$L^2(\Omega;\mathcal S'(\R^3))$ at each fixed time.  The only nonzero
contraction pairs the two color-$a$ legs at frequencies $\eta$ and $r$.
It imposes
\[
 \eta+r=0,
 \qquad
 \zeta=\xi,
\]
so the remaining first-chaos variable has color $a^\perp$ and output
frequency $\xi$.  Its cutoff multiplier is given by
\eqref{eq:cubic-contracted-cutoff-multiplier}.
\end{lemma}

\begin{proof}
At finite cutoff all deterministic Wiener kernels are square integrable and
have compact spatial-frequency support.  Apply the colored product formula
\eqref{eq:colored-cubic-product-formula} to the two-noise kernel defining
$V_{a^\perp,\Lambda}$ and the color-$a$ kernel defining
$\Psi_{a,\Lambda}$.  Orthogonality of the two Gaussian color spaces excludes
all cross-color contractions.  The single same-color contraction is evaluated
with \eqref{eq:cubic-same-color-covariance}; its delta distribution forces
$r=-\eta$, and the momentum constraint
$\eta+\zeta+r=\xi$ then gives $\zeta=\xi$.  Substitution in
\eqref{eq:cubic-cutoff-multiplier} yields
\eqref{eq:cubic-contracted-cutoff-multiplier}.  The multiple Wiener isometry
and deterministic Fubini theorem justify the identities in
$L^2(\Omega;\mathcal S')$.
\end{proof}

Let $p_{a,\Lambda}(t,r)$ be the covariance spectral density of
$\Psi_{a,\Lambda}(t)$.  The elementary stochastic-convolution estimates give,
for every $0<\theta<1/2$,
\begin{align}
  p_{a,\Lambda}(t,r)&\lesssim_T\la r\ra^{-2+2\beta_a},
  \label{eq:cubic-linear-density}\\
  p_{a,\Lambda}^{\Delta}(t,t';r)
  &\lesssim_{T,\theta}|t-t'|^{2\theta}
  \la r\ra^{-2+2\beta_a+2\theta}.
  \label{eq:cubic-linear-density-increment}
\end{align}

\begin{lemma}[Complexified momentum-fiber disintegration]
\label{lem:cubic-direct-integral}
Let
\[
 \mathcal G_a^{\mathbb R}
 :=\{F\in L^2([0,T]\times\R^3;\mathbb C):
 F(s,-\eta)=\overline{F(s,\eta)}\}
\]
with the real inner product from
\eqref{eq:hermitian-gaussian-hilbert-space}, and let
$\mathcal G_a^{\mathbb C}$ be its Hilbert-space complexification.  The map
from $\mathcal G_a^{\mathbb C}$ to the full complex space
$L^2([0,T]\times\R^3;\mathbb C)$ is unitary.  Consequently the changes of
variables
\[
 (\eta,\zeta)\mapsto(m=\eta+\zeta,\eta),
 \qquad
 (\eta,\zeta,r)\mapsto(\xi=\eta+\zeta+r,\eta,\zeta)
\]
induce the complex Hilbert direct-integral identifications
\begin{align}
 \mathcal G_a^{\mathbb C}\otimes_2\mathcal G_b^{\mathbb C}
 &\simeq\int_{\R^3}^{\oplus}
 L^2([0,T]^2\times\R^3_\eta)\dd m,
 \label{eq:two-noise-direct-integral}\\
 \mathcal G_a^{\mathbb C}\otimes_2\mathcal G_b^{\mathbb C}
 \otimes_2\mathcal G_a^{\mathbb C}
 &\simeq\int_{\R^3}^{\oplus}
 L^2([0,T]^3\times\R^6_{\eta,\zeta})\dd\xi.
 \label{eq:three-noise-direct-integral}
\end{align}
At finite cutoff, normalize the deterministic Wiener kernels so that the
first-Picard field and the linear field have measurable momentum fibers
$m\mapsto F_m^{\Lambda,\epsilon_{\mathrm P}}(t)$ and
$r\mapsto g_r^\Lambda(t)$.  With
\begin{equation}\label{eq:cubic-flagged-first-picard-density}
 q_{a^\perp,0,\Lambda}^{[\epsilon_{\mathrm P}]}(t,m)
 :=|m_\Lambda(m)|^{2\epsilon_{\mathrm P}}
 q_{a^\perp,0,\Lambda}(t,m),
\end{equation}
their fiber norms satisfy
\begin{equation}\label{eq:cubic-fiber-norm-identities}
 \|F_m^{\Lambda,\epsilon_{\mathrm P}}(t)\|^2
 =q_{a^\perp,0,\Lambda}^{[\epsilon_{\mathrm P}]}(t,m),
 \qquad
 \|g_r^\Lambda(t)\|^2
 =p_{a,\Lambda}(t,r)
\end{equation}
for almost every $m,r$.  Since the admissible multipliers are uniformly
bounded, the flagged density in
\eqref{eq:cubic-flagged-first-picard-density} is uniformly dominated by a
fixed multiple of the unflagged density.  Swapping the two color-$a$
slots preserves the total momentum $\xi$ and is unitary on every fiber of
\eqref{eq:three-noise-direct-integral}.  Hence
$\Sym_a=(I+\tau_{13})/2$ is a fiberwise orthogonal projection.

All these statements are made in the complexification of the same real
isonormal process.  The reality relation appears as conjugacy between the
$\xi$- and $-\xi$-fibers.
\end{lemma}

\begin{proof}
Let $J$ be the conjugation
\[
 (JF)(s,\eta):=\overline{F(s,-\eta)}
\]
on the full complex Fourier $L^2$ space.  Its fixed-point space is
$\mathcal G_a^{\mathbb R}$.  Every complex $F$ has the unique decomposition
\begin{equation}\label{eq:hermitian-complexification-decomposition}
 F=F_1+\ii F_2,
 \qquad
 F_1:=\frac{F+JF}{2},
 \qquad
 F_2:=\frac{F-JF}{2\ii},
\end{equation}
with $F_1,F_2\in\mathcal G_a^{\mathbb R}$.  If $G,H$ are fixed by $J$, then
$\langle G,H\rangle_{L^2}$ is real.  Therefore
$\|G+\ii H\|_{L^2}^2=\|G\|_2^2+\|H\|_2^2$, proving that
\eqref{eq:hermitian-complexification-decomposition} is the unitary
identification of the complexification with full complex Fourier $L^2$.
The real multiple Wiener integrals extend complex-linearly to this
complexification, and their conjugated covariance is the usual complex
Hilbert tensor inner product.  Thus tested Fourier covariance densities may
be computed in the full complex space without changing the underlying real
noise.

The displayed affine frequency maps and their inverses preserve Lebesgue
measure.  Fubini's theorem now gives
\eqref{eq:two-noise-direct-integral}--
\eqref{eq:three-noise-direct-integral}.  At finite cutoff the deterministic
kernels are jointly measurable, compactly supported in the spatial
frequencies, and square integrable.  Their direct-integral fibers can
therefore be chosen strongly measurable.  After absorbing the fixed
Fourier/Wiener normalization into the deterministic kernels, the colored
multiple Wiener isometries give
\eqref{eq:cubic-fiber-norm-identities}.
Finally, $\tau_{13}$ exchanges $\eta$ with
$r=\xi-\eta-\zeta$ while leaving $\eta+\zeta+r=\xi$ fixed.  It is therefore a
unitary involution on each output fiber, so $(I+\tau_{13})/2$ is an orthogonal
projection there.  The relation between the $\xi$ and $-\xi$ fibers follows
from $J$ and proves the final assertion.
\end{proof}

\begin{lemma}[Third-chaos kernel reduction]\label{lem:cubic-third-kernel-reduction}
Let $q_{a,\Lambda}^{(3)}$ be the covariance spectral density of
$\Gamma_{a,\Lambda}^{(3)}$.  Uniformly in the cutoff and in
$\epsilon_{\mathrm P}\in\{0,1\}$,
\begin{equation}\label{eq:cubic-third-kernel-reduction}
  q_{a,\Lambda}^{(3)}(t,\xi)
  \lesssim
  \int_{\R^3}
  |\chi^{\mathrm{res}}(m,\xi-m)|^2
  q_{a^\perp,0,\Lambda}^{[\epsilon_{\mathrm P}]}(t,m)
  p_{a,\Lambda}(t,\xi-m)\dd m.
\end{equation}
Here the flagged density is defined in
\eqref{eq:cubic-flagged-first-picard-density}; by uniform boundedness of the
cutoffs, the right side is dominated by the same expression with the ordinary
first Picard density from
\eqref{eq:spectral-density-definition-intro}.  The same estimate holds for
time increments, with one factor replaced by its increment density.  A cutoff difference is the sum of three terms: a
first-Picard kernel difference, a remaining linear-kernel difference, and, if
$\epsilon_{\mathrm O}=1$, an output scalar difference
$(m_\Lambda(\xi)-1)$.  Each term satisfies the corresponding estimate in
\eqref{eq:cubic-third-kernel-reduction}.
\end{lemma}

\begin{proof}
Let $\varphi\in C_c^\infty(\R^3_\xi)$ be a Fourier-side output test
function and test $\widehat\Gamma_{a,\Lambda}^{(3)}(t)$ against
$\overline\varphi$.  Under Lemma~\ref{lem:cubic-direct-integral}, the
unsymmetrized three-noise kernel of this tested random variable is the
measurable field
\[
 (m,r)\longmapsto
 \overline{\varphi(m+r)}\chi^{\mathrm{res}}(m,r)
 F_m^{\Lambda,\epsilon_{\mathrm P}}(t)\otimes g_r^\Lambda(t),
\]
up to the fixed Fourier normalization and the bounded auxiliary output scalar.
The two equal-color slots are symmetrized by the fiberwise orthogonal
projection $\Sym_a$.  In particular,
\begin{equation}\label{eq:cubic-symmetrization-contraction}
 \|\Sym_a G\|^2
 =\tfrac12\|G\|^2+\tfrac12\Re\langle G,\tau_{13}G\rangle
 \le\|G\|^2.
\end{equation}
Thus the permutation cross term is controlled by the contraction property of
$\Sym_a$.  Color orthogonality enters the Wick product formula by excluding
contractions between different noises.

The third Wiener isometry, the direct-integral norm identities, and
\eqref{eq:cubic-symmetrization-contraction} give
\begin{align*}
 \E|\langle\Gamma_{a,\Lambda}^{(3)}(t),\varphi\rangle|^2
 \lesssim\iint
 &|\varphi(m+r)|^2|\chi^{\mathrm{res}}(m,r)|^2\\
 &\times q_{a^\perp,0,\Lambda}^{[\epsilon_{\mathrm P}]}(t,m)
 p_{a,\Lambda}(t,r)\dd m\dd r.
\end{align*}
Changing variables $\xi=m+r$ identifies the bracket multiplying
$|\varphi(\xi)|^2$ with the right side of
\eqref{eq:cubic-third-kernel-reduction}.  Since this holds for every test
function, it proves the spectral-density bound.

For time increments, subtract the two tested Hilbert kernels and telescope one
factor at a time; the contraction property of $\Sym_a$ and Cauchy--Schwarz
produce the same estimate with one covariance density replaced by its
increment density.  For cutoff differences, write the tested unsymmetrized
kernel as
\[
 o_\Lambda(m+r)\,
 \overline{\varphi(m+r)}\chi^{\mathrm{res}}(m,r)
 F_m^{\Lambda,\epsilon_{\mathrm P}}\otimes g_r^\Lambda,
 \qquad
 o_\Lambda(\xi):=m_\Lambda(\xi)^{\epsilon_{\mathrm O}},
\]
and use the decomposition
\begin{align*}
 o_\Lambda F^{\Lambda,\epsilon_{\mathrm P}}\otimes g^\Lambda-F\otimes g
 ={}&(o_\Lambda-1)F^{\Lambda,\epsilon_{\mathrm P}}\otimes g^\Lambda
 +(F^{\Lambda,\epsilon_{\mathrm P}}-F)\otimes g^\Lambda
 +F\otimes(g^\Lambda-g).
\end{align*}
The first-Picard output projection is included in
$F^{\Lambda,\epsilon_{\mathrm P}}$.  Applying
$\Sym_a$, its contraction property, and Cauchy--Schwarz treats all three terms
within the Hermitian Gaussian Hilbert space.  The
finite-cutoff compact supports justify all Fubini steps; the uniform spectral
majorants then permit passage to the limit by dominated convergence.
\end{proof}

\begin{proposition}[Third-chaos spectral bounds]\label{prop:cubic-third-spectral}
For every $\eps>0$,
\begin{equation}\label{eq:cubic-third-spectral}
  \sup_{\Lambda\ge1}\sup_{t\le T}
  q_{a,\Lambda}^{(3)}(t,\xi)
  \lesssim_{T,\eps}\la\xi\ra^{-3+2\beta_{\Gamma,a}+\eps}.
\end{equation}
For every $0<\theta<\frac12-\betasum$,
\begin{equation}\label{eq:cubic-third-spectral-increment}
  q_{a,\Lambda}^{(3),\Delta}(t,t';\xi)
  \lesssim_{T,\eps,\theta}
  |t-t'|^{2\theta}\la\xi\ra^{-3+2\beta_{\Gamma,a}+2\theta+\eps}.
\end{equation}
The cutoff-difference densities tend pointwise to zero and satisfy the weighted
tail condition in \eqref{eq:path-criterion-tail} for every target exponent
$s<-\beta_{\Gamma,a}$.
\end{proposition}

\begin{proof}
Insert \eqref{eq:first-picard-spectral-static} with $\nu=0$ and
\eqref{eq:cubic-linear-density} into
\eqref{eq:cubic-third-kernel-reduction}.  On a resonant shell
\[
  \la m\ra\sim\la\xi-m\ra\sim M,
  \qquad M\gtrsim\la\xi\ra,
\]
the contribution is bounded by
\[
  M^3\,M^{-4+2\betasum+\eps}\,M^{-2+2\beta_a}
  =M^{-3+2\beta_{\Gamma,a}+\eps}.
\]
The dyadic sum over $M\gtrsim\la\xi\ra$ proves
\eqref{eq:cubic-third-spectral}.  For increments, use
\eqref{eq:first-picard-spectral-increment} and
\eqref{eq:cubic-linear-density-increment}; both terms give the shell bound
\[
  |t-t'|^{2\theta}M^{-3+2\beta_{\Gamma,a}+2\theta+\eps}.
\]
The condition $\theta<1/2-\betasum$ gives
$-3+2\beta_{\Gamma,a}+2\theta<-2+2\beta_a<0$.
For a prescribed final loss $\eps>0$, choose
\[
 0<\eps_0<\min\{\eps,
 3-2\beta_{\Gamma,a}-2\theta\}
\]
in the first-Picard and linear increment bounds.  The dyadic sum with
$\eps_0$ converges, and replacing $\eps_0$ by the larger final exponent
$\eps$ gives \eqref{eq:cubic-third-spectral-increment}.

For cutoff differences, use the three-term telescope from the proof of
\cref{lem:cubic-third-kernel-reduction}.  The remaining external
stochastic-convolution tail starts at
$M\gtrsim\max\{\la\xi\ra,\Lambda\}$ and is dominated by the same negative shell
power.  The first-Picard kernel tail tends pointwise to zero by
\cref{lem:first-picard-tail-integral} and is dominated by
$M^{-3+2\beta_{\Gamma,a}+\eps}$.  The auxiliary final-output difference is supported on
$N:=\la\xi\ra\gtrsim\Lambda$ and is bounded directly by the full
$N^{-3+2\beta_{\Gamma,a}+\eps}$ density majorant.  Dominated convergence therefore gives
fixed-output convergence for every projection convention.

If $N\le\Lambda$, the explicit external majorant is
$\Lambda^{-3+2\beta_{\Gamma,a}+\eps}$, whereas for $N>\Lambda$ the full majorant is
$N^{-3+2\beta_{\Gamma,a}+\eps}$.  For the low-output part,
\[
 \Lambda^{-3+2\beta_{\Gamma,a}+\eps}
 \sum_{N\lesssim\Lambda}N^{2s+3+\delta}
 \lesssim
 \Lambda^{-3+2\beta_{\Gamma,a}+\eps
 +\max\{2s+3+\delta,0\}+}.
\]
Choose $\eps,\delta>0$ so that
\[
 \eps+\delta<-2(s+\beta_{\Gamma,a}).
\]
If $2s+3+\delta>0$, the displayed exponent equals
$2s+2\beta_{\Gamma,a}+\eps+\delta<0$.  If
$2s+3+\delta\le0$, it is at most
$-3+2\beta_{\Gamma,a}+\eps<0$.  The output-projection tail is contained
in the region $N\gtrsim\Lambda$ and is treated by the full density
majorant.

The first-Picard internal tail requires a small distinction when the target
regularity is very negative.  Let $q_{V,>\Lambda}(M)$ denote the shell
supremum of the first-Picard difference density.  On a resonant shell of size
$M$, convolution with the remaining linear density contributes
\[
 M^3\,M^{-2+2\beta_a}q_{V,>\Lambda}(M)
 =M^{1+2\beta_a}q_{V,>\Lambda}(M).
\]
Put $A:=2s+3+\delta$.  After interchanging the nonnegative dyadic sums,
\begin{align}
 &\sum_N N^{A}\sum_{M\gtrsim N}
 M^{1+2\beta_a}q_{V,>\Lambda}(M)\notag\\
 &\qquad\lesssim
 \sum_M M^{1+2\beta_a+A_+}q_{V,>\Lambda}(M),
 \qquad A_+:=\max\{A,0\}.
 \label{eq:cubic-third-internal-tail-rearrangement}
\end{align}
Indeed, $\sum_{N\lesssim M}N^A\lesssim M^A$ when $A>0$, while this inner
sum is uniformly bounded when $A<0$; the endpoint $A=0$ costs only a
logarithm, absorbed into an arbitrarily small auxiliary loss.  Choose
$s_V<1/2-\betasum$ so that
\begin{equation}\label{eq:cubic-third-tail-choice-sV}
 1+2\beta_a+A_+<2s_V+3+\delta_V.
\end{equation}
Such a choice is possible for every $s<-\beta_{\Gamma,a}$.  When $A>0$, this amounts,
up to auxiliary losses, to $s_V>s+1/2+\beta_a$, which is compatible with
$s_V<1/2-\betasum$ since $s<-\beta_{\Gamma,a}$.  When $A\le0$, one may take
any $s_V>\beta_a-1$ sufficiently below $1/2-\betasum$.  The right side of
\eqref{eq:cubic-third-internal-tail-rearrangement} then tends to zero by the
first-Picard tail criterion \eqref{eq:first-picard-tail-sum} at regularity
$s_V$.  This proves the cubic tail for the full range $s<-\beta_{\Gamma,a}$.

For increment tails, replace $A_+$ by $(A+2\theta)_+$ and choose $s_V$ and
$\theta>0$ so that
$1+2\beta_a+(A+2\theta)_+<2s_V+3+\delta_V$.  The strict inequality
$s<-\beta_{\Gamma,a}$ leaves the required margin.  The same first-Picard tail criterion
completes the increment estimate.
\end{proof}

By \cref{lem:spectral-to-local-paths,prop:cubic-third-spectral},
\begin{equation}\label{eq:cubic-third-path-conclusion}
  \Gamma_a^{(3)}\in
  C_T\cC_{\loc}^{-\beta_{\Gamma,a}-}
  \cap L_T^\infty B_{2,\infty,\loc}^{-\beta_{\Gamma,a}-},
\end{equation}
with convergence in every finite probability moment and almost surely along
dyadic cutoffs.

\subsection{The first-chaos contraction}

Insert the remaining stochastic convolution into
\eqref{eq:cubic-contraction-before-fubini}.  Write the random factor there as
the uncut field,
\[
  \widehat\Psi_{a^\perp}(\xi,s)
  =\mathfrak h_{a^\perp}(\xi)
  \int_0^sK_{a^\perp}(s-u,\xi)
  \mathcal W_{a^\perp}(\dd u,\dd\xi),
\]
because its cutoff factor $m_\Lambda(\xi)$ already appears once in
\eqref{eq:cubic-contracted-cutoff-multiplier}.
At finite cutoff, stochastic Fubini is an identity in the first Gaussian
Hilbert space.  Hence
\begin{equation}\label{eq:cubic-first-integrated-representation}
  \widehat\Gamma_{a,\Lambda}^{(1)}(\xi,t)
  =\int_0^t\mathcal K_{a,\Lambda}(\xi;t,u)
  \mathcal W_{a^\perp}(\dd u,\dd\xi),
\end{equation}
where
\begin{equation}\label{eq:cubic-integrated-kernel}
\begin{aligned}
  \mathcal K_{a,\Lambda}(\xi;t,u)
  :=c_{\mathrm F}\,\mathfrak h_{a^\perp}(\xi)
  \int_u^t\int_{\R^3}
  &\mathfrak m_\Lambda^{[1]}(\xi,r)
  \chi^{\mathrm{res}}(\xi-r,r)
  K_{a^\perp}(t-s,\xi-r)\\
  &\times\sigma_a(r;s,t)
  K_{a^\perp}(s-u,\xi)\dd r\dd s.
\end{aligned}
\end{equation}
The limit is taken at the level of this integrated kernel.

For $0\le s\le t$ and $\omega_a(r)>0$, direct trigonometric integration gives
\begin{equation}\label{eq:cubic-covariance-formula}
\begin{aligned}
  \sigma_a(r;s,t)
  =\mathfrak q_a(r)\Bigg[{}&\frac{s}{2\omega_a(r)^2}
  \cos\bigl((t-s)\omega_a(r)\bigr)\\
  &+\frac{\sin((t-s)\omega_a(r))}{4\omega_a(r)^3}
   -\frac{\sin((t+s)\omega_a(r))}{4\omega_a(r)^3}\Bigg].
\end{aligned}
\end{equation}
For the wave zero mode the left side is used directly.  On every high shell
$|r|\sim M\ge2$, write
\begin{equation}\label{eq:cubic-covariance-split}
  \sigma_a=\sigma_{a,0}+R_a,
  \qquad
  \sigma_{a,0}(r;s,t)
  =\mathfrak q_a(r)\frac{s}{2\omega_a(r)^2}
  \cos\bigl((t-s)\omega_a(r)\bigr).
\end{equation}
Uniformly for both channels,
\begin{equation}\label{eq:cubic-covariance-shell-bounds}
  |\sigma_{a,0}(r;s,t)|\lesssim_T M^{-2+2\beta_a},
  \qquad
  |R_a(r;s,t)|\lesssim_T M^{-3+2\beta_a}.
\end{equation}

Put
\begin{equation}\label{eq:cubic-modulation}
  \Omega_a^{\bm\eps}(\xi,r)
  :=\eps_0\omega_{a^\perp}(\xi)
    +\eps_1\omega_{a^\perp}(\xi-r)
    +\eps_2\omega_a(r),
  \qquad \bm\eps\in\{\pm1\}^3,
\end{equation}
and let
\begin{equation}\label{eq:cubic-resonant-shell}
  \mathcal R_M(\xi)
  :=\left\{r\in\R^3:
  \la r\ra\sim M,\ \la\xi-r\ra\sim M\right\}.
\end{equation}
The resonant support implies $M\gtrsim\la\xi\ra$.  The phase estimate
has only the following three regimes.
\begin{center}
\small
\begin{tabularx}{\textwidth}{@{}p{0.18\textwidth}p{0.21\textwidth}YY@{}}
\toprule
frequency regime & phase control & sublevel or volume bound & contribution to
$\int_{\mathcal R_M(\xi)}\mathfrak m(\Omega)^{1-\delta}\dd r$ \\
\midrule
$M\sim1$
& direct kernel estimate
& bounded set
& $O_T(1)$ \\
$M\sim\la\xi\ra\gg1$
& fix $\rho=|r|$ and differentiate in $\lambda=|\xi-r|$;
  $|\partial_\lambda\Omega|\gtrsim1$
& $|\{|\Omega|\le L\}|\lesssim M^2\min\{M,L\}$
& $\lesssim M^{2+\delta+}$ \\
$M\ge C_0\la\xi\ra$
& the two high phases have distinct asymptotic speeds;
  $|\Omega|\gtrsim M$
& $|\mathcal R_M(\xi)|\lesssim M^3$
& $\lesssim M^{2+\delta}$ \\
\bottomrule
\end{tabularx}
\end{center}
The estimates hold for every sign vector
$\bm\eps\in\{\pm1\}^3$.  In the low-output regime, equal signs give a sum of
the two high phases and opposite signs give their distinct-speed difference.

For later use, write the sign branches as follows.  Put $N=\la\xi\ra$.
On a high resonant shell, inserting
\eqref{eq:cubic-covariance-split} and expanding the two sine propagators and
the leading cosine gives
\begin{equation}\label{eq:cubic-leading-sign-branch}
 \mathcal K_{a,M}^{\mathrm{lead}}(\xi;t,u)
 =\sum_{\bm\eps\in\{\pm1\}^3}
 \int_{\mathcal R_M(\xi)}
 A_{a,M}^{\bm\eps}(\xi,r)
 E_{a}^{\bm\eps}(\xi,r;t,u)
 \left(\int_u^t s e^{\ii s\Omega_a^{\bm\eps}(\xi,r)}\dd s\right)\dd r.
\end{equation}
Here $|E_a^{\bm\eps}|=1$, its $t$-dependence is a product of endpoint
phases with frequencies $O(M)$, and
\begin{equation}\label{eq:cubic-leading-amplitude-bound}
 |A_{a,M}^{\bm\eps}(\xi,r)|
 \lesssim_{\parvec}N^{-1+\beta_{a^\perp}}M^{-3+2\beta_a}
 \one_{\mathcal R_M(\xi)}(r).
\end{equation}
For the inhomogeneous output block, the factor
$N^{-1+\beta_{a^\perp}}$ in \eqref{eq:cubic-leading-amplitude-bound} is interpreted
as a bounded $T$- and profile-dependent constant.  Formula \eqref{eq:cubic-leading-sign-branch} follows
directly from
\[
 K_{a^\perp}(t-s,\xi-r),\qquad K_{a^\perp}(s-u,\xi),\qquad
 \cos((t-s)\omega_a(r)),
\]
and shows explicitly that the only $s$-modulations are those in
\eqref{eq:cubic-modulation}.  For later use, put
\[
 J_{u,t}(\Omega):=\int_u^t s e^{\ii s\Omega}\dd s.
\]
For $\Omega\ne0$, integration gives
\begin{equation}\label{eq:cubic-inner-time-endpoints}
 J_{u,t}(\Omega)
 =\frac{t e^{\ii t\Omega}-u e^{\ii u\Omega}}{\ii\Omega}
  +\frac{e^{\ii u\Omega}-e^{\ii t\Omega}}{(\ii\Omega)^2},
\end{equation}
while $J_{u,t}(0)=(t^2-u^2)/2$.  Hence, uniformly for
$0\le u\le t\le T$,
\begin{equation}\label{eq:cubic-inner-time-bound}
 |J_{u,t}(\Omega)|\lesssim_T\mathfrak m(\Omega),
 \qquad \mathfrak m(\Omega):=\min\{1,|\Omega|^{-1}\}.
\end{equation}
Thus both endpoints of the inner Volterra integral are included in the
kernel estimate.  The same representation also separates changes of endpoint
phases from changes of the upper integration limit.

\begin{lemma}[Continuous one-frequency phase layer]\label{lem:cubic-one-frequency-layer}
Let $0\le\delta<1$ and $\eps>0$.  For every high dyadic shell $M\ge2$ and every
$\bm\eps\in\{\pm1\}^3$,
\begin{equation}\label{eq:cubic-layer-estimate}
  \int_{\mathcal R_M(\xi)}
  \mathfrak m\!\left(\Omega_a^{\bm\eps}(\xi,r)\right)^{1-\delta}
  \dd r
  \lesssim_{\delta,\eps}M^{2+\delta+\eps},
  \qquad M\gtrsim\la\xi\ra.
\end{equation}
The implicit constant is uniform in $a\in\mathfrak C$ and locally uniform in the parameter vector.
\end{lemma}

\begin{proof}
We split into low-output high--high and balanced-output regions.

Suppose first that $M\ge C_0\la\xi\ra$ for a sufficiently large fixed $C_0$.
The output phase is $O(C_0^{-1}M)$, up to uniformly bounded low-frequency terms.
The two high phases have asymptotic speeds $\speedW$ and $\speedK$.  If their signs agree,
their sum has size comparable to $M$.  If their signs are opposite, then
\[
  \left|
  \omega_{a^\perp}(\xi-r)-\omega_a(r)
  \right|\gtrsim M,
\]
because $||\xi-r|-|r||\le|\xi|$, the massive correction is $O_{\parvec}(M^{-1})$, and $\speedgap>0$.
After choosing $C_0=C_0(\parvec)$ large enough and absorbing finitely many unit shells,
\begin{equation}\label{eq:cubic-low-output-gap}
  |\Omega_a^{\bm\eps}(\xi,r)|\gtrsim M
  \qquad\text{on }\mathcal R_M(\xi).
\end{equation}
Since $|\mathcal R_M(\xi)|\lesssim M^3$, this gives
$M^3M^{-(1-\delta)}=M^{2+\delta}$.

It remains to consider $M\sim\la\xi\ra$.  High shells then satisfy
$d:=|\xi|\sim M$.  Use the two-center variables
\[
  \rho=|r|,
  \qquad
  \lambda=|\xi-r|.
\]
After integrating the azimuthal angle, the volume form is
\begin{equation}\label{eq:cubic-two-center-volume}
  \dd r=2\pi\frac{\rho\lambda}{d}\dd\rho\dd\lambda.
\end{equation}
On the balanced region the density is $O(M)$.  For fixed $\rho$,
\[
  \Omega_a^{\bm\eps}
  =\eps_0\omega_{a^\perp}(d)
   +\eps_1\omega_{a^\perp}(\lambda)
   +\eps_2\omega_a(\rho),
\]
and
\[
  |\partial_\lambda\Omega_a^{\bm\eps}|
  =|\omega_{a^\perp}'(\lambda)|\gtrsim1
\]
on high balanced shells, for either choice of the dispersion $\omega_{a^\perp}$.
Consequently the $\lambda$-length of a sublevel set
$|\Omega_a^{\bm\eps}|\le L$ is $O(\min\{M,L\})$.  Multiplying by the
$\rho$-interval length $O(M)$ and by the density $O(M)$ gives
\begin{equation}\label{eq:cubic-sublevel-volume}
  \left|\left\{r\in\mathcal R_M(\xi):
  |\Omega_a^{\bm\eps}(\xi,r)|\le L\right\}\right|
  \lesssim M^2\min\{M,L\}.
\end{equation}
A dyadic modulation decomposition for $1\le L\lesssim M$ now yields
\[
  \int_{\mathcal R_M(\xi)}
  \mathfrak m(\Omega_a^{\bm\eps})^{1-\delta}\dd r
  \lesssim
  M^2+M^2\sum_{1\le L\lesssim M}L^\delta
  \lesssim_{\delta,\eps} M^{2+\delta+\eps}.
\]
For $\delta>0$ the $M^\eps$ loss may be removed; at the endpoint
$\delta=0$ it absorbs the logarithmic number of dyadic modulation layers.
This proves the lemma.
\end{proof}

The estimate uses the speed gap only in the low-output high--high region.  In the
balanced region it uses a single nondegenerate radial derivative.  This is the
one-frequency analogue of the three-propagator phase geometry in
\cref{app:first-picard}.

Let $\mathcal K_{a,M}$ denote the part of
\eqref{eq:cubic-integrated-kernel} in which
$r\in\mathcal R_M(\xi)$, with the finite cutoff removed.  Unit shells are
estimated directly.  Put $N=\la\xi\ra$.

\begin{lemma}[Single-shell contraction kernel]\label{lem:cubic-single-shell-kernel}
For every $\eps>0$,
\begin{equation}\label{eq:cubic-single-shell-kernel}
  \sup_{0\le u\le t\le T}
  |\mathcal K_{a,M}(\xi;t,u)|
  \lesssim_{T,\eps}N^{-1+\beta_{a^\perp}}M^{-1+2\beta_a+\eps},
  \qquad M\gtrsim N.
\end{equation}
\end{lemma}

\begin{proof}
For the covariance remainder in \eqref{eq:cubic-covariance-split}, estimate
absolutely.  On the shell,
\[
  |K_{a^\perp}(t-s,\xi-r)|\lesssim M^{-1},
  \qquad
  |R_a(r;s,t)|\lesssim_T M^{-3+2\beta_a},
\]
and
\[
  |\mathfrak h_{a^\perp}(\xi)K_{a^\perp}(s-u,\xi)|
  \lesssim_T N^{-1+\beta_{a^\perp}}.
\]
The $r$-volume is $O(M^3)$, so this contribution is
$O_T(N^{-1+\beta_{a^\perp}}M^{-1+2\beta_a})$.

For the leading covariance, expand the two sine propagators and the cosine in
\eqref{eq:cubic-covariance-split} into sign branches.  The $s$-dependent phase
is one of the modulations in \eqref{eq:cubic-modulation}, up to an
overall sign.  The static amplitude is
\[
  N^{-1+\beta_{a^\perp}}M^{-1}M^{-2+2\beta_a}
  =N^{-1+\beta_{a^\perp}}M^{-3+2\beta_a}.
\]
Moreover, \eqref{eq:cubic-inner-time-bound} gives
\[
  \left|\int_u^t s e^{\ii s\Omega}\dd s\right|
  \lesssim_T\mathfrak m(\Omega).
\]
Applying \cref{lem:cubic-one-frequency-layer} with $\delta=0$ gives,
for every $\eps>0$,
\[
  |\mathcal K_{a,M}^{\mathrm{lead}}(\xi;t,u)|
  \lesssim_{T,\eps}
  N^{-1+\beta_{a^\perp}}M^{-3+2\beta_a}M^{2+\eps}
  =N^{-1+\beta_{a^\perp}}M^{-1+2\beta_a+\eps}.
\]
This proves \eqref{eq:cubic-single-shell-kernel}.
\end{proof}

\begin{corollary}[Summed integrated kernel]\label{cor:cubic-summed-kernel}
For every $0<\eps<1-2\beta_a$, the dyadic sum converges and
\begin{equation}\label{eq:cubic-summed-kernel}
  \sup_{0\le u\le t\le T}
  |\mathcal K_a(\xi;t,u)|
  \lesssim_{T,\eps}\la\xi\ra^{-2+\beta_{\Gamma,a}+\eps}.
\end{equation}
The internal tail above a dyadic scale $L$ satisfies
\begin{equation}\label{eq:cubic-kernel-tail}
  \sup_{u,t}|\mathcal K_{a,>L}(\xi;t,u)|
  \lesssim_{T,\eps}
  N^{-1+\beta_{a^\perp}}\max\{N,L\}^{-1+2\beta_a+\eps}.
\end{equation}
\end{corollary}

\begin{proof}
Sum \eqref{eq:cubic-single-shell-kernel} over dyadic
$M\gtrsim N$.  The tail estimate follows by starting the same sum at
$M\gtrsim\max\{N,L\}$.
\end{proof}

\begin{lemma}[Two-sided time increments of the integrated kernel]
\label{lem:cubic-kernel-increments}
Let $0\le t<t'\le T$, put $h=t'-t$, and let
$0<\theta<1/2$.  For every $\eps>0$ and every resonant shell
$M\gtrsim N:=\la\xi\ra$,
\begin{align}
 \sup_{0\le u\le t}
 |\mathcal K_{a,M}(\xi;t',u)-\mathcal K_{a,M}(\xi;t,u)|
 &\lesssim_{T,\eps,\theta}
 h^\theta N^{-1+\beta_{a^\perp}}M^{-1+2\beta_a+\theta+\eps},
 \label{eq:cubic-shell-kernel-increment}\\
 \sup_{t<u\le t'}|\mathcal K_{a,M}(\xi;t',u)|
 &\lesssim_{T,\eps,\theta}
 h^\theta N^{-1+\beta_{a^\perp}}M^{-1+2\beta_a+\theta+\eps}.
 \label{eq:cubic-shell-kernel-boundary}
\end{align}
If, in addition,
\begin{equation}\label{eq:cubic-increment-summability-condition}
 2\beta_a+\theta+\eps<1,
\end{equation}
then summation of the internal shells gives
\begin{align}
 \sup_{0\le u\le t}
 |\mathcal K_a(\xi;t',u)-\mathcal K_a(\xi;t,u)|
 &\lesssim_{T,\eps,\theta}
 h^\theta N^{-2+\beta_{\Gamma,a}+\theta+\eps},
 \label{eq:cubic-full-kernel-increment}\\
 \sup_{t<u\le t'}|\mathcal K_a(\xi;t',u)|
 &\lesssim_{T,\eps,\theta}
 h^\theta N^{-2+\beta_{\Gamma,a}+\theta+\eps}.
 \label{eq:cubic-full-kernel-boundary}
\end{align}
The same four estimates hold uniformly with any of the cutoff factors in
\eqref{eq:cubic-contracted-cutoff-multiplier} inserted.
\end{lemma}

\begin{proof}
Consider first the covariance remainder $R_a$ in
\eqref{eq:cubic-covariance-split}.  On the common $s$-interval, the elementary
propagator estimate
\[
 |K_{a^\perp}(\tau+h,\zeta)-K_{a^\perp}(\tau,\zeta)|
 \lesssim_{T,\theta}h^\theta\la\zeta\ra^{-1+\theta}
\]
and \eqref{eq:cubic-covariance-formula} give
\begin{equation}\label{eq:cubic-covariance-remainder-increment}
 |R_a(r;s,t')-R_a(r;s,t)|
 \lesssim_{T,\theta}h^\theta M^{-3+2\beta_a+\theta}.
\end{equation}
Indeed, each $t$-dependent sine changes by at most
$h^\theta M^\theta$, while its denominator is $M^3$.  Combining either the
increment of the outer propagator with $|R_a|\lesssim M^{-3+2\beta_a}$, or the static
outer propagator with
\eqref{eq:cubic-covariance-remainder-increment}, and then using the factor
$N^{-1+\beta_{a^\perp}}$ from
$\mathfrak h_{a^\perp}(\xi)K_{a^\perp}(s-u,\xi)$ and the shell volume $O(M^3)$, gives
$h^\theta N^{-1+\beta_{a^\perp}}M^{-1+2\beta_a+\theta}$.  On the boundary interval of length at most
$h$, the absolute estimate is $hN^{-1+\beta_{a^\perp}}M^{-1+2\beta_a}$, which is bounded by
$C_T h^\theta N^{-1+\beta_{a^\perp}}M^{-1+2\beta_a+\theta}$.

For the leading covariance use the representation
\eqref{eq:cubic-leading-sign-branch}.  If
$E(t)=E_a^{\bm\eps}(\xi,r;t,u)$ and
$J(t)=J_{u,t}(\Omega_a^{\bm\eps}(\xi,r))$, then on the common interval
$0\le u\le t<t'$ the difference is written as
\begin{equation}\label{eq:cubic-common-time-telescope}
 E(t')J(t')-E(t)J(t)
 =[E(t')-E(t)]J(t)+E(t')J_{t,t'}(\Omega_a^{\bm\eps}).
\end{equation}
The first term changes only the endpoint phase, while the second is the new
upper-time interval.  The endpoint phase satisfies
\begin{equation}\label{eq:cubic-endpoint-phase-increment}
 |E_a^{\bm\eps}(\xi,r;t',u)
  -E_a^{\bm\eps}(\xi,r;t,u)|
 \lesssim_{T,\theta}h^\theta M^\theta.
\end{equation}
Multiplying by
$|\int_u^t s e^{\ii s\Omega}\dd s|
\lesssim_T\mathfrak m(\Omega)$, the amplitude bound
\eqref{eq:cubic-leading-amplitude-bound}, and
\cref{lem:cubic-one-frequency-layer} with $\delta=0$ yields
\[
 h^\theta N^{-1+\beta_{a^\perp}}M^{-3+2\beta_a+\theta}
 \int_{\mathcal R_M(\xi)}\mathfrak m(\Omega)\dd r
 \lesssim_{T,\eps,\theta}
 h^\theta N^{-1+\beta_{a^\perp}}M^{-1+2\beta_a+\theta+\eps}.
\]
The change of the upper endpoint of the oscillatory integral is
\[
 \int_t^{t'}s e^{\ii s\Omega}\dd s.
\]
Interpolation between its length bound $O_T(h)$ and the one-integration-by-
parts bound $O_T(\mathfrak m(\Omega))$ gives
\begin{equation}\label{eq:cubic-boundary-oscillatory-interpolation}
 \left|\int_t^{t'}s e^{\ii s\Omega}\dd s\right|
 \lesssim_T h^\theta
 \mathfrak m(\Omega)^{1-\theta}.
\end{equation}
Now \cref{lem:cubic-one-frequency-layer} with $\delta=\theta$ gives
\[
 h^\theta N^{-1+\beta_{a^\perp}}M^{-3+2\beta_a}M^{2+\theta+\eps}
 =h^\theta N^{-1+\beta_{a^\perp}}M^{-1+2\beta_a+\theta+\eps}.
\]
This proves \eqref{eq:cubic-shell-kernel-increment}.  If $t<u\le t'$, the
entire leading integral is over $[u,t']$ and the same interpolation
\eqref{eq:cubic-boundary-oscillatory-interpolation} proves
\eqref{eq:cubic-shell-kernel-boundary}.  Unit shells are controlled directly
by continuity of the wave and Klein--Gordon kernels and the finite length of
the time interval.

Under \eqref{eq:cubic-increment-summability-condition}, the dyadic sum over
$M\gtrsim N$ converges and gives
\eqref{eq:cubic-full-kernel-increment} and
\eqref{eq:cubic-full-kernel-boundary}.  The cutoff factors are independent of
$t,u,s$, are uniformly bounded, and do not alter either the phase layer or the
shell support, so the same estimates hold uniformly with them inserted.
\end{proof}

\subsection{Spectral estimates and completion of the construction}

By the first Wiener isometry applied to
\eqref{eq:cubic-first-integrated-representation}, the covariance spectral
density of $\Gamma_{a,\Lambda}^{(1)}$ is
\begin{equation}\label{eq:cubic-first-density-formula}
  q_{a,\Lambda}^{(1)}(t,\xi)
  =c_{\mathrm F}\int_0^t
  |\mathcal K_{a,\Lambda}(\xi;t,u)|^2\dd u.
\end{equation}

\begin{lemma}[Cutoff tails for the integrated first-chaos kernel]
\label{lem:cubic-first-cutoff-tail}
Fix the projection flags
$(\epsilon_{\mathrm P},\epsilon_{\mathrm O})\in\{0,1\}^2$.  Let
$\mathcal K_{a,>\Lambda}$ denote any term in the difference between the
cutoff kernel in \eqref{eq:cubic-integrated-kernel} and its limiting kernel,
after the finite-product telescope of
\cref{lem:finite-product-cutoff}.  Put
\[
 N:=\langle\xi\rangle,
 \qquad
 L_\Lambda(N):=\max\{N,c\Lambda\},
\]
where $c>0$ depends only on the cutoff plateau and the resonant support.  For
every sufficiently small $\varepsilon>0$,
\begin{equation}\label{eq:cubic-first-kernel-cutoff-tail}
 \sup_{0\le u\le t\le T}
 |\mathcal K_{a,>\Lambda}(\xi;t,u)|
 \lesssim_{T,\varepsilon}
 N^{-1+\beta_{a^\perp}}
 L_\Lambda(N)^{-1+2\beta_a+\varepsilon}.
\end{equation}
Let $0<\theta<1/2$ and assume
$2\beta_a+\theta+\varepsilon<1$.  Then
\begin{align}
 \sup_{0\le u\le t}
 |\mathcal K_{a,>\Lambda}(\xi;t',u)
  -\mathcal K_{a,>\Lambda}(\xi;t,u)|
 &\lesssim_{T,\varepsilon,\theta}
 |t'-t|^\theta
 N^{-1+\beta_{a^\perp}}
 L_\Lambda(N)^{-1+2\beta_a+\theta+\varepsilon},
 \label{eq:cubic-first-kernel-increment-cutoff-tail}\\
 \sup_{t<u\le t'}
 |\mathcal K_{a,>\Lambda}(\xi;t',u)|
 &\lesssim_{T,\varepsilon,\theta}
 |t'-t|^\theta
 N^{-1+\beta_{a^\perp}}
 L_\Lambda(N)^{-1+2\beta_a+\theta+\varepsilon}.
 \label{eq:cubic-first-kernel-boundary-cutoff-tail}
\end{align}
Consequently the static and increment covariance-density differences admit
majorants
\begin{align}
 q_{a,>\Lambda}^{(1)}(t,\xi)
 &\lesssim_{T,\varepsilon}
 N^{-2+2\beta_{a^\perp}}
 L_\Lambda(N)^{-2+4\beta_a+\varepsilon},
 \label{eq:cubic-first-density-cutoff-tail}\\
 \frac{q_{a,>\Lambda}^{(1),\Delta}(t,t';\xi)}
 {|t-t'|^{2\theta}}
 &\lesssim_{T,\varepsilon,\theta}
 N^{-2+2\beta_{a^\perp}}
 L_\Lambda(N)^{-2+4\beta_a+2\theta+\varepsilon}.
 \label{eq:cubic-first-increment-density-cutoff-tail}
\end{align}
The bounds are uniform over the four projection conventions.
\end{lemma}

\begin{proof}
For a cutoff factor carried by $r$, $-r$, or $\xi-r$, the resonant support
forces the internal shell to satisfy $M\gtrsim\Lambda$.  Summing
\eqref{eq:cubic-single-shell-kernel} from
$M\gtrsim L_\Lambda(N)$ proves
\eqref{eq:cubic-first-kernel-cutoff-tail}.  Summing
\eqref{eq:cubic-shell-kernel-increment} and
\eqref{eq:cubic-shell-kernel-boundary} from the same lower limit proves
\eqref{eq:cubic-first-kernel-increment-cutoff-tail}--
\eqref{eq:cubic-first-kernel-boundary-cutoff-tail}.  If the active high
factor is an output-frequency factor, then $N\gtrsim\Lambda$; the same
formulas follow from the full kernel bounds because
$L_\Lambda(N)\simeq N$.  The finite-product reduction treats repeated legs
and all projection flags uniformly.  Since the telescope contains only a
fixed number of terms, the square of the full kernel difference is bounded by
a fixed multiple of the sum of their squares.  Thus the estimates just proved
also hold for the complete cutoff difference.

Apply the first Wiener isometry.  If $t'<t$, the increment density is
\begin{align*}
 q_{a,>\Lambda}^{(1),\Delta}(t,t';\xi)
 ={}&c_{\mathrm F}\int_0^{t'}
 |\mathcal K_{a,>\Lambda}(\xi;t,u)
  -\mathcal K_{a,>\Lambda}(\xi;t',u)|^2\dd u\\
 &+c_{\mathrm F}\int_{t'}^t
 |\mathcal K_{a,>\Lambda}(\xi;t,u)|^2\dd u.
\end{align*}
Use \eqref{eq:cubic-first-kernel-increment-cutoff-tail} in the first integral
and \eqref{eq:cubic-first-kernel-boundary-cutoff-tail} in the second.  The
static density is obtained from the same isometry without subtraction.
Integration in $u$ contributes only a factor depending on $T$ and gives
\eqref{eq:cubic-first-density-cutoff-tail}--
\eqref{eq:cubic-first-increment-density-cutoff-tail}.  The case $t<t'$ is
identical after interchanging the two times.
\end{proof}

\begin{corollary}[Normalized first-chaos kernel majorant]
\label{cor:cubic-first-master-majorant}
Let $N=\langle\xi\rangle$, $L\ge1$, $0\le\theta<1/2$, and choose
$\varepsilon>0$ so that
$2\beta_a+\theta+\varepsilon<1$.  Define
\begin{equation}\label{eq:cubic-master-majorant}
 \mathfrak M_{a,\theta,\varepsilon}(N,L)
 :=N^{-1+\beta_{a^\perp}}
   \max\{N,L\}^{-1+2\beta_a+\theta+\varepsilon}.
\end{equation}
Then the following estimates hold uniformly over all four projection flags:
\begin{align}
 \sup_{u,t}|\mathcal K_{a,>L}(\xi;t,u)|
 &\lesssim \mathfrak M_{a,0,\varepsilon}(N,L),
 \label{eq:cubic-master-static}\\
 \sup_{0\le u\le t<t'}
 \frac{|\mathcal K_{a,>L}(\xi;t',u)
       -\mathcal K_{a,>L}(\xi;t,u)|}{|t'-t|^\theta}
 &\lesssim \mathfrak M_{a,\theta,\varepsilon}(N,L),
 \label{eq:cubic-master-common}\\
 \sup_{t<u\le t'}
 \frac{|\mathcal K_{a,>L}(\xi;t',u)|}{|t'-t|^\theta}
 &\lesssim \mathfrak M_{a,\theta,\varepsilon}(N,L).
 \label{eq:cubic-master-boundary}
\end{align}
Here $\mathcal K_{a,>L}$ denotes the sum of internal resonant shells
$M\gtrsim\max\{N,L\}$.  Any cutoff-difference term is bounded by the same
right-hand sides with $L=c\Lambda$; if the active cutoff leg is an output
frequency, it forces $N\gtrsim\Lambda$ and the same formula follows with
$\max\{N,L\}\simeq N$.  Consequently its static and normalized increment
covariance densities are bounded by
\begin{equation}\label{eq:cubic-master-density}
 \mathfrak M_{a,0,\varepsilon}(N,L)^2,
 \qquad
 \mathfrak M_{a,\theta,\varepsilon}(N,L)^2,
\end{equation}
respectively.
\end{corollary}

\begin{proof}
The static estimate is \eqref{eq:cubic-kernel-tail}; the common-time and
boundary estimates are obtained by summing
\eqref{eq:cubic-shell-kernel-increment} and
\eqref{eq:cubic-shell-kernel-boundary}.  The cutoff statement is
\Cref{lem:cubic-first-cutoff-tail}, after applying the finite-product
reduction to the active legs in the preceding cutoff table.  The first Wiener
isometry gives \eqref{eq:cubic-master-density}.  If $N$ or an internal
frequency lies in the inhomogeneous unit block, all multipliers are supported
in a fixed compact set, the propagator kernels are continuous at zero, and
the same bounds hold after increasing the constant; moreover, that block lies in the cutoff plateau for all sufficiently large $\Lambda$.
\end{proof}

\begin{proposition}[First-chaos spectral bounds]\label{prop:cubic-first-spectral}
For every $\eps>0$,
\begin{equation}\label{eq:cubic-first-spectral}
  \sup_{\Lambda\ge1}\sup_{t\le T}
  q_{a,\Lambda}^{(1)}(t,\xi)
  \lesssim_{T,\eps}\la\xi\ra^{-4+2\beta_{\Gamma,a}+\eps}.
\end{equation}
For every $0<\theta<1/2$,
\begin{equation}\label{eq:cubic-first-spectral-increment}
  q_{a,\Lambda}^{(1),\Delta}(t,t';\xi)
  \lesssim_{T,\eps,\theta}
  |t-t'|^{2\theta}
  \la\xi\ra^{-4+2\beta_{\Gamma,a}+2\theta+\eps}.
\end{equation}
For every $s<1/2-\beta_{\Gamma,a}$, one may choose
$0<\theta<1/2$ and $\delta>0$ such that the cutoff-difference densities
satisfy the tail criterion in \eqref{eq:path-criterion-tail}, uniformly over
the four projection conventions.
\end{proposition}

\begin{proof}
For the static estimate, choose
$0<\varepsilon_0<\min\{\varepsilon/2,1-2\beta_a\}$ in
\eqref{eq:cubic-summed-kernel}.  The Wiener isometry gives
\[
 q_{a,\Lambda}^{(1)}(t,\xi)
 \lesssim_T N^{-4+2\beta_{\Gamma,a}+2\varepsilon_0}
 \le N^{-4+2\beta_{\Gamma,a}+\varepsilon},
\]
which proves \eqref{eq:cubic-first-spectral}.  For the increment estimate,
choose
\[
 0<\varepsilon_0<\min\{\varepsilon/2,1-2\beta_a-\theta\}.
\]
Assume $t'<t$.  Splitting the first Wiener integral into the common interval
$[0,t']$ and the boundary interval $(t',t]$, and using
\eqref{eq:cubic-full-kernel-increment}--
\eqref{eq:cubic-full-kernel-boundary}, yields
\begin{align*}
 &\int_0^{t'}|\mathcal K_a(\xi;t,u)
  -\mathcal K_a(\xi;t',u)|^2\dd u\\
 &\qquad\lesssim_T |t-t'|^{2\theta}
 N^{-4+2\beta_{\Gamma,a}+2\theta+2\varepsilon_0},\\
 &\int_{t'}^t|\mathcal K_a(\xi;t,u)|^2\dd u
 \lesssim_T |t-t'|N^{-4+2\beta_{\Gamma,a}+2\varepsilon_0}\\
 &\qquad\lesssim_T |t-t'|^{2\theta}
 N^{-4+2\beta_{\Gamma,a}+2\theta+2\varepsilon_0}.
\end{align*}
This proves \eqref{eq:cubic-first-spectral-increment}.

It remains to verify the cutoff tail required by
\eqref{eq:path-criterion-tail}.  Let
$s<1/2-\beta_{\Gamma,a}$.  Choose
$\theta,\varepsilon,\delta>0$ such that
\begin{equation}\label{eq:cubic-first-tail-parameter-choice}
 2\theta+\varepsilon+\delta
 <1-2s-2\beta_{\Gamma,a},
 \qquad
 2\beta_a+\theta+\varepsilon<1.
\end{equation}
By \cref{lem:cubic-first-cutoff-tail}, the combined static and normalized
increment majorant is bounded by
\begin{equation}\label{eq:cubic-first-combined-tail-majorant}
 q_{a,>\Lambda}^{(1),(\theta)}(\xi)
 \lesssim
 N^{-2+2\beta_{a^\perp}}
 \max\{N,\Lambda\}^{-2+4\beta_a+2\theta+\varepsilon},
\end{equation}
after adjusting the plateau constant.  The output-frequency flags
satisfy the same formula because they force $N\gtrsim\Lambda$.

For $N>\Lambda$, \eqref{eq:cubic-first-combined-tail-majorant} becomes
\[
 q_{a,>\Lambda}^{(1),(\theta)}(\xi)
 \lesssim N^{-4+2\beta_{\Gamma,a}+2\theta+\varepsilon},
\]
and therefore
\[
 \sum_{N>\Lambda}N^{2s+3+\delta}
 \sup_{|\xi|\sim N}q_{a,>\Lambda}^{(1),(\theta)}(\xi)
 \lesssim
 \sum_{N>\Lambda}
 N^{2s-1+2\beta_{\Gamma,a}+2\theta+\varepsilon+\delta},
\]
which decays by \eqref{eq:cubic-first-tail-parameter-choice}.
For $N\le\Lambda$,
\begin{align}
 &\sum_{N\le\Lambda}N^{2s+3+\delta}
  \sup_{|\xi|\sim N}q_{a,>\Lambda}^{(1),(\theta)}(\xi)\notag\\
 &\qquad\lesssim
 \Lambda^{-2+4\beta_a+2\theta+\varepsilon}
 \sum_{N\le\Lambda}N^{2s+1+2\beta_{a^\perp}+\delta}\notag\\
 &\qquad\lesssim
 \Lambda^{-2+4\beta_a+2\theta+\varepsilon
 +\max\{2s+1+2\beta_{a^\perp}+\delta,0\}+}.
 \label{eq:cubic-first-low-output-tail-sum}
\end{align}
If the expression inside the positive part is positive, the exponent in
\eqref{eq:cubic-first-low-output-tail-sum} is
\[
 2s-1+2\beta_{\Gamma,a}+2\theta+\varepsilon+\delta+,
\]
which is negative by \eqref{eq:cubic-first-tail-parameter-choice}.  If it is
nonpositive, the exponent is
$-2+4\beta_a+2\theta+\varepsilon+$, which is negative because
$\beta_a<1/8$, $\theta<1/2$, and the auxiliary loss is chosen small.  Thus
\eqref{eq:path-criterion-tail} holds with polynomial decay.

The cutoff multipliers in
\eqref{eq:cubic-contracted-cutoff-multiplier} form a finite product with the
repeated $r$-leg listed twice.  Hence
\cref{lem:finite-product-cutoff} gives the same tail for all four projection
conventions and compares two non-nested admissible cutoffs through the sum of
their tails.  This completes the proof.
\end{proof}

Applying \cref{lem:spectral-to-local-paths} with
$\rho=1/2-\beta_{\Gamma,a}$ gives
\begin{equation}\label{eq:cubic-first-path-conclusion}
  \Gamma_a^{(1)}\in
  C_T\cC_{\loc}^{\frac12-\beta_{\Gamma,a}-}
  \cap L_T^\infty B_{2,\infty,\loc}^{\frac12-\beta_{\Gamma,a}-},
\end{equation}
with cutoff-independent convergence in every finite probability moment.  The
polynomial weighted tail in the proof of
\cref{prop:cubic-first-spectral}, followed by Markov and Borel--Cantelli as in
\cref{lem:spectral-to-local-paths}, gives almost-sure convergence along the
fixed-profile dyadic sequence.

\begin{proof}[Proof of \cref{thm:cubic-fullspace}]
The finite decomposition, the identity of the surviving contraction,
and the remaining random frequency are given by
\cref{lem:cubic-finite-wick-formula}.  The centered third-chaos convergence and
regularity are \cref{prop:cubic-third-spectral} and
\eqref{eq:cubic-third-path-conclusion}.  The first-chaos contraction is first
defined at finite cutoff by \eqref{eq:cubic-first-integrated-representation};
\cref{prop:cubic-first-spectral} and
\eqref{eq:cubic-first-path-conclusion} give its limiting construction.  The cutoff tail estimates apply to every family satisfying
Definition~\ref{def:admissible-cutoff}.  In particular, inserting one of the
auxiliary output projections in \eqref{eq:cubic-cutoff-variants} does not
change the limit.

It remains to verify the source topology.  Fix an exponent quadruple in
\eqref{eq:parameter-window}.  Choose $\kappa>0$ and $\delta>0$ so small that
\[
  \beta_{\Gamma,a}+\kappa+\delta<1-s_2,
  \qquad
  \beta_{\Gamma,a}+\kappa<1-\sigma.
\]
Both inequalities follow from the strict parameter window.  Locally,
\[
  B_{2,\infty}^{-\beta_{\Gamma,a}-\kappa}
  \hookrightarrow H^{-\beta_{\Gamma,a}-\kappa-\delta}
  \hookrightarrow H^{s_2-1},
  \qquad
  B_{2,\infty}^{-\beta_{\Gamma,a}-\kappa}
  \hookrightarrow B_{2,\infty}^{\sigma-1}.
\]
Since $\Gamma_a\in
L_T^\infty B_{2,\infty,\loc}^{-\beta_{\Gamma,a}-\kappa}$, integration over
$[0,T]$ proves \eqref{eq:cubic-source-regularity}.  This completes the proof.
\end{proof}

\end{document}